\documentclass{article}[11pt]
\usepackage{latexsym}
\usepackage{amsmath}
\usepackage{amsfonts}
\usepackage{amssymb}
\usepackage{mathrsfs}
\usepackage{graphicx}	
\usepackage{amsthm}
\usepackage[top=2.5cm, bottom=2.5cm, left=2.5cm, right=2.5cm]{geometry}
\usepackage{hyperref}
\usepackage{slashed,mathtools,color}
\usepackage{upgreek}
\usepackage{bbm}
\usepackage{comment}
\usepackage{xcolor}
\usepackage{halloweenmath}

\newcommand{\beq}{\begin{equation}}
\newcommand{\eeq}{\end{equation}}
\newcommand{\Pow}{P}
\def\oe{\upomega}

\newcommand{\mydim}{\mathfrak{D}}
\makeatletter
\newcommand{\blowupexp}{A_*}

\makeatother

\theoremstyle{plain}
\newtheorem{theorem}{Theorem}[section]
\newtheorem{proposition}[theorem]{Proposition}
\newtheorem{lemma}[theorem]{Lemma}

\theoremstyle{definition}
\newtheorem{definition}[theorem]{Definition}

\newtheorem{remark}[theorem]{Remark}

\numberwithin{equation}{section}

\allowdisplaybreaks[4] 

\swapnumbers
\begin{document}

\title{Stable Big Bang Formation for Einstein's Equations: 
The Complete Sub-Critical Regime}

\author{Grigorios Fournodavlos,\footnote{Princeton University, Mathematics Department, Fine Hall, Washington Road, Princeton, NJ 08544-1000, USA, gfournodavlos@uoc.gr.}\;\; Igor Rodnianski,\footnote{Princeton University, Mathematics Department, Fine Hall, Washington Road, Princeton, NJ 08544-1000, USA, irod@math.princeton.edu.}\;\; Jared Speck\footnote{Vanderbilt University, Mathematics Department, 1326 Stevenson Center,
Nashville, TN 37240, USA, jared.speck@vanderbilt.edu}}
%\address{}
%\email{}

\date{}

\maketitle

\begin{abstract}
For $(t,x) \in (0,\infty)\times\mathbb{T}^{\mydim}$, the generalized Kasner solutions
(which we refer to as Kasner solutions for short) 
are a family of explicit solutions to various Einstein-matter systems
that, exceptional cases aside, start out smooth but then
develop a Big Bang singularity as $t \downarrow 0$, i.e., a singularity
along an entire spacelike hypersurface, where various curvature scalars
blow up monotonically.
The family is parameterized by the Kasner exponents
$\widetilde{q}_1,\cdots,\widetilde{q}_{\mydim} \in \mathbb{R}$, 
which satisfy two algebraic constraints.
There are heuristics in the mathematical physics literature, going back more than 50 years,
suggesting that the Big Bang formation should be dynamically stable, that is, 
stable under perturbations of the Kasner initial data, given say at $\lbrace t = 1 \rbrace$,
as long as the exponents are ``sub-critical'' in the following sense: 
$\underset{\substack{I,J,B=1,\cdots, \mydim \\ I < J}}{\max}
\{\widetilde{q}_I+\widetilde{q}_J-\widetilde{q}_B\}<1$.
Previous works have rigorously shown the dynamic stability of the Kasner Big Bang singularity
under stronger assumptions: 
1) the Einstein-scalar field system with
$\mydim = 3$ and $\widetilde{q}_1 \approx \widetilde{q}_2 \approx \widetilde{q}_3 \approx 1/3$,
which corresponds to the stability of the Friedmann--Lema\^{\i}tre--Robertson--Walker solution's Big Bang
or 2) the Einstein-vacuum equations for $\mydim \geq 38$ with 
$\underset{I=1,\cdots,\mydim}{\max} |\widetilde{q}_I| < 1/6$.
In this paper, we prove that the Kasner singularity is dynamically stable
for \emph{all} sub-critical Kasner exponents,
thereby justifying the heuristics in the literature
in the full regime where stable monotonic-type curvature-blowup is expected.
We treat in detail the
$1+\mydim$-dimensional Einstein-scalar field system for all 
$\mydim \geq 3$
and
the $1+\mydim$-dimensional Einstein-vacuum equations for $\mydim \geq 10$;
both of these systems feature non-empty sets of sub-critical Kasner solutions.
Moreover, for the Einstein-vacuum equations in $1+3$ dimensions, 
where instabilities are in general expected, 
we prove that all singular Kasner solutions have dynamically stable
Big Bangs under polarized $U(1)$-symmetric perturbations of their initial data.
Our results hold for open sets of initial data in Sobolev spaces without symmetry,
apart from our work on polarized $U(1)$-symmetric solutions.

Our proof relies on a new formulation of Einstein's equations:
we use a constant-mean-curvature foliation,
and the unknowns are the 
scalar field,
the lapse, the 
components of the spatial connection and second fundamental form
relative to a Fermi--Walker transported spatial orthonormal frame, 
and the components of the orthonormal frame vectors
with respect to a transported spatial coordinate system.
In this formulation, the PDE evolution system for the 
structure coefficients of the orthonormal frame approximately diagonalizes
in a way that sharply reveals the significance of the Kasner exponent sub-criticality condition
for the dynamic stability of the flow: the condition leads to the time-integrability of many
terms in the equations, at least at the low derivative levels.
At the high derivative levels, the solutions that we study can be much more singular
with respect to $t$, and to handle this difficulty, 
we use $t$-weighted high order energies,
and we control nonlinear error terms by exploiting monotonicity induced by the $t$-weights 
and interpolating between the singularity-strength of the solution's
low order and high order derivatives.
Finally, we note that our formulation of Einstein's equations
highlights the quantities that might generate instabilities outside of
the sub-critical regime.

\bigskip


\bigskip


\end{abstract}

\parskip = 0 pt

\tableofcontents

\section{Introduction}
Our main results in this paper are proofs of stable Big Bang formation 
(i.e., curvature-blowup along an entire spacelike hypersurface)
for cosmological\footnote{By ``cosmological solutions,'' we mean ones with compact spatial topology.} 
solutions to the Cauchy problem for the Einstein-vacuum and
Einstein-scalar field systems.
All of our results hold for open sets of solutions
without symmetry, except for our results on polarized $U(1)$-symmetric solutions
to the Einstein-vacuum equations in $1+3$ dimensions.
We assume that initial data are given on the manifold 
$\Sigma_1 = \mathbb{T}^{\mydim} := [-\pi,\pi]^{\mydim}$
(with the endpoints identified),
where $\mydim \geq 3$ is the number of spatial dimensions. 
Later on, we provide a precise description of which kinds of data our results
apply to and how the value of $\mydim$ is tied to the data.
As we will explain, \emph{our results are sharp in the sense that they rigorously
confirm the dynamic stability of the singularity formation in the entire regime where
heuristics in the literature have suggested it might occur}. In particular,
our results significantly extend the prior results \cite{RodSp1,RodSp2,RodSp3},
which yield stable Big Bang formation for open sets of solutions without symmetry.
We refer readers to Theorem~\ref{thm:rough}
for a rough version of our main results and to Theorems~\ref{thm:precise} and \ref{thm:precise.U1}
for precise statements. 

The sharpness of our results is possible because we
have developed a new analytic framework for constant mean curvature (CMC)
foliations in which we study the \emph{components} of various spatial tensors
relative to an orthonormal ``spatial frame,'' obtained by Fermi--Walker transport,
as well as the connection coefficients 
and structure coefficients of the frame. We refer readers to Sect.\,\ref{sec:setup}
for the precise details behind the gauge and the corresponding formulation
of Einstein's equations that we use to derive estimates.
We also refer to Sect.\,\ref{subsec:pf.overview} for an overview of the proof.
Our framework allows us to precisely and efficiently
detect the terms in the equations that are integrable-in-time up to the singularity, which
is key to understanding the stability of the blowup.
Our framework also pinpoints the terms 
in the equations that might generate instabilities in other regimes;
see Remark~\ref{R:IDENTIFYOBSTRUCTION}.

\subsection{The Cauchy problem for the Einstein-scalar field equations}
\label{SS:CAUCHYPROBLEM}
\subsubsection{The Einstein-scalar field equations}
\label{SSS:EINSTEINSFEQUATIONS}
Relative to arbitrary coordinates, the Einstein-scalar field equations can be expressed as:
\begin{subequations}
\begin{align}\label{EE}
{\bf Ric}_{\mu\nu}
& = \partial_\mu \psi \partial_\nu \psi,
	\\
\square_{{\bf g}} \psi
& = 0. \label{SF}
\end{align}
\end{subequations}
In \eqref{EE}--\eqref{SF} and throughout,
${\bf Ric}$ is the Ricci curvature of the spacetime metric ${\bf g}$ (which has signature $(-,+,+,\cdots,+)$),
$\square_{{\bf g}} := ({\bf g}^{-1})^{\alpha \beta} {\bf D}_{\alpha} {\bf D}_{\beta}$
is the covariant wave operator of ${\bf g}$, ${\bf D}$ 
is the Levi-Civita connection of ${\bf g}$,
and $\psi$ is the scalar field.
Note that in the special case $\psi \equiv 0$,
the system is equivalent to the Einstein-vacuum equations. 

\subsubsection{The initial value problem formulation and the initial data}
\label{SSS:IVPFORMULATION}
It is well-known that the system \eqref{EE}--\eqref{SF}
has an initial value problem formulation in which 
sufficiently regular initial data give rise to unique solutions.
An initial data set for the system is defined to be
$(\Sigma_1,\mathring{g},\mathring{k},\mathring{\psi},\mathring{\phi})$,
where $\mathring{g}$ is a Riemannian metric on the manifold $\Sigma_1$
(in this paper, we assume that $\Sigma_1 := \mathbb{T}^{\mydim}$), 
$\mathring{k}$ is a symmetric two-tensor, and $\mathring{\psi},\mathring{\phi}$ are a pair of scalar functions.
We sometimes refer to such initial data as ``geometric initial data'' in order to distinguish them from
initial data for the reduced equations of Proposition~\ref{P:redeq};
initial data for those equations -- which are the main PDEs we study in this paper -- 
also involve gauge-dependent quantities,  
including initial data for an orthonormal spatial frame (which we discuss later on).
It is well-known that admissible geometric initial data must satisfy the Hamiltonian
and momentum constraint equations,
which are respectively:
\begin{subequations}
\begin{align}
\mathring{R}-|\mathring{k}|^2+(\mathrm{tr} \mathring{k})^2
\label{eq:hamconst}
& = 
\mathring{\phi}^2
+
|\mathring{\nabla} \mathring{\psi}|^2,
	\\
\label{eq:momconst} \mathring{\mathrm{div}} \mathring{k} - \mathring{\nabla} \text{tr} \mathring{k}
& = - \mathring{\phi} \mathring{\nabla} \mathring{\psi},
\end{align}
\end{subequations}
where $\mathring{\nabla}$ is the Levi-Civita connection of $\mathring{g}$
(with respect to which all covariant spatial operators along $\Sigma_1$ are defined)
and
$\mathring{R}$ is the scalar curvature of $\mathring{g}$.

\subsubsection{Globally hyperbolic developments}
\label{SSS:GLOBALLYHYPERBOLICDEVELOPMENTS}
A \emph{globally hyperbolic development} of the geometric initial data, 
which can be thought of as a solution to the initial value problem,
is a triplet 
$(\mathcal{M},{\bf g},\psi)$ 
and an embedding $i: \Sigma_1 \to \mathcal{M}$ such that:
\begin{itemize}
	\item $\mathcal{M}$ is a $1+\mydim$-dimensional spacetime manifold.
	\item ${\bf g}$ is a Lorentzian metric on $\mathcal{M}$ and $\psi$
	is a scalar function on $\mathcal{M}$ that together solve the equations \eqref{EE}--\eqref{SF}.
	\item $i(\Sigma_1)$ is a Cauchy hypersurface\footnote{In this paper, $i(\Sigma_1)$ will be a hypersurface
	of constant time with respect to a CMC time function $t$. To simplify the exposition, we will often slightly abuse notation 
	by suppressing the embedding and identifying $\Sigma_1$ with $i(\Sigma_1) = \lbrace t = 1 \rbrace \subset \mathcal{M}$.}  
		in $(\mathcal{M},{\bf g})$.
	\item The pullbacks (under $i$) of the first and second fundamental forms of the image surface $i(\Sigma_1)$
		(see Sect.\,\ref{sec:setup} for our sign conventions for the second fundamental form) are equal to 
		$\mathring{g},\mathring{k}$ respectively, and the pullbacks of the initial values of the scalar field $\psi$
		and its derivative with respect to the future unit normal to $i(\Sigma_1)$ are equal to $\mathring{\psi},\mathring{\phi}$
		respectively.
\end{itemize}

The fundamental work \cite{CBG} of Choquet-Bruhat and Geroch
shows that for sufficiently regular geometric initial data verifying
the constraints, there is  
a unique (up to diffeomorphism) maximal (classical) globally hyperbolic development
$(\mathcal{M}_{\textnormal{Max}},{\bf g}_{\textnormal{Max}},\psi_{\textnormal{Max}})$ 
(and a corresponding embedding $i_{\textnormal{Max}}: \Sigma_1 \to \mathcal{M}_{\textnormal{Max}}$ that we will suppress),
which we refer to as the ``maximal development'' for short.
Roughly, the solution furnished by \cite{CBG} is the largest possible classical solution to \eqref{EE}--\eqref{SF}
that is uniquely determined by the initial data; we refer the reader to \cite{Rin4} for detailed discussion of the maximal development. Although it is of philosophical importance to know that the 
maximal development exists and is unique, the results of \cite{CBG} do not reveal much about its structure.
Our goal in this article is to fully understand its structure for open sets of solutions
that exhibit curvature-blowup.

\subsection{Connections with Hawking's singularity theorem}
\label{SS:HAWKINGSTHEOREM}
Hawking's celebrated ``singularity theorem'' \cite{Hawk,HawkPen}
shows that for cosmological solutions,\footnote{See also \cite{HawkPen,Pen} for discussion
of the related -- but distinct -- ``singularity theorem'' by Penrose, which for non-compact Cauchy hypersurfaces
shows that the presence of a trapped surface in the initial data leads to geodesically incomplete solutions.} 
there exist large sets of regular initial data for the Einstein equations
such that the corresponding solutions eventually break down in the sense that the
spacetime is causally geodesically incomplete. 
The results apply to any matter model verifying the strong energy condition, 
including the scalar field model and the vacuum.
In particular, the version of Hawking's theorem stated as \cite[Theorem~9.5.1]{rW}
guarantees that under assumptions 
satisfied by the initial data featured in our main theorems,
the solution is such that all past-directed timelike geodesics are incomplete.
Although these works are of immense philosophical importance in general relativity and have had a great impact on the direction of the field, they are limited in that their proofs are by contradiction and
do not provide any information about the nature of the breakdown, aside from geodesic incompleteness. 
Through various telling examples, it is known that different kinds of breakdown are possible. A particularly sinister
scenario is found in the Taub--NUT and Kerr spacetimes, where the breakdown is not caused by any singularity in the metric (including its higher derivatives), but rather is caused by the development of a \emph{Cauchy horizon}, across which the solution can be smoothly extended in more than one way, signifying the failure of determinism past the maximal globally hyperbolic development of the data. 
A crucial point is that for the near-Kasner solutions covered by our main results, this sinister scenario does not occur; the geodesic incompleteness is caused by curvature-blowup at the boundary of the
maximal globally hyperbolic development.

\subsection{Remarks on Strong Cosmic Censorship}
\label{SS:SCC}
The Strong Cosmic Censorship\footnote{See \cite{Pen2} for the original formulation and \cite{Christ,Chrus} for more modern versions.} conjecture suggests that, ``generically,'' the maximal globally hyperbolic development
of the data is inextendible, roughly due to the formation of some kind of singularity.\footnote{One even hopes to rule out the possibility of
continuing the solution weakly past the boundary of the maximal development 
since, at least from the PDE point of view, in principle, it might be possible to make sense of weak solutions
in a neighborhood of a classical singularity; 
see the discussion on pg.\,13 of \cite{Christ4}.} 
Confirming some version of the Strong Cosmic Censorship conjecture, 
at least in a perturbative regime around explicit solutions, 
turns out to be extremely difficult, due to the strength of the nonlinearities in the system
and the possibility of complicated dynamics near singularities. 
For the near-Kasner solutions covered by our main results, their 
curvature-blowup shows that a $C^2$-extension
of the solution past the Big Bang is impossible.

It is important to appreciate that regularity considerations
are of crucial importance when defining what is meant by ``the Strong Cosmic Censorship conjecture;''
thanks to the remarkable work \cite{DL} on the $C^0$-stability of the Kerr Cauchy horizon,\footnote{The initial data 
considered in \cite[Theorem~1]{DL} are posed on a spacelike hypersurface in the black hole interior.
A full justification that these data are induced by open sets of black hole solutions that are settling
down to a Kerr black hole relies on forthcoming works by various authors. In particular, the justification relies 
on a quantitative version of the dynamic stability of the exterior region of Kerr, and
there have been a series of works that seem to be building towards a definitive proof 
its stability. We refer to \cite[Section~1.3]{DL} for further discussion. \label{FN:CONDITIONALKERR}} 
we now know that the $C^0$ formulation of the 
Strong Cosmic Censorship conjecture is not generically true. More precisely, in \cite{DL},
it was shown that for an open set of near-Kerr solutions, the metric can be continuously extended beyond the
Cauchy horizon. However, it is conceivable that these metrics generically do not enjoy any additional
regularity and in particular that they cannot even be extended past the Cauchy horizon as weak solutions
to Einstein's equations. It therefore remains possible that a revised version of the
Strong Cosmic Censorship conjecture is true, in which 
``generically, geodesic incompleteness is tied to breakdown at the boundary of the maximal development,''
where ``breakdown'' is defined to be
any loss of regularity that is sufficiently strong to prevent one from extending the solution
as a weak solution to Einstein's equations. There are works in spherical symmetry that support this possibility,
notably \cite{mD,Daf,LO}, where \cite{LO} has the compelling feature that it is a large data result.
More precisely, \cite{LO} proves that for an open and dense set of 
two-ended asymptotically flat initial data for the Einstein--Maxwell--(real)--scalar--field system in spherical
symmetry, the maximal development is $C^2$-inextendible.

\subsection{Beyond Hawking's singularity theorem}
\label{SS:HAWKINGSTHEOREM}
In the wake of Hawking's singularity theorem, 
there have been many works devoted towards understanding 
the precise cause of the geodesic incompleteness.
An interesting type of breakdown that has received extensive attention 
-- rigorous and otherwise --
over the past half-century
is the ``Kasner-like scenario,'' which concerns solutions 
whose metrics ${\bf g}$ are asymptotic to:
\begin{align} \label{E:LIMITINGKASNERLIKEMETRIC}
{\bf g}_{\textnormal{Limiting}}(t,x)
& = 
-dt \otimes dt
+
\sum_{I=1}^{\mydim}
t^{2 q_I(x)}\uptheta^I(x) \otimes \uptheta^I(x),\qquad \uptheta^I(x) =\uptheta^I_a(x)dx^a,
\end{align}
as $t \downarrow 0$ (i.e., towards the singularity). The form of ${\bf g}_{\textnormal{Limiting}}(t,x)$ is inspired by the Kasner solutions themselves, 
which we discuss in Sect.\,\ref{subsec:models}. 
It is important to note that the metrics ${\bf g}_{\textnormal{Limiting}}$ 
are not generally solutions to Einstein's equations.
However, in the special case 
that $\uptheta^I=dx^I$ and the $\lbrace q_I \rbrace_{I=1,\cdots,\mydim}$ are constants satisfying two 
algebraic constraints (see \eqref{sumpi}),
${\bf g}_{\textnormal{Limiting}}$ \emph{is} a solution, known as a Kasner solution
in the vacuum case and a ``generalized Kasner solution'' in the presence of matter
(for short, we sometimes refer to all such solutions simply as ``Kasner solutions'').
The Kasner solutions are spatially homogeneous and,
exceptional cases aside, exhibit monotonic Big Bang formation
(i.e., monotonic blowup of the spacetime Kretschmann scalar 
${\bf Riem}^{\alpha\mu\beta\nu}{\bf Riem}_{\alpha\mu\beta\nu}$
along a spacelike hypersurface) as $t \downarrow 0$,
as do the metrics ${\bf g}_{\textnormal{Limiting}}$. We stress that Big Bang formation
is consistent with the assertions of a $C^2$-inextendibility formulation
of the Strong Cosmic Censorship conjecture.
We also note that in the remainder of the paper, we often denote the (constant)
Kasner exponents by $\lbrace \widetilde{q}_I \rbrace_{I=1,\cdots,\mydim}$,
where the tilde emphasizes that they are associated to a ``background Kasner solution.''

A standout question, then, is: besides the explicit Kasner solutions
(which we describe in Sect.\,\ref{subsec:models}), 
are there \emph{any} other cosmological solutions
to Einstein's equations -- in particular ones with spatial dependence -- 
that are asymptotic to a metric of the form ${\bf g}_{\textnormal{Limiting}}$
and thus exhibit monotonic-type Big Bang formation? 
In an influential paper \cite{KL}, the authors gave heuristic arguments suggesting that in $1+3$ dimensions, 
cosmological solutions that are asymptotic to a metric of the form
${\bf g}_{\textnormal{Limiting}}$ should be \emph{non-generic} 
(in particular, unstable).
More precisely, their heuristics led them to deduce
``the absence of a real singularity in the general solution''
and that
``the general case of an arbitrary distribution of matter and gravitational field 
does not lead to the appearance of a singularity.'' Roughly,
the work \cite{KL} predicted that singularity formation along a spacelike hypersurface
is unstable because
singularities should form only for solutions
such that one of the gravitational degrees of freedom is \underline{inactive};
see \cite[Equation~(3.20)]{KL}.
In the subsequent work \cite{BKL}, 
a revised picture of singularities in 
cosmological solutions to Einstein's equations was proposed. 
Specifically, in \cite{BKL}, the authors used heuristic arguments to predict
that ``there exists a general solution which exhibits a physical singularity with respect to time,''
even for the Einstein-vacuum equations in $1+3$ dimensions.
In a more modern language, \cite{BKL} proposed that there are families of cosmological solutions 
that \textbf{i)} contain all gravitational degrees of freedom 
(e.g., $4$ functional degrees of freedom for the Einstein-vacuum equations in $1+3$ dimensions)
and \textbf{ii)} exhibit Big Bang formation along a spacelike hypersurface. Moreover, the authors argued
that ``generically'' (the meaning of ``generic'' was not rigorously defined),
solutions that exhibit Big Bang formation ``should'' -- unlike the Kasner solutions --
be highly oscillatory in time as the singularity is approached.
The alleged oscillatory behavior is sometimes referred to as the ``Mixmaster scenario,''
where the terminology goes back to Misner's important paper \cite{Mis}
on oscillatory solutions with Bianchi IX symmetry.
The oscillations are one of several features that have been conjectured to hold
for ``most'' cosmological
$1+3$-dimensional Einstein-vacuum solutions that have incomplete timelike geodesics. 
This picture has come to be known, somewhat imprecisely, as the ``BKL conjecture.''

\begin{remark}[Open problem]
	\label{R:OPENPROBLEM}
	In light of the above discussion, we would like to highlight the following open problem,
	brought to our attention by Mihalis Dafermos:
	construct \underline{any} open set of initial data without symmetry 
	for the Einstein-vacuum equations in $1+3$ dimensions such that the maximal development exhibits 
	a spacelike singularity.
\end{remark}

We highlight that, whatever one's interpretation of the BKL conjecture,
Dafermos--Luk's aforementioned work \cite{DL} shows 
(assuming, as mentioned in Footnote~\ref{FN:CONDITIONALKERR}, the stability of the exterior region of Kerr)
that some of its basic qualitative assertions 
fail to hold for $1+3$-dimensional Einstein-vacuum solutions corresponding to near-Kerr-black-hole initial data.
In particular, in \cite{DL}, the authors announced 
(see \cite[Theorem 2]{DL} and the discussion in \cite[Section~1.3.3]{DL}) 
their forthcoming follow-up work, which implies 
that Cauchy horizons develop for an \underline{open} set of 
near-Kerr asymptotically flat initial data for the Einstein-vacuum equations.
Since the Cauchy horizons are null, this shows that, at least in the setting of
black hole interiors, it is not generically true that incompleteness 
is tied to some kind of blowup along a spacelike hypersurface. 
Note that solutions arising from near-Kerr asymptotically flat initial data
have qualitatively distinct topologies compared to 
the solutions that are usually discussed in literature in connection with the BKL conjecture; 
that literature is centered on cosmological spacetimes, which have compact spatial topology.
Nonetheless, in view of Remark~\ref{R:OPENPROBLEM}, 
it is important to appreciate that as of the present,
the results of Dafermos--Luk provide 
the only known open sets of solutions to the Einstein-vacuum equations 
in $1+3$ dimensions without symmetry for which the precise nature 
of geodesic incompleteness has been understood.

Despite the points made above, within the class of \emph{spatially homogeneous} cosmological solutions,
there are rigorous\footnote{For a discussion of numerical work on singularity formation in Einstein's equations, see \cite{Ber} and the references therein.} 
results showing that some solutions exhibit oscillatory behavior towards a spacelike singularity.
Notable among these is Ringstr\"{o}m's paper \cite{Rin}, in which he showed that 
in the vacuum case and for various fluid matter models, 
solutions with Bianchi IX symmetry generically 
exhibit oscillatory\footnote{In the special case of a stiff fluid, 
which is also discussed in the next paragraph, 
Ringstr\"{o}m proved that the dynamics are monotonic towards the singularity.}
behavior towards their singularity. 
See also \cite{Beg,Bre,Dut,HU,LHWG} for related works.

In the wake of the works \cite{KL,BKL}, 
there were further heuristic works suggesting that
if the Einstein equations are coupled to a scalar field \cite{BK}
or a stiff fluid\footnote{A stiff fluid is such that the speed of sound is equal to the speed of light. It can be viewed as an analog of the scalar field model that allows for non-zero vorticity.} \cite{Bar},
or if one considers the Einstein-vacuum equations in $1+\mydim$ dimensions with $\mydim \geq 10$ \cite{DHS},
then the oscillations can be silenced, leading back to the Kasner-like scenario. More precisely, there ``should'' exist
open sets of initial data whose solutions exhibit 
\emph{monotonic} Big Bang formation. The essence of these works is that the following
``sub-criticality condition'' (which we sometimes refer to as a ``stability condition'')
for the Kasner exponents $\lbrace q_I(x) \rbrace_{I=1,\cdots,\mydim}$
might be sufficient to ensure the existence of sets of solutions 
-- containing all the gravitational degrees of freedom -- that have Kasner-like Big Bang 
singularities:
\begin{align}\label{Kasner.stability.cond.heur}
\mathop{\max_{I,J,B=1,\cdots,\mydim}}_{I < J} 
\{q_I(x)+q_J(x)-q_B(x)\}<1.
\end{align} 
The condition \eqref{Kasner.stability.cond.heur} 
is central\footnote{More precisely, our main results rely on the assumption that the background solution satisfies
\eqref{Kasner.stability.cond}, which is 
\eqref{Kasner.stability.cond.heur} in the special case of a generalized Kasner solution.} 
to our main results, and we will discuss its implications in detail below.

We highlight that, due to the constraints \eqref{sumpi},
for the Einstein-vacuum equations in $1 + \mydim$ dimensions,
the condition \eqref{Kasner.stability.cond.heur} can be satisfied only if $\mydim \geq 10$;
this algebraic fact was first observed in \cite{DHS}.
The papers described in the previous paragraph, 
which were in favor of the dynamic stability of the Big Bang,
were based on heuristic justifications of the claim
that the condition \eqref{Kasner.stability.cond.heur} ``should'' lead to
\emph{asymptotically velocity term dominated} (AVTD) behavior in perturbed solutions.
Roughly, AVTD behavior for a solution is such that in Einstein's equations,
the spatial derivative terms become negligible compared to the time derivative terms as the singularity is
approached. Put differently, AVTD behavior is such that the solution becomes asymptotic to
a truncated version of Einstein's equations in which all spatial derivative terms are thrown away.
Since the truncated equations are ODEs at each fixed spatial point $x$, one could say
that AVTD solutions are asymptotically $x$-parameterized ODE solutions.
As we will later explain, the condition 
\eqref{Kasner.stability.cond.heur} (see also \eqref{Kasner.stability.cond})
suggests that for perturbations of the Kasner solution, the Ricci tensor of the perturbed spatial metric,
which we denote by $Ric$, should satisfy, for some $\upsigma > 0$, 
$|Ric| \lesssim t^{-2 + \upsigma}$ as $t \downarrow 0$. It turns out that, when available, 
this bound leads to the time-integrability
of various terms in Einstein's equations. In turn,
the time-integrability is key to proving the AVTD nature of perturbations of Kasner solutions 
and for controlling the dynamics up to the singularity. In Sect.\,\ref{subsec:genKasner}, 
we provide a more detailed explanation of the significance of the bound $|Ric| \lesssim t^{-2 + \upsigma}$ 
for the proofs of our main results.

Clearly, any rigorous justification of the above circle of ideas requires, at a minimum, 
the construction of a gauge relative to which the AVTD behavior can be exhibited.
In the present paper, we introduce a general gauge $+$ framework 
for proving stable singularity formation for ``Kasner-like'' solutions with spatial dependence
and for proving the AVTD behavior. As in previous works on stable Big Bang formation \cite{RodSp1,RodSp2,RodSp3,Sp},
we rely on constant mean curvature foliations in which the level sets of the time function
$t$ have mean curvature\footnote{The mean curvature of a constant-time slice $\Sigma_t$ 
is defined to be the trace of its second fundamental form divided by the number of spatial dimensions $\mydim$.} 
equal to $-\frac{1}{\mydim t}$,
and we control the lapse
$n := [-({\bf g}^{-1})^{\alpha \beta} \partial_{\alpha} t \partial_{\beta} t]^{-1/2}$
via elliptic estimates.
The main new idea in our paper lies in our approach to controlling the dynamic ``spatial\footnote{By spatial tensorfields, we simply mean ones that are tangent to the level sets of the CMC time function $t$.} tensorfields'':
we construct a gauge for Einstein's equations in which the main dynamical unknowns 
are the \emph{components} of various spatial tensorfields relative to an orthonormal ``spatial frame''
$\lbrace e_I \rbrace_{I=1,\cdots,\mydim}$, obtained by Fermi--Walker transport 
(see equation \eqref{frame.prop} and Remark~\ref{R:FERMIWALKER}),
as well as the connection coefficients
$\upgamma_{IJB}: = {\bf g}({\bf D}_{e_I}e_J,e_B)$.
One of our key observations is: as a consequence of the
special structure of Einstein's equations and the Fermi--Walker transport equation \eqref{frame.prop},
\textbf{the frame is one degree more differentiable than naive estimates suggest}. More precisely,
the transport equation \eqref{dt.omega}, which is an equivalent formulation of \eqref{frame.prop},
 suggests that the frame vectorfield components
$\lbrace e_I^i \rbrace_{I,i=1,\cdots,\mydim}$ are only as regular as the second fundamental
form $k$ of $\Sigma_t$. However, our gauge allows us to prove that in fact, the connection
coefficients $\lbrace \upgamma_{IJB} \rbrace_{I,J,B = 1, \cdots, \mydim}$ 
of the frame enjoy the same Sobolev regularity as the components
$\lbrace k_{IJ} \rbrace_{I,J=1,\cdots,\mydim}$
where $k_{IJ} := k(e_I,e_J) = k_{cd} e_I^c e_J^d$;
this signifies a gain of one derivative for the connection coefficients.
Roughly, the gain in regularity stems from the fact that
$\lbrace \upgamma_{IJB} \rbrace_{I,J,B=1,\cdots,\mydim}$
and
$\lbrace k_{IJ} \rbrace_{I,J=1,\cdots,\mydim}$
satisfy a system of wave-like equations 
(coupled to $n$ and the scalar field)
that allow us to propagate the Sobolev regularity of their initial data.
We refer readers to Lemma~\ref{L:TOPORDERDIFFERENTIALENERGYIDENTITYFORGAMMAANDK}
for a differential version of the basic energy identity
that we use to obtain the desired regularity for $\upgamma$ and $k$.

A second key observation 
is that the \emph{structure coefficients} of the frame, namely\footnote{The identity ${\bf g}([e_I,e_J],e_B) = \upgamma_{IJB} + \upgamma_{JBI}$ is a simple consequence of the torsion-free property of the connection ${\bf D}$.}
${\bf g}([e_I,e_J],e_B) = \upgamma_{IJB} + \upgamma_{JBI}$,
satisfy an evolution equation system
(see Proposition~\ref{P:KEYEVOLUTIONSTRUCTURECOEFFICIENTS})
that is \textbf{diagonal} up to quadratic error terms,
and such that \emph{the strength of the main linear terms in the equations is controlled by
the Kasner stability condition \eqref{Kasner.stability.cond}.}
More precisely, we have
$\partial_t (\upgamma_{IJB} + \upgamma_{JBI})
=
-
\frac{(\widetilde{q}_{\underline{I}}+\widetilde{q}_{\underline{J}}-\widetilde{q}_{\underline{B}})}{t}
(\upgamma_{\underline{I}\underline{J}\underline{B}}+\upgamma_{\underline{J}\underline{B}\underline{I}})
+ \cdots
$,
where here and throughout the paper, we do not sum over repeated underlined indices.
From this equation and
the condition \eqref{Kasner.stability.cond}, we are able to prove
that there exists a constant $q < 1$ such that: 
\begin{align}\label{STRUCTURECOEFFICIENTKasner.stability.cond}
\mathop{\max_{I,J,B=1,\cdots,\mydim}}_{I < J}
t^q|\upgamma_{IJB} + \upgamma_{JBI}|
\lesssim data,
&&
(t,x) \in (0,1] \times \mathbb{T}^{\mydim},
\end{align}
where ``$data$'' denotes a small term that is controlled by the size of the
perturbation of the initial data from the Kasner data 
(in particular, ``$data$'' vanishes for Kasner solutions).
The estimate \eqref{STRUCTURECOEFFICIENTKasner.stability.cond} leads to the time-integrability of many
terms in the evolution equations,
allows us to rigorously justify the aforementioned spatial Ricci curvature bound\footnote{We also need to adequately control the first derivatives of the structure coefficients to obtain the desired bound for the Ricci curvature.} 
$|Ric| \lesssim t^{-2 + \upsigma}$,
and allows us to prove the AVTD behavior of perturbations of
any Kasner solution with exponents verifying \eqref{Kasner.stability.cond}.

\begin{remark}[A basis of structure coefficient functions]
\label{R:BASISOFSTRUCTURE}
The antisymmetry property\footnote{This is equivalent to the antisymmetry of the commutator
$[e_I,e_J]$ with respect to interchanges of $I$ and $J$.} 
$\upgamma_{IJB} + \upgamma_{JBI} = - (\upgamma_{JIB} + \upgamma_{IBJ})$,
which follows from \eqref{antisymmetricgamma},
implies that
$\lbrace \upgamma_{IJB} + \upgamma_{JBI} \ | \ 1 \leq I,J,B \leq \mydim, \ I < J  \rbrace$
forms a basis for the structure coefficient functions.
This explains the condition $I < J$ on 
LHS~\eqref{STRUCTURECOEFFICIENTKasner.stability.cond}.
We use this simple fact throughout the article without always explicitly mentioning it.
\end{remark}

\begin{remark}[Sharply identifying possible obstructions to stability: Three distinct indices]
\label{R:IDENTIFYOBSTRUCTION}
Recall that we only have to consider structure coefficients with $I < J$ (see Remark~\ref{R:BASISOFSTRUCTURE})
and that (aside from the trivial case of a single non-zero Kasner exponent equal to unity) we have
$\underset{I=1,\cdots,\mydim}{\max} |\widetilde{q}_I| < 1$
(see Remark~\ref{R:ALTERNATEDESCRIPTIONOFSTABILITYCONDITION}).
It follows that when $I < J$, unless all three indices are distinct, 
two of the terms in the sum $\widetilde{q}_{\underline{I}}+\widetilde{q}_{\underline{J}}-\widetilde{q}_{\underline{B}}$ 
must cancel each other, leaving us with a 
single term $q_{\textnormal{survivor}}$ satisfying $|q_{\textnormal{survivor}}| < 1$.
Recalling also the evolution equation
$\partial_t (\upgamma_{IJB} + \upgamma_{JBI})
= 
-
\frac{(\widetilde{q}_{\underline{I}}+\widetilde{q}_{\underline{J}}-\widetilde{q}_{\underline{B}})}{t}
(\upgamma_{\underline{I}\underline{J}\underline{B}}+\upgamma_{\underline{J}\underline{B}\underline{I}})
+ \cdots
$
mentioned above, we see that when
$I < J$, unless all three indices are distinct, 
the structure coefficient $\upgamma_{IJB} + \upgamma_{JBI}$
is expected to behave (modulo the error terms ``$\cdots$'')
like $t^{-q_{\textnormal{survivor}}}$. In particular, 
modulo the effect of the error terms ``$\cdots$,''
such structure coefficients
are integrable with respect to $t$ near $t = 0$ and are compatible
with our proof of the stability of the Big Bang.
Thus, for perturbations of Kasner solutions, 
the only structure coefficients 
$\upgamma_{IJB} + \upgamma_{JBI}$ (with $I < J$)
that in principle could
serve as an obstruction to stable Big Bang formation
are those such that the sum $\widetilde{q}_I + \widetilde{q}_J - \widetilde{q}_B$ is greater than $1$,
and this is possible only when all three indices are distinct;
the stability condition \eqref{Kasner.stability.cond}
is the assumption that this obstruction is absent.
\end{remark}

Finally, we highlight that our framework also extends to some symmetric sub-regimes of
regimes where Mixmaster-related instabilities might generally occur, such as in the vacuum case 
in $1+3$ dimensions. More precisely, one does not truly need the condition \eqref{Kasner.stability.cond}
to prove monotonic-type Big Bang formation;
our approach works as long as one can prove the estimate
\eqref{STRUCTURECOEFFICIENTKasner.stability.cond} (for some constant $q < 1$).
The point is that by imposing symmetries on solutions,
one can eliminate some of the gravitational degrees of freedom in the problem,
and it can become possible to prove the estimate \eqref{STRUCTURECOEFFICIENTKasner.stability.cond}
even if the condition \eqref{Kasner.stability.cond} fails.
Roughly, this is sometimes possible because symmetries can force some of the structure coefficients to vanish.
For example, in this paper, we treat in detail the case of
polarized $U(1)$-symmetric solutions to the $1+3$-dimensional Einstein-vacuum 
equations, and \emph{under symmetric perturbations, 
we prove the stability of the Big Bang 
for all Kasner solutions -- not just ones that satisfy \eqref{Kasner.stability.cond}}.
In the next section, we precisely describe the models that we treat in detail.
Moreover, in Sect.\,\ref{subsec:meth.app}, we describe other contexts in which
our methods are potentially applicable.

\subsection{The models}\label{subsec:models}
Our main results yield stable curvature-blowup for a subset of
the family of generalized Kasner solutions on $(0,\infty)\times\mathbb{T}^{\mydim}$,
which can be expressed as follows:
\begin{align}\label{gen.Kasner}
\widetilde{\bf g}
& = 
	-dt \otimes dt
	+
	\widetilde{g},
&
\widetilde{g}	
	& :=
	\sum_{I=1,\cdots,\mydim}t^{2\widetilde{q}_I} dx^I \otimes dx^I, 
&
\widetilde{\psi}
& 
= 
\widetilde{B} \log t.
\end{align}
The Kasner exponents $\lbrace \widetilde{q}_I \rbrace_{I=1,\cdots,\mydim}$
and $\widetilde{B}$ are constants constrained by the following two algebraic equations:
\begin{align}\label{sumpi}
\sum_{I=1}^{\mydim}\widetilde{q}_I
& =1,
& 
\sum_{I=1}^{\mydim}\widetilde{q}_I^2 
& = 1-\widetilde{B}^2.
\end{align}
The equations in \eqref{sumpi} are consequences of two other equations: 
\textbf{i)} the mean curvature condition $\text{tr} \widetilde{k}=-\frac{1}{t}$
(which we discuss in more detail later), 
where $\widetilde{k}$ is the second fundamental form of $\Sigma_t$
with respect to $\widetilde{\bf g}$,
and \textbf{ii)} the Hamiltonian constraint \eqref{eq:hamconst}.
One can check that under the above assumptions, the tensorfields $(\widetilde{{\bf g}},\widetilde{\psi})$
are solutions to the $1+\mydim$-dimensional Einstein-scalar field equations 
\eqref{EE}--\eqref{SF}.

Our main results come in two flavors.
In the first case, we make no symmetry assumptions on the initial data,
and our results yield the dynamic stability of the Kasner Big Bang singularity whenever
the exponents of the background Kasner solution themselves verify the sub-criticality condition \eqref{Kasner.stability.cond.heur}
(which we also refer to as the ``stability condition''), in which case it reads:
\begin{align}\label{Kasner.stability.cond}
\mathop{\max_{I,J,B=1,\cdots,\mydim}}_{I < J}
\{\widetilde{q}_I + \widetilde{q}_J - \widetilde{q}_B \}<1.
\end{align} 
Since our results imply that the final Kasner exponents of the perturbed singular solution are close to those of the background
(see \eqref{E:ONEOVERDHOLDEREXPONENTFORFINALKASNEREXPONENTS}), and since \eqref{Kasner.stability.cond} is an open condition, 
our perturbed Kasner-like solutions will satisfy the original condition \eqref{Kasner.stability.cond.heur} as well.

\begin{remark}
	\label{R:ALTERNATEDESCRIPTIONOFSTABILITYCONDITION}
	The Kasner constraints \eqref{sumpi} imply that, aside from the trivial case in which 
	one of the $\widetilde{q}_I$ is equal to $1$ and the others vanish
	(in which case the Kasner spacetime metric is flat),
	we must have $\underset{I=1,\cdots,\mydim}{\max} |\widetilde{q}_I| < 1$.
	Thus, assuming the Kasner exponent constraints, 
	we could replace
	\eqref{Kasner.stability.cond} with the following
	condition: 
	\begin{align} \label{EQUIVKasner.stability.cond}
		\mathop{\max_{I,J,B=1,\cdots,\mydim}}_{I \neq J \neq B \neq I}
		\{\widetilde{q}_I+\widetilde{q}_J-\widetilde{q}_B\}<1.
	\end{align}
	In stating our main results, we prefer to refer to the condition \eqref{Kasner.stability.cond} because
	the case $I=B$ explicitly indicates that
	$\widetilde{q}_J < 1$ for $J=2,\cdots,\mydim$,
	while the case $I=1$ with $J=B$ explicitly indicates
	that $\widetilde{q}_1 < 1$ too.
\end{remark}

In the second case, we consider polarized $U(1)$-symmetric solutions to the Einstein-vacuum
equations in $1+3$ dimensions and prove stable Big Bang formation for
\underline{symmetric} perturbations of \emph{any} Kasner solution 
(with exponents verifying the constraints \eqref{sumpi}, $\widetilde{B}=0$, and 
excluding the trivial case of a single non-zero Kasner exponent equal to unity). 
We emphasize that for polarized $U(1)$-symmetric solutions, 
the spatial connection coefficients featuring three distinct indices
automatically vanish  
(see Lemma~\ref{lem:gamma.U(1)} for a proof and Remark~\ref{R:IDENTIFYOBSTRUCTION} for a discussion of the relevance of this fact),
which leads to a simple proof of \eqref{STRUCTURECOEFFICIENTKasner.stability.cond}
(see the end of the proof of Proposition~\ref{prop:low}).

We will now describe these two setups in more detail.

\subsubsection{Regimes with no symmetry assumptions on the perturbed initial data}
Under the following assumptions, our results yield the stability of the 
Kasner Big Bang singularity for non-empty sets of background Kasner solutions:
\medskip
\begin{enumerate}
\item The Einstein-vacuum equations (i.e., $\psi = 0$) for $\mydim \geq 10$.
\item The Einstein-scalar field equations for $\mydim \geq 3$.
\end{enumerate}
As we have stressed, without symmetry,
we require that the background Kasner exponents satisfy the  stability condition
\eqref{Kasner.stability.cond},
which, for example,
for any $\mydim \geq 3$,
is satisfied when all Kasner exponents are positive
(which can be achieved in the presence of a non-zero scalar field, i.e., $\widetilde{B} \neq 0$).
Also, as was observed in \cite{DHS}, in vacuum (i.e., $\widetilde{B}=0$), the set of Kasner exponents satisfying the
condition \eqref{Kasner.stability.cond} is non-empty when $\mydim \geq 10$, while for $\mydim \leq 9$, 
\eqref{Kasner.stability.cond} is algebraically impossible, given the constraints \eqref{sumpi}.

\subsubsection{The definition of the polarized $U(1)$-symmetry class}\label{subsubsec:U1.data}
Our discussion in this section refers to polarized $U(1)$-symmetric solutions to
the Einstein-vacuum equations (i.e., $\psi \equiv 0$) 
on $I \times \mathbb{T}^3$, where $I$ is an interval of time. 
This symmetry class is defined as follows:
\medskip
\begin{enumerate}
\item {\bf Polarized $U(1)$-symmetric initial data.} There exists a non-degenerate,\footnote{That is, $\overline{X}$ has no vanishing points.}  hypersurface-orthogonal, spacelike Killing vectorfield $\overline{X}$ 
on $\Sigma_1 \simeq \mathbb{T}^3$ with $\mathbb{T}^1$ orbits such that
$\mathcal{L}_{\overline{X}}\mathring{g}=\mathcal{L}_{\overline{X}}\mathring{k}=0$, where $\mathcal{L}$ is the Lie derivative operator. Moreover, the second fundamental form of $\Sigma_1$ 
satisfies $\mathring{k}(\overline{X},\overline{Y})=0$ for every $\Sigma_1$-tangent
vectorfield $\overline{Y}$ such that 
$\mathring{g}(\overline{X},\overline{Y}) = 0$.
For such data,
we can construct coordinates\footnote{Although 
the coordinate functions $\lbrace x^i \rbrace_{i=1,2,3}$ are only locally defined,
the corresponding partial derivative vectorfield frame $\lbrace \partial_i \rbrace_{i=1,2,3}$ can
be extended to a smooth global frame on $\mathbb{T}^3$.} 
$\lbrace x^i \rbrace_{i=1,2,3}$ on $\Sigma_1$
such that all coordinate components of $\mathring{g}$ and $\mathring{k}$ are independent of $x^3$
and such that $\overline{X}=\partial_3$, i.e.
$\mathring{g}_{13} = \mathring{g}_{23} =\mathring{k}_{13} = \mathring{k}_{23} \equiv 0$;
see the discussion in \cite[Section~2]{IM}.

\item {\bf Polarized $U(1)$-symmetric solutions.} Einstein-vacuum spacetimes that arise from such data contain 
a non-degenerate, hypersurface-orthogonal, spacelike Killing vectorfield $X$, 
such that $X\big|_{\Sigma_1}=\overline{X}$.
In fact, relative to appropriately constructed CMC-transported spatial coordinates, we have
$X = \partial_3$; see Lemma~\ref{L:PROPOFSYM}.
\end{enumerate}
One can easily check that in $1+3$ spacetime dimensions in the vacuum case, 
the condition \eqref{Kasner.stability.cond} is violated by \emph{all} Kasner solutions,
i.e., by all Kasner exponents satisfying \eqref{sumpi} with $\widetilde{B}=0$. Indeed, 
the algebraic relations \eqref{sumpi} imply that at least one Kasner exponent must be negative 
and that:
\begin{align}\label{Kasner.cond.viol}
\mathop{\max_{I,J,B=1,2,3}}_{I < J}\{\widetilde{q}_I + \widetilde{q}_J-\widetilde{q}_B\} 
\geq 
1 - 2 \min_{B=1,2,3}\{\widetilde{q}_B\} > 1.
\end{align}
Hence, in $1+3$ spacetime dimensions in the vacuum case, 
without symmetries or other additional assumptions, the Kasner singularity 
might not be stable under perturbations of the Kasner initial data on $\Sigma_1$. 
However, we show that within the class of polarized 
$U(1)$-symmetric solutions, the Kasner singularity \emph{is} in fact stable.
There are both heuristic and analytic reasons for this phenomenon, 
which we discuss in Sections~\ref{subsec:genKasner} and \ref{subsec:pf.overview}.

\subsection{Rough version of the main theorem}\label{subsec:rough}
Given a ``background'' generalized Kasner solution \eqref{gen.Kasner}, within the regimes 
described in Sect.\,\ref{subsec:models},
we perturb its initial data on $\Sigma_1=\{t=1\}$ and study 
the corresponding maximal development in the past of $\Sigma_1$.
As in the previous works of the last two authors \cite{RodSp1,RodSp2,RodSp3,Sp}, in order to 
\emph{synchronize the singularity} along $\lbrace t = 0 \rbrace$, 
we use a constant mean curvature (CMC) foliation that is realized by the level sets $\Sigma_t$ of a time function 
$t\in(0,1]$; as we describe below, this gauge features an elliptic PDE, which involves an infinite speed of
propagation, allowing for a synchronization of the singularity.
Relative to ``transported'' spatial coordinates 
$\lbrace x^i \rbrace_{i=1,\cdots,\mydim}$, 
which by definition 
are constant along the integral curves of the future-directed unit normal to $\Sigma_t$,
the perturbed spacetime metric takes the form 
(see also \eqref{polarizedmetric} in the polarized $U(1)$-symmetric case):
\begin{align}\label{metric}
{\bf g} 
& 
=-n^2 dt \otimes dt
+
g_{cd} dx^c \otimes dx^d,
&
n
& =[- ({\bf g}^{-1})^{\alpha\beta}\partial_\alpha t \partial_\beta t]^{-\frac{1}{2}},
\end{align}
where $g$ is the first fundamental form of $\Sigma_t$ (i.e., the Riemannian metric 
on $\Sigma_t$ induced by ${\bf g}$) and 
$n > 0$ is the lapse of the $\Sigma_t$ foliation. The CMC condition is:
\begin{align}\label{INTROtrk}
\text{tr}k=-\frac{1}{t},
\end{align}
where $k$ is the second fundamental form of $\Sigma_t$. We emphasize that 
\eqref{INTROtrk} is the gauge condition tied to the infinite speed of propagation,
since it implies an elliptic equation for $n$ (see \eqref{n.eq}).
\begin{remark}[Initial CMC slice]\label{rem:CMC}
The condition \eqref{INTROtrk} presupposes that the data on the initial Cauchy hypersurface $\Sigma_1$
have constant mean curvature $\text{tr}k|_{\Sigma_1}=-1$. 
Such an assumption can be made without loss of generality for solutions that
start out close to background Kasner solutions. 
The reason is that for near-Kasner data (not necessarily CMC data), 
one can first use the standard wave coordinate gauge to
solve Einstein's equations in a neighborhood of $\Sigma_1$, 
and then prove the existence of a CMC slice in that neighborhood with the desired properties;
see \cite[Proposition 14.4]{RodSp2} and \cite[Theorem 4.2]{Bart}.
\end{remark}
\noindent {\bf Polarized $U(1)$-symmetric case.} In the polarized $U(1)$-symmetric vacuum case with $\mydim = 3$,
our setup will be such that $x^3$ corresponds to the symmetry. In particular, 
relative to the transported spatial coordinates $\lbrace x^i \rbrace_{i=1,2,3}$,
$n$, 
$\lbrace g_{ij} \rbrace_{i,j=1,2,3}$, and 
$\lbrace k_{ij} \rbrace_{i,j=1,2,3}$
will not depend on $x^3$. 
Moreover, $\partial_3$ will be a hypersurface-orthogonal Killing vectorfield, 
everywhere defined in the past of $\Sigma_1$ and with positive norm away from the singularity;
see Lemma~\ref{L:PROPOFSYM}.
\medskip

We now state a first, rough version of our main stability results.
See Theorems~\ref{thm:precise} and \ref{thm:precise.U1} for precise statements.
\begin{theorem}[Stable Big Bang formation (Rough version)]\label{thm:rough}
In $1 + \mydim$ spacetime dimensions,
consider an explicit generalized ``background'' Kasner solution \eqref{gen.Kasner}
whose Kasner exponents satisfy the condition
\eqref{Kasner.stability.cond}, which is possible for 
$\mydim \geq 3$ in the presence of a scalar field
and for $\mydim \geq 10$ in vacuum (i.e., with $\widetilde{B} = 0$ in \eqref{gen.Kasner}). 
These background solutions are dynamically stable under perturbations -- without symmetry -- 
of their initial data near their Big Bang singularities,
as solutions to the Einstein-scalar field equations in the case $\mydim \geq 3$,
and, when $\widetilde{B} = 0$, as solutions to the Einstein-vacuum equations in the case $\mydim \geq 10$.
Moreover, in $1+3$ spacetime dimensions,
\textbf{all} Kasner solutions (with $\widetilde{B} = 0$) are dynamically stable solutions to the Einstein-vacuum equations
under perturbations -- with polarized $U(1)$-symmetry --
near their Big Bang singularities, even though they all violate the 
condition \eqref{Kasner.stability.cond}.

More precisely, under the above assumptions, 
sufficiently regular perturbations (i.e., perturbations belonging to suitably high order Sobolev spaces) 
of the Kasner initial data on $\Sigma_1$ give rise to maximal developments 
that terminate in a Big Bang singularity to the past.  
In particular, 
the spacetime solutions in the past of $\Sigma_1$ are foliated by spacelike hypersurfaces 
$\Sigma_t$ that are equal to the level sets of a time function $t$
verifying the CMC condition $\mathrm{tr}k=-t^{-1}$,
and the perturbed Kretschmann scalars 
${\bf Riem}^{\alpha\mu\beta\nu}{\bf Riem}_{\alpha\mu\beta\nu}$
blow up like $t^{-4}$ as $t \downarrow 0$.
Finally, the perturbed solutions exhibit AVTD behavior 
as the singularity is approached
(see just below equation \eqref{Kasner.stability.cond.heur} for further discussion of the notion of ``AVTD''),
and various $t$-rescaled solution variables have regular limits as $t \downarrow 0$.
\end{theorem}
\subsection{Background on ``Kasner-like behavior:'' Heuristics}\label{subsec:genKasner}
We now aim to provide further background on our main results.
In Sect.\,\ref{subsec:rel.works},
we will discuss prior works in the literature. Many of those works 
concern solutions that exhibit ``Kasner-like behavior,'' a concept that we now discuss. 
We do not attempt to ascribe rigorous meaning to this terminology; rather, we will highlight some properties that are meant to capture the idea that a metric with spatial dependence is ``blowing up in a manner similar to the Kasner solutions.''
We find the discussion 
in \cite{DHS,KL} instructive, where the spacetime metric, to leading order near $t=0$, 
is \emph{assumed} to take the form:
\begin{align}\label{metric.heur}
{\bf g}
& =-dt \otimes dt
+
g,
&
g
&\,
\cong
\sum_{I=1}^{\mydim} t^{2q_I(x)} \uptheta^I(x) \otimes \uptheta^I(x),
&\uptheta^I & =\uptheta^I_a(x)dx^a,
\end{align}
where ``$\cong$'' means ``asymptotic to as $t \downarrow 0$,'' and
the scalar functions $\lbrace q_I(x) \rbrace_{I=1,\cdots,\mydim}$
satisfy the following (vacuum) analogs of \eqref{sumpi}:
\begin{align} \label{CONSTRAINTSmetric.heur}
	\sum_{I=1}^{\mydim} q_I(x)
	& =
	\sum_{I=1}^{\mydim} q_I^2(x)
	= 1.
\end{align}  
Note that in \eqref{metric.heur}, 
the one-forms\footnote{Recall that we do not sum over repeated underlined indices.} 
$\lbrace t^{q_{\underline{I}}(x)}\uptheta^{\underline{I}}(x)  \rbrace_{I=1,\cdots,\mydim}$ 
``represent the Kasner-like directions." Moreover,
although the metric components may vary in $x$, they are all monotonic in $t$ at fixed $x$. We stress that
our discussion here is heuristic in the sense that metrics of the form \eqref{metric.heur}
are not generally solutions to Einstein's equations, though they might approximate actual solutions.

Let $\lbrace k_{\ J}^I \rbrace_{I,J=1,\cdots,\mydim}$ 
denote\footnote{This notation should not be confused with the notation
``$k_{IJ}$'' that we use in the bulk of the article, where
$k_{IJ} := k_{cd} e_I^c e_J^d$ denotes the components of $k$ relative 
to a Fermi--Walker propagated orthonormal spatial frame.} the components of the type $\binom{1}{1}$
second fundamental form of $\Sigma_t$
with respect to the co-frame $\lbrace \uptheta^I(x) \rbrace_{I=1,\cdots,\mydim}$
and its basis-dual\footnote{If $\lbrace v_I \rbrace_{I=1,\cdots,\mydim}$ denotes the basis-dual
frame (i.e., $\uptheta^I(v_J) = \updelta_J^I$, where $\updelta_J^I$ is the Kronecker delta),
then relative to arbitrary coordinates $\lbrace y^i \rbrace_{i=1,\cdots,\mydim}$ 
on $\mathbb{T}^{\mydim}$, 
we have $\uptheta^I = \uptheta_c^I dy^c$,
$v_J = v_J^c \frac{\partial}{\partial y^c}$,
and
$k_{\ J}^I := k_{\ d}^c \uptheta_c^I v_J^d$,
where $k_{\ b}^a = (g^{-1})^{ac} k_{cb}$.} frame.
Standard computations yield that for metrics of the form \eqref{metric.heur}, 
we have $k_{\ J}^I \sim t^{-1}$. 
On the other hand, in coordinates such that the lapse 
$|{\bf g}(\partial_t,\partial_t)|^{1/2}$ is equal to $1$ (as on RHS~\eqref{metric.heur}),
the components
$k_{\ J}^I$ 
satisfy the following evolution equations:
\begin{align}\label{dt.k.heur}
\partial_t k_{\ J}^I-\text{tr}k k_{\ J}^I = Ric_{\ J}^I - {\bf Ric}_{\ J}^I,
\end{align}
where $Ric_{\ J}^I$ denotes a component of the type $\binom{1}{1}$ Ricci curvature of $g$ 
with respect to the co-frame $\lbrace \uptheta^I(x) \rbrace_{I=1,\cdots,\mydim}$
and its basis-dual frame, and similarly for ${\bf Ric}_{\ J}^I$.
\medskip

\noindent {\bf Heuristic criterion for Kasner-like behavior}
\begin{itemize}
\item If ${\bf Ric} = 0$ (e.g., if the metric ${\bf g}$ from \eqref{metric.heur} was already known
to be a solution to the Einstein-vacuum equations), 
then the leading order behavior $k_{\ J}^I \sim t^{-1}$ can easily be \emph{derived directly} from \eqref{dt.k.heur}
\emph{if} \textbf{i)} one knew that $\text{tr}k = - t^{-1} + \mathcal{O}(t^{-1+ \upsigma})$ for some $\upsigma>0$,
and \textbf{ii)}
one could prove the following pointwise estimate for $t$ larger than but close to $0$:
\begin{align}\label{heur.crit}
\max_{I,J=1,\cdots,\mydim}
|Ric_{\ J}^I|\lesssim t^{-2+\upsigma}.
\end{align}
In our main results, we impose the condition $\text{tr}k = - t^{-1}$ by using
constant mean curvature foliations. This gauge is not compatible with the ansatz
\eqref{metric.heur} because it generally requires the lapse
$[-({\bf g}^{-1})^{\alpha\beta}\partial_\alpha t\partial_\beta t]^{-\frac{1}{2}}$
to be different from unity. For convenience, we will downplay this issue in the present discussion.\footnote{Since we derive estimates showing that $|n-1| \lesssim t^{\upsigma}$, the non-constant lapse does not affect the heuristic analysis.}
In the presence of matter, the same conclusions $k_{\ J}^I \sim t^{-1}$ hold
if one can also show that\footnote{In our main results, in the case of the scalar field
matter model, we will prove (with the help of \eqref{EE}) pointwise estimates showing that 
$\underset{I,J=1,\cdots,\mydim}{\max} |{\bf Ric}(e_I,e_J)| \lesssim t^{-2+\upsigma}$, 
where $\lbrace e_I(t,x) \rbrace_{I=1,\cdots,\mydim}$
is an orthonormal spatial frame; this frame component bound is sufficient for the proof of our main results.
These technical estimates are in fact derived
in the proof of Lemma~\ref{L:SPATIALMETRICERRORTERMSPOINTWISE},
though it might not be immediately apparent from the statement of the lemma.} 
$\underset{I,J=1,\cdots,\mydim}{\max} |{\bf Ric}_{\ J}^I| \lesssim t^{-2+\upsigma}$.
\end{itemize}

\begin{remark}[Remarks on our use of time-dependent orthonormal frames]
	\label{R:TIMEDEPENDENTFRAMES}
	We make the following remarks:
	\begin{itemize}
	\item
	The above discussion of heuristics referred to the components
	of tensorfields with respect to the time-independent 
	co-frame $\lbrace \uptheta^I(x) \rbrace_{I=1,\cdots,\mydim}$
	and its basis-dual frame. Note that
	$\lbrace \uptheta^I(x) \rbrace_{I=1,\cdots,\mydim}$
	is not $\bf{g}$-orthonormal. Moreover,
	for general small perturbations of Kasner solutions, 
	there is no reason to believe that there 
	exists a time-independent co-frame in which the 
	perturbed metric is asymptotically of the form \eqref{metric.heur}. 
	Hence, we again stress that
	our approach is based on deriving 
	estimates for the components of tensorfields relative to an 
	orthonormal spatial frame $\lbrace e_I(t,x) \rbrace_{I=1,\cdots,\mydim}$ obtained by Fermi--Walker transport, 
	and that our use of an orthonormal frame 
	is crucial so that we can exploit the approximately diagonal nature
	of the structure coefficient evolution equations 
	(see Sect.\,\ref{SSS:INTROSTRUCTURECOEFFICIENTS}).
	\item
	In particular, in our main results, we will prove an analog of \eqref{heur.crit}
	for the components of $Ric$ relative to an orthonormal frame; 
	see Remark~\ref{R:CRUCIALSPATIALRICCIBOUND}.
	That is, instead of \eqref{heur.crit},
	our main results will rely on a proof of the following bound:
	\begin{align}\label{FRAMEheur.crit}
		|Ric|\lesssim t^{-2+\upsigma},
	\end{align}
	where LHS~\eqref{FRAMEheur.crit} denotes the usual invariant pointwise
	norm of the spatial Ricci tensor.
	\item We also highlight that
	we are able to close our estimates \emph{without} showing that the metric is asymptotic to a metric of the form
	\eqref{metric.heur}. In fact, we close the proof with only very weak information
	about the orthonormal frame $\lbrace e_I(t,x) \rbrace_{I=1,\cdots,\mydim}$
	and co-frame $\lbrace \upomega^I(t,x) \rbrace_{I=1,\cdots,\mydim}$:
	we prove only that their coordinate components 
 $\lbrace e_I^i\rbrace_{I,i=1,\cdots,\mydim}$
	and
	$\lbrace \upomega_i^I(t,x) \rbrace_{I,i=1,\cdots,\mydim}$
	are bounded in magnitude by $\lesssim t^{-q}$ for some $q \in (0,1)$
	depending on the background Kasner exponents; see also Remark~\ref{R:NOLIMITFORFRAME}.
\item Despite the previous comment, for the solutions under study, 
	we are able to prove the existence of ``final Kasner exponents'' 
	$\left\lbrace q_I^{(\infty)}(x) \right\rbrace_{I=1,\cdots,\mydim}$ as the singularity is approached;
	see Proposition~\ref{prop:tk}.
\end{itemize}
\end{remark}

\noindent {\bf Conditions for the validity of the heuristic criterion \eqref{heur.crit} for metrics of the form \eqref{metric.heur}}
\begin{itemize}
\item A computation using \eqref{metric.heur} shows that in the absence of special algebraic structure,
we typically have:\footnote{Note that the spatial \emph{coordinate} 
components $\lbrace Ric_{\ j}^i \rbrace_{i,j=1,\cdots,\mydim}$ of the type $\binom{1}{1}$ tensor
$Ric$ are bounded in magnitude by $\lesssim$ LHS~\eqref{heur.Ric}
and hence the inequality \eqref{heur.Ric} would imply the same bound for
$\underset{i,j=1,\cdots,\mydim}{\max} |Ric_{\ j}^i|$.
\label{FN:COORDINATECOMPONENTSOBEYSAMEBOUND}}
\begin{align} \label{heur.Ric}
\max_{I,J=1,\cdots,\mydim}
|Ric_{\ J}^I| 
\approx \mathop{\max_{I,J,B=1,\cdots,\mydim}}_{I < J} \{t^{2(q_B-q_I-q_J)} \}.
\end{align}

\item In view of \eqref{heur.Ric}, we see that the estimate \eqref{heur.crit} holds if: 
\begin{align}\label{heur.Kasner.cond}
\mathop{\max_{I,J,B=1,\cdots,\mydim}}_{I < J}\{q_I+q_J-q_B\}<1.
\end{align}
\item In $1+3$ spacetime dimensions in the vacuum case, 
where the condition \eqref{heur.Kasner.cond} is always violated 
(see \eqref{Kasner.cond.viol}), 
one can show that for metrics of the form \eqref{metric.heur},
the estimate \eqref{heur.crit} is valid if 
the following relation holds, where $\mathrm{d}$ denotes the exterior derivative operator:
\begin{align}\label{heur.omega.cond}
 \uptheta^{-} \wedge \mathrm{d} \uptheta^{-}=0,
\end{align}
where $q_-(x) < 0$ is the\footnote{Using the equations \eqref{CONSTRAINTSmetric.heur},
one can show that in the vacuum case with $\mydim = 3$,
aside from the trivial case in which 
one of the $q_I$ is equal to $1$ and the others vanish,
precisely one of the $q$'s must be negative.} negative Kasner-like exponent in \eqref{metric.heur}
and $\uptheta^-(x)$ is the corresponding one-form, i.e., these quantities are 
such that the tensor product
$\uptheta^-(x) \otimes \uptheta^-(x)$ is multiplied by the factor $t^{2 q_-(x)}$.
Standard calculations show that the condition \eqref{heur.omega.cond} \emph{eliminates} the terms responsible 
for the worst behavior on RHS~\eqref{heur.Ric}, 
which, if present, would have been more singular than RHS~\eqref{heur.crit}.
\end{itemize}

\noindent {\bf A geometric interpretation of the condition \eqref{heur.omega.cond} for metrics of the form \eqref{metric.heur}}
\begin{itemize}
\item The Frobenius Theorem states that \eqref{heur.omega.cond} is equivalent to the integrability of the $2$-dimensional subspaces
$V^{-}_p$ annihilated by $\uptheta^{-}$, where for $p \in \mathbb{T}^3$,
\begin{align}\label{Frob}
V^{-}_p=\{Y\in T_p \mathbb{T}^3:\uptheta^{-}_p(Y)=0\}.
\end{align}
We note that \eqref{heur.omega.cond} is equivalent to the existence of functions 
$u,v:\mathbb{T}^3 \to \mathbb{R}$
such that $\uptheta^{-}= u d v$.
\end{itemize}

As we already mentioned in Sect.\,\ref{subsec:models}, for
the models that we consider in our results without symmetry assumptions, it was already observed in 
\cite{BK,DHS} that the condition \eqref{heur.Kasner.cond} is not vacuous, at least in
the sense that there exist generalized Kasner 
(in particular, spatially homogeneous)
solutions whose exponents satisfy it.
We also stress that for solutions with $x$-dependence,
in the context of the heuristic works \cite{BK,DHS},
the condition \eqref{heur.Kasner.cond} can be interpreted as an inequality that should be satisfied by
the ``final Kasner exponents,'' i.e., the exponents $\lbrace q_I(x) \rbrace_{I=1,\cdots,\mydim}$ of the alleged asymptotic form \eqref{metric.heur}
of an alleged Kasner-like solution.
Our main results in fact justify the existence of ($x$-dependent) Kasner-like solutions
with  ``final Kasner exponents'' 
$\lbrace q_I(x) \rbrace_{I=1,\cdots,\mydim}$ verifying the stability condition
\eqref{heur.Kasner.cond},
at least when the data
are close to generalized Kasner solutions whose exponents verify the same condition;
see Proposition~\ref{prop:tk}.
Our proof of these facts relies, of course, on the open nature of the condition
\eqref{heur.Kasner.cond}.
 
The above discussion suggests that 
in $1+3$ spacetime dimensions in the vacuum case, 
$x$-dependent Kasner-like solutions can exist if
the ``polarization-type'' condition \eqref{heur.omega.cond} holds.
However, 
the condition \eqref{heur.omega.cond} 
refers to the structure of the metric ``at the singularity''
(i.e., since \eqref{metric.heur} is only supposed to capture the asymptotic structure of the metric, 
\eqref{heur.omega.cond} is a statement about the structure 
of the asymptotic behavior of the metric near the singularity),
and we are not aware of any ``general method'' for solutions without symmetry 
that allows one to ensure the validity of \eqref{heur.omega.cond} 
via assumptions on the initial data on $\Sigma_1$.
Nonetheless, for polarized $U(1)$-symmetric solutions,
discussed further below, the condition \eqref{heur.omega.cond} automatically holds.
\medskip

\noindent {\bf Polarized $U(1)$-symmetric metrics of the form \eqref{metric.heur} satisfy \eqref{heur.omega.cond}}

\begin{itemize}
\item Recall that we defined the polarized $U(1)$-symmetry class in Sect.\,\ref{subsubsec:U1.data}. Assume that $\partial_3$ is the hypersurface-orthogonal Killing vectorfield with $\mathbb{T}^1$-orbits. 
This will be the case in our study of solutions with symmetry; see Lemma~\ref{L:PROPOFSYM}. In addition, assume that the leading order expression \eqref{metric.heur} of the Kasner-like metrics in question respects the symmetry, i.e.,
assume that $\uptheta^3$ is proportional to $dx^3$ and that for $I=1,2,3$,  
$\partial_3 q_I=0$ and $\mathcal{L}_{\partial_3} \uptheta^I=0$. 
We divide the argument for the validity of \eqref{heur.omega.cond} 
into the following two cases, 
depending on the sign of the Kasner exponent associated to the norm of the Killing field $\partial_3$:

\item If $q_3 < 0$, where $q_3$ is the Kasner-like exponent corresponding to the direction of symmetry,
then for a metric of the form \eqref{metric.heur}, the validity of
\eqref{heur.omega.cond} (with $\uptheta^3$ in the role of $\uptheta^-$)
follows easily, since, by the previous point, 
$\uptheta^3$ is a scalar function multiple of $dx^3$.

\item Again assume that $\partial_3$ is the hypersurface-orthogonal Killing vectorfield with $\mathbb{T}^1$-orbits,
but now assume that
$q_3>0$ and (without loss of generality)
$q_1<0$, where $q_3$ is still the Kasner-like exponent corresponding to the direction of symmetry. 
Then the subspaces annihilated by the one-form $\uptheta^1=\uptheta^-$, corresponding to $q_1$, are $2$-dimensional and contain $\partial_3$. Let $Y$ be a unit-length vectorfield in the kernel of $\uptheta^1$ that is orthogonal to $\partial_3$. 
Fix a point $p \in \mathbb{T}^3$, and consider the integral curve $s \rightarrow a_p(s)$ of $Y$ passing through it, normalized by
$a_p(0) = p$. Then the image of $a_p$ times the orbits of $\partial_3$, i.e.,
$\mbox{\upshape Img}(a_p) \times [-\pi,\pi]_{x_3}$, is a surface whose tangent planes are exactly the kernel of $\uptheta^1$, since $\partial_3$ (being Killing and in the kernel of $\uptheta^1$) commutes with $Y$. Hence, the planes $V_p^-$ are integrable, and 
by the Frobenius Theorem, this is equivalent to the condition \eqref{heur.omega.cond}.
\end{itemize}

\subsection{Related works}\label{subsec:rel.works}
Before outlining the main ideas behind our proof of Theorem~\ref{thm:rough},
we first describe some prior results on Kasner-like singularities.
There are many such results, and we roughly divide them into
three categories.

\subsubsection{Big Bang formation under symmetry assumptions}\label{SSS:BigBangSymm}
There are many works that provide a detailed description of stable Big Bang formation, or more generally, spacelike singularity formation with AVTD behavior (e.g., in black hole interiors), for large sets 
of initial data on a smooth Cauchy hypersurface in a 
model with sufficient \emph{symmetry} such that the problem reduces to a system of ODEs or $1+1$-dimensional PDEs.
We further divide these results into sub-categories.

{\bf $\bullet$ The interior of black holes.} In Christodoulou's influential works 
\cite{Christ2,Christ3} on the spherically symmetric Einstein-scalar field system with large data, 
it was shown that black holes form and contain spacelike singularities in their interior,
where their Kretschmann scalars blow up.

{\bf $\bullet$ Polarized Gowdy-symmetry.} 
In \cite{CIM}, the authors studied polarized Gowdy-solutions\footnote{Roughly, Gowdy-solutions
are such that there exists a pair of spacelike Killing vectorfields $X$ and $Y$ such that the
twist constants $\epsilon_{\alpha \beta \gamma \delta} X^{\alpha} Y^{\beta} {\bf D}^{\gamma} X^{\delta}$
and
$\epsilon_{\alpha \beta \gamma \delta} X^{\alpha} Y^{\beta} {\bf D}^{\gamma} Y^{\delta}$
vanish, where $\epsilon$ is the spacetime volume form. Polarized Gowdy-solutions
satisfy one additional condition: $X$ and $Y$ are orthogonal.} 
to the Einstein-vacuum equations
and proved Strong Cosmic Censorship, that is, that for an open and dense set
of polarized Gowdy-symmetric initial data on $\mathbb{T}^3$ or $\mathbb{S}^2 \times \mathbb{S}^1$, 
the maximal globally hyperbolic development is inextendible, and causal geodesics
are generically inextendible in one direction due to curvature-blowup.
 
{\bf $\bullet$ Gowdy-symmetry.}
In \cite{Rin3}, Ringstr\"{o}m proved a similar result for
Gowdy-solutions with spatial topology $\mathbb{T}^3$, without the polarization assumption. See also the 
related works \cite{CL,Rinnew,Rin1} and the
survey article \cite{Rin2}.
The general Gowdy-case turned
out to be significantly more difficult to handle 
in view of a possible phenomenon that was shown to be absent in the polarized case: ``spikes.''
Roughly, spikes are regions where spatial derivatives can become large, i.e., regions
where solutions do not exhibit AVTD behavior. 
For an open and dense set of data in the topology of $C^{\infty}$, Ringstr\"{o}m
proved that a curvature singularity forms and that the solution exhibits 
Kasner-like behavior, except for possibly at a finite number of spikes.

{\bf $\bullet$ Polarized axi-symmetric initial data.} The Schwarzschild black hole singularity is highly unstable, 
as is shown by the fact that instead of singularities,
near-Schwarzschild Kerr solutions have Cauchy horizons inside their black holes,
and the metric can be smoothly extended across them. 
However, the Schwarzschild singularity was recently shown to be stable \cite{AF} 
as a solution to the Einstein-vacuum equations under symmetric perturbations,
specifically those perturbations whose solutions exhibit a hypersurface-orthogonal, spacelike, Killing vectorfield 
$X$ with $\mathbb{T}^1$ orbits.\footnote{Note that Kerr solutions, although axi-symmetric, 
do not contain a hypersurface-orthogonal Killing field.} This symmetry class is closely related 
to the polarized $U(1)$-symmetry class that we study in Theorem~\ref{thm:precise.U1},
as we now describe.
Compared to the polarized $U(1)$-symmetric solutions with $\mathbb{T}^3$ spatial topology that we study 
in the present paper,
the main difference in \cite{AF} is that
$X$ degenerates at a $2$-dimensional submanifold; since the vectorfield
$X$ in \cite{AF} is tangent to 2-spheres, 
such degeneracies are topologically unavoidable. This can be concretely seen already in the case of the 
background Schwarzschild metric in classical $(t,r,\theta,\phi)$ coordinates, 
where $X := \partial_{\phi}$ is the Killing field, and away from the singularity $\lbrace r = 0 \rbrace$,
its (square) norm 
${\bf g}_{\textnormal{Schwarzschild}}(\partial_\phi,\partial_\phi)=r^2\sin^2\theta$ vanishes at exactly $\theta=0,\pi$.
Apart from this extra feature of the degenerate Killing vectorfield and the difference in topology ($\mathbb{R}\times\mathbb{S}^2$ instead of the $\mathbb{T}^3$ topology considered here), 
the stability result of \cite{AF} can be seen to correspond\footnote{To see the correspondence, one must
re-parametrize the coordinate $r$ to proper time (recall that $r$ is a time function in the Schwarzschild black hole interior, whereas $t$ is a spatial coordinate, in the classical coordinate representation of the Schwarzschild metric).} 
to a special case of our symmetric blowup-results,
specifically Theorem~\ref{thm:precise.U1}, with background Kasner exponents $\widetilde{q}_1=-\frac{1}{3}$, 
$\widetilde{q}_2=\widetilde{q}_3=\frac{2}{3}$. The method of proof introduced in \cite{AF} is very much tied to the specific symmetry class, relying on a wave-maps reduction of the Einstein-vacuum equations, and it is therefore not applicable
to the non-symmetric solutions that we study in Theorem~\ref{thm:precise}. However, it seems that the use of the particular symmetry reduction in \cite{AF} allowed for the derivation of more refined asymptotic behaviors for the spatial components of the metric
compared to the results we derive in Theorem~\ref{thm:precise.U1}.

\subsubsection{The construction of solutions with Big Bang singularities -- without a proof of stability}
Numerous papers have provided a construction of solutions that exhibit a Kasner-like singularity. 
Most of these works concern cosmological spacetimes and employed Fuchsian techniques in regimes 
where the discussion in Sect.\,\ref{subsec:genKasner} suggests that one might expect the singularity formation
to be dynamically stable.

{\bf $\bullet$ Gowdy-symmetry.} The first result of this type \cite{KR} yielded the 
construction of analytic solutions with Gowdy-symmetry. 
The analyticity assumption was removed in \cite{Ren}. See also
\cite{St} for more general topologies and \cite{ABIL2} for a treatment in generalized wave gauges.

{\bf $\bullet$ Polarized and half-polarized $\mathbb{T}^2$-symmetry.} Analytic singularities in polarized $\mathbb{T}^2$-symmetry 
class were first constructed in \cite{IK}. The analyticity assumption was later removed in \cite{ABIL},
where the authors also constructed half-polarized solutions.

{\bf $\bullet$ Polarized or half-polarized $U(1)$-symmetry.} Polarized and half-polarized $U(1)$-symmetric
analytic solutions with $\mathbb{T}^3$ spatial topology were constructed in \cite{IM}. More
general topologies were later treated in \cite{CBIM}. We note that in these works, the authors defined their notion of 
polarized and half-polarized solutions at the singularity, i.e., at $t=0$, by eliminating free functions relative to a given ansatz, in the spirit of \eqref{metric.heur} and \eqref{heur.omega.cond}.

{\bf $\bullet$ Einstein-scalar field or stiff fluid.} The first construction of singular solutions without symmetries 
was carried out in \cite{AR}. The authors studied the Einstein-scalar field and Einstein-stiff fluid systems
and used Fuchsian techniques to construct analytic solutions whose ``final Kasner exponents'' 
(see the last point of Remark~\ref{R:TIMEDEPENDENTFRAMES} and Proposition~\ref{prop:tk})
are all positive.

{\bf $\bullet$ Sub-critical Einstein-matter systems.} In \cite{DHRW}, 
the authors extended the results of \cite{AR} by constructing
singular, analytic, Kasner-like solutions without symmetries
to various Einstein-matter systems
and to the Einstein-vacuum equations in $1+\mydim$ dimensions with $\mydim \geq 10$.
As in the present paper, the solution regimes treated in \cite{DHRW}
were sub-critical in the sense that the solutions exhibited the crucial
bound \eqref{FRAMEheur.crit} for the spatial Ricci curvature. Roughly,
our present work shows that an open set of solutions constructed
in \cite{DHRW} are \emph{dynamically stable} under Sobolev-class perturbations of their initial data
near their Big Bangs, at least in the vacuum and scalar field matter model cases.

{\bf $\bullet$ $1+3$ vacuum without symmetries.} As we alluded to in Sect.\,\ref{SS:HAWKINGSTHEOREM}, 
Kasner solutions might be unstable under general perturbations without symmetries, 
unless some kind of condition, such as a polarization condition of the type \eqref{heur.omega.cond}, is imposed. 
Nevertheless, in \cite{Klin}, the author constructed analytic 
Kasner-like singular solutions without symmetries,
demonstrating that such solutions exist, even though they might be unstable.
Moreover, in \cite{FL}, for distinct Kasner exponents,
the first author and Luk constructed Sobolev-class solutions that exhibit Kasner-like singularities. 
The solutions do not a priori enjoy any symmetry, 
but they satisfy the polarization condition \eqref{heur.omega.cond}.

{\bf $\bullet$ Asymptotically Schwarzschild on a 2-sphere.} Finally, we mention the
first author's work \cite{F}, which, in a Lorentz gauge, yielded the construction of a class
of spacetimes that converge to a portion of the Schwarzschild black hole singularity. 
The construction requires no symmetry or analyticity assumptions. While the 
construction does not yield a full spacelike singular hypersurface,
it does provide a spacelike singular 2-sphere.

\subsubsection{Stable Big Bang formation without symmetry assumptions}
The stability of some Kasner solutions towards their Big Bang singularities, 
without symmetries and for open sets of initial data, 
was only fairly recently shown by the last two authors.
For the scalar field and stiff fluid matter models, 
the stability of the (isotropic) 
Friedmann--Lema\^{\i}tre--Robertson--Walker (FLRW) solutions with $\mathbb{T}^3$ spatial topology
(i.e., $\widetilde{q}_1 = \widetilde{q}_2 = \widetilde{q}_3 =\frac{1}{3}$ and $\widetilde{B} = \sqrt{2/3}$)
was shown in \cite{RodSp1,RodSp2},
while the case of the scalar field matter model with $\mathbb{S}^3$ topology
was handled in \cite{Sp}.
The Einstein-vacuum equations
were handled in \cite{RodSp3} under a ``moderate anisotropy'' assumption on the Kasner exponents,
specifically $\underset{I=1,\cdots,\mydim}{\max} |\widetilde{q}_I|<\frac{1}{6}$,
which is possible in $1+\mydim$ spacetime dimensions when $\mydim \geq 38$. 
Some aspects of our analysis here are in the spirit of the analysis in \cite{RodSp3}.

\subsubsection{Conditional Kasner-like behavior}\label{SSS.cond.Kasner}
There are recent results that derive Kasner-like behavior for solutions under assumed bounds on certain key quantities. 
For example, assuming mainly scale invariant bounds on the Riemann curvature of Hubble-normalized time slices, 
Lott \cite{Lott} showed that the corresponding singular solutions converge to Kasner flows in appropriate topologies. 
Ringstr\"om \cite{Rin5,Rin6} derived sharp results on the geometry of Kasner-like solutions 
by assuming mainly bounds on the normalized Weingarten map.\footnote{The Weingarten map is the second fundamental form in type 
$\binom{1}{1}$ form, i.e., in the notation of the present paper, the tensorfield with components 
$k_{\ j}^i := (g^{-1})^{ia} k_{aj}$. \label{FN:WEINGARTEN}} 
Ringstr\"om's sharp estimates crucially rely on a frame that asymptotically diagonalizes the normalized Weingarten map. 
This frame is very different from ours, which is Fermi-propagated from $\Sigma_1$. 
It would be interesting to explore whether
a change of frames could yield more refined estimates for the spatial frame and connection coefficients
compared to the estimates we obtain in Theorem~\ref{thm:precise}; 
see also Remarks~\ref{rem:diag.frame} and \ref{R:SHARPERASYMPTOTICS}.

\subsection{Overview of our proof}\label{subsec:pf.overview}
Our proofs of Theorems~\ref{thm:precise} and \ref{thm:precise.U1} 
are based on deriving estimates for a set of reduced variables that
solve an elliptic-hyperbolic PDE system. Here we will summarize the main features
of the system and how its structures allow us to prove our main results.
We will confine our discussion to sketching proofs of various low order
and high order a priori estimates for near-Kasner initial data given
on $\Sigma_1 = \lbrace t = 1 \rbrace$. In practice, the low order and high order estimates
are coupled, and we derive them via a bootstrap argument.
The a priori estimates are sufficient to ensure 
that the solution exists on $(0,1] \times \mathbb{T}^{\mydim}$
(see Proposition~\ref{P:EXISTENCEONHALFSLAB}), 
which is the main step in the paper.
The proof of curvature-blowup and other aspects of the solution
are relatively straightforward consequences of the a priori estimates.
We will not discuss those results in this section; instead, we
refer readers to Sect.\,\ref{sec:sol} for those details.

\subsubsection{The gauge}
\label{SSS:GAUGE}
We use a constant mean curvature foliation
in which, for $t \in (0,1]$, the level sets $\Sigma_t$ of the time function $t$ satisfy
$\text{tr}k = - \frac{1}{t}$, where $k$ is the second fundamental form of $\Sigma_t$.
We also use spatial coordinates $\lbrace x^i \rbrace_{i=1,\cdots,\mydim}$
that are transported along the unit normals to $\Sigma_t$.
In this gauge, the spacetime metric satisfies
${\bf g}
	= - n^2 dt \otimes dt
				+
				g_{ab} dx^a \otimes dx^b
$,
where $n$ is the lapse and $g$ is the first fundamental form of $\Sigma_t$.
This setup is the same as in \cite{RodSp1,RodSp2,RodSp3}.
However, to derive the sharp results of the present paper, we 
use a crucial additional ingredient: we use Fermi--Walker transport to
construct a $\Sigma_t$-tangent orthonormal ``spatial frame'' 
$\lbrace e_I \rbrace_{I=1,\cdots,\mydim}$, which is globally defined in space. 
When supplemented with $e_0 := n^{-1} \partial_t$, we obtain an orthonormal spacetime frame.
We then formulate Einstein's equations in such a way that the unknowns are $n$, 
the \emph{components} $\lbrace e_I^i \rbrace_{I,i=1,\cdots,\mydim}$ of the orthonormal
frame with respect to the transported spatial coordinates,
the \emph{components} $\lbrace \oe_i^I \rbrace_{I,i=1,\cdots,\mydim}$ of the corresponding dual co-frame
with respect to the transported spatial coordinates,
the \emph{frame components} $k_{IJ} := k_{cd} e_I^c e_J^d$ of the second fundamental form with respect to the frame,
the \emph{connection coefficients} $\upgamma_{IJB} := g(\nabla_{e_I} e_J,e_B)$
of the spatial frame (where $\nabla$ is the Levi-Civita connection of $g$),
the future-directed-timelike-unit-normal-derivative of the scalar field, denoted by $e_0 \psi$,
and the \emph{spatial frame derivatives} $\lbrace e_I \psi \rbrace_{I=1,\cdots,\mydim}$
of the scalar field.\footnote{We never need to estimate $\psi$ itself 
since only its derivatives appear in the system \eqref{EE}--\eqref{SF}.} 
We refer readers to Sect.\,\ref{sec:setup} for the details.

\subsubsection{The lapse, the dynamic variables, and the ``less singular'' nature of spatial derivative terms}
\label{SSS:LAPSEANDDYNAMIC}
The lapse $n$ satisfies an elliptic PDE (see \eqref{n.eq})
with source terms depending on some of the other solution variables, specifically
the ``dynamic variables''
$e_I^i$,
$\upgamma_{IJB}$,
and $e_I \psi$.
Thus, to control $n$, we use elliptic estimates to control it in terms of
these dynamic variables.
These estimates are rather standard, and we will
not discuss them in detail here. We simply highlight
that it is crucial for our results that the right-hand
side of the elliptic lapse PDE depends only on the \emph{spatial derivatives} of various tensorfields,
i.e., there are no time derivative terms, the point being
that in the problem under study, spatial derivative terms are less singular with respect to $t$
compared to time derivative terms; this is a manifestation of AVTD behavior,
which we first mentioned in Sect.\,\ref{SS:HAWKINGSTHEOREM}. 
We refer readers to
Sect.\,\ref{sec:lapse} for a detailed proof of the lapse estimates.
To control the dynamic variables, including 
$e_I^i$,
$\oe_i^I$,
$k_{IJ}$,
$\upgamma_{IJB}$,
$e_0 \psi$,
and
$e_I \psi$,
we 
derive ``low order'' $L^{\infty}$ estimates and ``high order'' energy estimates
based on first-order formulations of the flow; we refer to 
Proposition~\ref{P:redeq} and Lemma~\ref{lem:psi.syst}
for the first-order formulations of the equations.
As we explained in the discussion above \eqref{STRUCTURECOEFFICIENTKasner.stability.cond},
we also crucially rely on the special ``diagonal structure'' 
exhibited by the PDE system satisfied by the \emph{structure coefficients} of the spatial frame.
We provide this PDE system in Proposition~\ref{P:KEYEVOLUTIONSTRUCTURECOEFFICIENTS},
and we will discuss it in more detail in Sect.\,\ref{SSS:INTROSTRUCTURECOEFFICIENTS}.

\subsubsection{Approximately diagonal form of the structure coefficient evolution equations} 
\label{SSS:INTROSTRUCTURECOEFFICIENTS}
Away from symmetry, to control the $\upgamma_{IJB}$'s, we rely on the crucial observation that the terms:
\begin{align}\label{gamma.set}
\lbrace S_{IJB} := \upgamma_{IJB} + \upgamma_{JBI} \ | \ 1 \leq I,J,B \leq \mydim, I < J  \rbrace
\end{align}
solve an evolution equation system whose ``main linear part'' is \emph{diagonal}
with coefficient magnitudes that are smaller than $t^{-1}$, provided the condition \eqref{Kasner.stability.cond} is 
satisfied by the background Kasner exponents; see equation
\eqref{dt.gamma-gammatilde2} for the precise equation, 
and equation \eqref{E:SCHEMATICSTRUCTURECOEFFICIENTEVOLUTION} for an abbreviated version. 
To caricature, the system is of the form $\dot{S} = \frac{M}{t} \cdot S + \cdots$, where $M$
is a diagonal matrix whose components verify $|M_{IJ}| < 1$ when \eqref{Kasner.stability.cond} holds. 
This allows us to prove that  
under \eqref{Kasner.stability.cond}, we have
$|S| \lesssim t^{-q}$ for some $q < 1$.
\textbf{This bound is crucial for the entire proof}, 
as we use it to show that the solutions exhibit AVTD behavior.
The variables $S_{IJB}$ in \eqref{gamma.set} 
are precisely the \emph{structure coefficients} 
of the spatial orthonormal frame $\lbrace e_I \rbrace_{I=1,\cdots,\mydim}$.
Here we note that by the simple identity \eqref{E:RECOVERGAMMAFROMSTRUCTURECOEFFCIENTS}, to control
all of the $\upgamma_{IJB}$'s, it suffices to control the structure coefficients.

Moreover, as we highlighted in Remark~\ref{R:IDENTIFYOBSTRUCTION},
even in cases such that the stability condition \eqref{Kasner.stability.cond} is violated,
only \emph{some} of the structure coefficients $S_{IJB}$
could possibly serve as an obstruction to proving the desired estimates: 
those with three distinct indices.
That is, our work essentially shows that 
in regimes where \eqref{Kasner.stability.cond} is violated
(such as the Einstein-vacuum equations in $1+3$ dimensions without symmetries),
\emph{any instabilities would arise from the combinations
$S_{IJB}$ with distinct indices}.
This observation is precisely what allows us to extend our 
stable blowup-results to the class of polarized $U(1)$-symmetric 
Einstein-vacuum solutions in $1+3$ dimensions: 
by considering a spatial orthonormal frame 
$\lbrace e_I \rbrace_{I=1,2,3}$ such that $e_3=(g_{33})^{-\frac{1}{2}}\partial_3$ corresponds to the normalized Killing direction (see Lemma~\ref{lem:U1}), 
we can conclude that the spatial connection coefficients 
with distinct indices are automatically zero (see Lemma~\ref{lem:gamma.U(1)}). Hence, 
the observations described above allow us to sufficiently control the non-zero structure
coefficients and prove stable blowup.

We also note that the less singular behavior (than $t^{-1}$) of the $\upgamma_{IJB}$'s is consistent with the normalized second fundamental form 
frame components $tk_{IJ}(t,x)$ having a continuous limit, $\upkappa_{IJ}^{(\infty)}(x)$, as $t \downarrow 0$, 
which is the main feature of a Kasner-like singularity (as we described in Sect.\,\ref{subsec:genKasner}). 
This is once again a manifestation of AVTD behavior.
The eigenvalues of $-\upkappa_{IJ}^{(\infty)}(x)$ can be viewed as the ``final, $x$-dependent'' 
Kasner exponents of the perturbed spacetime; see Proposition~\ref{prop:tk}. 

\subsubsection{The bootstrap argument and initial discussion of the behavior of the high order energies}
In practice, to prove our main results, we rely on a bootstrap argument in which
we assume that various low order and high order norms are small (indicating that the solution is near-Kasner) 
on a time interval $(T_{\textnormal{Boot}},1]$; see \eqref{Boots} for the precise bootstrap assumptions.
Then the main task becomes deriving strict improvements of the bootstrap assumptions for near-Kasner initial data,
where we remind the reader that the data are given along $\Sigma_1 = \lbrace t = 1 \rbrace$.
In the rest of Sect.\,\ref{subsec:pf.overview}, 
to illustrate the main ideas, we will not explain the full bootstrap argument
in detail, but will instead show how the different parts of the analysis consistently
fit together.
As a starting point, we note that our 
analysis will eventually show that
we have a top-order energy bound of the following form:
\begin{align} 
\begin{split} \label{E:HIGHORDERNERGIES}
	&
	t^{\blowupexp + 1} \| k_{IJ} \|_{\dot{H}^N}(\Sigma_t)
		+
	t^{\blowupexp + 1} \| \upgamma_{IJB} \|_{\dot{H}^N(\Sigma_t)}
		+
	t^{\blowupexp + q} \| e_I^i \|_{\dot{H}^N(\Sigma_t)}
		\\
	& \ \
	+
	t^{\blowupexp + 1} \| e_0 \psi \|_{\dot{H}^N(\Sigma_t)}
		+
	t^{\blowupexp + 1} \| e_I \psi \|_{\dot{H}^N(\Sigma_t)}
	\lesssim data,
\end{split}
\end{align}
where $q$ is as in Sect.\,\ref{SSS:INTROSTRUCTURECOEFFICIENTS} 
(see just above \eqref{E:INTROKBOUND} for further discussion)
and $\| \cdot \|_{\dot{H}^N}$ is a standard homogeneous Sobolev norm;
see Sect.\,\ref{sec:Boots} for the details.

We now highlight some crucial aspects of our analysis of the high order energies:
\begin{quote}
To close the proof and justify the estimate \eqref{E:HIGHORDERNERGIES}, 
we must first choose the parameter $\blowupexp$ to be \textbf{sufficiently large}, then choose 
the ``regularity parameter'' $N$ to be \textbf{sufficiently large relative to $\blowupexp$}, 
and finally choose $data$ to be sufficiently small, 
where for the rest of Sect.\,\ref{subsec:pf.overview}, ``$data$'' denotes a small number whose size 
is controlled by the closeness of the initial data to the Kasner data in
a high order Sobolev norm.
\end{quote}

\subsubsection{The behavior of the low order $L^{\infty}$ norms}
\label{SSS:INTROLOWORDERNORMS}
In this section, we will explain how the availability of a high order energy bound of the form
\eqref{E:HIGHORDERNERGIES} allows us to derive sharp $L^{\infty}$ estimates
for the solution variables at the low derivative levels.
We already stress that our proof fundamentally requires
that we prove much less singular (with respect to $t$) 
estimates at the low derivative levels
compared to \eqref{E:HIGHORDERNERGIES}; here, we are thinking of  
\eqref{E:HIGHORDERNERGIES} as a ``very singular estimate'' in the sense that $\blowupexp$ is large.
In particular, at the low
derivative levels, we must prove estimates for the perturbed $k_{IJ}$ and $e_0 \psi$ variables
showing that they are not more singular than their Kasner analogs, which blow up like $t^{-1}$.
To keep the presentation short, in most of the rest of Sect.\,\ref{subsec:pf.overview}, 
we will focus only on the estimates for $e_I^i$, $k_{IJ}$,
and $\upgamma_{IJB}$; the estimates for the scalar field can be obtained in a similar fashion.
Moreover, we again highlight that we derive control of the connection coefficients at the low
derivative levels by relying on the structure coefficients
$S_{IJB} := \upgamma_{IJB} + \upgamma_{JBI}$
(whereas for the energy estimates at the high derivative levels, 
we can work directly with the connection coefficients $\upgamma_{IJB}$).
Finally, we note that our discussion here will mainly concern the analysis away from symmetry
under the sub-criticality condition \eqref{Kasner.stability.cond}.

\begin{remark}[The frame is not precisely adapted]
\label{R:FRAMEFREEDOM}
It seems remarkable to us that away from symmetry,
for all sub-critical Kasner exponents,
we have a lot of freedom in constructing the orthonormal frame.
More precisely, in Sect.\,\ref{SS:CONSTRUCTIONOFTHEINTIALORTHONORMALFRAME},
we use the Gram--Schmidt algorithm to construct an initial orthonormal frame that is a perturbation of
the spatial coordinate frame $\lbrace \partial_i \rbrace_{i=1,\cdots,\mydim}$,
and then we propagate this frame using the Fermi--Walker transport equations \eqref{frame.prop}.
There is nothing special about our choice of initial data for the frame; any nearby initial data for the orthonormal
frame would have worked just as well.
In particular, we can close the estimates without using a spatial frame that is adapted to the perturbed Kasner directions,
that is, without the frame being aligned with the eigenvectors of the perturbed second fundamental form $k$;
see also Remark~\ref{R:NOLIMITFORFRAME}.
In fact, as of present, 
the only way we know how to close the top-order estimates is by using a Fermi--Walker-transported frame,
which is not generally aligned with eigenvectors of $k$.
In contrast,
many previous studies of Kasner-like singularities relied on a frame 
that is adapted to the eigenvectors of $k$
(see Sect.\,\ref{subsec:rel.works} for discussion of related works). 
\end{remark}

\begin{remark}[The role of $N_0$]
	\label{R:N0}
	In our main theorem, there appears a parameter $N_0 \geq 1$ that represents, roughly, 
	the number of derivatives that we sharply control
	in $\| \cdot \|_{L^{\infty}}$. We are free to choose it at the start of the bootstrap argument. 
	For example, $N_0=1$ is permissible. However, 
	the choice of $N_0$ will affect the minimal allowable size of $N$ 
	(see Theorem~\ref{thm:precise}). $N_0$ also captures the amount of regularity that the
	``limiting normalized solution variables'' enjoy along the Big Bang hypersurface $\Sigma_0$
	(see Sect.\,\ref{SS:LIMITINGFUNCTIONS}).
	We introduced $N_0$ mainly to clarify that for ``very smooth'' initial data that fall under the scope
	of our main results, the corresponding limiting solution variables will 
	inherit a quantifiable amount of the smoothness.  
	For convenience, in our heuristic discussion here, we will only discuss the case $N_0=1$, i.e., 
	the $L^{\infty}$ estimates at the level of the undifferentiated equations.
\end{remark}

To proceed, we let 
$\widetilde{e}_I^i(t)$
and
$\widetilde{k}_{IJ}(t) := \widetilde{k}_{cd}(t) \widetilde{e}_I^c(t) \widetilde{e}_J^d(t)$ 
respectively denote the background Kasner 
frame components
and
second fundamental form components;
see Sect.\,\ref{SS:BACKGROUNDVARIABLES} for the precise definitions.
We aim to sketch a proof of the following pointwise estimates for $(t,x) \in (0,1] \times \mathbb{T}^{\mydim}$,
where in what follows, $q$ and $\upsigma$ are fixed constants that satisfy
$0<2 \upsigma< 2\upsigma + \underset{\substack{I,J,B=1,\cdots, \mydim \\ I < J}}{\max}
\{|\widetilde{q}_B|,\widetilde{q}_I+\widetilde{q}_J-\widetilde{q}_B\} < q <1- 2 \upsigma$
(such constants exist whenever the sub-criticality condition \eqref{Kasner.stability.cond} holds):
\begin{subequations}
\begin{align}
	|t k_{IJ} - t \widetilde{k}_{IJ}|(t,x)
	& \lesssim 
		data,
			\label{E:INTROKBOUND} \\
	t^q |S_{IJB}|(t,x)
	& \lesssim
		data,
			\label{E:INTROSTRUCTURECOEFFICIENTBOUND} \\
	t^q |e_I^i-\widetilde{e}_I^i|(t,x)
	& \lesssim data.
		\label{E:INTROFRAMEBOUND}
\end{align}
\end{subequations}
The estimate \eqref{E:INTROKBOUND} is sharp and is of particular importance
because it is needed to control various ``borderline terms'' in the energy estimates,
as we explain in Sect.\,\ref{SSS:INTROHIGHORDERENERGYESTIMATES}. Similar remarks
apply for the $L^{\infty}$ estimates for $t e_0 \psi$ (which we do not discuss here).
The estimates \eqref{E:INTROSTRUCTURECOEFFICIENTBOUND}--\eqref{E:INTROFRAMEBOUND}
are not quite sharp with respect to powers of $t$, and we have chosen the power
$t^q$ on LHSs~\eqref{E:INTROSTRUCTURECOEFFICIENTBOUND}--\eqref{E:INTROFRAMEBOUND}
so as to allow for the simplest possible analysis. 
Estimates at the low derivative levels,
in the spirit of \eqref{E:INTROKBOUND}--\eqref{E:INTROFRAMEBOUND},
are sufficient to allow us to identify the limiting Kasner-like behavior 
of the perturbed solutions; see Proposition~\ref{prop:tk} for the details.
We stress that, although Proposition~\ref{prop:tk} shows
that the $t$-weighted scalar functions $t k_{IJ}$ and $t e_0\psi$ 
have non-trivial, regular limits as $t \downarrow 0$,
we do not obtain (or need!) analogous sharp limits for the frame components or spatial derivative-involving terms;
see Remark \ref{R:NOLIMITFORFRAME}.
\begin{remark}[Refined estimates with a different frame?]\label{rem:diag.frame}
It is conceivable that a different choice of orthonormal frame might yield 
sharper asymptotic estimates for the spatial frame and connection coefficients, 
as in \cite{Rin6}. However,
to close a bootstrap argument with a refined frame,
such as a frame that is adapted to the eigenvectors of $k$,
one would have to overcome serious technical difficulties, 
such as a potential loss of derivatives for the frame.
It would be interesting to understand whether such an approach is viable
for solutions without symmetry,
i.e., whether the entire proof can be carried out using a refined frame.
On the other hand, given the estimates we prove in Theorems~\ref{thm:precise} and
\ref{thm:precise.U1}, as a follow-up problem, 
one could try to derive sharper estimates for the asymptotics; 
see also Remark~\ref{R:SHARPERASYMPTOTICS}.

\end{remark}

\begin{remark}[The crucial bound for the spatial Ricci curvature]
	\label{R:CRUCIALSPATIALRICCIBOUND}
	Using the estimates 
	\eqref{E:INTROSTRUCTURECOEFFICIENTBOUND}--\eqref{E:INTROFRAMEBOUND}
	and similar estimates for the spatial derivatives of $S_{IJB}$,
	the algebraic identity \eqref{E:RECOVERGAMMAFROMSTRUCTURECOEFFCIENTS},
	and the spatial Ricci curvature frame component expression \eqref{spRic},
	one can conclude that $|Ric_{IJ}| := |Ric(e_I,e_J)| \lesssim data \times t^{-2 + \upsigma}$.
	This is a frame component analog of the classic sub-criticality condition \eqref{heur.crit},
	and in practice, one needs such an estimate to prove \eqref{E:INTROKBOUND}.
\end{remark}

To sketch the main ideas behind the proofs of \eqref{E:INTROKBOUND}--\eqref{E:INTROFRAMEBOUND}, we
note that the evolution equations for ${k_{IJ} - \widetilde{k}_{IJ}}$,
$S_{IJB}$, and
$e_I^i-\widetilde{e}_I^i$
can be caricatured as follows 
(see Proposition~\ref{P:KEYEVOLUTIONSTRUCTURECOEFFICIENTS} and Lemmas~\ref{L:EVOLUTIONEQUATIONSFORFRAMEANDCOFRAME}
and \ref{L:TOPORDERCOMMUTEDEQUATIONSFORGAMMANDK}
for the precise equations):
\begin{subequations}
\begin{align}
	\partial_t (k_{IJ} - \widetilde{k}_{IJ})
	+
	\frac{1}{t}
	{(k_{IJ} - \widetilde{k}_{IJ})}
	& = 
		e_I^i \cdot \partial \upgamma	
		+
		\upgamma \cdot \upgamma
		+
		\cdots,
			\label{E:SCHEMATICKEVOLUTION} \\
	\partial_t S_{IJB}
	+
	\frac{(\widetilde{q}_{\underline{I}}+\widetilde{q}_{\underline{J}}-\widetilde{q}_{\underline{B}})}{t}
	S_{\underline{I}\underline{J}\underline{B}}
	& =
			\cdots,
				\label{E:SCHEMATICSTRUCTURECOEFFICIENTEVOLUTION} 
					\\
	\partial_t(e_I^i-\widetilde{e}_I^i)
	+
	\frac{\widetilde{q}_{\underline{I}}}{t}({e_{\underline{I}}^i}-\widetilde{e}_{\underline{I}}^i)
& =
			\cdots,
			\label{E:SCHEMATICFRAMEEVOLUTION}
\end{align}
\end{subequations}
where $\cdots$ denotes similar or simpler error terms that we ignore to simplify the discussion,
and we recall that we do not sum over repeated underlined indices.

\begin{remark}[On the approximately diagonal structure of the evolution equations for the structure coefficients]
	\label{R:STRUCTURECOEFFICIENTDIAGAONAL}
	Note that \eqref{E:SCHEMATICSTRUCTURECOEFFICIENTEVOLUTION} shows that the
	$S_{IJB}$ solve an evolution equation system that is approximately diagonal,
	as we highlighted in Sect.\,\ref{SSS:INTROSTRUCTURECOEFFICIENTS}.
\end{remark}

Next, we note that the estimates \eqref{E:INTROSTRUCTURECOEFFICIENTBOUND} and
\eqref{E:INTROFRAMEBOUND} are easy to derive (modulo the omitted terms ``$\cdots$'')
via integrating factors as a consequence of
equations \eqref{E:SCHEMATICSTRUCTURECOEFFICIENTEVOLUTION}--\eqref{E:SCHEMATICFRAMEEVOLUTION}
and the definition of $q$. In reality, 
the proofs of \eqref{E:SCHEMATICKEVOLUTION}--\eqref{E:SCHEMATICFRAMEEVOLUTION}
must be handled simultaneously, via a bootstrap argument, due to coupling terms, but we will ignore this issue here;
see the proof of Proposition~\ref{prop:low} for the details.

Next, to illustrate the interplay between low order $L^{\infty}$ estimates and high order
energy estimates, we will now explain how to derive the bound
\eqref{E:INTROKBOUND} for $k_{IJ}$, assuming the high order energy bound \eqref{E:HIGHORDERNERGIES}
and the estimates \eqref{E:INTROSTRUCTURECOEFFICIENTBOUND} and \eqref{E:INTROFRAMEBOUND}.
To this end, we must explain how to control the term 
$e_I^i \cdot \partial \upgamma$ on RHS~\eqref{E:SCHEMATICKEVOLUTION}.
This term loses one derivative and must ultimately be handled with the help of
energy estimates (which we discuss in Sect.\,\ref{SSS:INTROHIGHORDERENERGYESTIMATES}), 
but as we explain, its $L^{\infty}$ norm
is sub-critical with respect to powers of $t$.
By this, we mean that the behavior of 
$e_I^i \cdot \partial \upgamma$
with respect to $t$ is
strictly less singular with respect to $t$, as $t \downarrow 0$,
compared to the terms on LHS~\eqref{E:SCHEMATICKEVOLUTION} 
(i.e., less singular than $t^{-2}$)
and thus, near the singularity, it is a negligible error term.
To see this, one can use standard Sobolev embedding and interpolation
estimates (see Lemmas~\ref{lem:basic.ineq} and \ref{lem:Sob.borrow})
to infer that there is a constant $\updelta_N > 0$ (depending on $N$)
such that $\updelta_N \rightarrow 0$ as $N \to \infty$ and
such that the following crucial estimate holds:
\begin{align} \label{E:INTROKEYINTERPOLATION}
\|\partial \upgamma \|_{L^{\infty}(\Sigma_t)}
	\lesssim 
	\| \upgamma \|_{L^{\infty}(\Sigma_t)}
	+
	\| \upgamma \|_{L^{\infty}(\Sigma_t)}^{1 - \updelta_N}
	\| \upgamma \|_{\dot{H}^N(\Sigma_t)}^{\updelta_N}.
\end{align}
Combining \eqref{E:INTROKEYINTERPOLATION} 
with 
\eqref{E:HIGHORDERNERGIES}
and
\eqref{E:INTROSTRUCTURECOEFFICIENTBOUND},
and using the fact that the connection coefficients $\upgamma$ are
linear combinations of the structure coefficients $S$ (see \eqref{E:RECOVERGAMMAFROMSTRUCTURECOEFFCIENTS}),
we find that
$\|\partial \upgamma \|_{L^{\infty}(\Sigma_t)}
	\lesssim 
	data \times t^{-q}
	+
	data
	\times
	t^{-(1 - \updelta_N)q}
	\times
	t^{- \updelta_N (\blowupexp + 1)}
$.
Thus, by choosing $N$ \underline{sufficiently large}, 
exploiting that \emph{$\blowupexp$ does not depend on $N$},
and that $\updelta_N \rightarrow 0$ as $N \to \infty$,
we find that:
\begin{align} \label{E:SECONDINTROKEYINTERPOLATION}
\|\partial \upgamma \|_{L^{\infty}(\Sigma_t)}
	\lesssim 
	data \times t^{-1 + \upsigma}.
\end{align}

\begin{quote}
	The importance of \eqref{E:INTROKEYINTERPOLATION} and \eqref{E:SECONDINTROKEYINTERPOLATION}
	is that they show that when $N$ is large,
	the singularity strength
	of $\|\partial \upgamma \|_{L^{\infty}(\Sigma_t)}$
	is not much worse than the singularity strength of
	$\|\upgamma \|_{L^{\infty}(\Sigma_t)}$,
	\emph{even if $\| \upgamma \|_{\dot{H}^N(\Sigma_t)}$
	obeys a much worse estimate of the form 
	$\| \upgamma \|_{\dot{H}^N(\Sigma_t)} \lesssim data \times t^{-(\blowupexp + 1)}$}
	(which is the bound afforded by the energy estimate \eqref{E:HIGHORDERNERGIES}).
\end{quote}

Hence, also using \eqref{E:INTROFRAMEBOUND},
we conclude that
$\| e_I^i \cdot \partial \upgamma \|_{L^{\infty}(\Sigma_t)}
	\lesssim 
	data \times t^{-2 + \upsigma}
$,
i.e., that this term is less singular than $t^{-2}$, as desired.
Let us now sketch the proof that these bounds imply the 
desired estimate \eqref{E:INTROKBOUND}.
Using these bounds,
multiplying the evolution equation \eqref{E:SCHEMATICKEVOLUTION} by $t$
and noting that the resulting LHS is equal to 
$\partial_t [t(k_{IJ} - \widetilde{k}_{IJ})]$,
and then using the fundamental theorem of calculus,
we deduce the pointwise bound:
\begin{align} \label{E:SCHEMINTEGRATEDPOINT}
\begin{split}
|t k_{IJ} - t \widetilde{k}_{IJ}|(t,x)
& \lesssim 
	|t k_{IJ} - t \widetilde{k}_{IJ}|(1,x)
		\\
& \ \
	+
	\int_t^1
		s |e_I^i|(s,x) \cdot |\partial \upgamma|(s,x)
	\, ds
	+
	\int_t^1
		s |\upgamma|(s,x) \cdot |\upgamma|(s,x)
	\, ds
	+
	\cdots
		\\
& \lesssim
	data
	+
	data
	\int_t^1
		s^{-1 + \upsigma}
	\, ds
	+
	\cdots
	\lesssim
	data
	+ 
	\cdots,	
\end{split}
\end{align}
which yields the desired bound \eqref{E:INTROKBOUND}, up to the error terms ``$\cdots$.''
We close this section by highlighting that in a fully detailed proof of \eqref{E:SCHEMINTEGRATEDPOINT},
the estimate \eqref{FRAMEheur.crit} is crucial for obtaining the power
$s^{-1 + \upsigma}$ in the next-to-last inequality in \eqref{E:SCHEMINTEGRATEDPOINT}.

\begin{remark}[How large does $N$ need to be?]
\label{rem:N}
The following natural question emerges from the above discussion: 
how large does $N$ need to be for the above scheme to work? 
The interpolation inequality \eqref{E:INTROKEYINTERPOLATION} already suggests 
that the rough estimate 
$N \sim \upsigma^{-1}$ 
is sufficient to guarantee that
$\| \partial \upgamma_{IJB} \|_{L^{\infty}(\Sigma_t)}$
is less singular than $t^{-1}$, 
given the high order energy bounds \eqref{E:HIGHORDERNERGIES} for $t^{\blowupexp + 1} \| \upgamma_{IJB} \|_{\dot{H}^N(\Sigma_t)}$
and the fact that $\upgamma$ satisfies 
$|\upgamma_{IJB}| \lesssim data \times t^{-1+\upsigma}$ and that $\updelta_N \sim N^{-1}$ 
(see the proof of Lemma~\ref{lem:Sob.borrow} for some details on the $N$-dependence of the
$\updelta_N$ that appear in our interpolation estimates). The precise largeness of $N$
needed for this argument to go through depends on $\mydim$ and the size of $\blowupexp$, which we discuss in Remark~\ref{rem:A}.
We also note that, even in the best case scenario,
we would expect the behavior of at least one of the connection coefficient norms
$\| \upgamma_{IJB} \|_{L^{\infty}(\Sigma_t)}$
to be at least as singular as $t^{- \underset{I=1,\cdots,\mydim}{\max} \widetilde{q}_I}$. 
Hence, in view of our choice \eqref{sigma,q} of the parameters $q$ and $\upsigma$,
we see that $\upsigma$ has order of magnitude at least as small as 
$1 - \underset{I=1,\cdots,\mydim}{\max} \widetilde{q}_I$.
In particular, this implies that
as the background Kasner exponents tend towards an extreme case, e.g., 
as $\underset{I=1,\cdots,\mydim}{\max} \widetilde{q}_I \to 1$ (and hence $\upsigma \to 0$), 
the number of derivatives $N$ we would need to close our bootstrap argument would tend to $\infty$.
Similarly, as
LHS~\eqref{Kasner.stability.cond} tends towards $1$,
our arguments would require $N$ to tend to $\infty$ as well.
Moreover, going back to Remark~\ref{R:N0}, 
we note that one could use similar interpolation arguments 
(see Lemmas~\ref{lem:basic.ineq} and \ref{lem:Sob.borrow})
to infer that choosing $N \sim N_0 \upsigma^{-1}$
would be sufficient to guarantee that
$\| \partial^{N_0} \upgamma_{IJB} \|_{L^{\infty}(\Sigma_t)}$
is less singular than $t^{-1}$.
\end{remark}

\subsubsection{The high order energy estimates}
\label{SSS:INTROHIGHORDERENERGYESTIMATES}
We now explain how we derive our top-order energy estimates for the dynamic
variables $e_I^i$, $k_{IJ}$, $\upgamma_{IJB}$, $e_0 \psi$, and $e_I \psi$,
that is, how we prove \eqref{E:HIGHORDERNERGIES}.
We will highlight the role played by the $L^{\infty}$ estimates of Sect.\,\ref{SSS:INTROLOWORDERNORMS}.
We first commute the evolution equations 
(recall that in our formulation, all of the evolution equations are first-order)
with $\partial^{\iota}$,
where $\partial^{\iota}$ is an $N^{th}$-order differential operator
corresponding to repeated differentiation with respect to the
transported spatial coordinate partial derivative vectorfields.
We then derive energy identities for solutions to the commuted equations,
where \emph{we incorporate $t^{\blowupexp + 1}$-weights into the identities}. 
Below we will explain the analytic role of the weights.
The energy identity for the scalar field is standard, and we will not discuss it in
detail here; we refer readers to Lemma~\ref{L:TOPORDERDIFFERENTIALENERGYIDENTITYFORSCALARFIELD} for
a differential version of that energy identity. Similar remarks apply for the energy
identity for the frame component functions $e_I^i$. 

However, 
the derivation of the energy identity for the second fundamental form frame components $k_{IJ}$ 
and the connection coefficients $\upgamma_{IJB}$ is more subtle,
since the identity corresponds to a surprising gain of one derivative for the connection coefficients, 
as we highlighted in Sect.\,\ref{SS:HAWKINGSTHEOREM}.
The identity can be derived using a modification of the approach used in \cite{RodSp1,RodSp2,RodSp3}.
The main difficulty is that the evolution equations \eqref{dt.k}--\eqref{dt.gamma} for $\upgamma$ and $k$
do not form a symmetric hyperbolic system, which, at first glance, seems to obstruct
the availability of a basic energy identity. 
However, one can use differentiation by parts and
the momentum constraint equation, as well as the special structure of the equations
relative to CMC foliations (see \eqref{momconst}),
to replace the problematic terms with source terms that enjoy a sufficient amount of
regularity. We refer readers to Lemma~\ref{L:TOPORDERDIFFERENTIALENERGYIDENTITYFORGAMMAANDK}
for this top-order energy identity, expressed in differential form.

We will now describe our top-order energy estimates.
We will give a simplified, schematic presentation in order
to focus on the main ideas.
We define the following top-order energy:\footnote{Note that as we have defined it,
the energy $\mathbb{E}_N$ scales linearly with respect
to the quantities that it controls. This is a different convention than is usually used
in the literature, in which energies are typically defined so
as to scale quadratically in the quantities that they control. 
Similar remarks apply to the energies we use in the proof of Lemma~\ref{L:PROPOFSYM}.}
\begin{align}
\begin{split}
\mathbb{E}_N(t) 
& := 
t^{\blowupexp + 1} \| k \|_{\dot{H}^N}
+
t^{\blowupexp + 1} \| \upgamma \|_{\dot{H}^N}
+
t^{\blowupexp + q} \sum_{I,i=1,\cdots,\mydim}  \| e_I^i \|_{\dot{H}^N}
	\\
& \ \
+
t^{\blowupexp + 1} \| e_0 \psi \|_{\dot{H}^N}
+
t^{\blowupexp + 1} \sum_{I=1,\cdots,\mydim} \| e_I \psi \|_{\dot{H}^N}.
\end{split}
\end{align}
We will sketch a proof that if $\blowupexp$ is chosen to be sufficiently large
and then $N$ is chosen to be sufficiently large such that 
the $L^{\infty}$ estimates of Sect.\,\ref{SSS:INTROLOWORDERNORMS} hold,
then we have the following bound:\footnote{In practice, we also derive top-order
energy estimates for the co-frame components $\lbrace \oe_i^I \rbrace_{I,i=1,\cdots,\mydim}$.} 
$\mathbb{E}_N(t) \leq C_N \times data$,
i.e., the estimate \eqref{E:HIGHORDERNERGIES} holds.
To obtain this bound, we combine the energy identities mentioned in the previous paragraphs
with elliptic estimates for the lapse,
and we use the $L^{\infty}$ estimates from Sect.\,\ref{SSS:INTROLOWORDERNORMS}
and interpolation to control the nonlinear error terms.
This allows us to derive the following energy integral inequality for $t \in (0,1]$
(see Proposition~\ref{P:TOPORDERENERGYINEQUALITY} for the precise inequalities),
\emph{where $C_*$ is a constant that captures the strength of the \textbf{borderline terms} 
in the equations and that can be chosen to be independent of $N$ and $\blowupexp$} (as long as ``$data$'' is small), 
while $C_N > 0$ is a large, $N$-dependent constant:
\begin{align} \label{E:INTROENERGYINEEQUALITYSCHEMATIC}
		\mathbb{E}_N^2(t)
	& \leq 
		data^2
		+
		(C_* - \blowupexp)
		\int_{t}^1
			\frac{\mathbb{E}_N^2(s)}{s}
		\, ds
		+
		C_N
		\int_{t}^1
			s^{-1 + \upsigma} \mathbb{E}_N^2(s)
		\, ds.
\end{align}
The crucial point is that 
if we choose $\blowupexp$ to be larger than $C_*$,
then the time integral on RHS~\eqref{E:INTROENERGYINEEQUALITYSCHEMATIC}
becomes non-positive, and we can discard it.
Finally, from
\eqref{E:INTROENERGYINEEQUALITYSCHEMATIC} and Gr\"{o}nwall's lemma,
we obtain that $\mathbb{E}_N(t) \leq C_N \times data$
as desired. This concludes our schematic discussion of the a priori estimates.

Some closing remarks are in order.
\begin{itemize}
	\item The negative definite integral 
		$
		- 
		\blowupexp
		\int_{t}^1 
			\frac{\mathbb{E}_N^2(s)}{s} 
		\, ds
		$
		on RHS~\eqref{E:INTROENERGYINEEQUALITYSCHEMATIC}
		arises from our energy identities, specifically from the $t^{\blowupexp + 1}$ weights
		that we have incorporated into them. This negative definite integral
		allows us to absorb the dangerous borderline error integral
		$
		C_*
		\int_{t}^1
			\frac{\mathbb{E}_N^2(s)}{s}
		\, ds
		$,
		but at the expense of forcing us to work with energies that are very degenerate
		near $t=0$.
	\item Above we mentioned the notion of a ``borderline term.''
		To handle such terms, we must rely on the
		sharp $L^{\infty}$ estimates from Sect.\,\ref{SSS:INTROLOWORDERNORMS};
		for borderline terms, there is ``no room'' in the $L^{\infty}$ estimates.
		In the context of energy estimates, 
		borderline terms contribute to the dangerous integral
		$C_*
		\int_{t}^1
			\frac{\mathbb{E}_N^2(s)}{s}
		\, ds
		$
		on RHS~\eqref{E:INTROENERGYINEEQUALITYSCHEMATIC}.
		One example of a borderline error integral
		is
		$
		\int_{t}^1
			\int_{\Sigma_s}
				s^{2(\blowupexp + 1)} \cdot k \cdot \partial^{\iota} \upgamma \cdot \partial^{\iota} \upgamma
			\, dx
		\, ds
		$,
		where $\partial^{\iota}$ is an $N^{th}$-order spatial differential operator of the type mentioned earlier.
		To bound this integral by 
		$C_*
		\int_{t}^1
			\frac{\mathbb{E}_N^2(s)}{s}
		\, ds
		$,
		we need to use the sharp estimate
		$\| k_{IJ} \|_{L^{\infty}(\Sigma_s)} \leq \frac{C_*}{s}$ implied by \eqref{E:INTROKBOUND}.
		If, instead of this sharp bound,
		we only knew that $\| k_{IJ} \|_{L^{\infty}(\Sigma_s)} \leq C_* s^{- (1 + \epsilon)}$ for some $\epsilon > 0$,
then on RHS~\eqref{E:INTROENERGYINEEQUALITYSCHEMATIC},
we would have
an additional error integral of the form
$
		C_*
		\int_{t}^1
			\frac{\mathbb{E}_N^2(s)}{s^{1 + \epsilon}}
		\, ds
$.
By virtue of Gr\"{o}nwall's lemma, 
this integral would lead to dramatically worse a priori estimates, 
which would in turn prevent us from closing
our bootstrap argument.
\end{itemize}

\begin{remark}[The size of $\blowupexp$] \label{rem:A}
It is possible, in principle, to compute how large $\blowupexp$ has to be for the above proof to work; 
one simply needs to derive an explicit upper bound for the constant $C_*$ 
on RHS~\eqref{E:INTROENERGYINEEQUALITYSCHEMATIC}. 
We will provide an outline of how to estimate $C_*$ (and thus $\blowupexp$), although we do not provide an explicit estimate.
To shorten the discussion, we will restrict our attention to the Einstein-vacuum equations, i.e., we will assume that $\psi = 0$.
In short, the constant $C_*$ can be controlled by the number of borderline terms in the top-order energy estimates 
and elliptic estimates and the size of the coefficients in front of these terms. 
More precisely, the borderline terms in the energy estimates for the connection coefficients
and second fundamental form are generated by the terms
on RHSs~\eqref{E:SECONDFUNDBORDER}, \eqref{E:CONNECTIONCOEFFICIENTBORDER}, and \eqref{E:MOMENTUMBORDER}.
The main terms driving the size of $C_*$ are the top-order ones with coefficients of size $\approx s^{-1}$ along
$\Sigma_s$, 
e.g., the term $\partial^{\iota} (n-1) \cdot \widetilde{k}$ on RHS~\eqref{E:SECONDFUNDBORDER}
and the terms
$
n
			\cdot
			\widetilde{k}
			\cdot
			\partial^{\iota} \upgamma
$
and	
$\widetilde{k} \cdot \partial^{\iota} \vec{e} n$
on RHS~\eqref{E:CONNECTIONCOEFFICIENTBORDER}.
These terms lead to borderline error integrals in the energy estimates, such as the integral 
$
		\int_{t}^1
			\int_{\Sigma_s}
			s^{2(\blowupexp + 1)} \cdot k \cdot \partial^{\iota} \upgamma \cdot \partial^{\iota} \upgamma
			\, dx
		\, ds
		$ 
mentioned above.
In each borderline term, the coefficient of the top-order term, specifically $k$ in the previous integral,
can be bounded in the norm $\| \cdot \|_{L^{\infty}(\Sigma_s)}$ by $(|\widetilde{q}_I|+C\epsilon) s^{-1}$ 
for some $I$, where $\widetilde{q}_I$ can be any of the background exponents and $C \epsilon$ can be as small as desired, 
by taking the initial data on $\Sigma_1$ to be sufficiently close to the Kasner data. 
More precisely, decomposing $k=\widetilde{k}+(k-\widetilde{k})$, we see that the factor
	$|\widetilde{q}_I| s^{-1}$ is generated by $\widetilde{k}$ (see \eqref{Kasnersol}), while the factor 
	$C\epsilon s^{-1}$ is generated by the bound \eqref{E:INTROKBOUND}.
	Since $|\widetilde{q}_I|<1$, by counting all borderline terms, we could
	crudely bound the contribution of these terms to the
	constant $C_*$ by
	$\leq \text{(number of borderline terms)}+C\varepsilon$. The sum of all the corresponding
	top-order error integrals would then be bounded by 
	$
		C_*
		\int_{t}^1
			\frac{\mathbb{E}_N^2(s)}{s}
		\, ds
$.
	We also stress that one encounters borderline terms in the elliptic estimates
	for the top-order derivatives of $n$
	(see, for example, the $C_*$-involving term on RHS~\eqref{PRECISE.n.high.est}),
	and that these borderline terms propagate into the top-order energy estimates via
	terms such as the one $\partial^{\iota} (n-1) \cdot \widetilde{k}$ mentioned above.
	In particular, these terms affect the size of $C_*$.
	We also note that the arguments given here allow for 
	the possibility that $C_*$ might increase with respect to $\mydim$.
	Finally, we note that the estimate for $C_*$ sketched here is not necessarily optimal.
	In fact, in the near-FLRW regime (where all Kasner exponents are nearly equal), 
	the last two authors \cite{RodSp1,RodSp2} showed that striking cancellations take place, 
	and $\blowupexp$ can in fact be taken very small, 
	i.e., $C_*=C\varepsilon$. 
	It is not known to us
	whether such cancellations exist for perturbations of highly anisotropic background Kasner solutions.
\end{remark}

\subsection{Applicability of the method}\label{subsec:meth.app}
\subsubsection{Polarized $\mathbb{T}^2$-symmetry}
\label{SSS:APPLICATIONSTOPOLARIZEDT2SYMMETRY}
We already mentioned that Kasner-like singularities have been constructed \cite{ABIL,IK} 
for the Einstein-vacuum equations in $1+3$ dimensions within the polarized $\mathbb{T}^2$-symmetry class.
This symmetry class contains vacuum spacetimes with two orthogonal, spacelike, Killing vectorfields 
$X,Y$ that commute, and it is more general than the polarized Gowdy-class in the sense that 
the twist constants, which measure the obstruction to the integrability of the $2$-dimensional orthogonal planes to $X,Y$, 
do not have to vanish. 
It turns out that polarized $\mathbb{T}^2$-symmetric solutions
can be can be viewed as special cases of
polarized $U(1)$-symmetric solutions in which one extra symmetry is present. 
This fact is not immediately apparent in the sense that 
our definition of polarized $U(1)$-symmetry
(recall the discussion in Sect.\,\ref{subsubsec:U1.data})
requires a spacelike Killing field to be hypersurface-orthogonal,
whereas the definition of polarized $\mathbb{T}^2$-symmetry does not
refer to hypersurface orthogonality.
Nevertheless, for polarized $\mathbb{T}^2$-symmetric solutions,
it is always possible\footnote{We are grateful to the authors of \cite{ABIO} for pointing this out to us.}
to construct coordinates such that one of the twist constants vanishes 
and such that one of the spatial coordinate partial derivative vectorfields associated
to the $\mathbb{T}^2$-symmetry is Killing \emph{and} hypersurface-orthogonal; see \cite[Section~2.2]{ABIO}, where this 
coordinate Killing vectorfield is denoted by ``$\partial_x$.'' Given such coordinates, 
the corresponding solutions can indeed be viewed as special cases of
polarized $U(1)$-symmetric solutions in which one extra symmetry is present. 
Hence, our results on polarized $U(1)$-symmetric solutions
imply, as a special case, that all (singular) Kasner solutions 
are also stable (as solutions to the Einstein-vacuum equations in $1+3$ dimensions)
near their Big Bangs under
polarized $\mathbb{T}^2$-symmetric perturbations. 
Here, by ``stable,'' we mean that the results of
Theorem~\ref{thm:precise.U1} hold for the near-Kasner polarized $\mathbb{T}^2$-symmetric solutions,
where the hypersurface-orthogonal Killing vectorfield
``$\partial_3$'' from Theorem~\ref{thm:precise.U1}
corresponds to the vectorfield ``$\partial_x$'' from \cite{ABIO}.
We refer to Remark~\ref{R:SHARPERASYMPTOTICS} for further discussion of
polarized $\mathbb{T}^2$-symmetric solutions and their asymptotics near the singularity.

\subsubsection{Potential further applications}
\label{SSS:POTENTIALFURTHERAPPLICATIONS}
Our approach could likely be adapted to prove stable Big Bang formation in other models that are not, strictly speaking, covered in the present paper. We mention here some interesting cases. 

{\bf $\bullet$ The stiff-fluid model, for $\mydim \geq 3$.}
This matter model reduces to the scalar field matter model in the case of vanishing vorticity. 
In \cite{RodSp2}, stable Big Bang formation was proved in the special case $\mydim = 3$ for the background
FLRW solution, in which
$\widetilde{q}_1=\widetilde{q}_2=\widetilde{q}_3 = 1/3$,
and the presence of matter is needed to ensure the validity of the Kasner exponent constraints \eqref{sumpi}.

{\bf $\bullet$ Perturbing around fixed, non-explicit, backgrounds/ solutions with large spatial
dependence.} The stability problems that we study in detail in this paper
concern perturbations of
explicit, \emph{spatially homogeneous}, singularity-forming solutions. 
However, one could try to use our methods to 
study perturbations of any of the singular solutions constructed in the works that we 
mentioned in Sect.\,\ref{subsec:rel.works}, including solutions with spatial dependence. 
From an analytical point of view, when dealing with background solutions that exhibit spatial dependence,
one encounters additional technical difficulties in the derivation of various estimates. 
In particular, when estimating the perturbed solution's higher spatial derivatives, 
one must control terms in which derivatives hit the background solution 
and thus do not have to be small (whereas in the present article, the background solution's spatial derivatives vanish). 
Nevertheless, our method is still potentially applicable. To simplify the approach, 
one could consider data with large spatial derivatives given on
a hypersurface close to the expected singularity, that is, on $\Sigma_{t_{Data}}$, with 
$t_{Data}$ larger than but close to $0$ 
(where $t_{Data}$ has to be chosen to be small in a manner that depends on the largeness of the data);
the point is that the smallness of the amount of time for which one needs to control the solution
can compensate for the largeness of the data. Moreover, by applying this philosophy to the setup of the present
paper, one could produce open sets of singularity-forming solutions that have ``substantial $x$-dependence.''

{\bf $\bullet$ Black hole interior.} There are numerous examples of black hole spacetimes containing a spacelike singularity, such as the classical Oppenheimer--Snyder model of gravitational collapse or the solutions 
detected by Christodoulou in his aforementioned studies \cite{Christ2,Christ3}
of the spherically symmetric Einstein-scalar field model. For the latter solutions, it would be interesting to see whether 
Kasner-like blowup holds for perturbations of solutions 
(in some class other than spherical symmetry, which was handled in \cite{Christ2,Christ3}). 
Compared to our work here, 
the difference in topology might pose additional analytical difficulties. 
Moreover, one would have to grapple with the question of whether 
the initial data given only in the interior of a black hole could arise as induced data of solutions to the global Cauchy problem.

\subsection{Paper outline}
In Sect.\,\ref{sec:setup}, we introduce our analytic framework, including the reduced solution variables
and a formulation of the Einstein-scalar field equations relative to CMC-transported spatial
coordinates with a Fermi--Walker transported orthonormal frame. 
In Sect.\,\ref{sec:Boots}, we define various norms and introduce 
our bootstrap assumptions for perturbations of Kasner solutions. Our bootstrap assumptions
involve $t$-weighted $L^{\infty}$ norms at the low derivative levels
and $t$-weighted Sobolev norms at the high derivative levels, where the $t$-weights 
are much smaller at the high derivative levels
(which corresponds to our allowing for very singular high order derivatives as $t \downarrow 0$).
In Sect.\,\ref{sec:basic.est}, we provide standard Sobolev and interpolation estimates that
we will use to control various error terms when we derive our main estimates.
In Sect.\,\ref{sec:mainest}, we derive the core estimates 
at both the low and high derivative levels. 
These estimates in particular yield a strict improvement of the bootstrap assumptions. 
Finally, in Sect.\,\ref{sec:sol}, we use the estimates of Sect.\,\ref{sec:mainest}
to prove our main theorems exhibiting the stability of the Kasner Big Bang singularity.

\subsection{Notation and conventions}
\label{SS:NOTATIONANDCONVENTIONS}
In the rest of the paper, we use the following notation and conventions.
\begin{itemize}
\item $\lbrace x^i \rbrace_{i=1,\cdots,\mydim}$ denote standard local spatial coordinates on $\mathbb{T}^{\mydim}$
that are transported in the sense described in Sect.\,\ref{SSS:SETUPFORMOFMETRIC},
and $\partial_i := \frac{\partial}{\partial x^i}$ denote the corresponding
spatial partial derivative vectorfields. The frame $\lbrace \partial_i \rbrace_{i=1,\cdots,\mydim}$
extends to a smooth global holonomic frame on $\mathbb{T}^{\mydim}$, and by abuse of notation,
we denote the globally defined vectorfields by the symbols $\partial_i$,
even though the coordinate functions are not globally defined.

\item Lowercase Latin ``spatial'' indices such as $a,b,i,j$ 
range over $\{1,\cdots,\mydim\}$ and correspond to the transported spatial coordinates 
$x^1,\cdots,x^{\mydim}$
(see Sect.\,\ref{sec:setup}).
For example, $g_{ij} := g(\partial_i,\partial_j)$.
Lowercase Greek ``spacetime'' indices such as $\alpha,\beta,\mu,\nu$ range over $\{0,1,\cdots,\mydim\}$
and usually correspond to the spacetime coordinates $t,x^1,\cdots,x^{\mydim}$, 
where the ``$0$'' index corresponds to $t$.
For example,
${\bf g}_{0i} = {\bf g}_{ti} := {\bf g}(\partial_t,\partial_i)$.
In a few instances, $\lbrace e_{\alpha} \rbrace_{\alpha=0,\cdots,\mydim}$
denotes an orthonormal spacetime frame, i.e.,
$\bf{g}(e_{\alpha},e_{\beta}) = \mathbf{m}_{\alpha \beta}$,
where $\mathbf{m}_{\alpha \beta} := \mbox{\upshape diag}(-1,1,\cdots,1)$.
Uppercase Latin ``spatial frame'' indices such as $A,B,I,J$ 
range over $\{1,\cdots,\mydim\}$ and, \textbf{with one exception}, 
correspond to the orthonormal spatial
frame $\lbrace e_I \rbrace_{I=1,\cdots,\mydim}$ 
or co-frame $\lbrace \oe^I \rbrace_{I=1,\cdots,\mydim}$ (see Sect.\,\ref{sec:setup}).
For example, $k_{IJ} := k(e_I,e_J) = k_{cd}e_I^c e_J^d$.
The exception is that for background Kasner tensors, uppercase
Latin indices denote their components with respect to the 
background Kasner orthonormal frame $\lbrace \widetilde{e}_I \rbrace_{I=1,\cdots,\mydim}$;
see Remark~\ref{R:NOTCOMPONENTSOFTENSORRELATIVETOFRAME} for further discussion.
We used primed indices, such as $a'$, in the same way we use their non-primed counterparts.

\item We use Einstein summation for repeated indices, including
frame indices. We stress that no metric is directly involved in contractions involving the frame indices. 
For example, $k_{IC}\upgamma_{CJB}$ stands for $\sum_{C=1}^{\mydim} k_{IC}\upgamma_{CJB}$,
where $\mydim$ is the number of spatial dimensions.

\item  If $\mathbf{X}$ is a vectorfield and $f$ is a scalar function,
then $\mathbf{X} f := \mathbf{X}^{\alpha} \partial_{\alpha} f$
denotes the derivative of $f$ in the direction $\mathbf{X}$.

\item $\lbrace dx^i \rbrace_{i=1,\cdots,\mydim}$ denotes the
globally defined basis-dual co-frame 
of
$\lbrace \partial_i \rbrace_{i=1,\cdots,\mydim}$,
i.e., $dx^i(\partial_j) := \updelta_j^i$, with $\updelta_j^i$ the Kronecker delta.

\item No summation of underlined terms. In a handful of key terms
that explicitly involve the Kasner exponents,
we will not use Einstein summation convention for some of the indices.
More precisely, in a given product, 
whenever there is no summation over a particular index,
we indicate this by underlining all instances of that index in
the product.
For example, there is no summation over the index $I$ in the following expression:
$\frac{\widetilde{q}_{\underline{I}}}{t}\upgamma_{\underline{I}JB}$.

\item If $\mathbf{X}$ and $\mathbf{Y}$ are vectorfields, then
$\mathbf{X} \mathbf{Y} f := \mathbf{X}^{\alpha} \partial_{\alpha} (\mathbf{Y}^{\beta} \partial_{\beta} f)$.
Similarly, if 
${\bf T}$ is a tensorfield and ${\bf D}$ denotes the Levi-Civita connection of ${\bf g}$,
then ${\bf D}_{\mathbf{X}} {\bf T} := \mathbf{X}^{\alpha} {\bf D}_{\alpha} {\bf T}$
and 
${\bf D}_{\mathbf{X}} {\bf D}_{\mathbf{Y}} {\bf T} := 
\mathbf{X}^{\alpha} {\bf D}_{\alpha} (\mathbf{Y}^{\beta} {\bf D}_{\beta} {\bf T})$.
In addition,
${\bf D}_{\mathbf{X} \mathbf{Y}}^2 {\bf T} := 
\mathbf{X}^{\alpha} \mathbf{Y}^{\beta} {\bf D}_{\alpha} {\bf D}_{\beta} {\bf T}$.
Note that in the latter expression, contractions are taken after covariant differentiation
and thus generally, ${\bf D}_{\mathbf{X} \mathbf{Y}}^2 {\bf T} \neq {\bf D}_{\mathbf{X}} {\bf D}_{\mathbf{Y}} {\bf T}$.

\item If $\mathbf{X}$ and $\mathbf{Y}$ are vectorfields, then 
${\bf g}(\mathbf{X},\mathbf{Y}) := {\bf g}_{\alpha \beta} \mathbf{X}^{\alpha} \mathbf{Y}^{\beta}$.
We use similar notation for contractions of higher-order tensorfields against vectorfields.
For example,
$${\bf Riem}(\mathbf{W},\mathbf{X},\mathbf{Y},\mathbf{Z})
:=
{\bf Riem}_{\alpha \beta \gamma \delta}\mathbf{W}^{\alpha} \mathbf{X}^{\beta} \mathbf{Y}^{\gamma} \mathbf{Z}^{\delta}.
$$

\item $\iota$ denotes a spatial multi-index. That is, for some positive integer $m$, 
	$\iota = (a_1,\cdots,a_m)$, where $a_i \in \lbrace 1, \cdots, \mydim \rbrace$ for $1 \leq i \leq m$
	and $|\iota| := m$ denotes the length of the index. 
	$\partial^{\iota} := \partial_{a_1} \cdots \partial_{a_m}$ denotes the corresponding
	$m^{th}$-order differential operator involving repeated differentiation
	with respect to the transported spatial coordinate partial derivative vectorfields.  
	$\iota_1 \cup \iota_2 = \iota$
	means that for some $r$ with $1 \leq r \leq m$,
	we have $\iota_1 = (a_{i_1},\cdots,a_{i_r})$ and $\iota_2 = (a_{i_{r+1}},\cdots,a_{i_m})$,
	where $(i_1,\cdots,i_m)$ is a permutation of $(1,\cdots,m)$
	such that $i_1 < i_2 < \cdots <i_r$ and $i_{r+1} < i_{r+2} < \cdots< i_m$.
	$\iota_1 \cup \iota_2 \cup \iota_3 = \iota$,
	$\iota_1 \cup \iota_2 \cup \iota_3 \cup \iota_4 = \iota$,
	etc.\ have analogous meanings.
	
	Note that our multi-index convention in $\mydim$ spatial dimensions differs from the more standard one,
	in which multi-indices $\alpha$ satisfy $\alpha\in\mathbb{N}^{\mydim}$. 
	For instance, in the more standard notation, in $3$ spatial dimensions,
	$\alpha := (1,1,0)$ corresponds to $\partial^\alpha=\partial_1\partial_2$, 
	whereas with our multi-index notation in $3$ spatial dimensions, 
	$\iota=(1,1)$ corresponds to $\partial^\iota=\partial_1\partial_1$.
\end{itemize}

\noindent \textbf{Parameters}

\begin{itemize}
	\item $\blowupexp \geq 1$ denotes a ``time-weight exponent parameter'' that is featured in 
		the high order solution norms from Definition~\ref{D:SOLUTIONNORMS}.
		To close our estimates, we will choose $\blowupexp$ to be large enough to overwhelm
		various universal constants $C_*$ (see below). This corresponds
		to our use of high order energies featuring large powers of $t$,
		which leads to weak high order energies near $t=0$.
	\item $0 < q < 1$ is a constant, fixed throughout the proof, 
		that bounds the crucial quantity 
		$$\mathop{\max_{I,J,B=1,\cdots,\mydim}}_{I < J} \{|\widetilde{q}_B|,\widetilde{q}_I+\widetilde{q}_J-\widetilde{q}_B\}.$$
	\item $\upsigma > 0$ is a small constant, fixed throughout the proof, 
		that we use to simplify the proofs of various estimates
		that ``have room in them.''
	\item $q$ and $\upsigma$ are constrained by  \eqref{sigma,q}.
	\item $N_0 \geq 1$ roughly corresponds to the number of derivatives of the solution that we control
		in $L^{\infty}$ (the precise derivative count depends on the solution variable -- 
		see Definition~\ref{D:SOLUTIONNORMS}).
	\item $N$ denotes the maximum number of times that we commute the equations with spatial derivatives
		(e.g., $k \in H^N(\Sigma_t)$ and $n \in H^{N+1}(\Sigma_t)$-- see Definition~\ref{D:SOLUTIONNORMS}).
		To close our estimates, we will choose $N$ to be sufficiently large 
		in a (non-explicit) manner that depends on $N_0, \blowupexp, \mydim, q,$ and $\upsigma$.
	\item $\updelta > 0$ is a small $(N,\mydim)$-dependent parameter that is allowed to vary from line to line
		and that is generated by the estimates of Lemma~\ref{lem:basic.ineq}.
		We use the convention that a sum of two $\updelta$'s is another $\updelta$.
		The only important feature of $\updelta$ that we exploit throughout the paper is the following:
		at fixed $\mydim$, we have $\lim_{N \to \infty} \updelta = 0$.
		In particular, if $\blowupexp$ is also fixed, then
		$\lim_{N \to \infty} \blowupexp \updelta = 0$.
	\item $\varepsilon$ is a small ``bootstrap parameter'' that, in our bootstrap argument,
		bounds the size of the solution norms; see \eqref{Boots}.
		The smallness of $\varepsilon$ needed to close the estimates is allowed
		to depend on the parameters $N, N_0, \blowupexp, \mydim, q,$ and $\upsigma$.
\end{itemize}

\noindent \textbf{Constants}
\begin{itemize}
	\item $C$ denotes a positive constant that is free to vary from line to line.
		$C$ can depend on $N, N_0, \blowupexp, \mydim, q,$ and $\upsigma$,
		but it can be chosen to be independent of
		all $\varepsilon > 0$ that are sufficiently small
		in the manner described just above.
	\item $C_*$ denotes a positive constant that is free to vary from line to line
		and that can depend on $\mydim$. Like $C$, 
		$C_*$ can be chosen to be independent of
		all $\varepsilon > 0$ that are sufficiently small
		in the manner described just above.
		However, unlike $C$, $C_*$ can be chosen to be \textbf{independent} of
		$N, N_0,$ and $\blowupexp$. $C_*$ can be chosen to be independent of $q$, and $\upsigma$,
		but that is less important in the sense that we view $q$ and $\upsigma$ to be
		fixed throughout the article.
		For example, $1 + C N! \varepsilon \leq C_*$
		while $N! = C$ and $N!/\upsigma = C$,
		where $C$ and $C_*$ are as above.
	\item We write $v \lesssim w$ to indicate that $v \leq C w$, with $C$ as above.
	\item We write $v = \mathcal{O}(w)$ to indicate that $|v| \leq C |w|$, with $C$ as above.
\end{itemize}

\subsection{Acknowledgments}
G.F.\ is supported by the \texttt{ERC grant 714408 GEOWAKI}, 
under the European Union's Horizon 2020 research and innovation program.
J.S.\ gratefully acknowledges support from from NSF grant \# 2054184,
from NSF CAREER grant \# 1914537, and from a Chancellor's Faculty Fellowship administered by Vanderbilt University.
I.R.\ gratefully acknowledges support from NSF grant \# DMS 2005464.
The authors would like to thank the anonymous referees for their careful reading of the original
manuscript and their insightful comments, 
which were of immense value during the revision process.

\section{Analytic setup and the formulation of the Einstein-scalar field equations}\label{sec:setup}
In this section, we introduce the framework that we will use to study perturbations of Kasner solutions.
In particular, we provide the formulation of the Einstein-scalar field equations 
that we will use to derive estimates.

\subsection{The reduced equations relative to a CMC-transported orthonormal frame}\label{subsec:redEE}
Our main goal in this section is to prove Proposition~\ref{P:redeq},
which provides the formulation of the Einstein-scalar equations that forms the starting point for our analysis.
We start by providing some basic constructions.

\subsubsection{The form of the spacetime metric, the lapse, and the transported spatial coordinates}
\label{SSS:SETUPFORMOFMETRIC}
Relative to CMC-transported spatial coordinates on a slab $(t,x) \in (T,1] \times \mathbb{T}^{\mydim}$,
the spacetime metric ${\bf g}$ takes the form:
\begin{align} \label{E:SPACETIMEMETRICRELATIVETOTRANSPORTEDCOORDINATES}
	{\bf g}
	& = - n^2 dt \otimes dt
				+
				g_{ab} dx^a \otimes dx^b,
\end{align}
where $n$ is the lapse and $g$ is the first fundamental form of the constant-time
slice $\Sigma_t := \lbrace (s,x) \in (T,1] \times \mathbb{T}^{\mydim} \ | \ s = t \rbrace$,
i.e., $g$ is the Riemannian metric induced by ${\bf g}$ on $\Sigma_t$.
Here and throughout, $t$ is the time function. In Sect.\,\ref{subsubsec:conn.coeff},
we state our CMC normalization condition for $t$.
The spatial coordinates $\lbrace x^i \rbrace_{i=1,\cdots,\mydim}$ are said to be
``transported'' because $n^{-1} \partial_t x^i = 0$, where $n^{-1} \partial_t$
is the future-directed unit normal to $\Sigma_t$.

\subsubsection{The orthonormal frame}
\label{SSS:ORTHONORMALFRAME}
Our proofs fundamentally rely on expressing Einstein's equations relative to 
an orthonormal frame:
\begin{align}\label{frame}
e_0
& =n^{-1}\partial_t,
& 
e_I 
& = e_I^c \partial_c,
&
I 
&
= 1,\cdots,\mydim,
\end{align}
where $e_0$ is the future-directed unit normal to $\Sigma_t$ 
(in particular, ${\bf g}(e_0,e_0) = -1$ and ${\bf g}(e_0,X) = 0$ for all $\Sigma_t$-tangent vectorfields $X$),
the ``spatial'' frame $\lbrace e_I \rbrace_{I=1,\cdots,\mydim}$ is $\Sigma_t$-tangent
and normalized by:
\begin{align} \label{E:ORTHONORMALSPATIALFRAMECONDITION}
	g(e_I,e_J)
	& = \updelta_{IJ},
	&
	\updelta_{IJ} = \mbox{ the Kronecker delta},
\end{align}
and the spatially-globally defined (see Sect.\,\ref{SS:NOTATIONANDCONVENTIONS})
scalar functions $\lbrace e_I^i \rbrace_{i=1,\cdots,\mydim}$ in \eqref{frame}
are the components of
$e_I$ relative to the transported spatial coordinates. 
Just below, we will describe how we construct the spatial frame.
We let $\lbrace \oe^I \rbrace_{I=1,\cdots,\mydim}$ 
denote the corresponding $\Sigma_t$-tangent one-forms that are a co-frame for the spatial frame
$\lbrace e_I \rbrace_{I=1,\cdots,\mydim}$, defined by: 
\begin{align} \label{E:DUALORTHONORMAL}
	\oe^I(e_J)
	& = \updelta_J^I,
\end{align}
where $\updelta_J^I$ is the Kronecker delta.
Note that $\oe^I = \oe_a^I dx^a$,
where the spatially-globally defined scalar functions 
$\lbrace \oe_i^I \rbrace_{i=1,\cdots,\mydim}$ are the components of
$\oe^I$ relative to the transported spatial coordinates.
Thus, we have:
\begin{align} \label{E:FRAMECOFRAMECONTRACTIONKRONECKER}
\oe_a^I e_J^a 
& = \updelta_J^I,
&
\oe_j^A e_A^i
& = \updelta_j^i,
&
I,J,i,j & = 1,\cdots,\mydim.
\end{align}
Moreover, since $\lbrace e_I \rbrace_{I=1,\cdots,\mydim}$ is orthonormal, 
$\oe^I$ is in fact the $g$-dual of $e_I$, that is:
\begin{align} \label{E:BASISDUALISMETRICDUAL}
\oe_i^I & = 
g_{ia} e_I^a,
&
I,i & = 1,\cdots,\mydim.
\end{align}
We also note that from \eqref{E:FRAMECOFRAMECONTRACTIONKRONECKER}
and the relation $e_I = e_I^c \partial_c$, it follows that:
\begin{align} \label{E:COORDINATEVECTORFIELDSINTERMSOFFRAMEVECTORFIELDS}
	\partial_i & = \oe_i^C e_C,
	&
i
&
= 1,\cdots,\mydim.
\end{align}

We now describe our construction of a spatial frame.
There is freedom in the construction; see Remark~\ref{R:FRAMEFREEDOM}.
In Sect.\,\ref{SS:CONSTRUCTIONOFTHEINTIALORTHONORMALFRAME},
we use the Gram--Schmidt process to construct an initial 
orthonormal spatial frame on $\Sigma_1$
that is suitable for proving our main results. 
Given this frame on $\Sigma_1$, we propagate it 
to slabs of the form $(T,1] \times \mathbb{T}^{\mydim}$
by solving the propagation equations:
\begin{align}\label{frame.prop}
{\bf D}_{e_0} e_I=n^{-1} (e_I n) e_0,
\end{align}
where ${\bf D}$ is the Levi-Civita connection of ${\bf g}$.

From equation \eqref{frame.prop}, it follows that the scalar functions 
$\lbrace e_I^i \rbrace_{I,i=1,\cdots,\mydim}$ satisfy a system of transport equations;
see \eqref{dt.omega}.
It is straightforward to check 
(for example, with the help of equation \eqref{De0e0})
that if ${\bf g}$ is $C^1$ on $(T,1] \times \mathbb{T}^{\mydim}$
and the initial spatial frame on $\Sigma_1$ is orthonormal and $C^1$,
then the frame $\lbrace e_I \rbrace_{I=1,\cdots,\mydim}$ 
obtained by propagating the initial frame via the transport equations 
\eqref{frame.prop} 
is orthonormal and tangent to $\Sigma_t$ for $t \in (T,1]$.
In particular, we have:
\begin{align} \label{E:FRAMEORTHONORMALITY}
	{\bf g}(e_{\alpha},e_{\beta}) 
	& = {\bf m}_{\alpha \beta},
	&
	\alpha, \beta
	& = 0,1,\cdots,\mydim,
\end{align}
where ${\bf m}_{\alpha \beta} := \mbox{\upshape diag}(-1,1,\cdots,1)$,
and:
\begin{align} \label{E:FRAMEISSIGMATANGENT}
	e_I t
	& = 0,
	&
	I = 1,\cdots,\mydim.
\end{align} 
Moreover, relative to the orthonormal frame $\lbrace e_{\alpha} \rbrace_{\alpha=0,1,\cdots,\mydim}$,
with ${\bf m}^{\mu \nu} := \mbox{\upshape diag}(-1,1,\cdots,1)$
and $\updelta^{IJ}$ the Kronecker delta, 
we have:
\begin{align} \label{E:INVERSESPACETIMEMETRICINTERMSOFFRAME}
	{\bf g}^{-1}
	& = {\bf m}^{\gamma \delta} e_{\gamma} \otimes e_{\delta},
	&
	g^{-1}
	& = \updelta^{CD} e_C \otimes e_D.
\end{align}
In addition, differentiating \eqref{E:FRAMEORTHONORMALITY}, we find that:
\begin{align} \label{E:SPACETIMECONNECTIONCOEFFICIENTSANTISYMMETRY}
	{\bf g}({\bf D}_{e_{\alpha}} e_{\beta}, e_{\gamma})
	& = - {\bf g}(e_{\beta}, {\bf D}_{e_{\alpha}} e_{\gamma}),
\end{align}	
which in particular implies that:
\begin{align} 	\label{E:VANISHINGSPACETIMECONNECTIONCOEFFICIENTS}
	{\bf g}({\bf D}_{e_{\alpha}} e_{\beta}, e_{\gamma}) 
	& = 0, 
	&& \mbox{if } \beta = \gamma.
\end{align}
We also note the following identity, which is straightforward
to verify using the form \eqref{E:SPACETIMEMETRICRELATIVETOTRANSPORTEDCOORDINATES}
of the metric:
\begin{align}\label{De0e0}
{\bf D}_{e_0}e_0
= n^{-1} (e_C n) e_C.
\end{align}

\begin{remark}[Fermi--Walker transport]
\label{R:FERMIWALKER}
When the frame initial data are $\Sigma_1$-tangent,
equation \eqref{frame.prop} is equivalent to the well-known Fermi--Walker transport equation for $e_I$
along the integral curves of $e_0$ (which, up to re-parametrization, are the same as the integral curves of $\partial_t$).
We remark that the ``standard'' Fermi--Walker transport equation is
${\bf D}_{e_0 } e_I =  n^{-1} (e_I n) e_0 - {\bf g}(e_I,e_0) n^{-1}(e_Cn)e_C$,
and that we have omitted the term
$- {\bf g}(e_I,e_0) n^{-1} (e_C n)e_C$
from RHS~\eqref{frame.prop}. This term vanishes 
in the present context because our frame initial data will verify 
${\bf g}(e_I,e_0)|_{\Sigma_1} = 0$
and
${\bf g}(e_I,e_J)|_{\Sigma_1} = \updelta_{IJ}$,
and these orthogonality conditions are propagated by solutions to equation \eqref{frame.prop}.
\end{remark}

\subsubsection{The second fundamental form, the CMC condition, and the connection coefficients}\label{subsubsec:conn.coeff}
Relative to the transported spatial coordinates  
$\lbrace x^i \rbrace_{i=1,\cdots,\mydim}$ on $\Sigma_t$, the second fundamental form $k$ of $\Sigma_t$ 
is the $\Sigma_t$-tangent tensorfield with components $k_{ij} := -{\bf g}({\bf D}_{\partial_i} e_0, \partial_j)$.
Hence,
the components of the second fundamental form with respect to the frame 
$\lbrace e_I \rbrace_{I=1,\cdots,\mydim}$ are:
\begin{align}\label{K}
k_{IJ} =
-{\bf g}({\bf D}_{e_I}e_0,e_J)
=
k_{JI},
\end{align}
where the symmetry property $k_{IJ} = k_{JI}$ is a well-known consequence
of the torsion-free property of ${\bf D}$ and the fact that the commutators $[e_I,e_J]$ are $\Sigma_t$-tangent
(and thus orthogonal to $e_0$).
Note that \eqref{E:FRAMEORTHONORMALITY}, \eqref{E:VANISHINGSPACETIMECONNECTIONCOEFFICIENTS}, 
and \eqref{K} imply that:
\begin{align} \label{E:DEIE0FRAMEEXPANSION}
	{\bf D}_{e_I} e_0
	& = 
			- k_{IC} e_C.
\end{align}

We now normalize the time function $t$ according to the CMC condition:
\begin{align}\label{trk}
\mathrm{tr}k 
& := k_{\ a}^a := (g^{-1})^{ab} k_{ab} = k_{AA} = -\frac{1}{t}.
\end{align}
It is well-known that \eqref{trk} leads to an elliptic equation for the lapse $n$ (see \eqref{n.eq}),
which means in particular that our gauge involves an \emph{infinite} speed of propagation.

In our analysis, we also study the spatial connection coefficients of the frame $\lbrace e_I \rbrace_{I=1,\cdots,\mydim}$,
which are defined by:
\begin{align} \label{gamma}
\upgamma_{IJB} :={\bf g}({\bf D}_{e_I}e_J,e_B)=g(\nabla_{e_I} e_J,e_B).
\end{align}
In \eqref{gamma} and throughout,
$\nabla$ denotes the Levi-Civita connection of $g$.
Note that 
\eqref{E:FRAMEORTHONORMALITY},
\eqref{E:SPACETIMECONNECTIONCOEFFICIENTSANTISYMMETRY},
\eqref{K}, and \eqref{gamma}
imply that:
\begin{align} \label{E:DEIEJFRAMEEXPANSION}
	{\bf D}_{e_I} e_J
	& = 
			-
			k_{IJ} e_0
			+
			\upgamma_{IJC} e_C,
	&
	\nabla_{e_I} e_J
	& = 
	\upgamma_{IJC} e_C.
\end{align}
Finally, by differentiating the relation ${\bf g}(e_J,e_B) = \updelta_{JB}$ with ${\bf D}_{e_I}$,
we deduce the following antisymmetry property:
\begin{align}\label{antisymmetricgamma}
\upgamma_{IJB} = -\upgamma_{IBJ}.
\end{align}

\subsubsection{Curvature tensors}
\label{SSS:CURVATURETENSORS}
Our sign conventions for the 
Riemann curvature ${\bf Riem}$ of ${\bf g}$,
the Ricci curvature ${\bf Ric}$ of ${\bf g}$,
and the scalar curvature ${\bf R}$ of ${\bf g}$,
are as follows
relative to the orthonormal frame $\lbrace e_{\alpha} \rbrace_{\alpha=0,1,\cdots,\mydim}$
constructed in Sect.\,\ref{SSS:ORTHONORMALFRAME},
where ${\bf m}^{\alpha \beta}$ is as in
\eqref{E:INVERSESPACETIMEMETRICINTERMSOFFRAME}:
\begin{subequations}
\begin{align} \label{E:RIEMSPAECTIME}
	{\bf g}
	\left({\bf D}^2_{e_{\alpha} e_{\beta}} e_\nu 
		- 
		{\bf D}^2_{e_{\beta} e_{\alpha}} e_\nu,
		e_\mu \right) 
	& :=
	{\bf Riem}(e_{\alpha},e_{\beta},e_{\mu},e_{\nu}),
		\\
	{\bf Ric}(e_{\alpha},e_{\beta})
	& :=
	{\bf m}^{\mu \nu}
	{\bf Riem}(e_{\alpha},e_{\mu},e_{\beta},e_{\nu}),
		\label{E:RICCISPAECTIME}\\
	{\bf R }
	& := {\bf m}^{\mu \nu} {\bf Ric}(e_{\mu},e_{\nu}).
	\label{E:SCALARCURVATURESPAECTIME}
\end{align}
\end{subequations}
Our sign conventions for the curvature of tensors of $g$, namely its Riemann curvature $Riem$, Ricci curvature $Ric$, 
and scalar curvature $R$, are analogous to the ones in 
\eqref{E:RIEMSPAECTIME}--\eqref{E:SCALARCURVATURESPAECTIME}.

\subsubsection{The reduced equations}
\label{SSS:REDUCEDEQUATIONS}
In the next proposition, we provide the PDEs that we use to study perturbations of generalized Kasner solutions.

\begin{proposition}[The reduced Einstein-scalar field equations relative to CMC-transported spatial coordinates
and a Fermi-Walker transported orthonormal frame]
\label{P:redeq}
Let $T \in (0,1)$, and let
$({\bf g},\psi)$ respectively be a Lorentzian metric and a scalar function on the manifold $(T,1] \times \mathbb{T}^{\mydim}$. 
Assume that $(T,1] \times \mathbb{T}^{\mydim}$ is equipped with 
a CMC time function $t$ and transported spatial coordinates $\lbrace x^i \rbrace_{i=1,\cdots,\mydim}$
such that the level sets $\Sigma_t = \lbrace t \rbrace \times \mathbb{T}^{\mydim}$
are CMC hypersurfaces normalized by \eqref{trk}, as is described in Sects.\,\ref{SSS:SETUPFORMOFMETRIC}--\ref{subsubsec:conn.coeff}
(in particular, the $\lbrace x^i \rbrace_{i=1,\cdots,\mydim}$ are coordinates on each $\Sigma_t$).
Let $n$ be the lapse, let $e_0 = n^{-1} \partial_t$ be the future-directed normal to $\Sigma_t$,
let $\lbrace e_I \rbrace_{I=1,\cdots,\mydim}$ 
be the $\Sigma_t$-tangent orthonormal spatial frame
described in Sects.\,\ref{SSS:SETUPFORMOFMETRIC}--\ref{subsubsec:conn.coeff},
and let $\lbrace \oe^I \rbrace_{I=1,\cdots,\mydim}$ be the 
corresponding $\Sigma_t$-tangent orthonormal spatial co-frame.
Let $\lbrace e_I^i \rbrace_{i=1,\cdots,\mydim}$ 
denote the components of $e_I$ with respect to the transported spatial coordinates,
and similarly for $\lbrace \oe_i^I \rbrace_{i=1,\cdots,\mydim}$.
Let $\lbrace k_{IJ} \rbrace_{I,J=1,\cdots,\mydim}$ 
denote the components of the second fundamental form of $\Sigma_t$ with respect to the orthonormal spatial
frame, and let $\lbrace \upgamma_{IJB} \rbrace_{I,J=1,\cdots,B}$ 
denote the connection coefficients of $\lbrace e_I \rbrace_{I=1,\cdots,\mydim}$,
as described in Sect.\,\ref{subsubsec:conn.coeff}.
Then the scalar functions 
$k_{IJ},\upgamma_{IJB},e_I^i,\oe_i^I,n,\psi$, 
$I,J,B,i=1,\cdots,\mydim$, 
satisfy all\footnote{We clarify that some of these equations, 
such as \eqref{dt.omega}--\eqref{dt.omega.inv} and \eqref{gamma.form},
are independent of the Einstein--scalar field equations
and follow from the constructions given in
Sects.\,\ref{SSS:SETUPFORMOFMETRIC}--\ref{subsubsec:conn.coeff}.} 
of the
following\footnote{Recall that $e_0 k_{IJ} := n^{-1} \partial_t(k_{IJ}) 
=n^{-1} \partial_t (k_{cd}e_I^c e_J^d)$,
$e_C \upgamma_{IJC} := e_C^c \partial_c (\upgamma_{IJC})$, etc.} 
``reduced'' equations
on $(T,1] \times \mathbb{T}^{\mydim}$
if and only if $({\bf g},\psi)$ are solutions 
to the Einstein-scalar field equations \eqref{EE}--\eqref{SF},
where the spacetime metric can be expressed in terms of the reduced variables
via the formulas
$
{\bf g}
	= - n^2 dt \otimes dt
				+
				g_{ab} dx^a \otimes dx^b
$
and
$g_{ij} = g(\partial_i,\partial_j) = \oe_i^A \oe_j^A$:

\medskip

\noindent \underline{\textbf{Evolution equations for the second fundamental form and connection coefficient components}}
\begin{subequations}
\begin{align}
\begin{split}
e_0 k_{IJ}
& = 
-n^{-1} e_I e_J n 
+
e_C \upgamma_{IJC}-e_I\upgamma_{CJC}
-
\frac{1}{t} k_{IJ}
	\label{dt.k} \\
& \ \
+
n^{-1}\upgamma_{IJC} e_C n
-
\upgamma_{DIC}\upgamma_{CJD}
-
\upgamma_{DDC}\upgamma_{IJC}
- 
(e_I \psi) e_J \psi,
\end{split}
	\\
\begin{split}
e_0 \upgamma_{IJB}
& = 
e_B k_{IJ} 
- 
e_J k_{BI} 
	\label{dt.gamma} 
	\\
&  \ \
-
k_{IC} \upgamma_{BJC}
-
k_{CJ} \upgamma_{BIC}
+
k_{IC} \upgamma_{JBC}
+
k_{BC} \upgamma_{JIC}
+
k_{IC} \upgamma_{CJB}
	\\
& \ \
+
n^{-1}(e_B n)k_{IJ}
-
n^{-1}(e_J n)k_{BI}.
\end{split}
\end{align}
\end{subequations}

\medskip

\noindent \underline{\textbf{Evolution equations for the frame components and co-frame components}}
\begin{subequations}
\begin{align}
e_0 e_I^i & = k_{IC} e_C^i,
	\label{dt.omega}
	\\
e_0 \oe_i^I  & = - k_{IC} \oe_i^C.
	\label{dt.omega.inv} 
\end{align}
\end{subequations}

\medskip

\noindent \underline{\textbf{Wave equation for the scalar field}}
\begin{align}
e_0 e_0 \psi
& = e_C e_C \psi
	- 
	\frac{1}{t} e_0 \psi
	+ 
	n^{-1} (e_C n) e_C \psi
	- 
	\upgamma_{CCD} e_D \psi.
	\label{boxpsi}
\end{align}

\medskip

\noindent \underline{\textbf{Elliptic lapse equation}}
\begin{align}
\begin{split}
e_C e_C(n-1)-t^{-2}(n-1)
& = 
	\upgamma_{CCD} e_D(n-1)
	+
2n e_C \upgamma_{DDC}
	\label{n.eq} \\
& \ \
	-
n
\left\lbrace 
\upgamma_{CDE}\upgamma_{EDC} 
+ 
\upgamma_{CCD}\upgamma_{EED}
+ 
(e_C\psi) 
e_C \psi 
\right\rbrace.
\end{split}
\end{align}

\medskip

\noindent \underline{\textbf{Hamiltonian and momentum constraint equations}}
\begin{subequations}
\begin{align}
2e_C \upgamma_{DDC}
-
\upgamma_{CDE}\upgamma_{EDC}
-
\upgamma_{CCD}\upgamma_{EED}
-
k_{CD} k_{CD} 
+ 
t^{-2}
& = (e_0\psi)^2 + (e_C \psi) e_C \psi,
	\label{Hamconst} \\
e_C k_{CI}
& =  
\upgamma_{CCD} k_{ID}
+
\upgamma_{CID} k_{CD}
-
(e_0 \psi) e_I \psi.
\label{momconst} 
\end{align}
\end{subequations}

\medskip

Finally, we also have the following formula:
\begin{align}\label{gamma.form}
\upgamma_{IJB}
&=
\frac{1}{2}
\left\lbrace
	\oe_c^B(e_Ie_J^c-e_Je_I^c)
	-
	\oe_c^I(e_Je_B^c-e_Be_J^c)
	+
	\oe_c^J(e_Be_I^c-e_I e_B^c)
\right\rbrace.
\end{align}
\end{proposition}
\begin{proof} We will prove in detail that solutions 
to the Einstein-scalar field equations \eqref{EE}--\eqref{SF}
yield solutions to the reduced equations stated in Proposition~\ref{P:redeq}.
The converse can be proved by similar arguments, 
and we will omit the details.

\medskip

\noindent\textbf{Proof of \eqref{dt.k}}:
We first use 
\eqref{frame.prop},
\eqref{E:FRAMEORTHONORMALITY},
\eqref{De0e0},
and \eqref{E:DEIEJFRAMEEXPANSION} to compute the following identity:
\begin{align}\label{R0i0j}
{\bf Riem}(e_0,e_I,e_0,e_J)
& ={\bf g}\left(({\bf D}_{e_0 e_I}^2 -{\bf D}_{e_I e_0}^2)e_J,e_0 \right) 
= e_0 k_{IJ}
-
k_{IC} k_{CJ}
+
n^{-1} \nabla_{e_I e_J}^2 n.
\end{align}
We then use
Gauss' equation, namely:
\begin{align} \label{E:GAUSS}
{\bf Riem}(e_C,e_I,e_D,e_J)
=
Riem(e_C,e_I,e_D,e_J)
+
k_{CD}
k_{IJ}
-
k_{CJ}
k_{ID},
\end{align}
Einstein's field equations \eqref{EE}, 
and \eqref{E:RICCISPAECTIME}
to rewrite LHS~\eqref{R0i0j} as follows:
\begin{align}\label{R0i0j2}
\begin{split}
{\bf Riem}(e_0,e_I,e_0,e_J)
& =-{\bf Ric}(e_I,e_J)
+
{\bf Riem}(e_C,e_I,e_C,e_J)
	\\
& =- (e_I \psi) e_J \psi
+ 
Ric(e_I,e_J)
+
\text{tr}k 
k_{IJ}
-
k_{IC}k_{JC}.
\end{split}
\end{align}
Next, we compute that the frame components of the Ricci tensor of $g$ 
can be expressed as follows:
\begin{align}\label{spRic}
\begin{split}
Ric(e_I,e_J)
& = Riem(e_C,e_I,e_C,e_J)=
g\left((\nabla_{e_C e_I}^2-\nabla_{e_I e_C}^2)e_J,e_C \right)
	\\
& 
= 
e_C \upgamma_{IJC}
-
e_I \upgamma_{CJC}
-
\upgamma_{CID} \upgamma_{DJC}
-
\upgamma_{CCD} \upgamma_{IJD}
+
\underbrace{\upgamma_{ICD} \upgamma_{DJC}
+
\upgamma_{ICD} \upgamma_{DCJ}}_0,
\end{split}
\end{align}
where we have noticed that the last two products
$
\upgamma_{ICD}\upgamma_{DJC}
+
\upgamma_{ICD}\upgamma_{DCJ}
$
cancel, due to the antisymmetry property \eqref{antisymmetricgamma}. 
Next, we use \eqref{gamma} to deduce the following identity for the factor
$\nabla_{e_I e_J}^2 n$ on RHS~\eqref{R0i0j}:
\begin{align} \label{LASTSTEPdt.k}
	\nabla_{e_I e_J}^2 n
	& = e_I e_J n - \upgamma_{IJC} e_C n.
\end{align}
The evolution equation \eqref{dt.k} for $k_{IJ}$ now follows from combining 
\eqref{R0i0j}--\eqref{LASTSTEPdt.k} 
and using the CMC condition \eqref{trk}.

\medskip

\noindent\textbf{Proof of \eqref{dt.gamma}}:
First, we take the $e_0$ derivative of \eqref{gamma} 
and use 
\eqref{frame.prop},
\eqref{E:FRAMEORTHONORMALITY},
\eqref{De0e0},
\eqref{E:DEIE0FRAMEEXPANSION},
\eqref{E:DEIEJFRAMEEXPANSION},
the symmetries of the curvature tensor,
and the Codazzi equations, namely:
\begin{align} \label{E:CODAZZI}
	(\nabla k)_{IJB} - (\nabla k)_{JIB} = {\bf Riem}(e_I,e_J,e_0,e_B)
\end{align}
(where throughout this proof, $\nabla k$ denotes the type $\binom{0}{3}$ $\Sigma_t$-tangent tensorfield
with coordinate components  $(\nabla k)_{abc} = \nabla_a k_{bc}$),
to compute:
\begin{align}\label{R0IJB}
\begin{split}
e_0 \upgamma_{IJB}
& = 
{\bf g}({\bf D}_{e_0 e_I}^2 e_J,e_B)
+
{\bf g}(({\bf D}_{e_0} e_I^{\alpha}) {\bf D}_{\alpha} e_J,e_B)
+
{\bf g}({\bf D}_{e_I} e_J,{\bf D}_{e_0} e_B)
	\\
& = {\bf Riem}(e_0,e_I,e_B,e_J)
+
{\bf g}({\bf D}_{e_I e_0}^2 e_J,e_B)
+
{\bf g}(({\bf D}_{e_0} e_I^{\alpha}) {\bf D}_{\alpha} e_J,e_B)
+
{\bf g}({\bf D}_{e_I}e_J,{\bf D}_{e_0}e_B)
 \\
& = {\bf Riem}(e_0,e_I,e_B,e_J)
+
{\bf g}({\bf D}_{e_I} ({\bf D}_{e_0} e_J),e_B)
	 \\
& \ \
+
{\bf g}(({\bf D}_{e_0} e_I^{\alpha}) {\bf D}_{\alpha} e_J,e_B)
-
{\bf g}(({\bf D}_{e_I} e_0^{\alpha}) {\bf D}_{\alpha} e_J,e_B)
+
{\bf g}({\bf D}_{e_I}e_J,{\bf D}_{e_0}e_B)
 \\
& = 
{\bf Riem}(e_B,e_J,e_0,e_I)
+
k_{IC} \upgamma_{CJB}
-
n^{-1}(e_J n)k_{IB}
+
n^{-1}(e_B n)k_{IJ}
 \\
& = 
(\nabla k)_{BJI}
-
(\nabla k)_{JBI}
+
k_{IC} \upgamma_{CJB}
-
n^{-1}(e_J n)k_{IB}
+
n^{-1}(e_B n)k_{IJ}
		\\
& =
e_B k_{IJ}
-
e_J k_{IB}
- 
k_{CJ} \upgamma_{BIC}
-
k_{IC} \upgamma_{BJC}
+
k_{CB} \upgamma_{JIC}
+
k_{IC} \upgamma_{JBC}
	 \\
& \ \
+
k_{IC}\upgamma_{CJB}
-
n^{-1}(e_J n)k_{IB}
+
n^{-1}(e_B n)k_{IJ},
\end{split}
\end{align}
which yields the desired evolution equation.

\medskip

\noindent\textbf{Proofs of \eqref{dt.omega}--\eqref{dt.omega.inv}}:
First, we use \eqref{frame}, 
\eqref{frame.prop}, 
\eqref{E:DEIE0FRAMEEXPANSION},
and the torsion-free property of the connection ${\bf D}$ 
to compute the following identity:
\begin{align} \label{INVARIANTframe}
(\partial_t e_I^c)\partial_c
& = [\partial_t,e_I^c\partial_c]
=
[\partial_t,e_I]=
{\bf D}_{\partial_t}e_I-{\bf D}_{e_I}(ne_0)
=
nk_{IC} e_C
=
n k_{IC}e_C^c\partial_c.
\end{align}
Considering the $i$ component of \eqref{INVARIANTframe}
relative to the transported spatial coordinates and again using \eqref{frame}, 
we arrive at the desired transport equation \eqref{dt.omega}.
Using the relations \eqref{E:FRAMECOFRAMECONTRACTIONKRONECKER}, we 
also deduce
\eqref{dt.omega.inv}
as a consequence of \eqref{dt.omega}.

\medskip

\noindent\textbf{Proof of \eqref{n.eq}}:
We simply take the $IJ$-trace
of equation \eqref{dt.k} and use the CMC condition \eqref{trk}
and the antisymmetry property \eqref{antisymmetricgamma}.

\medskip

\noindent\textbf{Proof of \eqref{Hamconst}}:
We first note that for solutions to the Einstein-scalar field equations \eqref{EE}--\eqref{SF},
the Hamiltonian constraint \eqref{eq:hamconst}
holds along all constant-time hypersurfaces,
that is:
$R - |k|^2 + (\mathrm{tr} k)^2
=
(e_0 \psi)^2
+
|\nabla \psi|^2
$. We refer to the discussion surrounding \cite[Equation~(10.2.30)]{rW} for a proof of this standard fact
in the context of the Einstein-vacuum equations, and we note that the arguments given there can be modified in a straightforward
fashion to apply to the Einstein-scalar field equations. 
We also note that the sign convention for $k$ used in \cite{rW} is the opposite of the one we use here.
From this equation, 
\eqref{E:INVERSESPACETIMEMETRICINTERMSOFFRAME},
\eqref{spRic},
the identity $R = Ric(e_C,e_C)$,
the antisymmetry property \eqref{antisymmetricgamma},
and the CMC condition \eqref{trk},
we arrive at \eqref{Hamconst}.

\medskip

\noindent\textbf{Proof of \eqref{momconst}}:
We first note that for solutions to the Einstein-scalar field equations \eqref{EE}--\eqref{SF}, 
the momentum constraint \eqref{eq:momconst}
holds along all constant-time hypersurfaces,
that is:
$
(\nabla k)_{CIC} 
- 
(\nabla k)_{ICC}
= 
- 
(e_0 \psi) e_I \psi
$.
We refer to the discussion surrounding \cite[Equation~(10.2.28)]{rW} for a proof of this standard fact
in the context of the Einstein-vacuum equations, and we note that the arguments given there can be modified in a straightforward
fashion to apply to the Einstein-scalar field equations. We again note that the sign convention for $k$
used in \cite{rW} is the opposite of the one we use here.
From this identity and the CMC condition \eqref{trk}, we find that
$
(\nabla k)_{CIC} 
= - (e_0 \psi) e_I \psi
$.
Next, using \eqref{E:INVERSESPACETIMEMETRICINTERMSOFFRAME}, 
\eqref{gamma},
and the Leibniz rule for covariant differentiation,
we find that
$
(\nabla k)_{CIC} 
= 
e_C k_{IC}
- 
\upgamma_{CID} k_{DC}
-
\upgamma_{CCD} k_{ID}
$.
Combining the above equations, we arrive at \eqref{momconst}.

\medskip

\noindent\textbf{Proof of \eqref{gamma.form}}:
This identity follows from the Koszul formula:
\begin{align}\label{Koszul}
\upgamma_{IJB}=&\,\frac{1}{2}\left\lbrace g([e_I,e_J],e_B)-g([e_J,e_B],e_I)+g([e_B,e_I],e_J) \right\rbrace
\end{align}
and the identity $[e_I,e_J]=\oe_l^C(e_Ie_J^l-e_J e_I^l) e_C$.

\medskip

\noindent\textbf{Proof of \eqref{boxpsi}}:
We first note that \eqref{SF} and \eqref{E:INVERSESPACETIMEMETRICINTERMSOFFRAME} imply that
$- e_0 e_0 \psi + e_C e_C \psi = - {}{({\bf D}_{e_0} e_0)^{\alpha}} {\bf D}_{\alpha} \psi 
+
({\bf D}_{e_C} e_C)^{\alpha} {\bf D}_{\alpha} \psi
$.
From this equation,
\eqref{De0e0},
and \eqref{E:DEIEJFRAMEEXPANSION},
we deduce that 
$
e_0 e_0 \psi 
= e_C e_C \psi
	+
	k_{CC} e_0 \psi
	+ 
	n^{-1} (e_C n) e_C \psi
	- 
	\upgamma_{CCD} e_D \psi
$.
From this equation and the CMC condition $k_{CC} = k_{\ a}^a = - \frac{1}{t}$,
we arrive at \eqref{boxpsi}.
\end{proof}

\subsection{Polarized $U(1)$-symmetry}
As we discussed in Sect.\,\ref{subsubsec:U1.data}, 
polarized $U(1)$-symmetric initial data on $\mathbb{T}^3$ for the Einstein-vacuum equations
are such that all coordinate components of $\mathring{g}$ and $\mathring{k}$ are independent of $x^3$
and such that $\mathring{g}_{13} = \mathring{g}_{23} =\mathring{k}_{13} = \mathring{k}_{23} \equiv 0$;
see also the discussion in \cite[Section~2]{IM}.

{
\subsubsection{Propagation of symmetry}
In this section, we show that for polarized $U(1)$-symmetric initial data for the Einstein-vacuum equations,
the corresponding solution to the equations of Proposition~\ref{P:redeq}
is such that all solution variables are independent of $x^3$
and such that
$g_{13} = g_{23} = k_{13} = k_{23} \equiv 0$.
Note that this implies that 
$(g^{-1})^{13} = (g^{-1})^{23} = k_{\ 3}^1 = k_{\ 1}^3 = k_{\ 3}^2 = k_{\ 2}^3 
= k^{13} = k^{23} \equiv 0$.
In particular, relative to CMC-transported spatial coordinates,
the corresponding spacetime metric takes the form:
\begin{align}\label{polarizedmetric}
{\bf g} 
& 
=-n^2 dt \otimes dt
+
\sum_{c,d=1,2}
g_{cd} dx^c \otimes dx^d
+
g_{33}
dx^3 \otimes dx^3
&
n
& =[-({\bf g}^{-1})^{\alpha\beta}\partial_\alpha t\partial_\beta t]^{-\frac{1}{2}}.
\end{align}
We provide the main propagation-of-symmetry result in the next lemma, namely Lemma~\ref{L:PROPOFSYM}. 
The result is standard, so we only provide the main steps in the proof.
We refer readers to \cite[Chapter XVI.3]{CB} for an alternate approach to propagating  
the polarized $U(1)$-symmetry via a wave map-type reduction of Einstein's equations. 
Although the approach of \cite[Chapter XVI.3]{CB}
is superior from a geometric point of view, Lemma~\ref{L:PROPOFSYM} is better adapted to 
the setup of the present paper because it directly refers to the
CMC-transported spatial coordinates, thus allowing us to avoid working with
multiple gauges.}

\begin{lemma}[Propagation of polarized $U(1)$-symmetry]
\label{L:PROPOFSYM}
Let $T \in (0,1)$, and let
${\bf g}$ be a solution to the Einstein-vacuum equations
(i.e., \eqref{EE} with $\psi \equiv 0$)
on the manifold $(T,1] \times \mathbb{T}^3$. 
Assume that $(T,1] \times \mathbb{T}^3$ is equipped with 
a CMC time function $t$ and transported spatial coordinates $\lbrace x^i \rbrace_{i=1,2,3}$
such that the level sets $\Sigma_t = \lbrace t \rbrace \times \mathbb{T}^3$
are CMC hypersurfaces normalized by \eqref{trk}, as is described in Sects.\,\ref{SSS:SETUPFORMOFMETRIC}--\ref{subsubsec:conn.coeff}
(in particular, the $\lbrace x^i \rbrace_{i=1,2,3}$ are coordinates on each $\Sigma_t$).
Assume that relative to the coordinates $(t,x^1,x^2,x^3)$,
the components of ${\bf g}$ belong to $C^2\left((T,1] \times \mathbb{T}^3 \right)$.
Let $(\mathring{g},\mathring{k})$ be the corresponding
data on $\Sigma_1$,
i.e., the first and second fundamental form
of $\Sigma_1$.
Assume that the data are polarized and $U(1)$-symmetric. More precisely,
assume that all coordinate components of $\mathring{g}$ and $\mathring{k}$ are independent of $x^3$
and that $\mathring{g}_{13} = \mathring{g}_{23} =\mathring{k}_{13} = \mathring{k}_{23} \equiv 0$.
Then relative to the CMC-transported spatial coordinates,
${\bf g}$ has the polarized form \eqref{polarizedmetric},
and the lapse $n$,
the coordinate components $\lbrace g_{ij} \rbrace_{i,j=1,2,3}$,
and the coordinate components $\lbrace k_{ij} \rbrace_{i,j=1,2,3}$
of the second fundamental form of $\Sigma_t$
are all independent of $x^3$.
\end{lemma}
\begin{proof}
We divide the proof into the following two steps:

\medskip

\noindent {\bf Step 1: Propagation of $U(1)$-symmetry.}
In Sect.\,\ref{sec:mainest}, we derive a priori estimates for the system of equations in Proposition~\ref{P:redeq} by commuting the equations with sufficiently many spatial derivatives. Here, to propagate the
$U(1)$-symmetry, we consider the case in which the equations are commuted
with only a single spatial derivative $\partial_3$. More precisely,
we consider the equations
\eqref{n.eq.diff}, 
\eqref{E:COMMUTEDFRAMEEVOLUTION}, 
\eqref{E:COMMUTEDCOFRAMEEVOLUTION}, 
\eqref{E:COMMUTEDKEQUATION}, 
\eqref{E:COMMUTEDCONNECTIONCOEFFICIENTEQUATION}, 
and \eqref{E:COMMUTEDMOMENUTMCONSTRAINT} with $\partial^\iota := \partial_3$
and $\Pow : = 0$.
It is straightforward to see that these commuted equations form a \emph{linear}
PDE system in the ``unknowns'' 
$\mathcal{U}_3 
:= 
\lbrace \partial_3 n \rbrace
\cup
\mathcal{V}_3
$,
$ 
\mathcal{V}_3
:=
\lbrace \partial_3 e_I^i, \partial_3 \upomega^I_i, \partial_3 k_{IJ}, \partial_3 \upgamma_{IJB} \rbrace_{I,J,B,i=1,2,3}$
such that, because we are assuming the existence of a classical $C^2$ solution,
all coefficients in front of the principal PDE terms are $C^1$, all inhomogeneous terms are continuous,
and all inhomogeneous terms are precisely linear in $\mathcal{U}_3$.
Hence, for $t \in (T,1]$,
we can use arguments similar to the ones we use in Sect.\,\ref{sec:mainest} 
to derive the elliptic estimate
$\| \partial_3 n \|_{L^2(\Sigma_t)} 
\lesssim 
\| \mathcal{V}_3 \|_{L^2(\Sigma_t)}
$
and the following energy estimate:
\begin{align} \label{E:GRONWALLREADYPROPOFU1SYMMETRY}
	\| \mathcal{V}_3 \|_{L^2(\Sigma_t)} 
	& \lesssim 
	\| \mathcal{V}_3 \|_{L^2(\Sigma_1)} 
	+ 
	\int_{s=t}^1 \| \mathcal{V}_3 \|_{L^2(\Sigma_s)} \, ds.
\end{align}
We clarify that on RHS~\eqref{E:GRONWALLREADYPROPOFU1SYMMETRY}, 
the implicit constants depend on the $C^2$ norm of the classical solution and thus
are allowed to grow as $t \downarrow T$. However, that is not relevant for the proof of the lemma.
Next, we note that by assumption, the initial data (on $\Sigma_1$) of $\mathcal{V}_3$ are trivial,
i.e., $\| \mathcal{V}_3 \|_{L^2(\Sigma_1)} = 0$.
Hence, applying Gr\"{o}nwall's lemma to \eqref{E:GRONWALLREADYPROPOFU1SYMMETRY}, 
we conclude that
$\| \mathcal{V}_3 \|_{L^2(\Sigma_t)} = 0$ (
and thus $\| \partial_3 n \|_{L^2(\Sigma_t)} = 0 = \| \mathcal{U}_3 \|_{L^2(\Sigma_t)}$ too)
for $t \in (T,1]$.
In total, we have shown that the $U(1)$-symmetry of the data is propagated to the entire region
$(t,x) \in (T,1] \times \mathbb{T}^3$ of classical existence,
and in particular, $\partial_3$ is a ${\bf g}$-Killing vectorfield.

\medskip
\noindent {\bf Step 2: Using the $U(1)$-symmetry to propagate the polarization.}
We now explain how to propagate the vanishing of 
$g_{13}=g(\partial_1,\partial_3)$,
$g_{23}=g(\partial_2,\partial_3)$,
$k_{13}= - {\bf g}(\partial_3,{\bf D}_{\partial_1}e_0)$,
and
$k_{23}= - {\bf g}(\partial_3,{\bf D}_{\partial_2}e_0)$
from the initial hypersurface $\Sigma_1$ 
to the entire region of classical existence.
To this end, we first derive PDEs satisfied by these variables
relative to the transported spatial coordinates.
We will exploit the fact, shown in Step 1,
that $\partial_3$ is ${\bf g}$-Killing. 
Specifically, the relevant PDEs for this step are the following
standard ``ADM equations'' and momentum constraint equation (see, for example, \cite[Proposition~3.1]{RodSp2}),
where in the remainder of the proof, $i,j \in \lbrace 1,2 \rbrace$,
and $\Gamma_{b \ c}^{\ a} := \frac{1}{2} (g^{-1})^{ad}(\partial_b g_{dc} + \partial_c g_{bd} - \partial_d g_{bc})$
are the Christoffel symbols of $g$ relative to the transported spatial coordinates:
\begin{align} 
\partial_t g_{i3} 
&= -2nk_{i3},
	\label{var.eq} 
	\\
\partial_t \partial_j g_{i3} 
&= -2n \partial_j k_{i3}
		-2 (\partial_j n) k_{i3},
	\label{var.diffeq} 
		\\
\partial_t k_{i3} 
& = - \nabla^2_{\partial_i \partial_3}n
+
n
Ric(\partial_i,\partial_3)
+
n
\text{tr}k
k_{i3}
-
2n(g^{-1})^{ab}k_{ia}k_{3b},
	\label{var.keq} 
	\\
(g^{-1})^{ab}\nabla_a k_{b3}
& = (g^{-1})^{ab}\partial_ak_{b3}
	-
	(g^{-1})^{ab}\Gamma_{a \ b}^{\ c} k_{c3}
	-
	(g^{-1})^{ab}\Gamma_{a \ 3}^{\ c} k_{bc}
	= 0.
	\label{mom.const.coord}
\end{align}
We clarify that \eqref{var.diffeq} follows from differentiating \eqref{var.eq} with $\partial_j$.
Next, using straightforward computations 
and exploiting that $\partial_3$ is ${\bf g}$-Killing,
we expand the Hessian of $n$ and the spatial Ricci components
as follows:
\begin{align} 
\begin{split} \label{Hess.i3.n}
-\nabla^2_{\partial_i \partial_3} n
& = \Gamma_{i \ 3}^{\ 1} \partial_1 n 
		+
		\Gamma_{i \ 3}^{\ 2} \partial_2 n 
=
\frac{1}{2}
(g^{-1})^{1a}(\partial_i g_{a3}-\partial_a g_{i3})
\partial_1 n
+
\frac{1}{2}
(g^{-1})^{2a}(\partial_i g_{a3}-\partial_a g_{i3})
\partial_2 n
	\\
& = 
\frac{1}{2}
(g^{-1})^{11}(\partial_i g_{13}-\partial_1 g_{i3})
\partial_1 n
+
\frac{1}{2}
(g^{-1})^{12}(\partial_i g_{23}-\partial_2 g_{i3})
\partial_1 n
+
\frac{1}{2}
(g^{-1})^{13}(\partial_i g_{33})
\partial_1 n
\\
&
\ \ 
+
\frac{1}{2}
(g^{-1})^{21}(\partial_i g_{13}-\partial_1 g_{i3})
\partial_2 n
+
\frac{1}{2}
(g^{-1})^{22}(\partial_i g_{23}-\partial_2 g_{i3})
\partial_2 n
+
\frac{1}{2}
(g^{-1})^{23} (\partial_i g_{33})
\partial_2 n,
\end{split}
\\
\begin{split} \label{Ric.i3}
Ric(\partial_i,\partial_3)
& 
= Ric(\partial_3,\partial_i)
=
\partial_a \Gamma_{3 \ i}^{\ a}
-
\partial_3 \Gamma_{a \ i}^{\ a}
+
\Gamma_{i \ 3}^{\ b} \Gamma_{b \ a}^{\ a}
-
\Gamma_{i \ a}^{\ b}\Gamma_{3 \ b}^{\ a}
	 \\
& = 
\frac{1}{2}
\partial_a\left[(g^{-1})^{ab}(\partial_i g_{3b}-\partial_b g_{i3}) \right]
+
\frac{1}{4}
(g^{-1})^{bd}
(g^{-1})^{ac}
(\partial_i g_{3c}-\partial_c g_{i3})
\partial_b g_{ad}
	\\
& \ \
	-\frac{1}{4}
	(g^{-1})^{bc}
	(g^{-1})^{ad}
	(\partial_i g_{ac} + \partial_a g_{ic}-\partial_c g_{ia})
	(\partial_b g_{3d}-\partial_d g_{3b}).
\end{split}
\end{align}
We now note that \eqref{var.eq}--\eqref{mom.const.coord} 
can be viewed as a \emph{linear} PDE system 
in the ``unknowns''
$$\mathcal{V} := (g_{13},g_{23},\partial_1 g_{13},\partial_1 g_{23},\partial_2 g_{13},\partial_2 g_{23},k_{13},k_{23})$$
such that, because we are assuming the existence of a classical $C^2$ solution,
all coefficients in front of the principal PDE terms are $C^1$, 
all inhomogeneous terms are continuous,
and all inhomogeneous terms are linear in $\mathcal{V}$.
To see this, we note that
the components $(g^{-1})^{13}$ and $(g^{-1})^{23}$ of $g^{-1}$ 
can expressed in terms of $g_{13}$ and $g_{23}$
via the following linear algebraic identity,
which follows easily from the identity $g_{ac} (g^{-1})^{cb} = \updelta_a^b$:
\begin{align}\label{g:inv:a.3}
\begin{pmatrix}
(g^{-1})^{13}  \\
(g^{-1})^{23} 
\end{pmatrix}	
& =-(g^{-1})^{33}
\begin{pmatrix}
g_{11} & g_{12} \\
g_{21} & g_{22} 
\end{pmatrix}^{-1}
\begin{pmatrix}
g_{13}\\
g_{23}
\end{pmatrix}.
\end{align}
Note that the positive definiteness of the $3 \times 3$ matrix $(g_{ab})_{a,b=1,2,3}$
implies the positive definiteness of
$2 \times 2$ sub-block $(g_{ab})_{a,b=1,2}$, which in turn implies the invertibility of
the matrix on RHS~\eqref{g:inv:a.3}.
To finish the proof of the lemma, we will derive energy estimates showing, in particular, that 
$\| \mathcal{V} \|_{L^2(\Sigma_t)} = 0$ for $t \in (T,1]$.
To this end, we first use \eqref{Hess.i3.n}--\eqref{Ric.i3}
and the above observations
to express \eqref{var.eq}--\eqref{mom.const.coord} in the following form,
where $\mathscr{L}(\mathcal{V})$ schematically denotes
terms that are linear\footnote{Here, we consider, e.g., 
terms of type $\mathcal{V} \cdot \mathcal{V}$ to be linear in $\mathcal{V}$
with continuous coefficients because the first factor of $\mathcal{V}$ is assumed to be continuous.} 
in $\mathcal{V}$ with continuous coefficients:
\begin{align}
\label{diff.var.eq}
\partial_t g_{i3} 
&= \mathscr{L}(\mathcal{V}),
	\\
\partial_t \partial_j g_{i3} 
&= -2n \partial_j k_{i3}
		+ 
		\mathscr{L}(\mathcal{V}),
	\label{diff.var.diffeq} 
		\\
\partial_t k_{i3} 
& = 
-
\frac{1}{2}
\sum_{a,b=1,2} 
n (g^{-1})^{ab} \partial_a \partial_b g_{i3}
+
\frac{1}{2}
\sum_{a,b=1,2}
n (g^{-1})^{ab}\partial_a \partial_i g_{b3}
+
\mathscr{L}(\mathcal{V}),
	\label{diff.var.keq} 
	\\
\sum_{a,b=1,2}
(g^{-1})^{ab} \partial_a k_{b3}
& = 
	\mathscr{L}(\mathcal{V}).
	\label{diff.mom.const.coord}
\end{align}
We clarify that, for example, we have used \eqref{g:inv:a.3}
to soak the term $\frac{1}{2} n (g^{-1})^{13}\partial_1 \partial_i g_{33}$ on RHS~\eqref{var.keq} 
(see the first term on RHS~\eqref{Ric.i3})
into the term $\mathscr{L}(\mathcal{V})$ on RHS~\eqref{diff.var.keq} 
and to soak the term $(g^{-1})^{13} \partial_1 k_{33}$
from \eqref{mom.const.coord} into the term $\mathscr{L}(\mathcal{V})$ on RHS~\eqref{diff.mom.const.coord}.

To derive energy estimates for solutions to \eqref{diff.var.eq}--\eqref{diff.mom.const.coord},
we will rely on the energy $\mathbb{E}(t) \geq 0$ defined by:
\begin{align}\label{H1.var}
\begin{split}
\mathbb{E}^2(t)
& := 
\int_{\Sigma_t}
\left\lbrace
\sum_{a,b,a',b' = 1,2}
\frac{1}{4}(g^{-1})^{bb'} (g^{-1})^{aa'} (\partial_b g_{a3}) \partial_{b'} g_{a'3}
+
\sum_{a,b = 1,2}(g^{-1})^{ab} k_{a3} k_{b3}
\right\rbrace
\, dx
	\\
& \ \
+
\int_{\Sigma_t}
\sum_{a,b = 1,2}
	(g^{-1})^{ab}
	g_{a3}
	g_{b3}
\, dx.
\end{split}
\end{align}
Note that the positive definiteness of the $3 \times 3$ matrix $((g^{-1})^{ab})_{a,b=1,2,3}$
implies the positive definiteness of
$2 \times 2$ sub-block $((g^{-1})^{ab})_{a,b=1,2}$. Hence, 
from definition \eqref{H1.var} and the definition of $\mathcal{V}$, we find that:
\begin{align} \label{E:ENERGYCOERCIVENESSFORPROPAGATINGU1SYMMETRY}
\mathbb{E}(t)
&
\approx
\sum_{a=1,2} \| g_{a3} \|_{L^2(\Sigma_t)}
+
\sum_{i=1,2} \sum_{a=1,2} \| \partial_i g_{a3} \|_{L^2(\Sigma_t)}
+
\sum_{a=1,2} \| k_{a3} \|_{L^2(\Sigma_t)}
\approx 
\| \mathcal{V} \|_{L^2(\Sigma_t)}.
\end{align}
Therefore, to show that $\mathcal{V}$ vanishes on the region of classical existence,
it suffices to show that $\mathbb{E}(t) = 0$ for $t \in (T,1]$.
To this end, we will show that the following estimate holds for $t \in (T,1]$:
\begin{align} \label{E:POLARIZEDGRONWALLREADY}
	\left|\frac{d}{dt} \mathbb{E}^2(t) \right|
	& \lesssim
		\mathbb{E}^2(t),
\end{align}
where on RHS~\eqref{E:POLARIZEDGRONWALLREADY} 
the implicit constants depend on the $C^2$ norm of the classical solution and thus
are allowed to grow as $t \downarrow T$. As in Step 1, 
this is not important for the proof.
Since our assumption of polarized initial data implies that $\mathbb{E}(1) = 0$,
it follows from \eqref{E:POLARIZEDGRONWALLREADY} and Gr\"{o}nwall's lemma
that $\mathbb{E}(t) = 0$ for $t \in (T,1]$, as desired.
In total, we have shown that the polarization condition $\mathcal{V} = 0$ 
is propagated from the data to the entire region
$(t,x) \in (T,1] \times \mathbb{T}^3$ of classical existence.

To complete the proof of the lemma, it remains for us to show \eqref{E:POLARIZEDGRONWALLREADY}.
To proceed, we differentiate \eqref{H1.var} with respect to time under the integral
and use equations \eqref{diff.var.eq}--\eqref{diff.var.keq} 
to substitute time derivatives with spatial derivatives,
thereby arriving at the following identity,
where in the rest of the proof, we freely integrate by parts from line to line,
and
``$\cdots$'' denotes error integrals that
can easily be bounded in magnitude 
by $\lesssim \mathbb{E}^2(t)$ by virtue of the Cauchy--Schwarz inequality:
\begin{align} \label{H1.var.est}
\begin{split}
\frac{d}{dt} \mathbb{E}^2(t)
& =
\int_{\Sigma_t}
	\sum_{a,b,a',b'=1,2} 
	(g^{-1})^{bb'}(g^{-1})^{aa'} 
	(-n \partial_bk_{a3})(\partial_{b'} g_{a'3})
\, dx
	\\
& \ \
+
\int_{\Sigma_t}
\sum_{a,b,c,d=1,2}
	(g^{-1})^{ab} k_{b3}
	\left\lbrace
	-n (g^{-1})^{cd}\partial_c \partial_d g_{a3}
	+
	n (g^{-1})^{cd} \partial_c \partial_a g_{d3} 
	\right\rbrace
\, dx
+ 
\cdots
\end{split}
		\\
\begin{split} \label{SECONDH1.var.est}
& = \mbox{(upon integrating by parts on the last line of RHS~\eqref{H1.var.est})}
	\\
& 
\int_{\Sigma_t}
\sum_{a,b,c,d = 1,2}	
	\left[\partial_c(n(g^{-1})^{ab}(g^{-1})^{cd}) \right] k_{b3} \partial_d g_{a3}
\, dx
	\\
& \ \
-
\int_{\Sigma_t}
	\sum_{a,b,c,d = 1,2}
	\left[\partial_a(n(g^{-1})^{ab}(g^{-1})^{cd}) \right] k_{b3} \partial_c g_{d3}
\, dx
	\\
&  \ \ 
-
\int_{\Sigma_t}
	\sum_{a,b,c,d = 1,2}
	\left((g^{-1})^{ab}\partial_a k_{b3} \right) n (g^{-1})^{cd} \partial_c g_{d3}
\, dx
+ 
\cdots
\end{split}
	\\
\begin{split} \label{THIRDH1.var.est}
& = \mbox{(upon using \eqref{diff.mom.const.coord} for substitution in the last integral on RHS~\eqref{SECONDH1.var.est})}
	\\
& 
\int_{\Sigma_t}
\sum_{a,b,c,d = 1,2}	
	\left[\partial_c(n(g^{-1})^{ab}(g^{-1})^{cd}) \right] k_{b3} \partial_d g_{a3}
\, dx
	\\
& \ \
-
\int_{\Sigma_t}
	\sum_{a,b,c,d = 1,2}
	\left[\partial_a(n(g^{-1})^{ab}(g^{-1})^{cd}) \right] k_{b3} \partial_c g_{d3}
\, dx
	+
	\cdots.
\end{split}
\end{align}
Since all terms on RHS~\eqref{THIRDH1.var.est} are bounded in magnitude by $\leq C \mathbb{E}^2(t)$,
we have therefore shown \eqref{E:POLARIZEDGRONWALLREADY} and finished the proof of the lemma.
\end{proof}

\subsubsection{The normalized Killing vectorfield in polarized $U(1)$-symmetry in $1+3$ dimensions}
In Lemma~\ref{lem:U1}, for polarized $U(1)$-symmetric Einstein-vacuum solutions with $\mydim = 3$
such that the transported coordinate vectorfield $\partial_3$ is Killing,
we construct an orthonormal spatial frame $e_1,e_2,e_3$ such that
$e_3$ is everywhere parallel to $\partial_3$
and such that $\mathcal{L}_{\partial_3} e_I = 0$ for $I=1,2,3$.
We use Lemma~\ref{lem:U1} in the proof of Theorem~\ref{thm:precise.U1},
i.e., in the proof of our symmetric stable blowup-results.

\begin{lemma}[Normalized Killing vectorfield]\label{lem:U1}
Suppose that on $(T,1] \times \mathbb{T}^3$, 
${\bf g}$ is a polarized and $U(1)$-symmetric $C^2$ metric of the form
\eqref{polarizedmetric},
where $\partial_3$ is the hypersurface-orthogonal Killing vectorfield,
and the components of ${\bf g}$ are independent of $x^3$.
Define $E_3 := (g_{33})^{-\frac{1}{2}}\partial_3$, and note that $g(E_3,E_3) = 1$.
Let $\mathring{e}_1,\mathring{e}_2$ be an orthonormal pair on $\Sigma_1$ that is orthogonal to $\partial_3$
along $\Sigma_1$ and that respects the symmetry, that is,
$\mathcal{L}_{\partial_3} \mathring{e}_1 = \mathcal{L}_{\partial_3}  \mathring{e}_2 = 0$,
where $\mathcal{L}$ denotes Lie differentiation.
In particular, $\lbrace \mathring{e}_1, \mathring{e}_2, E_3|_{\Sigma_1} \rbrace$
is an orthonormal frame on $\Sigma_1$;
we refer to Sect.\,\ref{SS:CONSTRUCTIONOFTHEINTIALORTHONORMALFRAME} for our construction of such a frame.
For $(t,x) \in (T,1] \times \mathbb{T}^3$, let $\lbrace e_I \rbrace_{I=1,2,3}$
be the orthonormal frame on $\Sigma_t$ obtained by solving the Fermi--Walker transport equations
\eqref{frame.prop} with initial data 
$e_1|_{\Sigma_1} := \mathring{e}_1$,
$e_2|_{\Sigma_1} := \mathring{e}_2$,
$e_3|_{\Sigma_1} := E_3|_{\Sigma_1}$.
Then on $(T,1] \times \mathbb{T}^3$,
we have $E_3 = e_3$,
and for $I=1,2,3$, we have $\mathcal{L}_{\partial_3} e_I = 0$.
In particular, $\lbrace e_1, e_2, e_3 = (g_{33})^{-\frac{1}{2}}\partial_3 \rbrace$ is an orthonormal
frame on $(T,1] \times \mathbb{T}^3$.
\end{lemma}
\begin{proof}
%We set $e_3:=(g_{33})^{-\frac{1}{2}}\partial_3$. 
We will show that ${\bf D}_{e_0}E_3=0$ on $(T,1] \times \mathbb{T}^3$. 
Then since $\partial_3$ is Killing (in particular, $E_3 n=0$),
this guarantees that $E_3$ satisfies the propagation equation \eqref{frame.prop}. 
Moreover, since ${\bf D}_{e_0} e_3=0$ and $e_3|_{\Sigma_1} = E_3|_{\Sigma_1}$,
ODE uniqueness then implies that $E_3=e_3$ on $(T,1] \times \mathbb{T}^3$, 
as is desired.
To show that ${\bf D}_{e_0}E_3=0$, we start by noting 
that since $\partial_3$ is the vectorfield of symmetry
and since ${\bf g}(e_0,E_3) = 0$,
the $e_0$ component of ${\bf D}_{e_0} E_3$ equals:
\begin{align}
{\bf g}({\bf D}_{e_0} E_3,e_0)=-{\bf g}(E_3,{\bf D}_{e_0}e_0)\overset{\eqref{De0e0}}{=}-n^{-1}(g_{33})^{-\frac{1}{2}}\partial_3n& =0.
\end{align}
%K
Similarly, using that ${\bf g}(E_3,E_3) = 1$,
we compute that
${\bf{g}}({\bf D}_{e_0}E_3,E_3) = \frac{1}{2} e_0 \left\lbrace {\bf{g}}(E_3,E_3) \right\rbrace = 0$
and hence the $\partial_3$ component of ${\bf D}_{e_0}E_3$ vanishes.
Thus, the desired relation ${\bf D}_{e_0}E_3 = 0$ follows 
from the vanishing of these two components as well as the following 
identities:
\begin{align}
{\bf g}({\bf D}_{e_0}E_3,\partial_i)
& 
= n^{-1}(g_{33})^{-\frac{1}{2}}\frac{1}{2}(\partial_t{\bf g}_{3i}+\partial_3{\bf g}_{ti}-\partial_i{\bf g}_{t3})=0,\qquad i=1,2.
\end{align} 
In obtaining these identities, we have used the identity
$
	{\bf g}({\bf D}_{\partial_t} \partial_3,\partial_i)
	=
	\frac{1}{2}(\partial_t{\bf g}_{3i}+\partial_3{\bf g}_{ti}-\partial_i{\bf g}_{t3})$,
the fact that $\partial_3$ is orthogonal to the elements of
$\lbrace \partial_t,\partial_1,\partial_2 \rbrace$,
the fact that $\partial_t$ is orthogonal to $\Sigma_t$,
and the fact that the components of ${\bf g}$ are independent of $x^3$. 

To show that $\mathcal{L}_{\partial_3} e_I = 0$ on $(T,1] \times \mathbb{T}^3$,
we commute equation \eqref{dt.omega} (which is equivalent to equation \eqref{frame.prop})
with $\mathcal{L}_{\partial_3}$. We find that for metrics ${\bf g}$ satisfying the assumptions
of the lemma, the scalar function array $\vec{\phi} := \lbrace \partial_3 e_I^i \rbrace_{I,i=1,2,3}$
satisfies a system of transport equations of the schematic form
$\partial_t \vec{\phi} = F \cdot \vec{\phi}$, where $F$ is smooth on $(T,1] \times \mathbb{T}^3$.
Moreover, the assumptions of the lemma guarantee that $\vec{\phi}|_{\Sigma_1} = 0$.
Hence, from ODE uniqueness, we find that $\vec{\phi} \equiv 0$ on $(T,1] \times \mathbb{T}^3$.
We have therefore proved the lemma.
\end{proof}

\begin{remark}
	Throughout the paper, in our analysis of polarized $U(1)$-symmetric solutions with $\mydim = 3$,
	we will always assume that $e_3$ is the $\partial_3$-parallel
	frame vectorfield constructed in Lemma~\ref{lem:U1}.
\end{remark}

\subsection{The background Kasner variables}
\label{SS:BACKGROUNDVARIABLES}
Our main results concern perturbations of the explicit generalized Kasner solutions presented in Sect.\,\ref{subsec:models}. 
Straightforward computations imply that the reduced variables
(see Sects.\,\ref{SSS:SETUPFORMOFMETRIC}--\ref{subsubsec:conn.coeff}) 
of the generalized Kasner solutions can be expressed as follows:
\begin{align}\label{Kasnersol}
\widetilde{n}:=1,\quad \widetilde{e}_I^i := t^{-\widetilde{q}_{\underline{I}}}\updelta_{\underline{I}}^i,
\quad
\widetilde{\oe}_i^I := t^{\widetilde{q}_{\underline{I}}}\updelta_i^{\underline{I}},
\quad 
\widetilde{e}_I:= t^{-\widetilde{q}_{\underline{I}}} \updelta_{\underline{I}}^c 
	\partial_c, \quad 
\widetilde{k}_{IJ}
:= -\frac{\widetilde{q}_{\underline{I}}}{t}\updelta_{\underline{I}J},\quad \widetilde{\upgamma}_{IJB}=0,
\quad\widetilde{\psi}=\widetilde{B}\log t,
\end{align}
where $\updelta_I^i$, $\updelta_i^I$, and $\updelta_{IJ}$ are Kronecker deltas,
and we recall that repeated underlined indices are not summed.
%i
\begin{remark}[The components of ``tilde-decorated'' tensors]
	\label{R:NOTCOMPONENTSOFTENSORRELATIVETOFRAME}
	Note that as defined in \eqref{Kasnersol},
	$\widetilde{k}_{IJ} = \widetilde{k}(\widetilde{e}_I,\widetilde{e}_J) \neq
	\widetilde{k}(e_I,e_J)$. Put differently, $\widetilde{k}_{IJ}$ denotes a component of
	$\widetilde{k}$ relative to the background Kasner-orthonormal frame $\lbrace \widetilde{e}_I \rbrace_{I=1,\cdots,\mydim}$, 
	rather than the perturbed $g$-orthonormal frame $\lbrace e_I \rbrace_{I=1,\cdots,\mydim}$.
	Similar remarks apply to other ``tilde-decorated'' tensors.
	That is, for tilde-decorated tensors, capital Latin indices denote components relative to the background Kasner frame
	or co-frame, whereas for non-tilde-decorated tensors, capital Latin indices denote components relative to the
	$g$-orthonormal frame or co-frame.
\end{remark}

\section{Norms, key parameters, and bootstrap assumptions}\label{sec:Boots}
The proofs of our main theorems rely on a continuity argument for solutions to the reduced 
equations of Proposition~\ref{P:redeq}.
We make bootstrap assumptions for the size of various norms of the perturbed solution 
on a time interval $(T_{\textnormal{Boot}},1]$ for some $T_{\textnormal{Boot}}\in(0,1)$. 
Then, in Proposition~\ref{prop:overall}, we derive a priori estimates 
for the perturbed solution that imply a strict improvement of the bootstrap assumptions
on $(T_{\textnormal{Boot}},1]$; this is the difficult part of the proof.
Once we have established a priori estimates,
standard arguments yield that the perturbed solution exists on $(0,1] \times \mathbb{T}^{\mydim}$
and satisfies the a priori estimates on $(0,1]$; see Proposition~\ref{P:EXISTENCEONHALFSLAB} for the details.
Based on the existence result and the a priori estimates, the proof of curvature-blowup as $t \downarrow 0$
and the derivation of other interesting properties of the solution 
are relatively straightforward; see Sect.\,\ref{sec:sol}. 
Our bootstrap assumptions are formulated in terms of various norms of the reduced variables along the $\Sigma_t$ slices, with well-chosen $t$-weights. 
Before stating the bootstrap assumptions, we will first define the norms
and the key parameters $q,\upsigma,\blowupexp,N,N_0$ that lie at the core of our framework.

\subsection{Running assumption}
\label{SS:RUNNINGASSUMPTION}
In the rest of the paper, it is understood that we are studying
general perturbations of a background generalized Kasner solution
whose Kasner exponents verify the stability condition \eqref{Kasner.stability.cond},
or that we are studying polarized $U(1)$-symmetric perturbations of an arbitrary vacuum Kasner solution
in $1+3$-dimensions; see Sect.\,\ref{subsec:models}.
We will often refrain from explicitly stating this assumption.

\subsection{Some additional differentiation notation}
\label{SS:ADDITIONALDIFFERENTIATIONNOTATION}
If $f$ is a scalar function, then $\vec{e} f := \lbrace e_C f \rbrace_{C=1,\cdots,\mydim}$,
where $\lbrace e_C \rbrace_{C=1,\cdots,\mydim}$ denotes the orthonormal spatial frame.
Similarly, $\vec{e} k := \lbrace e_C k_{IJ} \rbrace_{C,I,J=1,\cdots,\mydim}$,
$\vec{e} \upgamma := \lbrace e_C \upgamma_{IJB} \rbrace_{B,C,I,J=1,\cdots,\mydim}$,
$\vec{e} e := \lbrace e_C e_I^i \rbrace_{C,I,i=1,\cdots,\mydim}$,
and
$\vec{e} \oe := \lbrace e_C \oe_i^I \rbrace_{C,I,i=1,\cdots,\mydim}$.
Note that in the above expressions, \textbf{all quantities that are differentiated are scalar functions}.

\subsection{Sobolev norms of the reduced variables}
\label{SS:SOBOLEVNORMS}
For scalar functions $v$,
we define its norm $\|v\|_{L^2(\Sigma_t)} \geq 0$ by:
\begin{align} \label{E:STANDARDL2NORM}
	\|v\|_{L^2(\Sigma_t)}^2
	& := \int_{\Sigma_t} v^2(t,x) \, dx,
\end{align}
where $dx:=dx^1\cdots dx^{\mydim}$
denotes the Euclidean volume form on $\Sigma_t$.

We also define standard $H^M(\Sigma_t)$,
$\dot{H}^M(\Sigma_t)$, 
$W^{M,\infty}(\Sigma_t)$, 
and
$\dot{W}^{M,\infty}(\Sigma_t)$
norms
of scalar functions $v$:
\begin{align}\label{HM}
\|v\|_{H^M(\Sigma_t)}^2
& :=\sum_{|\iota |\leq M} \| \partial^{\iota}v\|_{L^2(\Sigma_t)}^2,
&
\|v\|_{\dot{H}^M(\Sigma_t)}^2
& 
:= \sum_{|\iota| = M} \| \partial^{\iota} v \|_{L^2(\Sigma_t)}^2,
	\\
\|v\|_{W^{M,\infty}(\Sigma_t)}& := \sum_{|\iota|\leq M}\|\partial^{\iota}v\|_{L^\infty(\Sigma_t)},
&
\|v\|_{\dot{W}^{M,\infty}(\Sigma_t)}& := \sum_{|\iota| = M}\|\partial^{\iota}v\|_{L^\infty(\Sigma_t)},
\end{align} 
where $\iota$ is a spatial multi-index,
$\partial^{\iota}$ is the corresponding operator involving repeated differentiation
with respect to the transported spatial coordinate vectorfields $\lbrace \partial_i \rbrace_{i=1,\cdots,\mydim}$
(see Sect.\,\ref{SS:NOTATIONANDCONVENTIONS}),
and $\|v\|_{L^\infty(\Sigma_t)} := \mbox{\upshape ess sup}_{x \in \mathbb{T}^{\mydim}} |v(t,x)|$.
As is standard, we write ``$L^{\infty}$'' instead of ``$W^{0,\infty}$.''
%Note that we do need to subtract the background variables in the definition of the high order norms, 
%since the norms are homogeneous and it does not make a difference.

If $v$ is a $\Sigma_t$-tangent tensorfield, then we define its 
$L^2(\Sigma_t)$,
$H^M(\Sigma_t)$, 
$\dot{H}^M(\Sigma_t)$,
$W^{M,\infty}(\Sigma_t)$,
and 
$\dot{W}^{M,\infty}(\Sigma_t)$
norms in an analogous fashion, but also summing over all ``frame indices.'' 
More precisely, with the background Kasner variables $\widetilde{k}_{IJ}$, etc., 
as defined in Sect.\,\ref{SS:BACKGROUNDVARIABLES}
(see in particular Remark~\ref{R:NOTCOMPONENTSOFTENSORRELATIVETOFRAME}), 
we define:
\begin{align}\label{HMnorms}
\begin{split}
\|k-\widetilde{k}\|_{H^M(\Sigma_t)}^2 :=\sum_{I,J=1}^{\mydim}\|k_{IJ}-\widetilde{k}_{IJ}\|_{H^M(\Sigma_t)}^2,
	\qquad
\|\upgamma\|_{H^M(\Sigma_t)}^2:=\sum_{I,J,B=1}^{\mydim}\|\upgamma_{IJB}\|_{H^M(\Sigma_t)}^2,
	\\
\|e-\widetilde{e}\|_{H^M(\Sigma_t)}^2:=\sum_{I,i=1}^{\mydim}\|e_I^i-\widetilde{e}_I^i\|_{H^M(\Sigma_t)}^2,
	\qquad
\|\oe-\widetilde{\oe}\|_{H^M(\Sigma_t)}^2:=\sum_{I,i=1}^{\mydim}\|\oe_i^I-\widetilde{\oe}_i^I\|_{H^M(\Sigma_t)}^2,
	\\
\|\vec{e} n \|_{H^M(\Sigma_t)}^2 
:=\sum_{I=1}^{\mydim}\|e_I n \|_{H^M(\Sigma_t)}^2,
\qquad
\|\vec{e} \psi \|_{H^M(\Sigma_t)}^2 :=\sum_{I=1}^{\mydim}\|e_I \psi \|_{H^M(\Sigma_t)}^2,
\end{split}
\end{align}
and:
\begin{align}\label{WMinfty.norms}
\begin{split}
\|k-\widetilde{k}\|_{W^{M,\infty}(\Sigma_t)}:=\sum_{I,J=1}^{\mydim}\|k_{IJ}-\widetilde{k}_{IJ}\|_{W^{M,\infty}(\Sigma_t)},
\qquad 
\|\upgamma\|_{W^{M,\infty}(\Sigma_t)}:=\sum_{I,J,B=1}^{\mydim}\|\upgamma_{IJB}\|_{W^{M,\infty}(\Sigma_t)},\\
\|e-\widetilde{e}\|_{W^{M,\infty}(\Sigma_t)}:=\sum_{I,i=1}^{\mydim} \|e_I^i-\widetilde{e}_I^i\|_{W^{M,\infty}(\Sigma_t)},
\qquad
\|\oe-\widetilde{\oe}\|_{W^{M,\infty}(\Sigma_t)}:= \sum_{I,i=1}^{\mydim} 
\|\oe_i^I-\widetilde{\oe}_i^I \|_{W^{M,\infty}(\Sigma_t)},
		\\
\|\vec{e} n \|_{W^{M,\infty}(\Sigma_t)}
:=\sum_{I=1}^{\mydim}\|e_I n \|_{W^{M,\infty}(\Sigma_t)},
\qquad
\|\vec{e} \psi \|_{W^{M,\infty}(\Sigma_t)} :=\sum_{I=1}^{\mydim}\|e_I \psi \|_{W^{M,\infty}(\Sigma_t)},
\end{split}
\end{align}
and similarly for the homogeneous norms 
(such as $\|k-\widetilde{k}\|_{\dot{H}^M(\Sigma_t)}^2 := 
\sum_{I,J=1}^{\mydim}\|k_{IJ}-\widetilde{k}_{IJ}\|_{\dot{H}^M(\Sigma_t)}^2$).

\subsection{Key parameters}
\label{SS:KEYPARAMETER}
We will formulate the bootstrap assumptions using two key parameters, namely $\upsigma,q$, which are any two 
fixed real numbers verifying the following inequalities:
\begin{align}\label{sigma,q}
\left\{\begin{array}{ll}
0<2 \upsigma< 2\upsigma + \displaystyle{\mathop{\max_{1 \leq I,J,B \leq \mydim }}_{I < J}}
\{|\widetilde{q}_B|,\widetilde{q}_I+\widetilde{q}_J-\widetilde{q}_B\} < q <1- 2 \upsigma,\qquad \text{non-symmetric cases},
\\
0<2 \upsigma<2 \upsigma+\max\{|\widetilde{q}_1|,|\widetilde{q}_2|,|\widetilde{q}_3|\}<q<1-2 \upsigma,\qquad\qquad\qquad\;\; 
\text{Polarized $U(1)$-symmetric 1+3 vacuum}.
\end{array}\right.
\end{align}
The set of Kasner exponents for which such parameters $\upsigma,q$ exist 
is non-empty and open in all the models that we consider; see Sect.\,\ref{subsec:models}.

Next, we introduce the positive-integer-valued parameters $N_0,N$, 
which, roughly speaking, represent the number of derivatives we will use to to control the solution
in $L^{\infty}$ at the low orders (i.e., derivative levels approximately equal to $N_0$) 
and in $L^2$ at the top-orders (i.e., derivative levels approximately equal to $N$);
we refer to Remark~\ref{rem:N0+1} for an important remark about the precise number of
low order derivatives that we use in our proof.
Our choice of $N_0$ and $N$ will be related to another parameter, $\blowupexp$, which
controls the strength of the $t$-weights (which will be of order $t^{\blowupexp}$) that we use in our 
high order energies. 
For our bootstrap argument to close, the parameters will have to satisfy the following inequalities:
\begin{align} \label{E:PARAMETERINEQUALITIES}
	N 
	&\gg N_0 \geq 1,
	&
	\blowupexp \gg 1,
\end{align}
where throughout the paper, we adjust the size of the parameters as necessary. 
Roughly, we will first choose $\blowupexp \geq 1$ to be large enough to dominate various order-unity
structural constants (denoted by the symbol ``$C_*$'' throughout the paper) in the PDEs.
We then fix any $N_0 \geq 1$. We will then choose $N$ to be sufficiently large in a manner 
that depends on 
$N_0, \blowupexp, q,$ and $\upsigma$
(as well as $\mydim$, the number of spatial dimensions).
See also Remarks~\ref{rem:N} and \ref{rem:A} for discussion on how to obtain crude estimates for $N$ and $\blowupexp$.

Finally, we will use a small parameter $0<\varepsilon \ll 1$ to capture the smallness
of the overall norms that measure the closeness of the perturbed solution to the background generalized Kasner metric. 
Roughly, for our bootstrap argument to close, we will first have to choose the other parameters as described above and then choose $\varepsilon$ to be sufficiently small in a manner that depends on
$N, N_0, \blowupexp, \mydim, q,$ and $\upsigma$.

\subsection{Definitions of the solution norms}
\label{SS:SOLUTIONNORMS}
In our bootstrap argument, 
we will rely on the $t$-weighted norms in the following definition.
Roughly, our main theorem shows that all of the norms
in the definition remain small throughout the entire interval $t \in (0,1]$ 
if they are small at $t=1$.

\begin{definition}[Solution norms]\label{D:SOLUTIONNORMS}
Recall that the parameter $N_0$ verifies $N_0 \geq 1$,
that $q, \upsigma \in (0,1)$ are the constants fixed in Sect.\,\ref{SS:KEYPARAMETER}, 
and that we introduced the notation ``$\vec{e} f$'' in Sect.\,\ref{SS:ADDITIONALDIFFERENTIATIONNOTATION}.
We define the low order norms:
\begin{subequations}
\begin{align}\label{norms.low}
\begin{split}
\mathbb{L}_{(e,\oe)}(t) 
& := \max\left\lbrace t^q\|e-\widetilde{e}\|_{W^{N_0,\infty}(\Sigma_t)}, \, t^q\|\oe-\widetilde{\oe}\|_{W^{N_0,\infty}(\Sigma_t)}
\right\rbrace,
	\\ 
\mathbb{L}_{(n)}(t)& :=
\max\left\lbrace t^{-\upsigma}\|n-1\|_{W^{N_0+1,\infty}(\Sigma_t)}, \, t^{q-\upsigma} \|\vec{e} n\|_{W^{N_0,\infty}(\Sigma_t)} \right\rbrace,
\\
\mathbb{L}_{(\upgamma,k)}(t)& :=
\max \left\lbrace t^q\|\upgamma\|_{W^{N_0,\infty}(\Sigma_t)}, \, t\|k-\widetilde{k}\|_{W^{N_0+1,\infty}(\Sigma_t)} \right\rbrace,
	\\
\mathbb{L}_{(\psi)}(t)  := &\,\max\left\lbrace
	t^q \|\vec{e} \psi\|_{W^{N_0,\infty}(\Sigma_t)},
	\,
	t \|e_0 \psi - \partial_t \widetilde{\psi} \|_{W^{N_0+1,\infty}(\Sigma_t)} \right\rbrace,
	\\
\mathbb{L}_{(e,\oe,\upgamma,k,\psi)}(t)
  & :=\mathbb{L}_{(e,\oe)}(t) + \mathbb{L}_{(\upgamma,k)}(t) + \mathbb{L}_{(\psi)}(t),
\end{split}
\end{align}
and the high order norms:
\begin{align}\label{norms.high}
\begin{split}
\mathbb{H}_{(e,\oe)}(t) & := \max \left\lbrace t^{\blowupexp + q}\|e\|_{\dot{H}^N(\Sigma_t)}, \, t^{\blowupexp + q}\|\oe\|_{\dot{H}^N(\Sigma_t)}
	\right\rbrace,
\\ 
\mathbb{H}_{(n)}(t) &:= \max \left\lbrace t^{\blowupexp} \|n\|_{\dot{H}^N(\Sigma_t)},
	\,
	t^{\blowupexp + 1} \|\vec{e} n \|_{\dot{H}^N(\Sigma_t)}
\right\rbrace,
\\
\mathbb{H}_{(\upgamma,k)}(t) &:= \max \left\{t^{\blowupexp + 1}\|\upgamma\|_{\dot{H}^N(\Sigma_t)}, \,
t^{\blowupexp + 1}\|k\|_{\dot{H}^N(\Sigma_t)} \right\},
	\\
\mathbb{H}_{(\psi)}(t) &:= \max\left\lbrace t^{\blowupexp + 1} \|\vec{e} \psi\|_{\dot{H}^N(\Sigma_t)},
\,
t^{\blowupexp + 1}\| e_0 \psi \|_{\dot{H}^N(\Sigma_t)}
\right\rbrace,
	\\
\mathbb{H}_{(e,\oe,\upgamma,k,\psi)}(t)
 & :=	\mathbb{H}_{(e,\oe)}(t) + \mathbb{H}_{(\upgamma,k)}(t) + \mathbb{H}_{(\psi)}(t).
\end{split}
\end{align}
We also find it convenient to define the following ``total norm'' for the ``dynamic'' variables
(i.e., the non-lapse\footnote{In Sect.\,\ref{sec:lapse}, we will use elliptic estimates
to show that the lapse can be controlled in terms of the dynamic variables;
see \eqref{E:LAPSECONTROLLEDBYDYNAMICVARIABLES}.} variables):
\begin{align} \label{E:TOTALDYNAMICNORM}
	\mathbb{D}(t)
	& := \mathbb{L}_{(e,\oe,\upgamma,k,\psi)}(t)
			+
			\mathbb{H}_{(e,\oe,\upgamma,k,\psi)}(t).
\end{align}
\end{subequations}
\end{definition}

\begin{remark}[Derivative counts involving $N_0$]
\label{rem:N0+1}
	Note that the low order norms in \eqref{norms.low}
	yield control over the ``kinetic'' (i.e., time-derivative-involving) terms
	$
	\lbrace k_{IJ} - \widetilde{k}_{IJ} \rbrace_{I,J=1,\cdots,\mydim}
	$
	and
	$
	e_0 \psi - \partial_t \widetilde{\psi}
	$
	at one derivative level higher than the remaining terms.
	This is important for our bootstrap argument, more precisely for our derivation 
	of the lower order 	
	estimates; see, for example, Lemma~\ref{L:STRUCTURECOFFEICIENTERRORTERMSPOINTWISE} 
	and the proof of \eqref{E:STRUCTURECOEFFICIENTSBORDERPOINTWISEBOUNDS}.
\end{remark}
\subsection{Bootstrap assumptions}
Our bootstrap assumptions are that there is
a ``bootstrap time'' $T_{\textnormal{Boot}}\in [0,1)$
such that:
\begin{align}\label{Boots}
\mathbb{D}(t)
+
\mathbb{L}_{(n)}(t)
+
\mathbb{H}_{(n)}(t)
& \leq \varepsilon,
&&
\forall t \in (T_{\textnormal{Boot}},1].
\end{align}
In the proof of our main theorem, such a $T_{\textnormal{Boot}} \in [0,1)$ 
will exist due to our near-Kasner assumptions on the data and Cauchy stability. 
\section{Basic estimates and identities}\label{sec:basic.est}
In this section, we provide some basic inequalities and commutation formulas that we will frequently 
use in our main estimates, i.e., in Sect.\,\ref{sec:mainest}.

\subsection{Interpolation and product inequalities} 
\label{SS:INTERPOLATIONANDPRODUCT}
In our ensuing analysis, we will control various error terms
with the help of the classical interpolation and Sobolev inequalities 
provided in the next lemma.
\begin{lemma}[Sobolev interpolation and product inequalities]
\label{lem:basic.ineq}
Let $v$ be a $\Sigma_t$-tangent tensorfield,
let $M_1,M_2$ be two non-negative integers,
and let $\iota_1,\cdots,\iota_R$ be spatial multi-indices such that
$\sum_{r=1}^R|\iota_r|=M_1$.
Then the following estimates hold,
where norms of tensorfields are defined as in Sect.\,\ref{SS:SOBOLEVNORMS},
and the implicit constants depend on $M_1$, $M_2$, and $\mydim$:
\begin{align}
\label{HM.interp}\|v\|_{\dot{H}^{M_1}(\Sigma_t)}\lesssim&\, \|v\|^{1-\frac{M_1}{M_2}}_{L^\infty(\Sigma_t)}\|v\|^{\frac{M_1}{M_2}}_{\dot{H}^{M_2}(\Sigma_t)} \lesssim \|v\|_{L^\infty(\Sigma_t)}+\|v\|_{\dot{H}^{M_2}(\Sigma_t)},
\quad \text{if $M_2\ge M_1$},\\
\label{WM.interp}\|v\|_{W^{M_1,\infty}(\Sigma_t)}\lesssim&\, \|v\|_{H^{M_1+1+\left\lfloor\frac{\mydim}{2} \right\rfloor}(\Sigma_t)}\lesssim \|v\|_{L^\infty(\Sigma_t)}+\|v\|_{\dot{H}^{M_2}(\Sigma_t)},\quad \text{if $M_2\ge M_1+1+\left\lfloor\frac{\mydim}{2} \right\rfloor$},\\
\label{Sob.prod}\|\partial^{\iota_1}v_1\cdots\partial^{\iota_R}v_R\|_{L^2(\Sigma_t)}\lesssim&\,\sum_{r=1}^R\|v_r\|_{\dot{H}^{M_1}(\Sigma_t)}\prod_{s\neq r}\|v_s\|_{L^\infty(\Sigma_t)},
\end{align}
where $\left\lfloor\frac{\mydim}{2} \right\rfloor$ is the integer part of $\frac{\mydim}{2}$. 

Moreover, if $1 \leq R_0 \leq R$ and
$\iota_1,\cdots,\iota_R$ are spatial multi-indices such that
$\sum_{r=1}^R|\iota_r|=M_1$ and $|\iota_{R-R_0+1}|,\cdots,|\iota_R| \leq M_1 - 1$,
then the following product inequality holds,
where implicit constant depends on $M_1$, $R$, $R_0$, and $\mydim$:
\begin{align}\label{Sob.prod2}
\begin{split}
\|\partial^{\iota_1}v_1\cdots\partial^{\iota_R}v_R\|_{L^2(\Sigma_t)}
& \lesssim
\sum_{r=1}^{R-R_0}
	\left(
		\|v_r\|_{W^{1,\infty}(\Sigma_t)}
		+
		\|v_r\|_{\dot{H}^{M_1}(\Sigma_t)}
	\right)
	\prod_{s\neq r}\|v_s\|_{W^{1,\infty}(\Sigma_t)}
		\\
& \ \
	+
	\sum_{r=R-R_0+1}^R
	\left(
		\|v_r\|_{W^{1,\infty}(\Sigma_t)}
		+
		\|v_r\|_{\dot{H}^{M_1-1}(\Sigma_t)}
	\right)
\prod_{s\neq r}\|v_s\|_{W^{1,\infty}(\Sigma_t)}.
\end{split}
\end{align}
\end{lemma}
\begin{proof}
The first inequality in \eqref{HM.interp} is immediate\footnote{Alternatively,
\eqref{HM.interp} could be derived as a straightforward consequence of Nirenberg's interpolation results \cite{lN1959},
an approach that has the added advantage that it is easy to generalize to topologies other than $\mathbb{T}^{\mydim}$.
\label{FN:ALTERNATEAPPROACHNIRENBERG}} 
from Plancherel's identity, H\"older's inequality, and the bound
			$\| v \|_{L^2(\Sigma_t)} 
		\lesssim
		\| v \|_{L^{\infty}(\Sigma_t)}
		$
		for scalar functions $v$
		(which holds because $\mathbb{T}^{\mydim}$ is compact).
The second inequality in \eqref{HM.interp} follows from the first and Young's inequality. 
In the case $\Sigma_t = \mathbb{R}^{\mydim}$, the inequality \eqref{Sob.prod}  
was proved as \cite[Lemma 6.16]{Rin4}, and the same proof works in the case
$\Sigma_t = \mathbb{T}^{\mydim}$.
The first inequality in \eqref{WM.interp} is standard Sobolev embedding, 
while the second inequality in \eqref{WM.interp} 
follows from applying \eqref{HM.interp} to the homogeneous norms $\dot{H}^{M_1'}(\Sigma_t)$ of $v$, 
for every $M_1'\leq M_1+1+\lfloor \frac{\mydim}{2}\rfloor$.
To derive \eqref{Sob.prod2}, we first note that either all derivatives act on one of the terms $v_1,\cdots,v_{R-R_0}$, say $v_1$, or there exist at least two factors having at least one derivative, say $v_1,v_a$, where $a > 1$. Then setting 
$u_1 := \partial v_1$ in the first case or $u_1:= \partial v_1$, $u_a:= \partial v_a$ in the second case, we apply 
\eqref{Sob.prod} and \eqref{HM.interp} 
to the product, where we view $u_1$ and $u_a$ to be terms in the product that are hit with
one fewer derivative than $v_1$ and $v_a$. This yields the
desired estimate.
\end{proof}
As an immediate application of Lemma~\ref{lem:basic.ineq},
we provide the next lemma, 
which yields control of the reduced solution variables at orders slightly higher than $N_0$.
The price we pay is that the estimates are slightly (when $N$ is large) 
more singular with respect to powers of $t$ compared to the very-low-order estimates. Nevertheless,
a small increase in the singularity strength is allowable
for treating error terms that are sub-critical with respect to powers of $t$.

\begin{lemma}[$L^{\infty}$ control at slightly higher orders than $N_0$ -- with only a mild increase in singularity strength for large $N$]
\label{lem:Sob.borrow}
Assume that the bootstrap assumptions \eqref{Boots} hold.
Then there exists a constant $\updelta = \updelta(N,\mydim)$ 
(which is free to vary from line to line)
such that $\updelta \rightarrow 0$ as $N \to \infty$ and  
such that if $N \geq N_0 + 4 + \left\lfloor\frac{\mydim}{2} \right\rfloor$, then the following estimates hold
for $t\in(T_{\textnormal{Boot}},1]$:
\begin{align}
\label{omega.borrow.est}
\|e-\widetilde{e}\|_{W^{{N_0+2},\infty}(\Sigma_t)}+\|\oe-\widetilde{\oe}\|_{W^{{N_0+2},\infty}(\Sigma_t)}\lesssim&\, t^{-q-\updelta \blowupexp}\left\lbrace\mathbb{L}_{(e,\oe)}(t)+\mathbb{H}_{(e,\oe)}(t)\right\rbrace,\\
\label{gamma.borrow.est}\|\upgamma-\widetilde{\upgamma}\|_{W^{{N_0+2},\infty}(\Sigma_t)}\lesssim&\, t^{-q-\updelta \blowupexp}\left\lbrace\mathbb{L}_{(\upgamma,k)}(t)+\mathbb{H}_{(\upgamma,k)}(t)\right\rbrace,\\
\label{k.borrow.est}\|k-\widetilde{k}\|_{W^{{N_0+2},\infty}(\Sigma_t)}\lesssim&\, t^{-1-\updelta \blowupexp}\left\lbrace\mathbb{L}_{(\upgamma,k)}(t)+\mathbb{H}_{(\upgamma,k)}(t)\right\rbrace,\\
\label{n.borrow.est}
\|n-1\|_{W^{{N_0+3},\infty}(\Sigma_t)}
+
t^q
\|\vec{e} n\|_{W^{{N_0+2},\infty}(\Sigma_t)}
\lesssim
&\, t^{\upsigma-\updelta \blowupexp}\left\lbrace\mathbb{L}_{(n)}(t)
+
\mathbb{H}_{(n)}(t)\right\rbrace,\\
\label{psi.borrow.est} t^{q} \| \vec{e} \psi\|_{W^{{N_0+2},\infty}(\Sigma_t)}+t\|\partial_t\psi\|_{W^{{N_0+2},\infty}(\Sigma_t)}\lesssim&\,t^{-\updelta \blowupexp}\left\lbrace\mathbb{L}_{(\psi)}(t)+\mathbb{H}_{(\psi)}(t)\right\rbrace.
\end{align}
\end{lemma}
\begin{proof}
The argument for all inequalities is essentially the same, so we only prove \eqref{n.borrow.est}. Using 
first \eqref{WM.interp} and then \eqref{HM.interp}, we find that for $N \geq N_0+4+\left\lfloor\frac{\mydim}{2} \right\rfloor$,
we have:
\begin{align*}
\|n-1\|_{W^{N_0+3,\infty}(\Sigma_t)} 
& \lesssim
\|n-1\|_{L^\infty(\Sigma_t)}+\|n-1\|_{\dot{H}^{N_0+4+\left\lfloor\frac{\mydim}{2} \right\rfloor}(\Sigma_t)}
\\
& \lesssim \|n-1\|_{L^\infty(\Sigma_t)}
+\|n-1\|_{L^\infty(\Sigma_t)}^{1-\widetilde{\updelta}}\|n-1\|^{\widetilde{\updelta}}_{\dot{H}^{N}(\Sigma_t)}
\\
& \leq t^{\upsigma}\mathbb{L}_{(n)}(t)+\big(t^{\upsigma}\mathbb{L}_{(n)}(t)\big)^{1-\widetilde{\updelta}}\big(t^{-\blowupexp}\mathbb{H}_{(n)}(t)\big)^{\widetilde{\updelta}}
\\
& = t^{\upsigma}\mathbb{L}_{(n)}(t)+t^{\upsigma-(\blowupexp+\upsigma)\widetilde{\updelta}}\mathbb{L}_{(n)}^{1-\widetilde{\updelta}}(t)\mathbb{H}_{(n)}^{\widetilde{\updelta}}(t)
\\
& \lesssim t^{\upsigma-\updelta \blowupexp}\left\lbrace \mathbb{L}_{(n)}(t)+\mathbb{H}_{(n)}(t) \right\rbrace,
\end{align*}
where $\widetilde{\updelta} : =\frac{N_0+4+\left\lfloor\frac{\mydim}{2} \right\rfloor}{N}\leq 1$, and for the last inequality, 
we used Young's inequality and set $\updelta : =\frac{\blowupexp+\upsigma}{\blowupexp}\widetilde{\updelta}$. 
It is clear that $\updelta \to 0$, as $N \rightarrow \infty$, at a rate that is independent 
of how large $\blowupexp \geq 1$ is. 
This yields \eqref{n.borrow.est} for the term $\|n-1\|_{W^{N_0+3,\infty}(\Sigma_t)}$.
The estimate for the term $t^q \|\vec{e} n\|_{W^{{N_0+2},\infty}(\Sigma_t)}$
would then follow from the Leibniz rule, 
the estimate for the term
$\|n-1\|_{W^{N_0+3,\infty}(\Sigma_t)}$,
and the estimate \eqref{omega.borrow.est} for the term $\|e-\widetilde{e}\|_{W^{{N_0+2},\infty}(\Sigma_t)}$
(which for purposes of exposition we assume to have already been proved).
We clarify that by this argument, the value of $\updelta$ corresponding to the estimate for 
$t^q \|\vec{e} n\|_{W^{{N_0+2},\infty}(\Sigma_t)}$
might be larger than the value of $\updelta$ for
$\|n-1\|_{W^{N_0+3,\infty}(\Sigma_t)}$, but nevertheless, all
``$\updelta$'s'' tend to $0$ as $N \to \infty$.
\end{proof}

\begin{remark}[$\updelta$ can vary from line to line]
\label{R:DELTAVARIES}
In the rest of the paper, $\updelta = \updelta(N,\mydim)$
denotes a small positive constant that is free to vary from line to line,
but that always has the property that
$\updelta \rightarrow 0$ as $N \to \infty$ (as in Lemma~\ref{lem:Sob.borrow}).
In particular, we sometimes express the sum of two $\updelta$'s as
another $\updelta$. 
\end{remark}

\begin{remark}[Smallness of $\updelta \blowupexp$]\label{rem:deltaA}
Later in the paper, when we use Lemma~\ref{lem:Sob.borrow} to derive estimates for the solution,
we will always assume (sometimes without explicitly mentioning it) 
that $\updelta \blowupexp$ is as small as we need it to be.
In particular, we assume that it is small enough such that $\updelta \blowupexp < \upsigma$
so that, for example, $t^{2\upsigma - \updelta \blowupexp} \leq t^{\upsigma}$ for $t \in (0,1]$. 
At fixed $\blowupexp$, the desired smallness can be ensured by choosing $N$ to be sufficiently large.
\end{remark}
\begin{remark}[Large interpolation constants are not an obstacle to stability]\label{rem:int.const}
The implicit constants in the interpolation inequalities of Lemmas~\ref{lem:basic.ineq} and \ref{lem:Sob.borrow} 
depend on $M_1, M_2, N_0, N$, and the number of spatial dimensions 
$\mydim$. One might worry, especially when taking $N$ sufficiently large to make $\updelta$ small, that the constants in the
elliptic and energy estimates, corresponding to the terms that we treat using these inequalities, can be quite large. 
While the constants ``$C$'' can in fact be large, largeness does not obstruct the proofs of our results. 
The reason is that we only apply these inequalities to handle two kinds of error terms: 
\textbf{i)} error terms that are sub-critical with respect to powers of $t$,
for which the largeness of $C$ is admissible within the context of our Gr\"{o}nwall estimates; and
\textbf{ii)} critical ``borderline'' products with one factor that yields a smallness factor of $\varepsilon$, so that the effective
coefficient $C \varepsilon$ can be made as small as one wants by choosing the bootstrap parameter $\varepsilon$ to be small (which is possible for initial data on $\Sigma_1$ that  are sufficiently close to the Kasner data).
In particular, in our estimates,
the implicit constants in Lemmas~\ref{lem:basic.ineq} and \ref{lem:Sob.borrow} do not affect the size of the
important constants ``$C_*$'' (see Sect.\,\ref{SS:NOTATIONANDCONVENTIONS} for our conventions for constants ``$C_*$'')
or the value of the parameter $\blowupexp$.
See also Sect.\,\ref{subsec:bord.vs.bbord} for further discussion of borderline and below-borderline terms.
\end{remark}
\subsection{Two simple commutation formulas}
\label{SS:SIMPLECOMMUTATIONFORMULATS}
To derive estimates for the solution's derivatives,
we will repeatedly commute the reduced equations with the transported spatial coordinate partial derivative vectorfields 
$\lbrace \partial_i \rbrace_{i=1,\cdots,\mydim}$, 
and we will use the following commutation relation
to uncover the structure of various error terms 
(see Sect.\,\ref{SS:NOTATIONANDCONVENTIONS} for our conventions for multi-indices):
\begin{align}\label{comm}
[\partial^{\iota},e_I]v=\sum_{\iota_1\cup \iota_2 =\iota,\,|\iota_2|<|\iota|}
(\partial^{\iota_1} e_I^c) \partial^{\iota_2}\partial_c v.
\end{align}
The identity \eqref{comm} follows easily from expanding $e_I=e_I^c\partial_c$. 
In our forthcoming analysis, we will sometimes use it silently.

We will also use the following commutation identity:
\begin{align}\label{comm[dt,ei]}
[\partial_t,e_I]
& 
=n k_{IC} e_C^c \partial_c,
\end{align}
which we derived in \eqref{INVARIANTframe}.

\section{Main estimates}\label{sec:mainest}
Our main goal in this section is to establish Proposition~\ref{prop:overall},
which forms the analytical cornerstone of the paper.
The proposition provides a priori estimates for perturbations of the Kasner background solution
and in particular yields improvements of the bootstrap assumptions when the data are 
sufficiently near-Kasner. We also highlight that for near-Kasner data,
the a priori estimates 
and standard arguments 
collectively imply that the solution
exists on the entire half-slab $(0,1] \times \mathbb{T}^{\mydim}$ and enjoys the quantitative properties
afforded by the a priori estimates; see Proposition~\ref{P:EXISTENCEONHALFSLAB} for those details.

\subsection{Statement of the main a priori estimates}
\label{SS:MAINAPRIORI}
In the next proposition, we state our main a priori estimates.
The proof is located in Sect.\,\ref{PROOFOFprop:overall}.
In the sections that precede it,
we will establish a series of preliminary identities and estimates
for $n, \upgamma, k$, the frame $\lbrace e_I \rbrace_{I=1,\cdots,\mydim}$,
and the co-frame $\lbrace \oe^I \rbrace_{I=1,\cdots,\mydim}$.
The proof of the proposition essentially amounts to combining the preliminary results. 

\begin{proposition}[The main a priori estimates]
\label{prop:overall}
Let $(n,k_{IJ},\upgamma_{IJB},e_I^i,\oe_i^I,\psi)_{I,J,B,i=1,\cdots,\mydim}$
be a solution to the reduced equations of Proposition~\ref{P:redeq} on
$(T_{\textnormal{Boot}},1] \times \mathbb{T}^{\mydim}$.
Recall that $\mathbb{D}(t)$ is the total norm of the dynamic variables
and that $\mathbb{L}_{(n)}(t)$ and $\mathbb{H}_{(n)}(t)$ are norms of the lapse
(see Definition~\ref{D:SOLUTIONNORMS}).
Let $\mathring{\upepsilon}$ denote the initial value of the total norm of the dynamic variables:
\begin{align}\label{E:INITIALNORMOFDYNAMICVARIABLES}
\mathring{\upepsilon}
& := \mathbb{D}(1)
=
\mathbb{L}_{(e,\oe)}(1)
+
\mathbb{L}_{(\upgamma,k)}(1)
+
\mathbb{L}_{(\psi)}(1)+\mathbb{H}_{(e,\oe)}(1)
+
\mathbb{H}_{(\upgamma,k)}(1)
+
\mathbb{H}_{(\psi)}(1).
\end{align}
Assume that the bootstrap assumptions \eqref{Boots} hold for $t \in (T_{\textnormal{Boot}},1]$.
If $\blowupexp$ is sufficiently large and $N_0 \geq 1$, then
there exists a constant $C_{N,N_0,\blowupexp,\mydim,q,\upsigma} > 0$ 
such that if $N$ is sufficiently large in a manner 
that depends on $N_0, \blowupexp, \mydim, q,$ and $\upsigma$,
and if $\varepsilon$ is sufficiently small (in a manner that depends on $N, N_0, \blowupexp, \mydim, q,$ and $\upsigma$), 
then the following estimate holds for $t \in (T_{\textnormal{Boot}},1]$:
\begin{align}\label{overall.est}
\mathbb{D}(t)
+
\mathbb{L}_{(n)}(t) 
+
\mathbb{H}_{(n)}(t)
& \leq C_{N,N_0,\blowupexp,\mydim,q,\upsigma} \mathring{\upepsilon}.
\end{align}
%aA
In particular, if $C_{N,N_0,\blowupexp,\mydim,q,\upsigma} \mathring{\upepsilon} < \varepsilon$, 
then \eqref{overall.est} yields a strict improvement of the bootstrap assumptions \eqref{Boots}.
\end{proposition}

\subsection{Schematic notation}
\label{SS:SCHEMATIC}
We will use schematic notation to simplify the presentation of various formulas when the precise
structure of the terms is not important. $\partial$ denotes an arbitrary partial derivative
with respect to one of the transported spatial coordinate vectorfields.
$k$ denotes an arbitrary element of the array $(k_{IJ})_{I,J=1,\cdots,\mydim}$
of components of the second fundamental form with respect to the orthonormal frame.
$\partial^{\iota} k$ denotes an arbitrary element of the array 
$(\partial^{\iota} k_{IJ})_{I,J=1,\cdots,\mydim}$.
Similarly, $\upgamma$ denotes an arbitrary element of the array $(\upgamma_{IJB})_{I,J,B=1,\cdots,\mydim}$
and $\partial^{\iota} \upgamma$ denotes an arbitrary element of the array 
$(\partial^{\iota} \upgamma_{IJB})_{I,J,B=1,\cdots,\mydim}$.
$e$ denotes an arbitrary element of the array $(e_I^i)_{I,i=1,\cdots,\mydim}$,
while
$\oe$ denotes an arbitrary element of the array $(\oe_i^I)_{I,i=1,\cdots,\mydim}$.
If $f$ is a scalar function, then $\vec{e} f$ denotes the array $(e_I f)_{I=1,\cdots,\mydim}$.

As an example, with the help of the notation from Sect.\,\ref{SS:NOTATIONANDCONVENTIONS},
we can express the commutator
$
\partial^{\iota}  (n e_C \upgamma_{IJC})
-
n e_C \partial^{\iota} \upgamma_{IJC}
$
in the following schematic form:
$
\sum_{\iota_1\cup \iota_2 \cup \iota_3 = \iota, |\iota_3| < |\iota|}	
			\partial^{\iota_1} n
			\cdot
			\partial^{\iota_2} e 
			\cdot
			\partial 
			\partial^{\iota_3}
			\upgamma
$.
We remark that we use schematic notation only when the overall signs and precise
numerical coefficients in front of the terms
is not important. Thus, \textbf{when using schematic notation for terms,
we do not account for their overall signs or precise numerical coefficients.}

\subsection{Borderline terms vs.\ Junk terms}
\label{subsec:bord.vs.bbord}
In our top-order energy estimates, we encounter some delicate error terms that cannot be treated by Gr\"{o}nwall's lemma uniformly in $T_{\textnormal{Boot}}{}{\in(0,1)}$. That is, if treated crudely, these terms would 
prevent us from deriving an energy estimate that would lead to an improvement of our bootstrap assumptions. 
We described one example of such a term at the end of Sect.\,\ref{SSS:INTROHIGHORDERENERGYESTIMATES}.
Let us revisit this issue in more detail. In our top-order energy estimates, 
we encounter ``borderline'' error integrands with the following strength:
\begin{align} \label{E:BORDERLINEERRORTERM}
\frac{1}{t} \cdot t^{2 \blowupexp+2}\partial^{\iota} \upgamma \cdot \partial^{\iota}\upgamma,
	\qquad
\frac{1}{t} \cdot t^{2 \blowupexp+2}\partial^{\iota}k \cdot \partial^{\iota}k,
	\qquad
\frac{1}{t} \cdot t^{2 \blowupexp+2} \partial^{\iota}\upgamma \cdot \partial^{\iota}(e_I n).
\end{align}
The difficulty is that the integrands in \eqref{E:BORDERLINEERRORTERM} 
are more singular than the energy density itself
due to the factors of $\frac{1}{t}$. To handle these error terms, we 
exploit the following crucial fact, which we must justify in our analysis:
	\begin{quote}
		In the energy identities, the \underline{coefficients} of all of 
		the borderline terms can be bounded by a uniform constant $C_*$, independent of $\blowupexp$ and $N$,
		as long as the bootstrap parameter $\varepsilon$ is sufficiently small
		(in a manner that is allowed to depend on $N$ and $\blowupexp$).
		Such terms contribute to the $C_*$-multiplied integrals on the
		right-hand side of the energy inequalities of Proposition~\ref{P:TOPORDERENERGYINEQUALITY}.
	\end{quote}
{We refer readers to Remark~\ref{R:MEANINGOFBORDERLINE} for further comments on our use
of the terminology ``borderline.''}

At this point, the role of the $t^{2 \blowupexp+2}$ weights in our energy identities emerges:
the weights \emph{also} generate borderline terms 
(roughly, when the $\partial_t$ derivative falls on the weights in the energy identities)
of the same strength as those in \eqref{E:BORDERLINEERRORTERM},
but unlike the terms in \eqref{E:BORDERLINEERRORTERM},
the error terms generated by the weights have a \underline{favorable sign} towards the singularity
with an overall coefficient that is proportional to $\blowupexp$.
These terms contribute to the favorable $-\blowupexp$-multiplied integrals
on the right-hand side of the energy inequalities of Proposition~\ref{P:TOPORDERENERGYINEQUALITY}.
Thus, if $\blowupexp$ is chosen sufficiently large, the overall coefficient $C_* - \blowupexp$ of the borderline terms becomes negative, 
and in our energy estimates, the corresponding integral has a ``good sign'' and can be discarded. 
We again stress that for this argument to work,
it is crucial that $C_*$ can be chosen to be independent of $\blowupexp,N$, at least when $\varepsilon$ is small.

On the other hand, there are many terms in the energy estimates that are ``junk'' in the sense that they can be bounded by our norms times a factor of strength $C t^{-1+\upsigma}$. Although ``$C$'' 
is allowed to depend on $\blowupexp,N$, and other parameters 
(cf.\ Remark~\ref{rem:int.const} regarding the size of the constants $C$ in the interpolation inequalities), 
such terms do not pose any difficulty in 
the a priori energy estimates. The reason is that $Ct^{-1+\upsigma}$ is integrable in time near $t=0$
and thus, in the context of Gr\"{o}nwall's lemma, 
the factor $C t^{-1+\upsigma}$ causes only finite growth of our energies, 
which is perfectly compatible with our bootstrap argument and our proof of stability.

\begin{remark}[``Border'' and ``Junk'' notation]
	\label{R:BORDERANDNJUNK}
	To help the reader navigate the energy estimates,
	in our ensuing analysis, we label error terms that generate borderline (in the sense above) error terms
	with the superscript ``Border,'' and we label error terms that generate junk (in the sense above) error
	terms with the superscript ``Junk.'' See, for example
	the terms
	$t^{\Pow-1} \mathfrak{K}_{IJ}^{(\textnormal{Border};\iota)}$
	and $t^{\Pow} \mathfrak{K}_{IJ}^{(\textnormal{Junk};\iota)}$
	on RHS~\eqref{E:COMMUTEDKEQUATION}.
	
	We sometimes use similar notation to distinguish between ``borderline terms'' and ``junk terms'' 
	in our pointwise estimates; see, however, Remark~\ref{R:MEANINGOFBORDERLINE}.
\end{remark}

\subsection{Control of the lapse $n$ in terms of the dynamic solution variables}\label{sec:lapse}
Our main goal in this subsection is to prove the following proposition,
which yields control of the lapse in terms of the remaining ``dynamic'' solution variables.
This is a preliminary step in our derivation of a priori estimates for all solution variables. 
The proof of the proposition relies on elliptic estimates and the bootstrap assumptions \eqref{Boots}
and is located in Sect.\,\ref{SSS:PROOFOFPROPLAPSE}. Before proving the proposition, we first establish some preliminary identities and estimates.

\begin{proposition}[Estimates for the lapse in terms of the dynamic solution variables]
\label{prop:n}
Recall that 
$\mathbb{L}_{(n)}(t)$,
$\mathbb{H}_{(n)}(t)$,
$\mathbb{H}_{(\upgamma,k)}(t)$,
and $\mathbb{D}(t)$ are norms from Definition~\ref{D:SOLUTIONNORMS}.
Under the assumptions of Proposition~\ref{prop:overall},
there exists a constant $C_* > 0$ \underline{independent of $N, N_0,$ and $\blowupexp$}
and a constant $C = C_{N,N_0,\blowupexp,\mydim,q,\upsigma} > 0$ 
such that if $N_0 \geq 1$ and $N$ is sufficiently large in a manner 
that depends on $N_0, \blowupexp, \mydim, q,$ and $\upsigma$,
and if $\varepsilon$ is sufficiently small (in a manner that depends on $N, N_0, \blowupexp, \mydim, q,$ and $\upsigma$), 
then the following estimates hold for $t \in (T_{\textnormal{Boot}},1]$:
\begin{align}
\label{n.low.est}
\| n-1 \|_{W^{N_0+1,\infty}(\Sigma_t)}
+
t^q \| \vec{e} n \|_{W^{N_0,\infty}(\Sigma_t)}
& \leq C t^{\upsigma} \mathbb{D}(t).
\end{align}

Moreover, if $\iota$ is any spatial multi-index with $|\iota| = N$,
then we have:
\begin{subequations}
\begin{align}  
t^{\blowupexp + 1} \| \partial^{\iota} \vec{e} n \|_{L^2(\Sigma_t)}
+
t^{\blowupexp} \| \partial^{\iota} n \|_{L^2(\Sigma_t)}
& \leq 
C_* t^{\blowupexp + 1}
\| \partial^{\iota} \upgamma \|_{L^2(\Sigma_t)}
+
C t^{\upsigma} \mathbb{D}(t),
	\label{PRECISE.n.high.est}  \\
t^{\blowupexp + 1}\| \vec{e} n \|_{\dot{H}^N(\Sigma_t)}
+
t^{\blowupexp} \| n \|_{\dot{H}^N(\Sigma_t)}
& \leq 
C_* 
\mathbb{H}_{(\upgamma,k)}(t)
+
C t^{\upsigma} \mathbb{D}(t).
\label{n.high.est} 
\end{align}
\end{subequations}

Finally, the lapse norms are bounded by the dynamic variable norm:
\begin{align}\label{E:LAPSECONTROLLEDBYDYNAMICVARIABLES}
		\mathbb{L}_{(n)}(t)
		+
		\mathbb{H}_{(n)}(t)
		& \leq C \mathbb{D}(t).
\end{align}
\end{proposition}

\subsubsection{Equations for controlling the lapse}
\label{SSS:EQUATIONSFORCONTROLLINGLAPSE}
We start by deriving the elliptic equations satisfied by the derivatives of the lapse.

\begin{lemma}[The commuted lapse equation]
\label{L:COMMUTEDLAPSEEQUATION}
For solutions $n$ to the lapse equation \eqref{n.eq}
and spatial coordinate multi-indices $\iota$ with $|\iota|\leq N$,
the following equation holds:
\begin{align}\label{n.eq.diff}
e_C  \partial^{\iota} e_C (n-1)
-
t^{-2} \partial^{\iota}(n-1)
& = 
	2n e_D \partial^{\iota} \upgamma_{CCD}
	+
	\mathfrak{N}^{(\iota)},
\end{align}
where:
\begin{align}
\begin{split}
\mathfrak{N}^{(\iota)}
& := 
		\sum_{\iota_1 \cup \iota_2 = \iota, |\iota_2| < |\iota|}	
		\partial^{\iota_1} e 
		\cdot 
		\partial \partial^{\iota_2} \vec{e} n
		\label{E:LAPSEPDECOMMUTEDTHROUGHOUTERFRAMEDERIVATIVEERRORTERM}
			\\
& \ \
			+
			\sum_{\iota_1 \cup \iota_2 \cup \iota_3 = \iota, |\iota_3| < |\iota|}	
			\partial^{\iota_1} n
			\cdot
			\partial^{\iota_2} e 
			\cdot
			\partial 
			\partial^{\iota_3}
			\upgamma
		+
		\sum_{\iota_1\cup \iota_2 = \iota}
		\partial^{\iota_1} \upgamma \cdot \partial^{\iota_2} \vec{e} n
		 \\
& \ \
	+ \sum_{\iota_1 \cup \iota_2 \cup \iota_3 = \iota}
		\partial^{\iota_1} n
		\cdot
		\partial^{\iota_2} \upgamma
		\cdot
		\partial^{\iota_3} \upgamma
	+ \sum_{\iota_1 \cup \iota_2 \cup \iota_3 = \iota}
		\partial^{\iota_1} n
		\cdot
		\partial^{\iota_2} \vec{e} \psi
		\cdot
		\partial^{\iota_3} \vec{e} \psi.
\end{split}
\end{align}
\end{lemma}
\begin{proof}
	\eqref{n.eq.diff} follows from differentiating \eqref{n.eq} with $\partial^{\iota}$
	and using the commutation formula \eqref{comm} and the Leibniz rule.
\end{proof}

\subsubsection{A standard elliptic identity}
\label{SSS:ELLIPTICIDENTITYFORLAPSE}
In the next lemma, we provide a standard elliptic identity for
the lapse. We will use the identity to establish 
$L^2$-control of the lapse at the top order.

\begin{lemma}[Elliptic identity for $n$]
	\label{L:ELLIPTICDIFFERENTIALLAPSEID}
	Let $\iota$ be a spatial coordinate multi-index with $1 \leq |\iota| \leq N$.
	Then for solutions to equation \eqref{n.eq.diff},
	the following identity holds:
	\begin{align} \label{E:ELLIPTICDIFFERENTIALLAPSEID}
		\begin{split}
		&
		t^{2 \blowupexp+2}
		(\partial^{\iota} e_C n) 
		\partial^{\iota} e_C n
		+
		t^{2 \blowupexp} (\partial^{\iota} n)^2
			\\
		& = 
		2n 
		(t^{\blowupexp + 1} \partial^{\iota} e_D n) 
		(t^{\blowupexp + 1} \partial^{\iota} \upgamma_{CCD})
		 -
				(t^{\blowupexp} \partial^{\iota} n)
				(t^{A+2} \mathfrak{N}^{(\iota)})
				+
				t^{2 \blowupexp+2}
				\mathfrak{R}^{(\iota)}
					 \\
		& \ \ 
		+
		\partial_c
		\left\lbrace
			t^{2 \blowupexp+2}
			e_C^c 
			(\partial^{\iota} e_C n) 
			\partial^{\iota} n
		\right\rbrace
	-	
	\partial_c 
	\left\lbrace 
		2t^{2 \blowupexp+2}
		n e_D^c 
		(\partial^{\iota} n)
		\partial^{\iota} \upgamma_{CCD}
	\right\rbrace,
	\end{split}
		\\
\begin{split}
\mathfrak{R}^{(\iota)}
	 &:=
		(\partial^{\iota} e_C n) 
		([\partial^{\iota},e_C])
		 n
		-
		2n 
		([\partial^{\iota}, e_D] n) 
		\partial^{\iota} \upgamma_{CCD}
				\label{E:TOPORDERLAPSEELLIPTICIDENTITYANNOYINGTERMS}	
				\\
		&  \ \
				-
					(\partial_c e_C^c)
					(\partial^{\iota} e_C n)
					\partial^{\iota} n
				+
				\left\lbrace
					\partial_c (2n e_D^c)
				\right\rbrace
					(\partial^{\iota} n)
					\partial^{\iota} \upgamma_{CCD}.
	\end{split}
	\end{align}
\end{lemma}

\begin{proof}
We first multiply \eqref{n.eq.diff} with $-\partial^{\iota} n$ and 
	differentiate by parts in the top-order terms after expanding $e_C=e_C^c\partial_c$ and $e_D=e_D^c\partial_c$
	to obtain the following identity:
\begin{align} \label{E:FIRSTSTEPELLIPTICDIFFERENTIALLAPSEID}
\begin{split}
&
	-
	\partial_c\big\{e_C^c  (\partial^{\iota} e_Cn)  \partial^{\iota}n\big\}
	+
	(\partial_ce_C^c)(\partial^\iota e_Cn)\partial^\iota n
	+
	(\partial^{\iota} e_Cn) [e_C,\partial^{\iota}]n
	+
	(\partial^{\iota} e_Cn)  \partial^{\iota}e_Cn
	+
	t^{-2} (\partial^{\iota}n)^2
	\\
 & =
	- 
	2 \partial_c\big\{n (\partial^{\iota}n) e_D^c \partial^{\iota} \upgamma_{CCD}\big\}
	+ 
	\left\lbrace
		\partial_c (2n e_D^c)
	\right\rbrace 
	(\partial^{\iota}n) \partial^{\iota} \upgamma_{CCD}
	+
	2 n ([e_D,\partial^\iota]n) \partial^{\iota} \upgamma_{CCD}
		\\
	&	\ \
	+ 2n (\partial^\iota e_Dn ) \partial^{\iota} \upgamma_{CCD}
	-
	 (\partial^{\iota}n) \mathfrak{N}^{(\iota)}
	\end{split}
\end{align}
 Multiplying \eqref{E:FIRSTSTEPELLIPTICDIFFERENTIALLAPSEID} by $t^{2 \blowupexp+2}$ 
and rearranging the terms, we arrive at the desired identity \eqref{E:ELLIPTICDIFFERENTIALLAPSEID}.
\end{proof}

\subsubsection{Control of the error terms in the top-order commuted lapse equation}
\label{SSS:TOPORDERLAPSEINHOMOGENEOUSTERML2ESTIMATE}
In the next lemma, we derive $L^2$-control of the error terms in the top-order commuted lapse equation.

\begin{lemma}[$L^2$-control of the error terms in the top-order commuted lapse equation]
	\label{L:TOPORDERLAPSEINHOMOGENEOUSTERML2ESTIMATE}
	Recall that $\mathbb{D}(t)$ is the total norm of the dynamic variables from Definition~\ref{D:SOLUTIONNORMS}.
	Let
	$\mathfrak{N}^{(\iota)}$
	and $\mathfrak{R}^{(\iota)}$
	denote the lapse equation error terms
	defined respectively 
	in
	\eqref{E:LAPSEPDECOMMUTEDTHROUGHOUTERFRAMEDERIVATIVEERRORTERM}
	and
	\eqref{E:TOPORDERLAPSEELLIPTICIDENTITYANNOYINGTERMS}
	(these terms appear on the right-hand side of \eqref{E:ELLIPTICDIFFERENTIALLAPSEID}).
	Under the assumptions of Proposition~\ref{prop:n},
	there exists a constant $C = C_{N,N_0,\blowupexp,\mydim,q,\upsigma} > 0$ 
such that the following estimates hold for $t \in (T_{\textnormal{Boot}},1]$:
	\begin{align} \label{E:TOPORDERLAPSEINHOMOGENEOUSTERML2ESTIMATE}
		t^{A_*+2} 
		\sum_{|\iota| = N} 
		\| \mathfrak{N}^{(\iota)} \|_{L^2(\Sigma_t)}
		& \leq
			{
			C
		\varepsilon 
		t^{2 \upsigma}
		\mathbb{D}(t),}
				\\
	{t^{2 \blowupexp+2}}\int_{\Sigma_t}
		|\mathfrak{R}^{(\iota)}|
	\, dx	
	& 
	\leq 
	C
	\varepsilon 
		t^{2 \upsigma}
		\mathbb{D}(t)
		\left\lbrace
			t^{\blowupexp}\|\partial^{\iota} n\|_{L^2(\Sigma_t)}
			+
			t^{\blowupexp + 1}\|\partial^{\iota} \vec{e}n\|_{L^2(\Sigma_t)}
			+
			\mathbb{D}(t)
		\right\rbrace,\qquad \text{if $|\iota|=N$}.
		\label{E:TOPORDERLAPSEELLIPTICIDENTITYANNOYINGTERMSL1ESTIMATE}
	\end{align}
\end{lemma}

\begin{proof}
Using the inequalities provided by Lemma~\ref{lem:basic.ineq},
	it is straightforward to estimate every product term in
		the expressions 
		\eqref{E:LAPSEPDECOMMUTEDTHROUGHOUTERFRAMEDERIVATIVEERRORTERM}
		and
		\eqref{E:TOPORDERLAPSEELLIPTICIDENTITYANNOYINGTERMS}
by accounting for the control afforded by our bootstrap assumptions \eqref{Boots} 
and taking into account the powers of $t$ 
featured in the solution norms of Definition~\ref{D:SOLUTIONNORMS}.
We provide the details for two representative terms.
First, using Lemma~\ref{lem:basic.ineq} and Definition~\ref{D:SOLUTIONNORMS}, 
we see that
the following term in $\mathfrak{N}^{(\iota)}$ 
(i.e., the third sum on RHS~\eqref{E:LAPSEPDECOMMUTEDTHROUGHOUTERFRAMEDERIVATIVEERRORTERM}) 
satisfies:
\begin{align} \label{REPTERMESTIMATEPROOFOFL:TOPORDERLAPSEINHOMOGENEOUSTERML2ESTIMATE}
\left
\|\sum_{\iota_1\cup\iota_2=\iota}
	\partial^{\iota_1} \upgamma 
	\cdot 
	\partial^{\iota_2} \vec{e} n
\right\|_{L^2(\Sigma_t)}
& \lesssim 
\| \upgamma \|_{L^{\infty}(\Sigma_t)}
\| \vec{e} n \|_{\dot{H}^N(\Sigma_t)}
+
\| \vec{e} n \|_{L^{\infty}(\Sigma_t)}
\| \upgamma \|_{\dot{H}^N(\Sigma_t)}
\lesssim 
\varepsilon t^{-\blowupexp-1-q}\mathbb{D}(t).
\end{align}
The factor of
$\varepsilon$ on RHS~\eqref{REPTERMESTIMATEPROOFOFL:TOPORDERLAPSEINHOMOGENEOUSTERML2ESTIMATE} 
comes from the bootstrap assumptions \eqref{Boots} 
and the fact that the LHS is quadratic with respect to quantities that vanish
for the background Kasner solution. Hence, multiplying
\eqref{REPTERMESTIMATEPROOFOFL:TOPORDERLAPSEINHOMOGENEOUSTERML2ESTIMATE}
with $t^{\blowupexp + 2}$ and recalling our assumptions \eqref{sigma,q} on the parameters 
$q, \upsigma$, 
we conclude that the resulting term is 
$\leq \mbox{RHS}~\eqref{E:TOPORDERLAPSEINHOMOGENEOUSTERML2ESTIMATE}$ as desired. 

We now give a second example, this time for the $L^1(\Sigma_t)$-type 
inequality \eqref{E:TOPORDERLAPSEELLIPTICIDENTITYANNOYINGTERMSL1ESTIMATE}. 
Specifically, we bound a term in $\mathfrak{R}^{(\iota)}$
(the second term on RHS~\eqref{E:TOPORDERLAPSEELLIPTICIDENTITYANNOYINGTERMS}) as follows
by using \eqref{comm},
Lemma~\ref{lem:basic.ineq}, Definition~\ref{D:SOLUTIONNORMS}, and the bootstrap assumptions \eqref{Boots}
(which in particular imply that $\| n \|_{L^{\infty}(\Sigma_t)} \leq 2$):
\begin{align} \label{STEPINPROOFOFE:TOPORDERLAPSEELLIPTICIDENTITYANNOYINGTERMSL1ESTIMATE}
\begin{split}
& 
	\int_{\Sigma_t}
	\left|
		2n([\partial^{\iota}, e_D] n)(\partial^{\iota} \upgamma_{CCD})
	\right| \, dx
		\\
& \lesssim
\| n \|_{L^{\infty}(\Sigma_t)}
\sum_{\iota_1 \cup \iota_2 = \iota, \, |\iota_2|<|\iota|}
\left\|(\partial^{\iota_1} e_I^c) (\partial^{\iota_2} \partial_c n) \right\|_{L^2(\Sigma_t)}
\| \partial^{\iota} \upgamma_{CCD} \|_{L^2(\Sigma_t)}
	\\
& \lesssim
	\left\lbrace
		\| e - \widetilde{e} \|_{W^{1,\infty}(\Sigma_t)}
		\| n - 1 \|_{W^{1,\infty}(\Sigma_t)}
		+
		\| e - \widetilde{e} \|_{W^{1,\infty}(\Sigma_t)}
		\| n \|_{\dot{H}^N(\Sigma_t)}
		+	
		\| n - 1 \|_{W^{1,\infty}(\Sigma_t)}
		\| e - \widetilde{e}  \|_{\dot{H}^N(\Sigma_t)}
	\right\rbrace
		\\
& \ \ \ \
	\times \| \upgamma \|_{\dot{H}^N(\Sigma_t)}
	\\
& \leq 
C
\varepsilon (t^{- \blowupexp - 1 - q + \upsigma} + t^{-2 \blowupexp - 1 - q} + t^{-2 \blowupexp - q + \upsigma}) \mathbb{D}^2(t).
\end{split}
\end{align}
Multiplying \eqref{STEPINPROOFOFE:TOPORDERLAPSEELLIPTICIDENTITYANNOYINGTERMSL1ESTIMATE} by $t^{2 \blowupexp+2}$ 
and using the inequality \eqref{sigma,q}, we deduce that the resulting term 
is $\leq \mbox{RHS}~\eqref{E:TOPORDERLAPSEELLIPTICIDENTITYANNOYINGTERMSL1ESTIMATE}$ as desired.  

The remaining terms that need to be bounded can be handled with similar arguments, and we omit the details.
\end{proof}

\subsubsection{Proof of Proposition~\ref{prop:n}}
\label{SSS:PROOFOFPROPLAPSE}
Throughout this proof, we will silently assume that $N$ is large enough such that
we can use the smallness of $\updelta \blowupexp$
described in Remark~\ref{rem:deltaA}. 

\medskip

\noindent \textbf{Proof of \eqref{n.low.est}}.
First, for $|\iota|\leq N_0+1$,
we use \eqref{n.eq.diff} to solve for
$e_C e_C \partial^{\iota} (n-1) 
- 
t^{-2} \partial^{\iota} (n-1)$
and then bound the resulting terms
in $L^\infty$ using
the bootstrap assumptions and Lemma~\ref{lem:Sob.borrow},
in particular bounding all terms involving $n-1$ and its derivatives by $\lesssim t^{\upsigma - \updelta \blowupexp}$,
which yields the following pointwise estimate for $|\iota|\leq N_0+1$
(see Remark~\ref{R:DELTAVARIES}):
\begin{align}\label{n.op.est}
\begin{split}
\left|e_C e_C \partial^{\iota} (n-1) 
- 
t^{-2} \partial^{\iota} (n-1) \right|
& 
\lesssim
\left|2n e_D^c\partial_c \partial^{\iota} \upgamma_{CCD}
	+
	\mathfrak{N}^{(\iota)} 
	+ 
	e_C^c 
	\partial_c 
	\left\lbrace[e_C^d,\partial^{\iota}] \partial_d (n-1) \right\rbrace \right|
		\\
& \lesssim
t^{-2q-\updelta \blowupexp}
\mathbb{D}(t).
\end{split}
\end{align}
From \eqref{n.op.est} and the maximum principle, noting that 
$e_Ce_C \partial^{\iota} (n-1) \leq 0 \;(\ge0)$ at the maxima (minima) of $\partial^{\iota} (n-1)$ in $\Sigma_t$,
and using the inequalities in \eqref{sigma,q},
we find that
$
\|t^{-2} (n-1)\|_{W^{N_0+1,\infty}(\Sigma_t)}
\lesssim
t^{-2+\upsigma}
\mathbb{D}(t)
$. Multiplying this inequality by $t^2$, 
we arrive at the desired estimate \eqref{n.low.est}
for the first term $\| n-1 \|_{W^{N_0+1,\infty}(\Sigma_t)}$ on the LHS.
To complete the proof of \eqref{n.low.est}, we 
must show that $t^q \| \vec{e} n \|_{W^{N_0,\infty}(\Sigma_t)} \lesssim t^{\upsigma} \mathbb{D}(t)$.
Since $e_I n = e_I^c \partial_c n$, the desired estimate is a simple consequence of
the already obtained bound $\| n-1 \|_{W^{N_0+1,\infty}(\Sigma_t)} \lesssim t^{\upsigma} \mathbb{D}(t)$
and the estimate $t^q \| \vec{e} \|_{W^{N_0,\infty}(\Sigma_t)} \lesssim 1$,
which follows from the bootstrap assumptions,
the definition of the background Kasner scalar functions $\widetilde{e}_I^i$ 
given in \eqref{Kasnersol},
and the inequalities in \eqref{sigma,q}.

\medskip

\noindent \textbf{Proof of \eqref{PRECISE.n.high.est}--\eqref{n.high.est}}.
We will show that there are constants $C_* > 0$ and $C > 0$, as in the statement of Prop.\,\ref{prop:n}, such that 
for each spatial multi-index $\iota$ with $|\iota| = N$, we have:
\begin{align} \label{mainstep.n.eq.diff}
\begin{split}
		 t^{2 \blowupexp + 2}
		\| \partial^{\iota} \vec{e} n\|_{L^2(\Sigma_t)}^2
		+
		t^{2 \blowupexp} \|\partial^{\iota} n \|_{L^2(\Sigma_t)}^2
		\leq&\, {\frac{1}{2}}
					t^{2 \blowupexp+2} \| \partial^{\iota} \vec{e} n \|_{L^2(\Sigma_t)}^2
					+{\frac{1}{2}}
					t^{2 \blowupexp} \| \partial^{\iota} n \|_{L^2(\Sigma_t)}^2 
						\\
&+					C_* 
					t^{2 \blowupexp + 2}
					\| \partial^{\iota} \upgamma \|_{L^2(\Sigma_t)}^2
		+
					C
					\varepsilon 
				t^{2 \upsigma}
				\mathbb{D}^2(t).
				\end{split}
\end{align}
Once we have proved \eqref{mainstep.n.eq.diff}, 
we absorb the first two terms on RHS~\eqref{mainstep.n.eq.diff}
back into the left, 
at the expense of doubling the constants in front of the remaining terms.
Afterward, taking the square root, we conclude \eqref{PRECISE.n.high.est}.
We then sum the square of \eqref{PRECISE.n.high.est} over all $\iota$ with $|\iota|=N$ 
and take the square root,
thereby concluding, in view of Definition~\ref{D:SOLUTIONNORMS},
the desired estimate \eqref{n.high.est}.

It remains for us to prove \eqref{mainstep.n.eq.diff}.
We integrate equation \eqref{E:ELLIPTICDIFFERENTIALLAPSEID} over
$\mathbb{T}^{\mydim}$ with respect to $dx$,
note that the integrals of the last {two} terms on RHS~\eqref{E:ELLIPTICDIFFERENTIALLAPSEID} vanish,
use the Cauchy--Schwarz inequality for integrals,
and use the estimate $\| n \|_{L^{\infty}(\Sigma_t)} \leq 2$ 
(which follows from the bootstrap assumptions)
to obtain:
\begin{align}\label{n.enest}
\begin{split}
t^{2 \blowupexp+2} \|\partial^{\iota} \vec{e} n \|_{L^2(\Sigma_t)}^2
+
t^{2 \blowupexp} \|\partial^{\iota} n \|_{L^2(\Sigma_t)}^2
&	\leq
	C_* 
	\| t^{\blowupexp + 1} \partial^{\iota} \vec{e} n \|_{L^2(\Sigma_t)}
	\| t^{\blowupexp + 1} \partial^{\iota} \upgamma \|_{L^2(\Sigma_t)}
	\\
& \ \
	+
	\| t^{\blowupexp} \partial^{\iota} n \|_{L^2(\Sigma_t)}
	\| t^{\blowupexp + 2} \mathfrak{N}^{(\iota)} \|_{L^2(\Sigma_t)}
	+
	\int_{\Sigma_t}
		t^{2 \blowupexp+2}
		|\mathfrak{R}^{(\iota)}|
	\, dx.
	\end{split}
\end{align}
From \eqref{n.enest}, 
the error estimates in 
Lemma~\ref{L:TOPORDERLAPSEINHOMOGENEOUSTERML2ESTIMATE}, our bootstrap assumptions \eqref{Boots},
Young's inequality,
and Definition~\ref{D:SOLUTIONNORMS},
we conclude when $\varepsilon$ is sufficiently small,
the desired bound \eqref{mainstep.n.eq.diff} holds
(for a different $C_*$, which is nevertheless independent of $\blowupexp$, $N_0$, and $N$).

\medskip

\noindent \textbf{Proof of \eqref{E:LAPSECONTROLLEDBYDYNAMICVARIABLES}}.
The estimate \eqref{E:LAPSECONTROLLEDBYDYNAMICVARIABLES} follows easily from Definition~\ref{D:SOLUTIONNORMS}
and the estimates \eqref{n.low.est}--\eqref{n.high.est}.\hfill $\qed$

\subsection{Preliminary identities and inequalities for $k$, $\upgamma$, $e$, and $\oe$}
\label{SS:PRELIMINARYRESULTSFORFRAMESECONDFUNDANDCONNECTION}
In this section, 
we derive preliminary low order and high order 
identities and inequalities
for
$\upgamma$, $k$, $e$, and $\oe$ by using the evolution equations
\eqref{dt.k}--\eqref{dt.gamma}
and
\eqref{dt.omega}--\eqref{dt.omega.inv},
as well as the key evolution equations for the structure coefficients
provided by Proposition~\ref{P:KEYEVOLUTIONSTRUCTURECOEFFICIENTS}.
Roughly, we control the inhomogeneous terms in their evolution equations
in terms of our solution norms, and we derive differential versions of our energy identities.
In Sects.\,\ref{SS:INTEGRALINEQUALITYFORLOWNORMS}--\ref{PROOFOFprop:overall},
we will combine these preliminary results with related ones
for the lapse and scalar field to derive our main a priori estimates,
i.e., to prove Proposition~\ref{prop:overall}.

\subsubsection{The key evolution equation verified by the structure coefficients}
\label{SSS:STRUCTURECOEFFICIENTS}
To control the connection coefficients $\upgamma_{IJB}$ at the low derivative levels,
we will use the following proposition, which provides
evolution equations for the structure coefficients 
$\upgamma_{IJB}+\upgamma_{JBI}$
of the orthonormal spatial frame $\lbrace e_I \rbrace_{I=1,\cdots,\mydim}$.
Although its proof is simple, the proposition is of profound significance for our main results.
As we mentioned in Sect.\,\ref{subsec:pf.overview}, 
the main virtues of the proposition are: it shows that up to error terms,
the evolution equation system for the structure coefficients is diagonal,
and it shows that the strength of the main linear terms driving the dynamics 
is controlled by the Kasner stability condition \eqref{Kasner.stability.cond}.
The connection coefficients themselves can be controlled in terms of the structure
coefficients via the identity \eqref{E:RECOVERGAMMAFROMSTRUCTURECOEFFCIENTS}.

\begin{proposition}[The key evolution equations for the structure coefficients of the orthonormal frame]
	\label{P:KEYEVOLUTIONSTRUCTURECOEFFICIENTS}
		For solutions to the equations of Proposition~\ref{P:redeq},
		the structure coefficients of the orthonormal frame $\lbrace e_I \rbrace_{I=1,\cdots,\mydim}$,
		namely 
		%$\upgamma_{\underline{I}J\underline{I}}$ with $I \neq J$ and
		$\upgamma_{IJB}+\upgamma_{JBI}$ with 
		$I < J$ (see Remark~\ref{R:BASISOFSTRUCTURE}),
		verify the following evolution equations,
		whose left-hand sides exhibit a \textbf{diagonal structure},
		where the Kasner background scalars 
		$\lbrace \widetilde{e}_I^i \rbrace_{I,i=1,\cdots,\mydim}$
		and
		$\lbrace \widetilde{k}_{IJ} \rbrace_{I,J=1,\cdots,\mydim}$ 
		are defined in \eqref{Kasnersol} 
		(see also Remark~\ref{R:NOTCOMPONENTSOFTENSORRELATIVETOFRAME})
		and we recall that we do not sum underlined repeated indices:
\begin{align}
\begin{split}
& \partial_t(\upgamma_{IJB}+\upgamma_{JBI})
+
\frac{(\widetilde{q}_{\underline{I}}+\widetilde{q}_{\underline{J}}-\widetilde{q}_{\underline{B}})}{t}
(\upgamma_{\underline{I}\underline{J}\underline{B}}+\upgamma_{\underline{J}\underline{B}\underline{I}})
\label{dt.gamma-gammatilde2} 
\\
& 
= 
(n-1)
\left\lbrace
	k_{IC} \upgamma_{CJB}
	-
	k_{CI} \upgamma_{BJC}
	-
	k_{JC} \upgamma_{BIC}
	+
	k_{CI} \upgamma_{JBC}
	+
	k_{BC} \upgamma_{JIC}
\right\rbrace
		\\
& \ \
	+
	(n-1) 
	\left\lbrace
		k_{JC}\upgamma_{CBI}
		-
		k_{CJ} \upgamma_{IBC}
		-
		k_{BC} \upgamma_{IJC}
		+
		k_{CJ} \upgamma_{BIC}
		+
		k_{IC} \upgamma_{BJC}
	\right\rbrace
		\\
& \ \
	+
	(k_{IC}-\widetilde{k}_{IC}) \upgamma_{CJB}
	-
	(k_{CI} - \widetilde{k}_{CI})\upgamma_{BJC}
	-
	(k_{JC} - \widetilde{k}_{JC})
	\upgamma_{BIC}
	+
	(k_{CI} - \widetilde{k}_{CI})
	\upgamma_{JBC}
	+
	(k_{BC} - \widetilde{k}_{BC})
	\upgamma_{JIC}
		 \\
& \ \
		+
		(k_{JC} - \widetilde{k}_{JC})
		\upgamma_{CBI}
		-
		(k_{CJ} - \widetilde{k}_{CJ})
		\upgamma_{IBC}
		-
		(k_{BC} - \widetilde{k}_{BC})
		\upgamma_{IJC}
		+
		(k_{CJ} - \widetilde{k}_{CJ})
		\upgamma_{BIC}
		+
		(k_{IC} - \widetilde{k}_{IC})
		\upgamma_{BJC}
			\\
	& \ \
	+
	n(e_I^c-\widetilde{e}_I^c)\partial_ck_{BJ}
	-
	n(e_J^c-\widetilde{e}_J^c)\partial_ck_{BI}
	+
	n \widetilde{e}_I^c\partial_ck_{BJ}
	-
	n \widetilde{e}_J^c\partial_ck_{BI}
	+
	(e_I n)k_{BJ}
	-
	(e_J n)k_{BI}.
\end{split}
\end{align}

Moreover, for spatial coordinate multi-indices $\iota$ with $|\iota|\leq N_0$,
the following evolution equation holds:
\begin{align} \label{E:SCHEMATICSTRUCTURECOEFFICIENTEVOLUTIONEQUATION}
\partial_t 
[t^q \partial^{\iota}(\upgamma_{IJB}+\upgamma_{JBI})]
& = 
\left\lbrace
	q
	-
	(\widetilde{q}_{\underline{I}}+\widetilde{q}_{\underline{J}}-\widetilde{q}_{\underline{B}})
\right\rbrace
t^{q-1}
\partial^{\iota}
(\upgamma_{\underline{I}\underline{J}\underline{B}}+\upgamma_{\underline{J}\underline{B}\underline{I}})
+
t^q
\mathfrak{S}_{IJB}^{(\textnormal{Border};\iota)} 
+ 
t^q
\mathfrak{S}_{IJB}^{(\textnormal{Junk};\iota)},
\end{align}
where:
\begin{subequations}
\begin{align}
	\mathfrak{S}_{IJB}^{(\textnormal{Border};\iota)} 
	& := 	\sum_{\iota_1 \cup \iota_2  = \iota}
				\partial^{\iota_1} (k - \widetilde{k})
				\cdot
				\partial^{\iota_2} \upgamma
				+
				\sum_{\iota_1 \cup \iota_2 = \iota}
				n
				\cdot
				\partial^{\iota_1} (e - \widetilde{e})
				\cdot
				\partial \partial^{\iota_2} k,
						\label{E:STRUCTURECOEFFICIENTERRORTERMBORDER} 
						\\
	\begin{split}
	\mathfrak{S}_{IJB}^{(\textnormal{Junk};\iota)} 
	& := \sum_{\iota_1 \cup \iota_2 \cup \iota_3  = \iota}
				\partial^{\iota_1} (n-1)
				\cdot
				\partial^{\iota_2} k
				\cdot
				\partial^{\iota_3} \upgamma
				+
				\sum_{\iota_1\cup\iota_2=\iota}\partial^{\iota_1} n
				\cdot
				\widetilde{e}
				\cdot
				\partial \partial^{\iota_2} k
			 \label{E:STRUCTURECOEFFICIENTERRORTERMJUNK}\\
		& \ \
				+
				\sum_{\iota_1 \cup \iota_2 \cup \iota_3  = \iota, \,1 \leq |\iota_1|}
				\partial^{\iota_1} n
				\cdot
				\partial^{\iota_2} (e - \widetilde{e})
				\cdot
				\partial \partial^{\iota_3} k
				+
				\sum_{\iota_1 \cup \iota_2  = \iota} 
				\partial^{\iota_1} \vec{e} n
				\cdot
				\partial^{\iota_2} k.
\end{split}
\end{align}
\end{subequations}

Finally, the scalar function
$\upgamma_{IJB}$ can be expressed as a linear combination
of three structure coefficients:
\begin{align}
	\label{E:RECOVERGAMMAFROMSTRUCTURECOEFFCIENTS}
	\upgamma_{IJB}
	& = 
		\frac{1}{2}
			\left\lbrace
				\upgamma_{IJB}
				+
				\upgamma_{JBI}
			\right\rbrace
			+
			\frac{1}{2}
			\left\lbrace
				\upgamma_{BJI}
				+
				\upgamma_{JIB}
			\right\rbrace
			+
			\frac{1}{2}
			\left\lbrace
				\upgamma_{BIJ}
				+
				\upgamma_{IJB}
			\right\rbrace.
\end{align}
\end{proposition}

\begin{remark}[Connection between equation \eqref{dt.gamma-gammatilde2} and the stability condition 
\eqref{Kasner.stability.cond}]
	\label{R:IMPORTANCEOFSTABILITYCONDITIONMADECLEAR}
	If we were to ignore the terms on RHS~\eqref{dt.gamma-gammatilde2},
	then equation \eqref{dt.gamma-gammatilde2} would allow us to conclude
	that $|\upgamma_{IJB}+\upgamma_{JBI}| \lesssim t^{- (\widetilde{q}_I + \widetilde{q}_J - \widetilde{q}_B)}$.
	This makes the significance of the stability condition \eqref{Kasner.stability.cond} for
	equation \eqref{dt.gamma-gammatilde2} clear: under this condition,
	the quantity $\underset{\substack{I,J,B=1,\cdots, \mydim \\ I < J}}{\max}
	|\upgamma_{IJB}+\upgamma_{JBI}|$
	is integrable in $t$ near $0$,
	and by \eqref{E:RECOVERGAMMAFROMSTRUCTURECOEFFCIENTS} (cf.\ Remark~\ref{R:BASISOFSTRUCTURE}), 
	$\underset{I,J,B=1,\cdots,\mydim}{\max} |\upgamma_{IJB}|$ is also integrable in $t$.
	In our ensuing analysis, we will in fact control the terms on RHS~\eqref{dt.gamma-gammatilde2} and 
	show that $\underset{I,J,B=1,\cdots,\mydim}{\max} |\upgamma_{IJB}|$ is integrable,
	which is a crucial step in our proof of stable blowup.
\end{remark}

\begin{remark}
Interestingly, if we were to try to control the $\upgamma_{IJB}$'s at the low derivative levels
by using the formula \eqref{gamma.form} and separately controlling each of the factors
$e_I^i,\oe_i^I$, 
then we would not be able to close our estimates for the full range of Kasner
exponents verifying the stability condition \eqref{Kasner.stability.cond}.
In fact, since RHS~\eqref{gamma.form} is cubic in $e_I^i,\oe_i^I$ and their derivatives,
the crudest version of that approach would yield only
$|\upgamma_{IJB}| \lesssim t^{-3q}$, which, when $q$ is near $1$, is far too singular
for proving stability. Moreover, the evolution equation \eqref{dt.gamma} 
for the $\upgamma_{IJB}$'s is not diagonal at the linear level and thus, a crude treatment based only 
on this equation would lead to far too singular estimates\footnote{However, the structure of equation \eqref{dt.gamma} 
is sufficient for our top-order energy estimates, which are allowed to be much more singular
within the scope of our approach; this explains why in Lemma~\ref{L:TOPORDERCOMMUTEDEQUATIONSFORGAMMANDK},
we derive commuted versions of equation \eqref{dt.gamma} to set up our energy estimates for $\upgamma$ and $k$.} 
for the connection coefficients at the lower derivative levels.
Thus, the diagonal structure revealed by Proposition~\ref{P:KEYEVOLUTIONSTRUCTURECOEFFICIENTS}
is essential to our overall argument.
\end{remark}

\begin{proof}[Proof of Proposition~\ref{P:KEYEVOLUTIONSTRUCTURECOEFFICIENTS}]
	Equations 
	\eqref{dt.gamma-gammatilde2} 
	follow from
	the evolution equation \eqref{dt.gamma},
	the definition of the background Kasner scalar functions in \eqref{Kasnersol},
	the antisymmetry property \eqref{antisymmetricgamma},
	and straightforward algebraic computations.
	\eqref{E:SCHEMATICSTRUCTURECOEFFICIENTEVOLUTIONEQUATION}
	then follows from differentiating \eqref{dt.gamma-gammatilde2} with $\partial^{\iota}$,
	applying the product rule,
	multiplying both sides of the resulting identity by $t^q$,
	and then commuting the factor of $t^q$ under the operator $\partial_t$ on the LHS
	and accounting for the commutator $[t^q,\partial_t]$.
	
	\eqref{E:RECOVERGAMMAFROMSTRUCTURECOEFFCIENTS} is an immediate consequence of the Koszul formula for an orthonormal frame and the antisymmetry property \eqref{antisymmetricgamma}.
\end{proof}

\subsubsection{Pointwise estimates for the error terms in the structure coefficient evolution equations}
\label{SSS:STRUCTURECOFFEICIENTERRORTERMSPOINTWISE}
In the next lemma, we derive pointwise estimates at the low derivative levels
for the error terms from Proposition~\ref{P:KEYEVOLUTIONSTRUCTURECOEFFICIENTS}.

\begin{lemma}[Pointwise estimates for the error terms in the structure coefficient evolution equations at
orders $\leq N_0$]
	\label{L:STRUCTURECOFFEICIENTERRORTERMSPOINTWISE}
Recall that $\mathbb{D}(t)$ is the total norm of the dynamic variables from Definition~\ref{D:SOLUTIONNORMS}.
Assume that the bootstrap assumptions \eqref{Boots} hold.
There exists a constant $C = C_{N,N_0,\blowupexp,\mydim,q,\upsigma} > 0$ 
such that if $N_0 \geq 1$ and $N$ is sufficiently large in a manner 
that depends on $N_0, \blowupexp, \mydim, q,$ and $\upsigma$,
and if $\varepsilon$ is sufficiently small 
(in a manner that depends on $N, N_0, \blowupexp, \mydim, q,$ and $\upsigma$), 
then the following pointwise estimates hold on $(T_{\textnormal{Boot}},1] \times \mathbb{T}^{\mydim}$
for the error terms
$
\mathfrak{S}_{IJB}^{(\textnormal{Border};\iota)}
$
and
$
\mathfrak{S}_{IJB}^{(\textnormal{Junk};\iota)}
$
defined in \eqref{E:STRUCTURECOEFFICIENTERRORTERMBORDER}--\eqref{E:STRUCTURECOEFFICIENTERRORTERMJUNK}:
	\begin{subequations}
	\begin{align}
		\begin{split}
		\sum_{|\iota| \leq N_0} 
		\sum_{\substack{I,J,B = 1, \cdots, \mydim 
			\\ I < J}}
		t^q 
		|\mathfrak{S}_{IJB}^{(\textnormal{Border};\iota)}|(t,x)
		& \leq C \varepsilon t^{q-1} \sum_{|\iota| \leq N_0}
		\sum_{I,J,B = 1, \cdots, \mydim}
		|\partial^{\iota} \upgamma_{IJB}|(t,x)
					\label{E:STRUCTURECOEFFICIENTSBORDERPOINTWISEBOUNDS}
						\\
		& \ \
				+
				C \varepsilon t^{q-1} \sum_{|\iota| \leq N_0}
				\sum_{I,i=1,\cdots,\mydim}
				|\partial^{\iota} (e_I^i - \widetilde{e}_I^i)|(t,x),
		\end{split} \\
		\sum_{|\iota| \leq N_0} 
		\sum_{\substack{I,J,B = 1, \cdots, \mydim 
			\\ I < J}}
		t^q 
		|\mathfrak{S}_{IJB}^{(\textnormal{Junk};\iota)}|(t,x)
		& \leq C t^{-1+\upsigma} \mathbb{D}(t).
			\label{E:STRUCTURECOEFFICIENTSJUNKPOINTWISEBOUNDS}
	\end{align}
	\end{subequations}
\end{lemma}

\begin{proof}
	Based on equations \eqref{E:STRUCTURECOEFFICIENTERRORTERMBORDER}--\eqref{E:STRUCTURECOEFFICIENTERRORTERMJUNK},
	the estimates
	\eqref{E:STRUCTURECOEFFICIENTSBORDERPOINTWISEBOUNDS}--\eqref{E:STRUCTURECOEFFICIENTSJUNKPOINTWISEBOUNDS}
	follow as straightforward consequences of
	\eqref{Kasnersol},
	the inequalities in \eqref{sigma,q},
	Definition~\ref{D:SOLUTIONNORMS},
	the bootstrap assumptions \eqref{Boots}, 
	and the already derived low order estimates \eqref{n.low.est} for $n$.
	Note in particular that we have used the fact that
	the low order norm \eqref{norms.low} controls $k-\widetilde{k}$
	at derivative levels $\leq N_0+1$ (see Remark \ref{rem:N0+1});
	for example, for $|\iota| \leq N_0$,
	this allows us to pointwise bound the magnitude of the sum
	$\sum_{\iota_1 \cup \iota_2 = \iota}
				n
				\cdot
				\partial^{\iota_1} (e - \widetilde{e})
				\cdot
				\partial \partial^{\iota_2} k
	$ on RHS~\eqref{E:STRUCTURECOEFFICIENTERRORTERMBORDER}
	by $\lesssim \varepsilon t^{-1} \sum_{|\iota| \leq N_0}|\partial^{\iota} (e - \widetilde{e})|(t,x)$.
\end{proof}

\subsubsection{Absence of certain structure coefficients in polarized $U(1)$-symmetry}	
In the next lemma, we show that for polarized $U(1)$-symmetric metrics with $\mydim=3$,
relative to an orthonormal spatial frame of the type provided by Lemma~\ref{lem:U1},
all structure coefficients with three distinct indices vanish.
As we explained in Remark~\ref{R:IDENTIFYOBSTRUCTION}, 
this vanishing is crucial for
the proof of our main results in the case of the
Einstein-vacuum equations in $1+3$ dimensions
under polarized $U(1)$-symmetry.

\begin{lemma}[The vanishing of key variables in polarized $U(1)$-symmetry]\label{lem:gamma.U(1)}
Suppose that $\mydim = 3$ and that
${\bf g}$ is a polarized $U(1)$-symmetric metric
satisfying the hypotheses and conclusions of Lemma~\ref{lem:U1}.
Moreover, let $\lbrace e_1,e_2,e_3 \rbrace$ be an orthonormal spatial frame satisfying 
the hypotheses and conclusions of Lemma~\ref{lem:U1}.
In particular, $e_3=(g_{33})^{-\frac{1}{2}}\partial_3$ 
and $\mathcal{L}_{\partial_3} e_I = 0$ for $I=1,2,3$,
where $\partial_3$ is the hypersurface-orthogonal Killing vectorfield. Then the following spatial connection coefficients vanish:
\begin{align}\label{gamma.U(1)}
\upgamma_{123}=\upgamma_{231}=\upgamma_{312}=0.
\end{align}
Moreover, under the same assumptions, we have:
\begin{align} \label{E:U1VANISHINGSTRUCTURECOEFFICIENTS}
		\upgamma_{IJB}
		+
		\upgamma_{JBI}
		& = 
		\begin{cases}
			0, & \mbox{ if } I = J,
				\\
			0, & \mbox{ if } I,J,B \mbox{ are distinct}.
		\end{cases}
\end{align}
\end{lemma} 
\begin{proof}
Under the assumptions and conclusions of Lemma~\ref{lem:U1},
$\partial_3$ is parallel to $e_3$ and orthogonal to $\partial_1$ and $\partial_2$,
and we have $e_1^3=e_2^3 = e_3^1 = e_3^2 = \oe_3^1 = \oe_3^2 = \oe_1^3 = \oe_2^3 = e_3 e_I^i=0$. 
Hence, using \eqref{gamma.form} we compute:
\begin{align*}
\upgamma_{123}=&\,\frac{1}{2}\left\lbrace\oe_l^3(e_1e_2^l-e_2e_1^l)
-\oe_l^1(e_2e_3^l-e_3e_2^l)
+\oe_l^2(e_3e_1^l-e_1e_3^l)\right\rbrace=0,\\
\upgamma_{231}=&\,\frac{1}{2}\left\lbrace\oe_l^1(e_2e_3^l-e_3e_2^l)
-\oe_l^2(e_3e_1^l-e_1e_3^l)
+\oe_l^3(e_1e_2^l-e_2e_1^l)\right\rbrace=0,\\
\upgamma_{312}=&\,\frac{1}{2}\left\lbrace\oe_l^2(e_3e_1^l-e_1e_3^l)
-\oe_l^3(e_1e_2^l-e_2e_1^l)
+\oe_l^1(e_2e_3^l-e_3e_2^l)\right\rbrace=0,
\end{align*}
which yields \eqref{gamma.U(1)}. A more conceptual justification of the above computations is that in the present setting, 
$g([e_I,e_J],e_B)=0$
whenever $I,J,B$ are distinct indices;
using this fact and the Koszul formula, we conclude that $\upgamma_{IJB} = 0$ 
whenever $I,J,B$ are distinct indices, as desired.

\eqref{E:U1VANISHINGSTRUCTURECOEFFICIENTS} follows from \eqref{gamma.U(1)} 
and the antisymmetry property \eqref{antisymmetricgamma}.
\end{proof}
\begin{remark}[The role of polarized $U(1)$-symmetry]\label{rem:gamma.U1}
In proving our stable Big Bang formation results for the Einstein-vacuum equations in $1+3$ dimensions, 
there is precisely one way in which our polarized $U(1)$-symmetry assumption is important for our analysis: 
it allows us to use the results of Lemma~\ref{lem:gamma.U(1)};
see also Remark~\ref{R:IDENTIFYOBSTRUCTION}
and the end of the proof of Proposition~\ref{prop:low}.
Put differently, if we were to assume the conclusions 
\eqref{gamma.U(1)}--\eqref{E:U1VANISHINGSTRUCTURECOEFFICIENTS} of the lemma, then 
the rest of our proof of stable Big Bang formation would go through.
\end{remark}

\subsubsection{Commuted evolution equations for $e$ and $\oe$}
\label{SSS:FRAMEANDCOFRAMEEQUATIONS}
In this section, we provide the evolution equations that we will use to control
the scalar functions
$\lbrace e_{I}^i \rbrace_{I,i=1,\cdots,\mydim}$ 
and
$\lbrace \oe_i^I \rbrace_{I,i=1,\cdots,\mydim}$
as well as their derivatives.

\begin{lemma}[Evolution equations for 
$\lbrace e_{I}^i \rbrace_{I,i=1,\cdots,\mydim}$,
$\lbrace \oe_i^I \rbrace_{I,i=1,\cdots,\mydim}$, 
and their derivatives]
\label{L:EVOLUTIONEQUATIONSFORFRAMEANDCOFRAME}
The evolution equations \eqref{dt.omega}--\eqref{dt.omega.inv}
can be rewritten as follows,
where the Kasner background scalars 
$\lbrace \widetilde{e}_I^i \rbrace_{I,i=1,\cdots,\mydim}$,
$\lbrace \widetilde{\oe}_i^I \rbrace_{I,i=1,\cdots,\mydim}$,
and
$\lbrace \widetilde{k}_{IJ} \rbrace_{I,J=1,\cdots,\mydim}$ 
are defined in \eqref{Kasnersol} 
(see also Remark~\ref{R:NOTCOMPONENTSOFTENSORRELATIVETOFRAME}),
and we recall that we do not sum over repeated underlined indices:
\begin{subequations}
\begin{align}\label{dt.omega-omegatilde}
\begin{split}
\partial_t({e_I^i}-\widetilde{e}_I^i)
+
\frac{\widetilde{q}_{\underline{I}}}{t}({e_{\underline{I}}^i}-\widetilde{e}_{\underline{I}}^i)
& = 
(n-1) k_{IC} (e_C^i - \widetilde{e}_C^i)
+
(k_{IC}-\widetilde{k}_{IC})(e_C^i - \widetilde{e}_C^i)
	\\
& \ \
+
(n-1) \widetilde{k}_{IC} \widetilde{e}_C^i
+
n
(k_{IC}-\widetilde{k}_{IC})
\widetilde{e}_C^i,
\end{split}
	\\
\begin{split}
 \partial_t(\oe_i^I-\widetilde{\oe}_i^I)
-
\frac{\widetilde{q}_{\underline{I}}}{t}
(\oe_i^{\underline{I}}-\widetilde{\oe}_i^{\underline{I}})
& 
= 
-
(n-1) k_{IC} (\oe_i^C - \widetilde{\oe}_i^C)
-
(k_{IC}-\widetilde{k}_{IC})(\oe_i^C - \widetilde{\oe}_i^C)
\label{dt.omega-omegatilde.inv}
	\\
& \ \ 
-
(n-1) \widetilde{k}_{IC} \widetilde{\oe}_i^C
-
n
(k_{IC}-\widetilde{k}_{IC})
\widetilde{\oe}_i^C.
\end{split}
\end{align}
\end{subequations}
Moreover, let $\iota$ be a spatial multi-index with $|\iota| \leq N$,
and let $\Pow \geq 0$ be a real number.
Then the following equations hold:
\begin{subequations}
\begin{align} 
\label{E:COMMUTEDFRAMEEVOLUTION}
	\partial_t
	[t^{\Pow}\partial^{\iota}({e_I^i}-\widetilde{e}_I^i)]
	& = 
		(\Pow - \widetilde{q}_{\underline{I}}) 
		t^{\Pow-1} \partial^{\iota} ({e_{\underline{I}}^i}-\widetilde{e}_{\underline{I}}^i)
		+
		t^{\Pow}
		\mathfrak{E}_I^{i;(\textnormal{Border};\iota)}
		+
		t^{\Pow}
		\mathfrak{E}_I^{i;(\textnormal{Junk};\iota)},\\
	\partial_t [t^{\Pow}\partial^{\iota}(\oe^I_i-\widetilde{\oe}^I_i)]
	& = 
		(\Pow + \widetilde{q}_{\underline{I}}) 
		t^{\Pow-1} \partial^{\iota} (\oe^I_i-\widetilde{\oe}^I_i)
		+
		t^{\Pow}
		\mathfrak{O}_I^{i;(\textnormal{Border};\iota)}
		+
		t^{\Pow}
		\mathfrak{O}_I^{i;(\textnormal{Junk};\iota)},
	\label{E:COMMUTEDCOFRAMEEVOLUTION}	
\end{align}
\end{subequations}
where:
\begin{subequations}
\begin{align}
	\mathfrak{E}_I^{i;(\textnormal{Border};\iota)}
	& :=\sum_{\iota_1 \cup \iota_2 \cup \iota_3 = \iota}
				\partial^{\iota_1} (n - 1)
				\cdot
				\partial^{\iota_2} k
				\cdot
				\partial^{\iota_3} (e - \widetilde{e})
				+
			\sum_{\iota_1 \cup \iota_2 = \iota}
			\partial^{\iota_1} (k - \widetilde{k})
			\cdot
			\partial^{\iota_2} (e - \widetilde{e}),
				\label{E:FRAMEBORDERERRORTERMS} \\
\mathfrak{E}_I^{i;(\textnormal{Junk};\iota)}
& :=\partial^{\iota} (n-1) \cdot \widetilde{k} \cdot \widetilde{e}
	+
	\sum_{\iota_1 \cup \iota_2 = \iota}
		\partial^{\iota_1} n
		\cdot
		\partial^{\iota_2} (k - \widetilde{k})
		\cdot
		\widetilde{e},
			\label{E:FRAMEJUNKERRORTERMS}
			\\
	\mathfrak{O}_I^{i;(\textnormal{Border};\iota)}
	& :=\sum_{\iota_1 \cup \iota_2 \cup \iota_3 = \iota}
				\partial^{\iota_1} (n - 1)
				\cdot
				\partial^{\iota_2} k
				\cdot
				\partial^{\iota_3} (\oe - \widetilde{\oe})
				+
			\sum_{\iota_1 \cup \iota_2 = \iota}
			\partial^{\iota_1} (k - \widetilde{k})
			\cdot
			\partial^{\iota_2} (\oe - \widetilde{\oe}),
				\label{E:COFRAMEBORDERERRORTERMS} \\
\mathfrak{O}_I^{i;(\textnormal{Junk};\iota)}
& :=\partial^{\iota} (n-1) \cdot \widetilde{k} \cdot \widetilde{\oe}
	+
	\sum_{\iota_1 \cup \iota_2 = \iota}
		\partial^{\iota_1} n
		\cdot
		\partial^{\iota_2} (k - \widetilde{k})
		\cdot
		\widetilde{\oe}.
			\label{E:COFRAMEJUNKERRORTERMS}
\end{align}
\end{subequations}
\end{lemma}

\begin{proof}
	\eqref{dt.omega-omegatilde}--\eqref{dt.omega-omegatilde.inv}
	follow from equations \eqref{dt.omega}--\eqref{dt.omega.inv}
	and straightforward algebraic computations.
	\eqref{E:COMMUTEDFRAMEEVOLUTION}--\eqref{E:COMMUTEDCOFRAMEEVOLUTION}
	then follow from differentiating \eqref{dt.omega-omegatilde}--\eqref{dt.omega-omegatilde.inv}
	with $\partial^{\iota}$,
	using the Leibniz rule,
	multiplying both sides of the resulting equations by $t^{\Pow}$,
	and commuting the factors of $t^{\Pow}$ under the operator $\partial_t$ on the LHSs
	and accounting for the commutator $[t^{\Pow},\partial_t]$.
\end{proof}

\subsubsection{Pointwise estimates for the error terms in the frame component evolution equations}
\label{SSS:FRAMECOMPONENTERRORTERMSPOINTWISE}
In this section, at the low derivative levels,
we derive pointwise estimates for the error terms in the evolution equations of 
Lemma~\ref{L:EVOLUTIONEQUATIONSFORFRAMEANDCOFRAME}.

\begin{lemma}[Pointwise estimates for the error terms in the evolution equations for 
	$\partial^{\leq N_0} (e-\widetilde{e})$ and $\partial^{\leq N_0} (\oe-\widetilde{\oe})$]
Recall that $\mathbb{D}(t)$ is the total norm of the dynamic variables from Definition~\ref{D:SOLUTIONNORMS}.
Assume that the bootstrap assumptions \eqref{Boots} hold.
There exists a constant $C = C_{N,N_0,\blowupexp,\mydim,q,\upsigma} > 0$ 
such that if $N_0 \geq 1$ and $N$ is sufficiently large in a manner 
that depends on $N_0, \blowupexp, \mydim, q,$ and $\upsigma$,
and if $\varepsilon$ is sufficiently small (in a manner that depends on $N, N_0, \blowupexp, \mydim, q,$ and $\upsigma$), 
then the error terms
$\mathfrak{E}_I^{i;(\textnormal{Border};\iota)}$,
$\mathfrak{E}_I^{i;(\textnormal{Junk};\iota)}$, 
$\mathfrak{O}_I^{i;(\textnormal{Border};\iota)}$, 
and $\mathfrak{O}_I^{i;(\textnormal{Junk};\iota)}$
defined in
\eqref{E:FRAMEBORDERERRORTERMS}--\eqref{E:COFRAMEJUNKERRORTERMS}
verify the following pointwise estimates for $(t,x) \in (T_{\textnormal{Boot}},1] \times \mathbb{T}^{\mydim}$,
where the Kasner background scalars
$\lbrace \widetilde{e}_I^i \rbrace_{I,i=1,\cdots,\mydim}$
and
$\lbrace \widetilde{\oe}_i^I \rbrace_{I,i=1,\cdots,\mydim}$
are defined in \eqref{Kasnersol}:
	\begin{subequations}
	\begin{align} 
		\sum_{|\iota| \leq N_0} 
		\sum_{I,i = 1, \cdots, \mydim}
		t^q |\mathfrak{E}_I^{i;(\textnormal{Border};\iota)}|(t,x)
		& \leq C \varepsilon t^{q-1}{}{\sum_{|\iota| \leq N_0} 
		\sum_{I,i = 1, \cdots, \mydim}
		|\partial^{\iota}(e_I^i - \widetilde{e}_I^i)|(t,x)},
			\label{E:FRAMEBORDERPOINTWISEBOUNDS} \\
		\sum_{|\iota| \leq N_0} 
		\sum_{I,i = 1, \cdots, \mydim}
		t^q |\mathfrak{E}_I^{i;(\textnormal{Junk};\iota)}|(t,x)
		& \leq C t^{-1+\upsigma} \mathbb{D}(t),
			\label{E:FRAMEJUNKPOINTWISEBOUNDS}\\
			\sum_{|\iota| \leq N_0} 
		\sum_{I,i = 1, \cdots, \mydim}
		t^q |\mathfrak{O}_I^{i;(\textnormal{Border};\iota)}|(t,x)
		& \leq C \varepsilon t^{q-1}\sum_{|\iota| \leq N_0} 
		\sum_{I,i = 1, \cdots, \mydim}
		|\partial^{\iota}(\oe_i^I - \widetilde{\oe}_i^I)|(t,x),
			\label{E:COFRAMEBORDERPOINTWISEBOUNDS} \\
		\sum_{|\iota| \leq N_0} 
		\sum_{I,i = 1, \cdots, \mydim}
		t^q |\mathfrak{O}_I^{i;(\textnormal{Junk};\iota)}|(t,x)
		& \leq C t^{-1+\upsigma} \mathbb{D}(t).
			\label{E:COFRAMEJUNKPOINTWISEBOUNDS}
	\end{align}
	\end{subequations}
\end{lemma}

\begin{proof}
	The lemma follows from the expressions
	\eqref{E:FRAMEBORDERERRORTERMS}--\eqref{E:COFRAMEJUNKERRORTERMS},
	the bootstrap assumptions,
	the definition of the lower order norms \eqref{norms.low}, 
	the explicit formulas \eqref{Kasnersol}, 
	the inequalities in \eqref{sigma,q}, and the already derived low order estimates \eqref{n.low.est} for $n$.\footnote{Note in particular that we do not use the interpolation inequalities of 
	Lemma~\ref{lem:Sob.borrow} in this proof.}
\end{proof}

\subsubsection{$L^2$-control of the error terms in the top-order commuted frame component evolution equations}
\label{SSS:FRAMECOMPONENTERRORTERMSTOPORDERL2CONTROL}
In this section, at the top-order derivative level,
we derive $L^2$ estimates for the error terms in the evolution equations of 
Lemma~\ref{L:EVOLUTIONEQUATIONSFORFRAMEANDCOFRAME}.

\begin{lemma}[$L^2$-control of the error terms in the top-order commuted frame component evolution equations]
	\label{L:FRAMEL2CONTROLOFERRORTERMS}
	Recall that 
$\mathbb{H}_{(\upgamma,k)}(t)$,
$\mathbb{H}_{(e,\oe)}(t)$,
and $\mathbb{D}(t)$ are norms from Definition~\ref{D:SOLUTIONNORMS},
and assume that the bootstrap assumptions \eqref{Boots} hold.
There exists a constant $C = C_{N,N_0,\blowupexp,\mydim,q,\upsigma} > 0$ 
such that if $N_0 \geq 1$ and $N$ is sufficiently large in a manner 
that depends on $N_0, \blowupexp, \mydim, q,$ and $\upsigma$,
and if $\varepsilon$ is sufficiently small (in a manner that depends on $N, N_0, \blowupexp, \mydim, q,$ and $\upsigma$), 
then the error terms
$\mathfrak{E}_I^{i;(\textnormal{Border};\iota)}$,
$\mathfrak{E}_I^{i;(\textnormal{Junk};\iota)}$,
$\mathfrak{O}_I^{i;(\textnormal{Border};\iota)}$, 
and $\mathfrak{O}_I^{i;(\textnormal{Junk};\iota)}$
defined in
\eqref{E:FRAMEBORDERERRORTERMS}--\eqref{E:COFRAMEJUNKERRORTERMS}
verify the following $L^2$ estimates for
$t \in (T_{\textnormal{Boot}},1]$:
	\begin{subequations}
	\begin{align} 
	\label{E:FRAMEL2CONTROLBORDERLINEERROR}
		t^{\blowupexp + q}
		\sqrt{
		\sum_{|\iota| = N}
		\sum_{I,i=1,\cdots,\mydim}
		\| \mathfrak{E}_I^{i;(\textnormal{Border};\iota)} \|_{L^2(\Sigma_t)}^2}
		& \leq C \varepsilon t^{-1} \mathbb{H}_{(\upgamma,k)}(t)
				+
				C \varepsilon t^{-1} \mathbb{H}_{(e,\oe)}(t)
				+
				C t^{-1 + \upsigma} \mathbb{D}(t),
			\\
		t^{\blowupexp + q}
		\sqrt{
		\sum_{|\iota| = N}
		\sum_{I,i=1,\cdots,\mydim}
		\| \mathfrak{E}_I^{i;(\textnormal{Junk};\iota)} \|_{L^2(\Sigma_t)}^2}
		& \leq
				C t^{-1 + \upsigma} \mathbb{D}(t),
			\label{E:FRAMEL2CONTROLJUNKERROR}\\
			\label{E:COFRAMEL2CONTROLBORDERLINEERROR}
		t^{\blowupexp + q}
		\sqrt{
		\sum_{|\iota| = N}
		\sum_{I,i=1,\cdots,\mydim}
		\| \mathfrak{O}_I^{i;(\textnormal{Border};\iota)} \|_{L^2(\Sigma_t)}^2}
		& \leq C \varepsilon t^{-1} \mathbb{H}_{(\upgamma,k)}(t)
				+
				C \varepsilon t^{-1} \mathbb{H}_{(e,\oe)}(t)
				+
				C t^{-1 + \upsigma} \mathbb{D}(t),
			\\
		t^{\blowupexp + q}
		\sqrt{
		\sum_{|\iota| = N}
		\sum_{I,i=1,\cdots,\mydim}
		\| \mathfrak{O}_I^{i;(\textnormal{Junk};\iota)} \|_{L^2(\Sigma_t)}^2}
		& \leq
				C t^{-1 + \upsigma} \mathbb{D}(t).
			\label{E:COFRAMEL2CONTROLJUNKERROR}
	\end{align}
	\end{subequations}
	
	\end{lemma}

	\begin{proof}
		The lemma follows from the 
		expressions \eqref{E:FRAMEBORDERERRORTERMS}--\eqref{E:COFRAMEJUNKERRORTERMS},
		the explicit formulas \eqref{Kasnersol},
		the inequalities in \eqref{sigma,q},
		Definition~\ref{D:SOLUTIONNORMS},
		the bootstrap assumptions,
		the product inequality \eqref{Sob.prod},
		and the already derived estimates \eqref{n.low.est}--\eqref{E:LAPSECONTROLLEDBYDYNAMICVARIABLES} for $n$.
	\end{proof}

\subsubsection{Commuted equations for $k$ and $\upgamma$}
\label{SSS:METRICCOMMUTEDEQUATIONS}
In this section, we provide the evolution equations that we will use to control
the scalar functions
$\lbrace k_{IJ} \rbrace_{I,J=1,\cdots,\mydim}$
and
$\lbrace \upgamma_{IJB} \rbrace_{I,J,B=1,\cdots,\mydim}$ 
as well as their derivatives.

\begin{lemma}[$\partial^{\iota}$-commuted equations for $\upgamma$ and $k$]
\label{L:TOPORDERCOMMUTEDEQUATIONSFORGAMMANDK}
		Let $\iota$ be a spatial multi-index with $|\iota| \leq N$,
		and let $\Pow \geq 0$ be a real number.
		Then for solutions to the equations of Proposition~\ref{P:redeq},
		the following evolution equations hold,
		where the Kasner background scalars $\lbrace \widetilde{k}_{IJ} \rbrace_{I,J=1,\cdots,\mydim}$ 
		and $\widetilde{\psi}$
		are defined in \eqref{Kasnersol} 
		(see also Remark~\ref{R:NOTCOMPONENTSOFTENSORRELATIVETOFRAME}):
\begin{subequations}
\begin{align}
\begin{split}
\partial_t [t^{\Pow} \partial^{\iota} (k_{IJ} - \widetilde{k}_{IJ})]
& = (\Pow-1) t^{\Pow-1} \partial^{\iota} (k_{IJ} - \widetilde{k}_{IJ})
		+
		t^{\Pow} n e_C \partial^{\iota} \upgamma_{IJC}
		- 
		t^{\Pow} n e_I \partial^{\iota} \upgamma_{CJC}
		-
		t^{\Pow} e_I \partial^{\iota} e_J n
			\label{E:COMMUTEDKEQUATION} 
			\\
	& \ \
		+
		t^{\Pow-1}\mathfrak{K}_{IJ}^{(\textnormal{Border};\iota)}
		+
		t^{\Pow} \mathfrak{K}_{IJ}^{(\textnormal{Junk};\iota)},
\end{split} \\
\begin{split}
\partial_t (t^{\Pow} \partial^{\iota} \upgamma_{IJB})
& = \Pow t^{\Pow-1} \partial^{\iota} \upgamma_{IJB}
		+
		t^{\Pow} n e_B \partial^{\iota} k_{JI} 
		- 
		t^{\Pow} n e_J \partial^{\iota} k_{BI} 
		\label{E:COMMUTEDCONNECTIONCOEFFICIENTEQUATION} \\
& \ \
		+
		t^{\Pow} \mathfrak{G}_{IJB}^{(\textnormal{Border};\iota)}
		+
		t^{\Pow} \mathfrak{G}_{IJB}^{(\textnormal{Junk};\iota)},
	\end{split} \\
	t^{\Pow} e_C \partial^{\iota} k_{CI}
	& =
	t^{\Pow} \mathfrak{M}_I^{(\textnormal{Border};\iota)}
	+
	t^{\Pow} \mathfrak{M}_I^{(\textnormal{Junk};\iota)},
	\label{E:COMMUTEDMOMENUTMCONSTRAINT}
\end{align}
\end{subequations}
where:
\begin{subequations}
\begin{align}
\mathfrak{K}_{IJ}^{(\textnormal{Border};\iota)}
& :=	
	\partial^{\iota} (n-1) \cdot \widetilde{k}
	+
	\sum_{\iota_1 \cup \iota_2= \iota}
	\partial^{\iota_1} (n-1) 
	\cdot 
	\partial^{\iota_2} (k - \widetilde{k}),
			\label{E:SECONDFUNDBORDER} \\
	\begin{split}
	\mathfrak{K}_{IJ}^{(\textnormal{Junk};\iota)}
	& := 
			\sum_{\iota_1 \cup \iota_2= \iota,\,|\iota_2|<|\iota|}
			\partial^{\iota_1} e
			\cdot
			\partial \partial^{\iota_2} \vec{e} n
		+
		\sum_{\iota_1 \cup \iota_2= \iota,\,|\iota_2|<|\iota|}
			\partial^{\iota_1} \upgamma 
			\cdot 
			\partial^{\iota_2} \vec{e} n
					\label{E:SECONDFUNDJUNK} \\
	& \ \
		+
		\sum_{\iota_1 \cup \iota_2 \cup \iota_3= \iota,\,|\iota_3|<|\iota|}
			\partial^{\iota_1} n
			\cdot
			\partial^{\iota_2} e
			\cdot 
			\partial^{\iota_3} \partial \upgamma 	
		+
		\sum_{v \in \lbrace \upgamma, \vec{e} \psi \rbrace}
		\sum_{\iota_1 \cup \iota_2 \cup \iota_3= \iota}
			\partial^{\iota_1} n
			\cdot
			\partial^{\iota_2} v 
			\cdot 
			\partial^{\iota_3} v,
	\end{split}
		\\
	\mathfrak{G}_{IJB}^{(\textnormal{Border};\iota)}
	& := 
			n
			\cdot
			\widetilde{k}
			\cdot
			\partial^{\iota} \upgamma
			+
			\sum_{\iota_1 \cup \iota_2 = \iota}
			n
			\partial^{\iota_1} (k - \widetilde{k}) 
			\cdot
			\partial^{\iota_2} \upgamma
				+
	\widetilde{k} \cdot \partial^{\iota} \vec{e} n
	+
	\sum_{\iota_1 \cup \iota_2= \iota}
	\partial^{\iota_1} (k - \widetilde{k}) 
	\cdot 
	\partial^{\iota_2} \vec{e} n,
		\label{E:CONNECTIONCOEFFICIENTBORDER} \\
\mathfrak{G}_{IJB}^{(\textnormal{Junk};\iota)}
& := 
			\sum_{\iota_1 \cup \iota_2 \cup \iota_3= \iota,\,|\iota_1| \geq 1}
			\partial^{\iota_1} n
			\cdot
			\partial^{\iota_2} k
			\cdot 
			\partial^{\iota_3} \upgamma
		+
		\sum_{\iota_1 \cup \iota_2 \cup \iota_3= \iota,\,|\iota_3| < |\iota|}
			\partial^{\iota_1} n
			\cdot
			\partial^{\iota_2} e
			\cdot 
			\partial^{\iota_3} \partial k,
	\label{E:CONNECTIONCOEFFICIENTJUNK}  \\
\mathfrak{M}_I^{(\textnormal{Border};\iota)}
& :=
\widetilde{k} \cdot \partial^{\iota} \upgamma
+
\sum_{\iota_1 \cup \iota_2= \iota}
 \partial^{\iota_1} (k - \widetilde{k}) 
\cdot 
\partial^{\iota_2}
\upgamma
+
\partial_t \widetilde{\psi} \cdot \partial^{\iota} \vec{e} \psi
+
\sum_{\iota_1 \cup \iota_2= \iota}
 \partial^{\iota_1} (e_0 \psi - \partial_t \widetilde{\psi}) 
\cdot 
\partial^{\iota_2}
\vec{e} \psi,
	\label{E:MOMENTUMBORDER} \\
\mathfrak{M}_I^{(\textnormal{Junk};\iota)}
& :=
\sum_{\iota_1 \cup \iota_2= \iota, \, |\iota_2| < |\iota|}
\partial^{\iota_1} e
\cdot
\partial \partial^{\iota_2} k.
\label{E:MOMENTUMJUNK}
\end{align}
\end{subequations}
\end{lemma}

\begin{proof}
	Equations \eqref{E:COMMUTEDKEQUATION}--\eqref{E:COMMUTEDCONNECTIONCOEFFICIENTEQUATION} 
	follow from straightforward computations
	based on first multiplying equations \eqref{dt.k}--\eqref{dt.gamma} by $n$, 
	using that $\partial_t = n e_0$,
	differentiating the resulting equations with $\partial^{\iota}$,
	applying the Leibniz rule,
	multiplying both sides of the resulting identities by $t^{\Pow}$,
	and then commuting the factor of $t^{\Pow}$ under the operator $\partial_t$ on the LHSs
	and accounting for the commutator $[t^{\Pow},\partial_t]$.
	Similarly, equation \eqref{E:COMMUTEDMOMENUTMCONSTRAINT}
	follows from
	differentiating equation \eqref{momconst} with $\partial^{\iota}$,
	applying the Leibniz rule,
	and then multiplying both sides of the resulting identity by $t^{\Pow}$.
\end{proof}

\subsubsection{Pointwise estimates for the error terms in the spatial metric evolution equations}
\label{SSS:SPATIALMETRICERRORTERMSPOINTWISE}
In this section, 
we derive pointwise estimates for the error terms in the equations of 
Lemma~\ref{L:TOPORDERCOMMUTEDEQUATIONSFORGAMMANDK}
that we will later use to control $k-\widetilde{k}$
at derivative levels $\leq N_0+1$.

\begin{lemma}[Pointwise estimates for the error terms in the evolution equations for 
$\partial^{\leq N_0+1} (k-\widetilde{k})$]
\label{L:SPATIALMETRICERRORTERMSPOINTWISE}
Recall that $\mathbb{D}(t)$ is the total norm of the dynamic variables from Definition~\ref{D:SOLUTIONNORMS}.
Assume that the bootstrap assumptions \eqref{Boots} hold.
There exists a constant $C = C_{N,N_0,\blowupexp,\mydim,q,\upsigma} > 0$ 
such that if $N_0 \geq 1$ and $N$ is sufficiently large in a manner 
that depends on $N_0, \blowupexp, \mydim, q,$ and $\upsigma$,
and if $\varepsilon$ is sufficiently small 
(in a manner that depends on $N, N_0, \blowupexp, \mydim, q,$ and $\upsigma$), 
then the following pointwise estimates hold for $(t,x) \in (T_{\textnormal{Boot}},1] \times \mathbb{T}^{\mydim}$,
where $\mathfrak{K}_{IJ}^{(\textnormal{Border};\iota)}$ and $\mathfrak{K}_{IJ}^{(\textnormal{Junk};\iota)}$
are defined in \eqref{E:SECONDFUNDBORDER}--\eqref{E:SECONDFUNDJUNK}:
	\begin{subequations}
	\begin{align}
		\sum_{|\iota| \leq N_0+1} 
		\sum_{I,J=1,\cdots,\mydim}
		{t}\left|
			 n e_C \partial^{\iota} \upgamma_{IJC}
			- 
			 n e_I \partial^{\iota} \upgamma_{CJC}
			-
			 e_I \partial^{\iota} e_J n
		\right|(t,x)
		& \leq C t^{-1+\upsigma} \mathbb{D}(t){,}	
			\label{E:SECONDFUNDPOINTWISEERRORTERMESTIMATE1} 
			\\
			\sum_{|\iota| \leq N_0+1} 
			\sum_{I,J=1,\cdots,\mydim}
				|\mathfrak{K}_{IJ}^{(\textnormal{Border};\iota)}|(t,x)
		+
		\sum_{|\iota| \leq N_0+1}
		\sum_{I,J=1,\cdots,\mydim} 
			t |\mathfrak{K}_{IJ}^{(\textnormal{Junk};\iota)}|(t,x)
		& \leq C t^{-1+\upsigma} \mathbb{D}(t).
			\label{E:SECONDFUNDPIONTWISEERRORTERMESTIMATE2}
	\end{align}
	\end{subequations}
\end{lemma}

\begin{proof}
	The lemma follows from
	the explicit formulas \eqref{Kasnersol}, 
	the inequalities in \eqref{sigma,q},
	Definition~\ref{D:SOLUTIONNORMS},
	the bootstrap assumptions,
	the interpolation estimates of Lemma~\ref{lem:Sob.borrow} (see Remark \ref{rem:deltaA}), 
	and the already derived lower order estimate \eqref{n.low.est} for $n-1$.
\end{proof}

\begin{remark}[On the meaning of ``Borderline'']
\label{R:MEANINGOFBORDERLINE}
Quantities featuring the superscript ``Borderline'' contain terms that are
either borderline with respect to our low order estimates \emph{or} our high order estimates (or both).
For example, the estimate \eqref{E:SECONDFUNDPIONTWISEERRORTERMESTIMATE2} reveals that at the lower orders,
$\mathfrak{K}_{IJ}^{(\textnormal{Border};\iota)}$ is not a borderline term (see also \eqref{SUBSTEP2MAINSTEPlow.est}), 
while the presence of the $C_*$-involving term on RHS~\eqref{E:KTOPORDERBORDERTERMSL2ESTIMATE}
in Lemma~\ref{L:METRICL2CONTROLOFTOPORDERERRORTERMS} below shows that
$\mathfrak{K}_{IJ}^{(\textnormal{Border};\iota)}$ is indeed borderline in the context
of our top-order energy estimates. Similar remarks apply to other quantities
featuring the superscript ``Borderline."
\end{remark}

\subsubsection{Differential energy identity for the second fundamental form and connection coefficients}
\label{SSS:METRICDIFFERENTIALENERGYIDENTITY}
We will derive our top-order energy estimates for the second fundamental form and connection
coefficients by integrating the differential identity provided by the following lemma.

\begin{lemma}[Top-order differential energy identity for $\lbrace k_{IJ} \rbrace_{I,J=1,\cdots,\mydim}$ 
and $\lbrace \upgamma_{IJB} \rbrace_{I,J,B=1,\cdots,\mydim}$]
	\label{L:TOPORDERDIFFERENTIALENERGYIDENTITYFORGAMMAANDK}
	Let $\iota$ be a top-order spatial multi-index,
	i.e., $|\iota|=N$.
	Then for solutions to the $\partial^{\iota}$-commuted equations
	\eqref{E:COMMUTEDKEQUATION}--\eqref{E:COMMUTEDMOMENUTMCONSTRAINT} with $\Pow := \blowupexp + 1$, 
	the following differential energy identity holds, where 
	the error terms
	$\mathfrak{K}_{IJ}^{(\textnormal{Border};\iota)}$, 
	$\mathfrak{K}_{IJ}^{(\textnormal{Junk};\iota)}$, 
	$\mathfrak{G}_{IJB}^{(\textnormal{Border};\iota)}$, 
	$\mathfrak{G}_{IJB}^{(\textnormal{Junk};\iota)}$, 
	$\mathfrak{M}_I^{(\textnormal{Border};\iota)}$, 
	and
	$\mathfrak{M}_I^{(\textnormal{Junk};\iota)}$ 
	are defined in \eqref{E:SECONDFUNDBORDER}--\eqref{E:MOMENTUMJUNK}:
	\begin{align} \label{E:TOPORDERDIFFERENTIALENERGYIDENTITYFORGAMMAANDK} 
	\begin{split}	
		&
		\partial_t
			\left\lbrace
				(t^{\blowupexp + 1}\partial^{\iota}k_{IJ}) (t^{\blowupexp + 1} \partial^{\iota}k_{IJ})
			\right\rbrace
			+
			\frac{1}{2}
			\partial_t
			\left\lbrace
				(t^{\blowupexp + 1} \partial^{\iota}\upgamma_{IJB})
				(t^{\blowupexp + 1} \partial^{\iota}\upgamma_{IJB})
			\right\rbrace
				\\
	& =  \frac{2 \blowupexp}{t} (t^{\blowupexp + 1} \partial^{\iota} k_{IJ}) (t^{\blowupexp + 1}  \partial^{\iota}k_{IJ})
			+
			\frac{(\blowupexp + 1)}{t} (t^{\blowupexp + 1} \partial^{\iota} \upgamma_{IJB}) (t^{\blowupexp + 1} \partial^{\iota} \upgamma_{IJB})
				\\
& \ \
		+ 2 (t^{\blowupexp + 1} \partial^{\iota} k_{IJ}) 
			\left(t^{\blowupexp} \mathfrak{K}_{IJ}^{(\textnormal{Border};\iota)} + t^{\blowupexp + 1} \mathfrak{K}_{IJ}^{(\textnormal{Junk};\iota)} \right)
				\\
& \ \
		+ (t^{\blowupexp + 1} \partial^{\iota} \upgamma_{IJB}) 
			\left(t^{\blowupexp + 1} \mathfrak{G}_{IJB}^{(\textnormal{Border};\iota)} + t^{\blowupexp + 1} \mathfrak{G}_{IJB}^{(\textnormal{Junk};\iota)} \right)
						 \\
	& \ \
		+
		2  
		(t^{\blowupexp + 1} \partial^{\iota} e_J n)
		\left(t^{\blowupexp + 1} \mathfrak{M}_J^{(\textnormal{Border};\iota)} + t^{\blowupexp + 1} \mathfrak{M}_J^{(\textnormal{Junk};\iota)} \right)
			\\
&  \ \	
		+
		2  
		n 
		(t^{\blowupexp + 1} \partial^{\iota} \upgamma_{CJC})
		\left(t^{\blowupexp + 1} \mathfrak{M}_J^{(\textnormal{Border};\iota)} + t^{\blowupexp + 1} \mathfrak{M}_J^{(\textnormal{Junk};\iota)} \right)
			 \\
	& \ \
		+
		2
		(\partial_c e_I^c) 
		(t^{\blowupexp + 1} \partial^{\iota} e_J n) 	
		(t^{\blowupexp + 1} \partial^{\iota} k_{IJ})
			\\
&  \ \
		+
		2
		\left\lbrace
			\partial_c(n e_I^c)
		\right\rbrace
		(t^{\blowupexp + 1}\partial^{\iota} k_{IJ}) 
		(t^{\blowupexp + 1}\partial^{\iota} \upgamma_{CJC})
		-
		2 
		\left\lbrace
			\partial_c (n e_C^c)
		\right\rbrace
		(t^{\blowupexp + 1} \partial^{\iota} k_{IJ})  
		(t^{\blowupexp + 1}\partial^{\iota} \upgamma_{IJC})
			\\
	& \ \
		-
		2 
		\partial_c
		\left\lbrace
			t^{2 \blowupexp+2} e_I^c (\partial^{\iota} e_J n) \partial^{\iota} k_{IJ}
		\right\rbrace
		- 
		2
		\partial_c
		\left\lbrace
			t^{2 \blowupexp+2} e_I^c n (\partial^{\iota} k_{IJ}) \partial^{\iota} \upgamma_{CJC}
		\right\rbrace
			\\
	& \ \
		+
		2 
		\partial_c
		\left\lbrace
			t^{2 \blowupexp+2}
			n e_C^c
			(\partial^{\iota} k_{IJ})  \partial^{\iota} \upgamma_{IJC}
		\right\rbrace.
	\end{split}
	\end{align}
\end{lemma}

\begin{proof}
	The proof is a calculation that, although lengthy, is 
	straightforward; hence, we only explain the main steps.
	We first note that $\partial^{\iota} \widetilde{k}_{IJ} = 0$
	and thus we can ignore the formal presence of this term on LHS~\eqref{E:COMMUTEDKEQUATION}.
	Next, we expand LHS~\eqref{E:TOPORDERDIFFERENTIALENERGYIDENTITYFORGAMMAANDK} using the Leibniz rule.
	When $\partial_t$ falls on $t^{\blowupexp + 1} \partial^{\iota}k_{IJ}$,
	we plug in \eqref{E:COMMUTEDKEQUATION} with $\Pow := \blowupexp + 1$.
	When $\partial_t$ falls on $t^{\blowupexp + 1} \partial^{\iota}\upgamma_{IJB}$,
	we plug in \eqref{E:COMMUTEDCONNECTIONCOEFFICIENTEQUATION} with $\Pow := \blowupexp + 1$.
	We then differentiate the resulting terms by parts. 
	Next, we use the (differentiated) momentum constraint \eqref{E:COMMUTEDMOMENUTMCONSTRAINT}
	with $\Pow := \blowupexp + 1$
	to substitute for the terms $t^{\blowupexp + 1} e_I \partial^{\iota} k_{IJ}$ in the product
	$2 (t^{\blowupexp + 1} \partial^{\iota} e_J n) \cdot t^{\blowupexp + 1} e_I \partial^{\iota} k_{IJ}$
	(which is ``present'' in the sense that it is needed to cancel a corresponding
	product obtained from 
	expanding the third-to-last term
	$
		-
		2 
		\partial_c
		\left\lbrace
			t^{2 \blowupexp+2} e_I^c (\partial^{\iota} e_J n) \partial^{\iota} k_{IJ}
		\right\rbrace
	$
	on RHS~\eqref{E:TOPORDERDIFFERENTIALENERGYIDENTITYFORGAMMAANDK}).
	Similarly,
	we use \eqref{E:COMMUTEDMOMENUTMCONSTRAINT}
	with $\Pow := \blowupexp + 1$
	to substitute for the terms $t^{\blowupexp + 1} e_I \partial^{\iota} k_{IJ}$ in the product
	$2 n t^{\blowupexp + 1} e_I \partial^{\iota} k_{IJ} \cdot t^{\blowupexp + 1} \partial^{\iota} \upgamma_{CJC}$
	(which is ``present'' in the sense that it is needed to cancel a corresponding
	product obtained from 
	expanding the next-to-last term
	$
	- 
		2
		\partial_c
		\left\lbrace
			t^{2 \blowupexp+2} e_I^c n (\partial^{\iota} k_{IJ}) \partial^{\iota} \upgamma_{CJC}
		\right\rbrace
	$
	on RHS~\eqref{E:TOPORDERDIFFERENTIALENERGYIDENTITYFORGAMMAANDK}).
\end{proof}
\begin{remark}[Comments tied to the momentum constraint and well-posedness in CMC-transported spatial coordinates]\label{rem:mom.const.use}
	The (differentiated) momentum constraint \eqref{E:COMMUTEDMOMENUTMCONSTRAINT} 
	plays a crucial role in our proof of Lemma~\ref{L:TOPORDERDIFFERENTIALENERGYIDENTITYFORGAMMAANDK};
	without this constraint equation, the corresponding differential energy identity would have featured terms 
	involving one too many derivatives of $k_{IJ}$, which in turn would have led to a fatal loss of one derivative
	in the top-order estimates. 
	An alternate way to overcome the derivative loss
	is to use spatial harmonic coordinates on each time slice $\Sigma_t$, as in \cite{lAvM2003}. 
	However, such coordinates lead to the presence of a non-zero shift vector in the coordinate expression for the spacetime
	metric, and it is not currently known whether the corresponding error terms are compatible with a proof of 
	stable Big Bang formation. We also emphasize that for similar reasons, 
	the momentum constraint equation plays a crucial role
	in proving local well-posedness for Einstein's equations in
	CMC-transported spatial coordinates; see \cite[Theorem~14.1]{RodSp2}.
	Moreover, we also highlight that while energy identities such as \eqref{E:TOPORDERDIFFERENTIALENERGYIDENTITYFORGAMMAANDK} 
	can be used to derive a priori energy estimates for solutions to the \emph{nonlinear} 
	reduced equations 
	(where by ``reduced,'' we roughly mean gauge-dependent equations in the spirit of the ones stated in Proposition~\ref{P:redeq}),
	the proof of local well-posedness given by \cite[Theorem~14.1]{RodSp2}
	relies on a modified system, which can be shown to be equivalent to the nonlinear reduced equations
	(and hence, by the ``if and only if'' aspect of Proposition~\ref{P:redeq}, equivalent to Einstein's equations too)
	for initial data that satisfy the constraints and the CMC condition \eqref{INTROtrk} at $t=1$. 
	The key advantage of the modified system is that 
	it does not involve constraint equations; this allows one to show that solutions to
	linearized versions of the modified system also enjoy good energy estimates, 
	which is important for the standard iteration/contraction mapping schemes that are used
	in proofs of local well-posedness for quasilinear equations.
\end{remark}

\subsubsection{Control of the error terms in the top-order commuted spatial metric equations}
\label{SSS:TOPORDERSPATIALMETRICINHOMOGENEOUSTERML2ESTIMATE}
In this section, at the top-order derivative level,
we derive $L^2$ estimates for the error terms in the equations of 
Lemma~\ref{L:TOPORDERCOMMUTEDEQUATIONSFORGAMMANDK}.

\begin{lemma} [$L^2$-control of the error terms in the top-order commuted evolution equations for $k$ and $\upgamma$]
	\label{L:METRICL2CONTROLOFTOPORDERERRORTERMS}
		Recall that $\mathbb{H}_{(\upgamma,k)}$,
$\mathbb{H}_{(\psi)}$,
and $\mathbb{D}(t)$ are norms from Definition~\ref{D:SOLUTIONNORMS},
and assume that the bootstrap assumptions \eqref{Boots} hold.
	Recall that the error terms
	$\mathfrak{K}_{IJ}^{(\textnormal{Border};\iota)}$,
	$\mathfrak{K}_{IJ}^{(\textnormal{Junk};\iota)}$,
	$\mathfrak{G}_{IJB}^{(\textnormal{Border};\iota)}$,
	$\mathfrak{G}_{IJB}^{(\textnormal{Junk};\iota)}$,
	$\mathfrak{M}_I^{(\textnormal{Border};\iota)}$,
	and
	$\mathfrak{M}_I^{(\textnormal{Junk};\iota)}$
	are defined in \eqref{E:SECONDFUNDBORDER}--\eqref{E:MOMENTUMJUNK}.
There exists a constant $C_* > 0$ \underline{independent of $N, N_0,$ and $\blowupexp$}
and a constant $C = C_{N,N_0,\blowupexp,\mydim,q,\upsigma} > 0$ 
such that if $N$ is sufficiently large in a manner 
that depends on $N_0, \blowupexp, \mydim, q,$ and $\upsigma$,
and if $\varepsilon$ is sufficiently small (in a manner that depends on $N, N_0, \blowupexp, \mydim, q,$ and $\upsigma$), 
then the following estimates hold for $t \in (T_{\textnormal{Boot}},1]$:
	\begin{subequations}
	\begin{align}
		t^{\blowupexp}
		\sqrt{
		\sum_{|\iota| = N}
		\sum_{I,J=1,\cdots,\mydim}
		\|  \mathfrak{K}_{IJ}^{(\textnormal{Border};\iota)}\|_{L^2(\Sigma_t)}^2}
		& \leq 
				C_* t^{-1} \mathbb{H}_{(\upgamma,k)}(t)
					+
				C t^{-1 + \upsigma} \mathbb{D}(t),
			\label{E:KTOPORDERBORDERTERMSL2ESTIMATE}	\\
		t^{\blowupexp + 1}
		\sqrt{
		\sum_{|\iota| = N}
		\sum_{I,J,B=1,\cdots,\mydim}
		\|  \mathfrak{G}_{IJB}^{(\textnormal{Border};\iota)}\|_{L^2(\Sigma_t)}^2}
		& \leq 
				C_* t^{-1} \mathbb{H}_{(\upgamma,k)}(t)
				+
				C t^{-1 + \upsigma} \mathbb{D}(t),
					\label{E:CONNECTIONTOPORDERBORDERTERMSL2ESTIMATE} \\
		t^{\blowupexp + 1}
		\sqrt{
		\sum_{|\iota| = N}
		\sum_{I=1,\cdots,\mydim}
		\|  \mathfrak{M}_I^{(\textnormal{Border};\iota)}\|_{L^2(\Sigma_t)}^2}
		& \leq 
			C_* t^{-1} \mathbb{H}_{(\upgamma,k)}(t)
				+
				C_* t^{-1} \mathbb{H}_{(\psi)}(t)
				+
				C t^{-1 + \upsigma} \mathbb{D}(t),
				\label{E:CONSTRAINTTOPORDERBORDERTERMSL2ESTIMATE}
	\end{align}
	\end{subequations}
	
	\begin{subequations}
	\begin{align}
		t^{\blowupexp + 1}
		\sqrt{
		\sum_{|\iota| = N}
		\sum_{I,J=1,\cdots,\mydim}
		\| \mathfrak{K}_{IJ}^{(\textnormal{Junk};\iota)} \|_{L^2(\Sigma_t)}^2}
		& \leq C t^{-1 + \upsigma} \mathbb{D}(t),	
			\label{E:KTOPORDERJUNKTERMSL2ESTIMATE} 
				\\
		t^{\blowupexp + 1}
		\sqrt{
		\sum_{|\iota| = N}
		\sum_{I,J,B=1,\cdots,\mydim}
		\| \mathfrak{G}_{IJB}^{(\textnormal{Junk};\iota)} \|_{L^2(\Sigma_t)}^2}
		& \leq C t^{-1 + \upsigma} \mathbb{D}(t),	
			\label{E:CONNECTIONTOPORDERJUNKTERMSL2ESTIMATE} 
				\\
		t^{\blowupexp + 1}
		\sqrt{
		\sum_{|\iota| = N}
		\sum_{I=1,\cdots,\mydim}
		\| \mathfrak{M}_I^{(\textnormal{Junk};\iota)} \|_{L^2(\Sigma_t)}^2}
		& \leq C t^{-1 + \upsigma} \mathbb{D}(t).
		\label{E:CONSTRAINTTOPORDERJUNKTERMSL2ESTIMATE}
	\end{align}
	\end{subequations}
\end{lemma}

\begin{proof}
 We will give the proofs of \eqref{E:KTOPORDERBORDERTERMSL2ESTIMATE}
	and \eqref{E:KTOPORDERJUNKTERMSL2ESTIMATE}.
	The remaining estimates can be proved using similar arguments, and we omit the details.
	To prove \eqref{E:KTOPORDERBORDERTERMSL2ESTIMATE}, 
	we let $\iota$ be any spatial multi-index with $|\iota| = N$.
	We multiply both sides of \eqref{E:SECONDFUNDBORDER} by $t^{\blowupexp}$ and take the
	$\| \cdot \|_{L^2(\Sigma_t)}$ norm.
	Using the bootstrap assumptions, 
	the explicit formulas \eqref{Kasnersol},
	the inequalities in \eqref{sigma,q},
	Definition~\ref{D:SOLUTIONNORMS},
	and the product estimate \eqref{Sob.prod},
	we find that
	$
	t^{\blowupexp} \|  \mathfrak{K}_{IJ}^{(\textnormal{Border};\iota)}\|_{L^2(\Sigma_t)}
	\leq 
	C_* t^{A_*-1} \| \partial^{\iota} n \|_{L^2(\Sigma_t)}
	+
	C \varepsilon t^{A_*-1} \| n \|_{\dot{H}^N(\Sigma_t)}
	+
	C \varepsilon t^{-1+\upsigma} \mathbb{D}(t)
	$.
	We then square this estimate, sum over all $\iota$ with $|\iota| = N$,
	sum over all $1 \leq I,J \leq \mydim$, and then take the square root.
	We find that
	$\mbox{LHS~\eqref{E:KTOPORDERBORDERTERMSL2ESTIMATE}}
	\leq (C_* + C \varepsilon) t^{A_*-1} \| n \|_{\dot{H}^N(\Sigma_t)}
	+
	C \varepsilon t^{-1+\upsigma} \mathbb{D}(t)
	\leq 
	C_* t^{A_*-1} \| n \|_{\dot{H}^N(\Sigma_t)}
	+
	C \varepsilon t^{-1+\upsigma} \mathbb{D}(t)
	$.
	From this bound and the already derived high order estimate \eqref{n.high.est} for $n$,
	we arrive at the desired bound \eqref{E:KTOPORDERBORDERTERMSL2ESTIMATE}.
	
	The estimate \eqref{E:KTOPORDERJUNKTERMSL2ESTIMATE}
	can be proved by multiplying equation \eqref{E:SECONDFUNDJUNK}
	by $t^{\blowupexp + 1}$ and combing arguments similar to the ones we used above
	with the estimates of Lemma~\ref{lem:basic.ineq}.
\end{proof}

\subsection{Preliminary identities and inequalities for the scalar field $\psi$}
\label{SS:SCALARFIELDESTIMATES}
This section is an analog of Sect.\,\ref{SS:PRELIMINARYRESULTSFORFRAMESECONDFUNDANDCONNECTION}
for the scalar field $\psi$. That is,
we derive preliminary low order and high order 
identities and inequalities for $\psi$ by using the wave equation \eqref{boxpsi}. 
In order to avoid the time derivative of $n$ appearing as an error term in the equations
(which would unnecessarily complicate our derivation of the main estimates), 
we treat $e_0\psi$, $\lbrace e_I\psi \rbrace_{I=1,\cdots,\mydim}$ 
as separate variables satisfying a first-order system derived from the wave equation, 
cf.\ \cite{Sp}. 
Roughly, we bound the inhomogeneous terms in the evolution equations
in terms of our solution norms, and we derive an energy identity in differential form.
In Sects.\,\ref{SS:INTEGRALINEQUALITYFORLOWNORMS}--\ref{PROOFOFprop:overall},
we will combine these preliminary results with related ones
for $n$, $\lbrace k_{IJ} \rbrace_{I,J=1,\cdots,\mydim}$, 
$\lbrace \upgamma_{IJB} \rbrace_{I,J,B=1,\cdots,\mydim}$, 
$\lbrace e_I^i \rbrace_{I,i=1,\cdots,\mydim}$,
and
$\lbrace \oe_i^I \rbrace_{I,i=1,\cdots,\mydim}$
to derive our main a priori estimates,
i.e., to prove Proposition~\ref{prop:overall}.

\subsubsection{Commuted evolution equations for $e_0 \psi$ and $e_I \psi$}
\label{SSS:SCALARFIELDEQUATIONS}
In this section, we provide the first-order evolution equations that we will use to control
the scalar functions
$e_0 \psi$
and
$\lbrace e_I \psi \rbrace_{I=1,\cdots,\mydim}$
as well as their derivatives.

\begin{lemma}[The first-order evolution system for $e_0 \psi$, $\lbrace e_I \psi \rbrace_{I=1,\cdots,\mydim}$, and their derivatives]
\label{lem:psi.syst}
For solutions to the equations of Proposition~\ref{P:redeq},
the ${\bf g}$-orthonormal frame derivatives of $\psi$, 
namely $e_0\psi$ and $\lbrace e_I \psi \rbrace_{I=1,\cdots,\mydim}$, 
satisfy the following first-order symmetric hyperbolic system,
where the Kasner background scalars 
$\widetilde{\psi}=\widetilde{B}\log t$
and
$\lbrace \widetilde{k}_{IJ} \rbrace_{I,J=1,\cdots,\mydim}$
are defined in \eqref{Kasnersol}
(see also Remark~\ref{R:NOTCOMPONENTSOFTENSORRELATIVETOFRAME}),
and we recall that
we do not sum over repeated underlined indices:
\begin{subequations}
\begin{align}
\label{dt.e0psi}
\partial_t [t(e_0\psi-\partial_t \widetilde{\psi})]
& 
=
t n e_C e_C \psi
-
t n \upgamma_{CCD} e_D \psi
+ 
t (e_C n) e_C \psi
-
(n-1) \partial_t \widetilde{\psi} -
(n-1) (e_0 \psi - \partial_t \widetilde{\psi}),
	\\
\begin{split}
\label{dt.eipsi}\partial_t e_I \psi
&
=
-
\frac{\widetilde{q}_{\underline{I}}}{t} e_{\underline{I}} \psi
+
n e_I e_0 \psi
+
(n - 1) k_{IC} e_C \psi
+
(k_{IC} - \widetilde{k}_{IC}) e_C \psi
+
(e_I n) 
\partial_t \widetilde{\psi}\\
&	\ \
+
(e_I n) (e_0 \psi - \partial_t \widetilde{\psi}).
\end{split}
\end{align}
\end{subequations}

Moreover, if $\iota$ is a spatial coordinate multi-index
and $\Pow \geq 0$ is any real number, 
then the following equations hold:
\begin{subequations}
\begin{align} \label{dt.e0psi.diff}
\begin{split}
\partial_t [t^{\Pow}\partial^{\iota} (e_0\psi-\partial_t\widetilde{\psi})]
& =
(\Pow - 1) [t^{\Pow-1}\partial^{\iota} (e_0\psi-\partial_t\widetilde{\psi})]
+
t^{\Pow} n e_C \partial^{\iota} e_C \psi	
	 \\
& \ \
+
t^{\Pow-1}
\mathfrak{P}^{(\textnormal{Border};\iota)}
+
t^{\Pow}
\mathfrak{P}^{(\textnormal{Junk};\iota)},
\end{split} \\
\label{dt.eipsi.diff}
\partial_t (t^{\Pow} \partial^{\iota} e_I \psi)
& = 
	(\Pow - \widetilde{q}_{\underline{I}}) t^{\Pow-1} \partial^{\iota} e_{\underline{I}} \psi
	+
	t^{\Pow} n e_I \partial^{\iota} e_0 \psi
	+
t^{\Pow}
\mathfrak{Q}_I^{(\textnormal{Border};\iota)}
+
t^{\Pow}
\mathfrak{Q}_I^{(\textnormal{Junk};\iota)},
\end{align}
\end{subequations}
where:
\begin{subequations}
\begin{align}
	\mathfrak{P}^{(\textnormal{Border};\iota)}
	& := 	
				\partial^{\iota} (n-1) \cdot \partial_t \widetilde{\psi}
				+
				\sum_{\iota_1 \cup \iota_2 = \iota}
				\partial^{\iota_1} (n-1)
				\cdot
				\partial^{\iota_2} (e_0 \psi - \partial_t \widetilde{\psi}),
		\label{E:SCALARFIELDTIMEBORDER} \\
\begin{split}
\mathfrak{P}^{(\textnormal{Junk};\iota)}
& :=
	\sum_{\iota_1\cup \iota_2\cup \iota_3= \iota,\,|\iota_3|<|\iota|}
	\partial^{\iota_1}n 
	\cdot
	\partial^{\iota_2}e
	\cdot
	\partial \partial^{\iota_3} \vec{e} \psi
		\label{E:SCALARFIELDTIMEJUNK}  \\
& \ \
	+
	\sum_{\iota_1\cup \iota_2\cup \iota_3= \iota}
	\partial^{\iota_1}n 
	\cdot
	\partial^{\iota_2} \upgamma
	\cdot
	\partial^{\iota_3} \vec{e} \psi
	+
	\sum_{\iota_1 \cup \iota_2 = \iota}
	\partial^{\iota_1} \vec{e} n
	\cdot
	\partial^{\iota_2} \vec{e} \psi,	
		\end{split} \\
\mathfrak{Q}_I^{(\textnormal{Border};\iota)}
& :=		
\sum_{\iota_1 \cup \iota_2 = \iota}
\partial^{\iota_1}(k - \widetilde{k}) \cdot \partial^{\iota_2} \vec{e} \psi
+
\partial^{\iota}  \vec{e} n 
\cdot
\partial_t \widetilde{\psi}
+
\sum_{\iota_1 \cup \iota_2 = \iota}
\partial^{\iota_1} \vec{e} n 
\cdot 
\partial^{\iota_2} (e_0 \psi - \partial_t \widetilde{\psi})
				\label{E:SCALARFIELDSPACEBORDER}  
					\\
& \ \ 
	+ 
	\sum_{\iota_1\cup \iota_2 = \iota,\,|\iota_2|<|\iota|}
	\partial^{\iota_1} e
	\cdot
	\partial \partial^{\iota_2} e_0 \psi,
	\notag
	\\
\mathfrak{Q}_I^{(\textnormal{Junk};\iota)}
& :=	
	\sum_{\iota_1\cup \iota_2\cup \iota_3= \iota,\,|\iota_3|<|\iota|}
	\partial^{\iota_1} (n-1) 
	\cdot
	\partial^{\iota_2} e
	\cdot
	\partial \partial^{\iota_3} e_0 \psi
	+
	\sum_{\iota_1 \cup \iota_2 \cup \iota_3 = \iota}
	\partial^{\iota_1} 
	(n - 1) 
	\cdot
	\partial^{\iota_2} k
	\cdot 
	\partial^{\iota_3} \vec{e} \psi.
				\label{E:SCALARFIELDSPACEJUNK} 
\end{align}
\end{subequations}

\end{lemma}
\begin{proof}
Equation \eqref{dt.e0psi} follows from multiplying both sides of \eqref{boxpsi} by $nt$, 
using that $\partial_t (t\partial_t \widetilde{\psi}) = 0$,
and carrying out straightforward algebraic computations.
\eqref{dt.eipsi} follows from the identity
$\partial_t = n e_0$,
the commutation identity \eqref{comm[dt,ei]},
and straightforward algebraic computations.
\eqref{dt.e0psi.diff}--\eqref{dt.eipsi.diff} then follow from 
differentiating \eqref{dt.e0psi}--\eqref{dt.eipsi} with $\partial^{\iota}$,
using the Leibniz rule,
multiplying both sides of the resulting equations by $t^{\Pow-1}$ and $t^{\Pow}$ respectively,
commuting the factors of $t^{\Pow-1}$ and $t^{\Pow}$ under the operator $\partial_t$ on the LHSs,
and accounting for the commutators $[t^{\Pow-1},\partial_t]$ and $[t^{\Pow},\partial_t]$.
\end{proof}

\subsubsection{Pointwise estimates for the error terms in the scalar field evolution equations}
\label{SSS:SCALARFIELDERRORTERMSPOINTWISE}
In this section, 
we derive the pointwise estimates for the error terms in the equations of 
Lemma~\ref{lem:psi.syst}
that we will later use to control $e_0\psi-\partial_t\widetilde{\psi}$
at derivative levels $\leq N_0+1$
and 
$e_I \psi$
at derivative levels $\leq N_0$.

\begin{lemma}[Pointwise estimates for the error terms in the evolution equations for 
	$\partial^{\leq N_0+1} (e_0\psi-\partial_t\widetilde{\psi})$ and 
	$\lbrace \partial^{\leq N_0} e_I \psi \rbrace_{I=1,\cdots,\mydim}$]
	Recall that $\mathbb{D}(t)$ is the total norm of the dynamic variables from Definition~\ref{D:SOLUTIONNORMS},
	and assume that the bootstrap assumptions \eqref{Boots} hold.
	Recall that the error terms
	$\mathfrak{P}^{(\textnormal{Border};\iota)}$,
	$\mathfrak{P}^{(\textnormal{Junk};\iota)}$,
	$\mathfrak{Q}_I^{(\textnormal{Border};\iota)}$,
	and
	$\mathfrak{Q}_I^{(\textnormal{Junk};\iota)}$
	are defined in \eqref{E:SCALARFIELDTIMEBORDER}--\eqref{E:SCALARFIELDSPACEJUNK}.
There exists a constant $C = C_{N,N_0,\blowupexp,\mydim,q,\upsigma} > 0$ 
such that if $N$ is sufficiently large in a manner 
that depends on $N_0, \blowupexp, \mydim, q,$ and $\upsigma$,
and if $\varepsilon$ is sufficiently small (in a manner that depends on $N, N_0, \blowupexp, \mydim, q,$ and $\upsigma$), 
then the following pointwise 
estimates hold for $(t,x) \in (T_{\textnormal{Boot}},1] \times \mathbb{T}^{\mydim}$:
	\begin{subequations}
	\begin{align} 
			\sum_{|\iota| \leq  N_0+1}
			t|n e_C \partial^{\iota} e_C \psi|(t,x)
			& \leq C t^{-1 + \upsigma} \mathbb{D}(t),
			 \label{E:TIMEFRAMEDERIVATIVESCALARFIELDDERIVATIVELOSSTERMPOINTWISEBOUNDS} 
				\\
		\sum_{|\iota| \leq N_0+1}
		|\mathfrak{P}^{(\textnormal{Border};\iota)}|(t,x)
		+
		\sum_{|\iota| \leq N_0+1}
		t
		|\mathfrak{P}^{(\textnormal{Junk};\iota)}|(t,x)
		& \leq C t^{-1 + \upsigma} \mathbb{D}(t),
			\label{E:TIMEDERIVATIVESCALARFIELDBORDERANDJUNKPOINTWISEBOUNDS} 
			\\
		\begin{split}
		\sum_{|\iota| \leq N_0} 
		\sum_{I = 1, \cdots, \mydim}
		t^q
		\left|
			 n e_I \partial^{\iota} e_0 \psi
		\right|(t,x)
		& \leq 
			C \varepsilon 
			\sum_{|\iota| \leq N_0} 
			\sum_{I,i=1,\cdots,\mydim}
			t^{q-1} |\partial^{\iota} (e_I^i - \widetilde{e}_I^i)|(t,x)
				\label{E:SPATIALFRAMEDERIVATIVESCALARFIELDANNOYINGDERIVATIVELOSSTERMPOINTWISEBOUNDS}  
					\\
		& \ \
			+
			C t^{-1 + \upsigma} \mathbb{D}(t),
	\end{split}	
			\\
		\begin{split}
		\sum_{|\iota| \leq N_0} 
		\sum_{I = 1,\cdots, \mydim}
		t^q |\mathfrak{Q}_I^{(\textnormal{Border};\iota)}|(t,x)
		& \leq C \varepsilon 
		\sum_{|\iota|\leq N_0}
		\sum_{I=1,\cdot, \mydim} 
		t^{q-1} |\partial^{\iota} e_I \psi|(t,x)
				\label{E:SPATIALFRAMEDERIVATIVESCALARFIELDBORDERPOINTWISEBOUNDS} 
				\\
& \ \ 				
		+
		C \varepsilon 
		\sum_{|\iota| \leq N_0} 
		\sum_{I,i=1,\cdots,\mydim}
		t^{q-1} |\partial^{\iota}(e_I^i - \widetilde{e}_I^i)|(t,x)
			\\
& \ \
		+
		C t^{-1 + \upsigma} \mathbb{D}(t),
	\end{split}
			\\
		\sum_{|\iota| \leq N_0} 
		\sum_{I = 1, \cdots, \mydim}
		t^q |\mathfrak{Q}_I^{(\textnormal{Junk};\iota)}|(t,x)
		& \leq C t^{-1+\upsigma} \mathbb{D}(t).
			\label{E:SPATIALFRAMEDERIVATIVESCALARFIELDJUNKPOINTWISEBOUNDS}
	\end{align}
	\end{subequations}
\end{lemma}

\begin{proof}
 We apply the same arguments we used in the proof of 
	Lemmas~\ref{L:STRUCTURECOFFEICIENTERRORTERMSPOINTWISE}
	and
	Lemma~\ref{L:SPATIALMETRICERRORTERMSPOINTWISE},
	taking into account the structure of the terms
	on RHS~\eqref{E:SCALARFIELDTIMEBORDER}--\eqref{E:SCALARFIELDSPACEJUNK}
	and the fact that the low order norm \eqref{norms.low}
	controls $e_0 \psi$ at up to derivative level $N_0+1$ 
	(in particular, we use this fact to derive 
	\eqref{E:TIMEDERIVATIVESCALARFIELDBORDERANDJUNKPOINTWISEBOUNDS}--\eqref{E:SPATIALFRAMEDERIVATIVESCALARFIELDANNOYINGDERIVATIVELOSSTERMPOINTWISEBOUNDS}).
	We also clarify that to obtain \eqref{E:SPATIALFRAMEDERIVATIVESCALARFIELDANNOYINGDERIVATIVELOSSTERMPOINTWISEBOUNDS},
	we use the triangle inequality to bound the summand on the LHS by 
	$
	\leq
	t^q
	|n (e_I^c - \widetilde{e}_I^c) \partial_c \partial^{\iota} e_0 \psi|
	+
	t^q
	|n \widetilde{e}_I^c \partial_c \partial^{\iota} e_0 \psi|
	$
	and then bound (rather inefficiently) the first product by the term
	$
	C 
	\varepsilon 
			\sum_{|\iota| \leq N_0} 
			\sum_{I,i=1,\cdots,\mydim}
			t^{q-1} |\partial^{\iota} (e_I^i - \widetilde{e}_I^i)|(t,x)
	$
	on RHS~\eqref{E:SPATIALFRAMEDERIVATIVESCALARFIELDANNOYINGDERIVATIVELOSSTERMPOINTWISEBOUNDS}
	and, with the help of \eqref{Kasnersol} and \eqref{sigma,q}, the second product by the term
	$
	C t^{-1 + \upsigma} \mathbb{D}(t)
	$
	on RHS~\eqref{E:SPATIALFRAMEDERIVATIVESCALARFIELDANNOYINGDERIVATIVELOSSTERMPOINTWISEBOUNDS}.
\end{proof}

\subsubsection{Differential energy identity for the scalar field}
\label{SSS:SCALARFIELDDIFFERENTIALENERGYIDENTITY}
We will derive our top-order energy estimates for the scalar field by integrating the differential identity provided by the following lemma.

\begin{lemma}[Top-order differential energy identity for $e_0 \psi$ and $\lbrace e_I \psi \rbrace_{I=1,\cdots,\mydim}$]
	\label{L:TOPORDERDIFFERENTIALENERGYIDENTITYFORSCALARFIELD}
	Let $\iota$ be a top-order spatial multi-index,
	i.e., $|\iota|=N$.
	Then for solutions to the $\partial^{\iota}$-commuted equations
	\eqref{dt.e0psi.diff}--\eqref{dt.eipsi.diff}
	with $\Pow := \blowupexp + 1$,
	the following differential energy identity holds,
	where the error terms
	$\mathfrak{P}^{(\textnormal{Border};\iota)}$,
	$\mathfrak{P}^{(\textnormal{Junk};\iota)}$,
	$\mathfrak{Q}_I^{(\textnormal{Border};\iota)}$,
	and
	$\mathfrak{Q}_I^{(\textnormal{Junk};\iota)}$
	are defined in \eqref{E:SCALARFIELDTIMEBORDER}--\eqref{E:SCALARFIELDSPACEJUNK}:
	\begin{align} \label{E:TOPORDERDIFFERENTIALENERGYIDENTITYFORSCALARFIELD}
	\begin{split}
		&
		\partial_t
			\left\lbrace
				(t^{\blowupexp + 1}\partial^{\iota} e_0 \psi)^2
			\right\rbrace
			+
			\partial_t
			\left\lbrace
				(t^{\blowupexp + 1} \partial^{\iota} e_I \psi)
				(t^{\blowupexp + 1} \partial^{\iota} e_I \psi)
			\right\rbrace
				\\
	& =  \frac{2 \blowupexp}{t} (t^{\blowupexp + 1} \partial^{\iota} e_0\psi)^2
			+
			\frac{2 (\blowupexp + 1 - \widetilde{q}_I)}{t} 
			(t^{\blowupexp + 1} \partial^{\iota} e_I \psi)
			(t^{\blowupexp + 1} \partial^{\iota} e_I \psi)
				 \\
& \ \
		+ 2 (t^{\blowupexp + 1} \partial^{\iota} e_0 \psi) 
			\left(t^{\blowupexp} \mathfrak{P}^{(\textnormal{Border};\iota)} + t^{\blowupexp + 1} \mathfrak{P}^{(\textnormal{Junk};\iota)} \right)
				\\
& \ \
	+ 2 (t^{\blowupexp + 1} \partial^{\iota} e_I \psi) 
				\left(t^{\blowupexp + 1} \mathfrak{Q}_I^{(\textnormal{Border};\iota)} + t^{\blowupexp + 1} \mathfrak{Q}_I^{(\textnormal{Junk};\iota)} \right)
						\\
	& \ \
		-
		2
		\left\lbrace
			\partial_c
			(n e_C^c)
		\right\rbrace
		(t^{\blowupexp + 1}\partial^{\iota} e_C \psi) 
		(t^{\blowupexp + 1}\partial^{\iota} e_0 \psi)
		+ 
		2
		t^{2 \blowupexp+2} 
		\partial_c
		\left\lbrace
			n e_C^c (\partial^{\iota} e_C \psi) (\partial^{\iota} e_0 \psi)
		\right\rbrace.
	\end{split}
	\end{align}
\end{lemma}

\begin{proof}
	This lemma follows from straightforward calculation, so we only explain the main steps.
	We first note that $\partial^{\iota} \partial_t \widetilde{\psi} = 0$
	and thus we can ignore the formal presence of this term on LHS~\eqref{dt.e0psi.diff}.
	Next, we expand LHS~\eqref{E:TOPORDERDIFFERENTIALENERGYIDENTITYFORSCALARFIELD} using the Leibniz rule.
	When $\partial_t$ falls on $t^{\blowupexp + 1}\partial^{\iota} e_0 \psi$,
	we plug in \eqref{dt.e0psi.diff} with $\Pow := \blowupexp + 1$.
	When $\partial_t$ falls on $t^{\blowupexp + 1} \partial^{\iota} e_I \psi$,
	we plug in \eqref{dt.eipsi.diff} with $\Pow := \blowupexp + 1$.
	Also differentiating by parts,
	we arrive at the desired identity \eqref{E:TOPORDERDIFFERENTIALENERGYIDENTITYFORSCALARFIELD}.
\end{proof}

\subsubsection{Control of the error terms in the top-order commuted scalar field evolution equations}
\label{SSS:TOPORDERCALARFIELDINHOMOGENEOUSTERML2ESTIMATE}
In this section, at the top-order derivative level,
we derive $L^2$ estimates for the error terms in the evolution equations of 
Lemma~\ref{lem:psi.syst}.

\begin{lemma}[$L^2$-control of the error terms in the top-order commuted scalar field equations]
	\label{L:SCALARFIELDL2CONTROLOFERRORTERMS}
	Recall that 
	$\mathbb{H}_{(\upgamma,k)}$,
	$\mathbb{H}_{(\psi)}$,
and $\mathbb{D}(t)$ are norms from Definition~\ref{D:SOLUTIONNORMS},
and assume that the bootstrap assumptions \eqref{Boots} hold.
	Recall that the error terms
	$\mathfrak{P}^{(\textnormal{Border};\iota)}$,
	$\mathfrak{P}^{(\textnormal{Junk};\iota)}$,
	$\mathfrak{Q}_I^{(\textnormal{Border};\iota)}$,
	and
	$\mathfrak{Q}_I^{(\textnormal{Junk};\iota)}$
	are defined in \eqref{E:SCALARFIELDTIMEBORDER}--\eqref{E:SCALARFIELDSPACEJUNK}.
There exists a constant $C_* > 0$ \underline{independent of $N, N_0,$ and $\blowupexp$}
and a constant $C = C_{N,N_0,\blowupexp,\mydim,q,\upsigma} > 0$ 
such that if $N$ is sufficiently large in a manner 
that depends on $N_0, \blowupexp, \mydim, q,$ and $\upsigma$,
and if $\varepsilon$ is sufficiently small (in a manner that depends on $N, N_0, \blowupexp, \mydim, q,$ and $\upsigma$), 
then the following estimates hold for $t \in (T_{\textnormal{Boot}},1]$:
	\begin{subequations}
	\begin{align}
	\begin{split}	
		&
		t^{\blowupexp}
		\sqrt{
		\sum_{|\iota| = N}
		\| \mathfrak{P}^{(\textnormal{Border};\iota)} \|_{L^2(\Sigma_t)}^2}
		+
		t^{\blowupexp + 1}
		\sqrt{
		\sum_{|\iota| = N}
		\sum_{I=1,\cdots,\mydim}
		\| \mathfrak{Q}_I^{(\textnormal{Border};\iota)} \|_{L^2(\Sigma_t)}^2}
			\\
		 & \leq C_* t^{-1} \mathbb{H}_{(\upgamma,k)}(t)
				+
				C_* t^{-1} \mathbb{H}_{(\psi)}(t)
				+
				C t^{-1 + \upsigma} \mathbb{D}(t),
	\end{split}
	\end{align}
	
	\begin{align}
	\begin{split}
		&
		t^{\blowupexp + 1}
		\sqrt{
		\sum_{|\iota| = N}
		\| \mathfrak{P}^{(\textnormal{Junk};\iota)} \|_{L^2(\Sigma_t)}^2}
		+
		t^{\blowupexp + 1}
		\sqrt{
		\sum_{|\iota| = N}
		\sum_{I=1,\cdots,\mydim}
		\| \mathfrak{Q}_I^{(\textnormal{Junk};\iota)} \|_{L^2(\Sigma_t)}^2}
			\\
		 & \leq C t^{-1 + \upsigma} \mathbb{D}(t).
	\end{split}
	\end{align}
	\end{subequations}
\end{lemma}

\begin{proof}
	We apply the same arguments that we used
	in the proof of Lemma~\ref{L:METRICL2CONTROLOFTOPORDERERRORTERMS}
	to the terms on RHSs~\eqref{E:SCALARFIELDTIMEBORDER}--\eqref{E:SCALARFIELDSPACEJUNK}.
\end{proof}

\subsection{Integral inequality for the low order solution norms}
\label{SS:INTEGRALINEQUALITYFORLOWNORMS}
In the next proposition, we combine some of the results derived earlier in Sect.\,\ref{sec:mainest}
to obtain an integral inequality for the low order solution norms. 
In Sect.\,\ref{SS:TOPORDERENRGYINTEGRALINEQUALITIES}, we will derive 
a related integral inequality for the high order solution norms.
Then, in Sect.\,\ref{PROOFOFprop:overall},
we will combine the two integral inequalities 
and carry out the proof of our main a priori estimates.

\begin{proposition}[Integral inequality for the low order solution norms]
\label{prop:low}
Recall that $\mathbb{L}_{(e,\oe,\upgamma,k,\psi)}(t)$ is a low order norm and that
$\mathbb{D}(t)$ is the total norm of the dynamic solution variables
(see Definition~\ref{D:SOLUTIONNORMS}).
Under the assumptions of Proposition~\ref{prop:overall},
including the bootstrap assumptions \eqref{Boots},
there exists a constant $C = C_{N,N_0,\blowupexp,\mydim,q,\upsigma} > 0$ 
such that if $N_0 \geq 1$ and $N$ is sufficiently large in a manner 
that depends on $N_0, \blowupexp, \mydim, q,$ and $\upsigma$,
and if $\varepsilon$ is sufficiently small (in a manner that depends on $N, N_0, \blowupexp, \mydim, q,$ and $\upsigma$), 
then the following estimate holds for $t \in (T_{\textnormal{Boot}},1]$:
\begin{align}\label{low.est}
\mathbb{L}_{(e,\oe,\upgamma,k,\psi)}^2(t)
& \leq 
C
\mathbb{L}_{(e,\oe,\upgamma,k,\psi)}^2(1)
+
C
\int_t^1
	s^{-1 + \upsigma}
	\mathbb{D}^2(s)
\, ds.
\end{align}
\end{proposition}
\begin{proof}
The polarized $U(1)$-symmetric case will require an additional observation, which we provide at the end of the proof.

\medskip

\noindent{\bf The proof except for the polarized $U(1)$-symmetric case.} Recall \eqref{E:RECOVERGAMMAFROMSTRUCTURECOEFFCIENTS} and Remark~\ref{R:BASISOFSTRUCTURE}.
We define the scalar function $Q(t,x) \geq 0$
as follows, where the background Kasner scalars are defined in
Sect.\,\ref{SS:BACKGROUNDVARIABLES} and we suppress the $(t,x)$
arguments on RHS~\eqref{Q}:
\begin{align} \label{Q}
\begin{split}
	Q^2
	=
	Q^2(t,x)
	& := 
			\sum_{|\iota| \leq N_0}
			\mathop{\sum_{I,J,B=1,\cdots,\mydim}}_{I < J}
			\left[t^q \partial^{\iota}(\upgamma_{IJB} + \upgamma_{JBI})\right]^2
			+
			\sum_{|\iota| \leq N_0+1}
			\sum_{I,J=1,\cdots,\mydim}
			\left[t \partial^{\iota}(k_{IJ} - \widetilde{k}_{IJ}) \right]^2 
		\\
	& \ \
			+
			\sum_{|\iota| \leq N_0}
			\sum_{I,i=1,\cdots,\mydim}
			\left[t^q\partial^{\iota}(e_I^i - \widetilde{e}_I^i) \right]^2
			+
			\sum_{|\iota| \leq N_0}
			\sum_{I,i=1,\cdots,\mydim}
			\left[t^q\partial^{\iota}(\oe_i^I - \widetilde{\oe}_i^I) \right]^2
				 \\
	& \ \
			+
			\sum_{|\iota| \leq N_0+1}
			\left[t\partial^{\iota} (e_0 \psi - \partial_t \widetilde{\psi}) \right]^2
			+
			\sum_{|\iota| \leq N_0}
			\sum_{I=1,\cdots,\mydim}
			[t^q\partial^{\iota} e_I \psi]^2.
\end{split}
\end{align}
Throughout the proof, 
we will silently use the estimates
$C^{-1}\| Q \|_{L^{\infty}(\Sigma_t)} \leq \mathbb{L}_{(e,\oe,\upgamma,k,\psi)}(t) \leq C \| Q \|_{L^{\infty}(\Sigma_t)}$
and $\| Q \|_{L^{\infty}(\Sigma_t)} \leq C \mathbb{D}(t)$,
which follow easily from the definitions of the quantities involved
and the identity \eqref{E:RECOVERGAMMAFROMSTRUCTURECOEFFCIENTS}.
In particular, to prove \eqref{low.est},
it suffices to derive the following pointwise bound for $Q^2(t,x)$:
\begin{align} \label{E:SUFFICIENTBOUND.low.est}
	Q^2(t,x)
	& \lesssim
		\mathbb{L}_{(e,\oe,\upgamma,k,\psi)}^2(1)
		+
		\int_t^1
			s^{-1 + \upsigma}
			\mathbb{D}^2(s)
		\, ds.
\end{align}
To prove \eqref{E:SUFFICIENTBOUND.low.est},
we will derive the following pointwise bound for $(t,x) \in (T_{\textnormal{Boot}},1] \times \mathbb{T}^{\mydim}$:
\begin{align} \label{MAINSTEPlow.est}
	\begin{split}
	Q^2(t,x)
	& \leq 
		C \mathbb{L}_{(e,\oe,\upgamma,k,\psi)}^2(1)
		+
		(C \varepsilon - 4 \upsigma)
		\int_t^1
			s^{-1}
			\sum_{|\iota| \leq N_0}
			\mathop{\sum_{I,J,B=1,\cdots,\mydim}}_{I < J}
			\left[s^q \partial^{\iota}(\upgamma_{IJB} + \upgamma_{JBI})(s,x) \right]^2
		\, ds
			\\
	& \ \
		+
		(C \varepsilon - 4 \upsigma)
		\int_t^1
			s^{-1}
			\sum_{|\iota| \leq N_0}
			\sum_{I=1,\cdots,\mydim}
			\left[s^q \partial^{\iota} e_I \psi(s,x) \right]^2
		\, ds
			\\
	& \ \
		+
		(C \varepsilon - 4 \upsigma)
		\int_t^1
			s^{-1}
			\sum_{|\iota| \leq N_0}
			\sum_{I,i=1,\cdots,\mydim}
			\left[s^q \partial^{\iota}(e_I^i - \widetilde{e}_I^i)(s,x) \right]^2
		\, ds
		\\
	& \ \
		+
		(C \varepsilon - 4 \upsigma)
		\int_t^1
			s^{-1}
			\sum_{|\iota| \leq N_0}
			\sum_{I,i=1,\cdots,\mydim}
			\left[s^q \partial^{\iota}(\oe_i^I - \widetilde{\oe}_i^I)(s,x) \right]^2
		\, ds
			\\
	& \ \
		+
		C
		\int_t^1
			s^{-1 + \upsigma}
			\mathbb{D}^2(s)
		\, ds.
\end{split}
\end{align}
Then for $\varepsilon$ sufficiently small, the first four integrals on RHS~\eqref{MAINSTEPlow.est} are
negative, and we can discard them; the desired bound \eqref{E:SUFFICIENTBOUND.low.est} then follows.

It remains for us to prove \eqref{MAINSTEPlow.est}. 
We will show that the following pointwise estimates hold
for $(t,x) \in (T_{\textnormal{Boot}},1] \times \mathbb{T}^{\mydim}$,
where to condense the notation, we omit the arguments
$(t,x)$ on the LHSs and the integrand arguments $(s,x)$
on the RHSs:
\begin{align} \label{SUBSTEP1MAINSTEPlow.est}
\begin{split}	
	\sum_{|\iota| \leq N_0}
	\mathop{\sum_{I,J,B=1,\cdots,\mydim}}_{I < J}
	\left[t^q \partial^{\iota}(\upgamma_{IJB} + \upgamma_{JBI}) \right]^2
	& \leq 
		C \mathbb{L}_{(\upgamma,k)}^2(1)
			\\
	& \ \
		+
		(C \varepsilon - 4 \upsigma)
		\int_t^1
			s^{-1}
			\sum_{|\iota| \leq N_0}
			\mathop{\sum_{I,J,B=1,\cdots,\mydim}}_{I < J}
			\left[s^q \partial^{\iota}(\upgamma_{IJB} + \upgamma_{JBI}) \right]^2
		\, ds
		\\
	& \ \
		+
		C \varepsilon
		\int_t^1
			s^{-1}
			\sum_{|\iota| \leq N_0}
			\sum_{I,i=1,\cdots,\mydim}
			\left[s^q \partial^{\iota} (e_I^i - \widetilde{e}_I^i) \right]^2
		\, ds
		\\
	& \ \
		+
		C
		\int_t^1
			s^{-1 + \upsigma}
			\mathbb{D}^2(s)
		\, ds,
	\end{split}
			\\
	\begin{split}
	\sum_{|\iota| \leq N_0+1}
	\sum_{I,J=1,\cdots,\mydim}
	\left[t \partial^{\iota}(k_{IJ} - \widetilde{k}_{IJ}) \right]^2
	& \leq 
		C \mathbb{L}_{(\upgamma,k)}^2(1)
			\label{SUBSTEP2MAINSTEPlow.est}	 \\
	& \ \
		+
		C
		\int_t^1
			s^{-1 + \upsigma}
			\mathbb{D}^2(s)
		\, ds,
	\end{split}
			\\
	\begin{split}
		\sum_{|\iota| \leq N_0}
		\sum_{I,i=1,\cdots,\mydim}
		\left[t^q\partial^{\iota} (e_I^i - \widetilde{e}_I^i) \right]^2
	& \leq 
			C \mathbb{L}_{(e,\oe)}^2(1)
			\label{SUBSTEP3MAINSTEPlow.est}
				\\
	& \ \
		+
		(C \varepsilon - 4 \upsigma)
		\int_t^1
			s^{-1}
			\sum_{|\iota| \leq N_0}
			\sum_{I,i=1,\cdots,\mydim}
			\left[s^q\partial^{\iota}(e_I^i - \widetilde{e}_I^i) \right]^2
		\, ds
			\\
	& \ \
		+
		C
		\int_t^1
			s^{-1 + \upsigma}
			\mathbb{D}^2(s)
		\, ds,
	\end{split}
				\\
	\begin{split}
	\sum_{|\iota| \leq N_0}
		\sum_{I,i=1,\cdots,\mydim}
		\left[t^q\partial^{\iota} (\oe_i^I - \widetilde{\oe}_i^I) \right]^2
	& \leq 
			C \mathbb{L}_{(e,\oe)}^2(1)
			\label{ANOTHERSUBSTEP3MAINSTEPlow.est}
				\\
	& \ \
		+
		(C \varepsilon - 4 \upsigma)
		\int_t^1
			s^{-1}
			\sum_{|\iota| \leq N_0}
			\sum_{I,i=1,\cdots,\mydim}
			\left[s^q\partial^{\iota}(\oe_i^I - \widetilde{\oe}_i^I) \right]^2
		\, ds
			\\
	& \ \
		+
		C
		\int_t^1
			s^{-1 + \upsigma}
			\mathbb{D}^2(s)
		\, ds,
	\end{split}
				\\
	\begin{split}
	\sum_{|\iota| \leq N_0+1}
	\left[t \partial^{\iota}(e_0 \psi - \partial_t \widetilde{\psi}) \right]^2
	& \leq 
		C \mathbb{L}_{(\psi)}^2(1)
			\label{SUBSTEP4MAINSTEPlow.est}	 
			\\
	& \ \
		+
		C
		\int_t^1
			s^{-1 + \upsigma}
			\mathbb{D}^2(s)
		\, ds,
	\end{split}
			\\
	\begin{split}
	\sum_{|\iota| \leq N_0}
		\sum_{I=1,\cdots,\mydim}
		[t^q\partial^{\iota} e_I \psi]^2
	& \leq 
			C \mathbb{L}_{(\psi)}^2(1)
			\label{SUBSTEP5MAINSTEPlow.est}
				\\
	& \ \
		+ 
		(C \varepsilon - 4 \upsigma)
		\int_t^1
			s^{-1}
			\sum_{|\iota| \leq N_0}
			\sum_{I=1,\cdots,\mydim}
			[s^q \partial^{\iota} e_I \psi]^2
		\, ds
				\\
	& \ \
		+
		C \varepsilon
		\int_t^1
			s^{-1}
			\sum_{|\iota| \leq N_0}
			\sum_{I,i=1,\cdots,\mydim}
			\left[s^q\partial^{\iota}(e_I^i - \widetilde{e}_I^i) \right]^2
		\, ds
				\\
	& \ \
		+
		C
		\int_t^1
			s^{-1 + \upsigma}
			\mathbb{D}^2(s)
		\, ds.
\end{split}
\end{align}
Then adding \eqref{SUBSTEP1MAINSTEPlow.est}--\eqref{SUBSTEP5MAINSTEPlow.est},
we arrive at \eqref{MAINSTEPlow.est}.

To prove \eqref{SUBSTEP1MAINSTEPlow.est}, we first multiply 
equation \eqref{E:SCHEMATICSTRUCTURECOEFFICIENTEVOLUTIONEQUATION}
by $2 [t^q \partial^{\iota}(\upgamma_{IJB}+\upgamma_{JBI})]$
to obtain the evolution equation 
$\partial_t \left\lbrace [t^q \partial^{\iota}(\upgamma_{IJB}+\upgamma_{JBI})]^2 \right\rbrace
= 2 [t^q \partial^{\iota}(\upgamma_{IJB}+\upgamma_{JBI})] 
\times \mbox{RHS~\eqref{E:SCHEMATICSTRUCTURECOEFFICIENTEVOLUTIONEQUATION}}$.
We then integrate this equation in time over $[t,1]$ with respect to $ds$,
apply the fundamental theorem of calculus,
and then sum the resulting identity  
over all $\iota$ with $|\iota| \leq N_0$
and over all
$I,J,B=1,\cdots,\mydim$ with $I < J$.
In the resulting identity, we place the term
$
\sum_{|\iota| \leq N_0}
\underset{I < J}{\underset{I,J,B=1,\cdots,\mydim}{\sum}}
\left[t^q \partial^{\iota}(\upgamma_{IJB} + \upgamma_{JBI})(t,x) \right]^2
$
on the left-hand side (as the only term on LHS~\eqref{SUBSTEP1MAINSTEPlow.est}),
while the resulting initial data term (on $\Sigma_1$)
is $\leq$ the term $C \mathbb{L}_{(\upgamma,k)}^2(1)$ on RHS~\eqref{SUBSTEP1MAINSTEPlow.est}.
Next, noting that the first term
$\left\lbrace
	q
	-
	(\widetilde{q}_{\underline{I}}+\widetilde{q}_{\underline{J}}-\widetilde{q}_{\underline{B}})
\right\rbrace
t^{q-1}
\partial^{\iota}
(\upgamma_{\underline{I}\underline{J}\underline{B}}+\upgamma_{\underline{J}\underline{B}\underline{I}})
$
on RHS~\eqref{E:SCHEMATICSTRUCTURECOEFFICIENTEVOLUTIONEQUATION}
generates the integrals
$
	- 2
	\left\lbrace
		q
		-
		(\widetilde{q}_{\underline{I}}+\widetilde{q}_{\underline{J}}-\widetilde{q}_{\underline{B}})
	\right\rbrace
	\int_t^1
			s^{-1}
			\sum_{|\iota| \leq N_0}
			\underset{I < J}{\underset{I,J,B=1,\cdots,\mydim}{\sum}}
			[s^q \partial^{\iota}(\upgamma_{IJB} + \upgamma_{JBI})]^2
	\, ds
$
(where the overall minus sign in front of these integrals
is correct because $t \in (T_{\textnormal{Boot}},1]$
and $
\sum_{|\iota| \leq N_0}
\underset{I < J}{\underset{I,J,B=1,\cdots,\mydim}{\sum}}
\left[t^q \partial^{\iota}(\upgamma_{IJB} + \upgamma_{JBI})(t,x) \right]^2
$ is on LHS~\eqref{SUBSTEP1MAINSTEPlow.est}),
we can use \eqref{sigma,q} to bound these integrals
by $\leq$ the negative-definite term 
$$
	- 4 \upsigma
		\int_t^1
			s^{-1}
			\sum_{|\iota| \leq N_0}
			\underset{I < J}{\underset{I,J,B=1,\cdots,\mydim}{\sum}}
			[s^q \partial^{\iota}(\upgamma_{IJB} + \upgamma_{JBI})]^2
		\, ds
$$
on RHS~\eqref{SUBSTEP1MAINSTEPlow.est}.
Finally, with the help of the identity \eqref{E:RECOVERGAMMAFROMSTRUCTURECOEFFCIENTS},
the error term estimates 
\eqref{E:STRUCTURECOEFFICIENTSBORDERPOINTWISEBOUNDS}--\eqref{E:STRUCTURECOEFFICIENTSJUNKPOINTWISEBOUNDS},
and Young's inequality,
we see that the terms generated by the remaining terms
on RHS~\eqref{E:SCHEMATICSTRUCTURECOEFFICIENTEVOLUTIONEQUATION}
are $\leq$ the sum of the remaining terms 
on RHS~\eqref{SUBSTEP1MAINSTEPlow.est}
as desired.

The estimate \eqref{SUBSTEP2MAINSTEPlow.est} follows
from a similar argument based on equation \eqref{E:COMMUTEDKEQUATION}
with $\Pow :=1$ and $|\iota| \leq N_0 + 1$
and the error term estimates 
\eqref{E:SECONDFUNDPOINTWISEERRORTERMESTIMATE1}--\eqref{E:SECONDFUNDPIONTWISEERRORTERMESTIMATE2}.

The estimate \eqref{SUBSTEP3MAINSTEPlow.est} follows
from a similar argument based on equation \eqref{E:COMMUTEDFRAMEEVOLUTION}
with $\Pow := q$
and the error term estimates
\eqref{E:FRAMEBORDERPOINTWISEBOUNDS}--\eqref{E:FRAMEJUNKPOINTWISEBOUNDS}.
The estimate \eqref{ANOTHERSUBSTEP3MAINSTEPlow.est}
can be proved via similar arguments based on equation \eqref{E:COMMUTEDCOFRAMEEVOLUTION}
with $\Pow := q$
and the error term estimates
\eqref{E:COFRAMEBORDERPOINTWISEBOUNDS}--\eqref{E:COFRAMEJUNKPOINTWISEBOUNDS}. 

The estimate \eqref{SUBSTEP4MAINSTEPlow.est} follows
from a similar argument based on equation \eqref{dt.e0psi.diff}
with $\Pow := 1$ and $|\iota| \leq N_0 + 1$
and the error term estimates
\eqref{E:TIMEFRAMEDERIVATIVESCALARFIELDDERIVATIVELOSSTERMPOINTWISEBOUNDS}--\eqref{E:TIMEDERIVATIVESCALARFIELDBORDERANDJUNKPOINTWISEBOUNDS}.

Finally, the estimate \eqref{SUBSTEP5MAINSTEPlow.est} follows
from a similar argument based on equation \eqref{dt.eipsi.diff}
with $\Pow := q$
and the error term estimates
\eqref{E:SPATIALFRAMEDERIVATIVESCALARFIELDANNOYINGDERIVATIVELOSSTERMPOINTWISEBOUNDS}--\eqref{E:SPATIALFRAMEDERIVATIVESCALARFIELDJUNKPOINTWISEBOUNDS}.
This completes the proof except in the polarized $U(1)$-symmetric case.

\medskip

\noindent{\bf The proof in the polarized $U(1)$-symmetric case.}
By \eqref{E:U1VANISHINGSTRUCTURECOEFFICIENTS}, 
in polarized $U(1)$-symmetry with $\mydim = 3$,
the structure coefficient $\upgamma_{IJB}+\upgamma_{JBI}$
vanishes unless $I=B \neq J$
(in which case \eqref{antisymmetricgamma} implies $\upgamma_{IJB}+\upgamma_{JBI} = \upgamma_{IJB}$ -- though this identity is not needed
for our results)
or $B=J \neq I$.
The key point is that for the non-zero structure coefficients, when $I=B$,
the factor $\frac{\widetilde{q}_{\underline{I}}+\widetilde{q}_{\underline{J}}-\widetilde{q}_{\underline{B}}}{t}$
on LHS~\eqref{dt.gamma-gammatilde2}
reduces to
$\frac{\widetilde{q}_{\underline{J}}}{t}$,
and similarly, when $B=J$,
the factor $\frac{\widetilde{q}_{\underline{I}}+\widetilde{q}_{\underline{J}}-\widetilde{q}_{\underline{B}}}{t}$
on LHS~\eqref{dt.gamma-gammatilde2}
reduces to
$\frac{\widetilde{q}_{\underline{I}}}{t}$.
Hence, using the definition \eqref{sigma,q} of $q$ 
in the polarized $U(1)$-symmetric case, we can repeat the proof of \eqref{SUBSTEP1MAINSTEPlow.est}
given above in the non-symmetric case 
-- but making the change 
$\frac{(\widetilde{q}_{\underline{I}}+\widetilde{q}_{\underline{J}}-\widetilde{q}_{\underline{B}})}{t}
\rightarrow
\frac{\widetilde{q}_{\underline{J}}}{t}
$
or
$\frac{(\widetilde{q}_{\underline{I}}+\widetilde{q}_{\underline{J}}-\widetilde{q}_{\underline{B}})}{t}
\rightarrow
\frac{\widetilde{q}_{\underline{I}}}{t}
$
in the relevant spots -- to derive the desired estimates.
\end{proof}

\subsection{Integral inequality for the high order solution norms}
\label{SS:TOPORDERENRGYINTEGRALINEQUALITIES}
In the next proposition, we combine some of the results derived earlier in Sect.\,\ref{sec:mainest}
to obtain an integral inequality for the high order solution norms. 

\begin{proposition}[Top-order energy integral inequalities]
\label{P:TOPORDERENERGYINEQUALITY}
Recall that $\mathbb{H}_{(\upgamma,k)}$,
$\mathbb{H}_{(\psi)}$,
$\mathbb{H}_{(e,\oe)}$,
and $\mathbb{D}(t)$ are norms from Definition~\ref{D:SOLUTIONNORMS}.
Under the assumptions of Proposition~\ref{prop:overall},
including the bootstrap assumptions \eqref{Boots},
there exists a constant $C_* > 0$ \underline{independent of $N, N_0,$ and $\blowupexp$}
and a constant $C = C_{N,N_0,\blowupexp,\mydim,q,\upsigma} > 0$
such that if $N_0 \geq 1$ and $N$ is sufficiently large in a manner 
that depends on $N_0, \blowupexp, \mydim, q,$ and $\upsigma$,
and if $\varepsilon$ is sufficiently small (in a manner that depends on $N, N_0, \blowupexp, \mydim, q,$ and $\upsigma$), 
then the following estimates hold for $t \in (T_{\textnormal{Boot}},1]$:
\begin{subequations}
\begin{align}
\label{E:TOPORDERENERGYINEQUALITYFORSPATIALMETRIC}
\begin{split}
\mathbb{H}_{(\upgamma,k)}^2(t)
 & \leq
	C \mathbb{H}_{(\upgamma,k)}^2(1)
		\\
& \ \
+
(C_* - \blowupexp) 
\int_t^1
	s^{-1} \mathbb{H}_{(\upgamma,k)}^2(s)
\, ds
+ 
C_* 
\int_t^1
	s^{-1} \mathbb{H}_{(\psi)}^2(s)
\, ds
+
C 
\int_t^1
	s^{-1+\upsigma} \mathbb{D}^2(s)
\, ds,
\end{split} \\
\begin{split}
\mathbb{H}_{(\psi)}^2(t)
 & \leq
	C \mathbb{H}_{(\psi)}^2(1)
\label{E:TOPORDERENERGYINEQUALITYFORSCALARFIELD} \\
&  \ \
+
C_*
\int_t^1
	s^{-1} \mathbb{H}_{(\upgamma,k)}^2(s)
\, ds
+ 
(C_* - \blowupexp)  
\int_t^1
	s^{-1} \mathbb{H}_{(\psi)}^2(s)
\, ds
+
C 
\int_t^1
	s^{-1+\upsigma} \mathbb{D}^2(s)
\, ds,
	\end{split} \\
\begin{split}
\mathbb{H}_{(e,\oe)}^2(t)
& \leq
C \mathbb{H}_{(e,\oe)}^2(1)
\label{E:TOPORDERENERGYINEQUALITYFORSPATIALFRAME}\\
& \ \
+
C_*
\int_t^1
	s^{-1} \mathbb{H}_{(\upgamma,k)}^2(s)
\, ds
+
(C_* - \blowupexp) 
\int_t^1
	s^{-1} \mathbb{H}_{(e,\oe)}^2(s)
\, ds
+
C 
\int_t^1
	s^{-1+\upsigma} \mathbb{D}^2(s)
\, ds.
\end{split}
\end{align}
\end{subequations}
\end{proposition}
\begin{proof}
	We stress that throughout the proof, $C$ and $C_*$ 
	denote constants that have the properties stated in Sect.\,\ref{SS:NOTATIONANDCONVENTIONS},
	and that these constants can vary from line to line. In particular,
	the ``final constants'' appearing in 
	\eqref{E:TOPORDERENERGYINEQUALITYFORSPATIALMETRIC}--\eqref{E:TOPORDERENERGYINEQUALITYFORSPATIALFRAME}
	do not have to coincide with the constants appearing in the proof.
	
	To prove \eqref{E:TOPORDERENERGYINEQUALITYFORSPATIALMETRIC},
	we first integrate the differential energy identity \eqref{E:TOPORDERDIFFERENTIALENERGYIDENTITYFORGAMMAANDK}
	over $[t,1] \times \mathbb{T}^{\mydim}$ with respect to $ds \, dx$, 
	note that the integrals of the last three (perfect-spatial-derivative) 
	terms on RHS~\eqref{E:TOPORDERDIFFERENTIALENERGYIDENTITYFORGAMMAANDK} vanish,
	sum the resulting identity over all $\iota$ with $|\iota|= N$, 
	use \eqref{PRECISE.n.high.est} to control the top-order derivatives
	of the lapse,
	use the estimates 
	$\| n - 1 \|_{W^{1,\infty}(\Sigma_t)} \lesssim t^{\upsigma}$
	and $\| e \|_{W^{1,\infty}(\Sigma_t)} \lesssim t^{-1 + 2 \upsigma}$
	(which are simple consequences of 
		\eqref{Kasnersol},
		the inequalities in \eqref{sigma,q},
		and the bootstrap assumptions),
	and use the Cauchy--Schwarz inequality for integrals and sums
	and Young's inequality
	to deduce that the following estimate holds for $t \in (T_{Boot)},1]$,
	where $C_* > 0$ and $C > 0$ are as in the statement of the proposition:
	\begin{align} \label{E:FIRSTSTEPTOPORDERENERGYINEQUALITYFORSPATIALMETRIC}
	\begin{split}
			&
			\sum_{|\iota| = N}
			\sum_{I,J=1,\cdots,\mydim}
			t^{2 \blowupexp+2}
			\|  \partial^{\iota} k_{IJ} \|_{L^2(\Sigma_t)}^2
			+
			\frac{1}{2}
			t^{2 \blowupexp+ 2} 
			\sum_{|\iota| = N}
			\sum_{I,J,B=1,\cdots,\mydim}
			\|  \partial^{\iota} \upgamma_{IJB} \|_{L^2(\Sigma_t)}^2
				\\
		& 
		\leq 
		C \mathbb{H}_{(\upgamma,k)}^2(1)
			\\
& \ \		
		+
		(C_* - \blowupexp)
		\int_t^1
			\left\lbrace
				\sum_{|\iota| = N}
				\sum_{I,J=1,\cdots,\mydim}
				s^{2 \blowupexp+ 1} \|  \partial^{\iota} k_{IJ} \|_{L^2(\Sigma_s)}^2
				+
				\sum_{|\iota| = N}
			\sum_{I,J,B=1,\cdots,\mydim}
				s^{2 \blowupexp+ 1} \|  \partial^{\iota} \upgamma_{IJB} \|_{L^2(\Sigma_s)}^2
			\right\rbrace
		\, ds
			\\
	& \ \
		+
		\sum_{|\iota| = N}
		\sum_{I,J=1,\cdots,\mydim}
		\int_t^1
			s^{2\blowupexp + 1} 
				\left\| \mathfrak{K}_{IJ}^{(\textnormal{Border};\iota)} \right\|_{L^2(\Sigma_s)}^2
		\, ds
			 \\
		& \ \
		+
		\sum_{|\iota| = N}
		\sum_{I,J,B=1,\cdots,\mydim}
		\int_t^1
			s^{2 \blowupexp+3} 
			\left\| \mathfrak{G}_{IJB}^{(\textnormal{Border};\iota)}  \right\|_{L^2(\Sigma_s)}^2
		\, ds
		 \\
		& \ \
		+
		\sum_{|\iota| = N}
		\sum_{J=1,\cdots,\mydim}
		\int_t^1
			s^{2 \blowupexp+3} 
			\left\| \mathfrak{M}_J^{(\textnormal{Border};\iota)} \right\|_{L^2(\Sigma_s)}^2
		\, ds
				\\
	& \ \
		+
		C
		\sum_{|\iota| = N}
		\sum_{I,J=1,\cdots,\mydim}
		\int_t^1
				\| s^{\blowupexp + 1} \partial^{\iota} k_{IJ} \|_{L^2(\Sigma_s)}
				\left\| s^{\blowupexp + 1} \mathfrak{K}_{IJ}^{(\textnormal{Junk};\iota)} \right\|_{L^2(\Sigma_s)}
		\, ds
			\\
	& \ \
		+
		C
		\sum_{|\iota| = N}
		\sum_{I,J,B=1,\cdots,\mydim}
		\int_t^1
			\| s^{\blowupexp + 1} \partial^{\iota} \upgamma_{IJB} \|_{L^2(\Sigma_s)}
			\left\| s^{\blowupexp + 1} \mathfrak{G}_{IJB}^{(\textnormal{Junk};\iota)}  \right\|_{L^2(\Sigma_s)}
		\, ds
			 \\
	& \ \
		+
		C
		\sum_{|\iota| = N}
		\sum_{I,J,B,E=1,\cdots,\mydim}
		\int_t^1
			\left\lbrace
				\| s^{\blowupexp + 1} \partial^{\iota} \upgamma_{IJB} \|_{L^2(\Sigma_s)} + s^{\upsigma} \mathbb{D}(s)
			\right\rbrace
			\left\| s^{\blowupexp + 1} \mathfrak{M}_E^{(\textnormal{Junk};\iota)} \right\|_{L^2(\Sigma_s)}
		\, ds 
				\\
	& \ \
		+
		C
		\sum_{|\iota| = N}
			\sum_{I,J,B,E,F=1,\cdots,\mydim}
		\int_t^1
			s^{-1+\upsigma}
			\| s^{\blowupexp + 1} \partial^{\iota} k_{IJ} \|_{L^2(\Sigma_s)}
			\left\lbrace
				\| s^{\blowupexp + 1} \partial^{\iota} \upgamma_{BEF} \|_{L^2(\Sigma_s)} 
				+ 
				s^{\upsigma} \mathbb{D}(s)
			\right\rbrace
		\, ds
			\\
	& \ \
		+
		C 
		\int_t^1
			s^{-1+\upsigma} \mathbb{D}^2(s)
		\, ds.
	\end{split}
	\end{align}
	Using Lemma~\ref{L:METRICL2CONTROLOFTOPORDERERRORTERMS},
	we deduce that the three integrals involving
	the borderline terms 
	$\mathfrak{K}_{IJ}^{(\textnormal{Border};\iota)}$,
	$\mathfrak{G}_{IJB}^{(\textnormal{Border};\iota)}$,
	and $\mathfrak{M}_J^{(\textnormal{Border};\iota)}$
	are bounded by:
	\begin{align*}
	& \leq 
		C_* 
	\int_t^1
		s^{-1}
		\left\lbrace
			\mathbb{H}_{(\upgamma,k)}^2(s)
			+
			\mathbb{H}_{(\psi)}^2(s)
		\right\rbrace
	\, ds
	+
		C \int_t^1
				s^{-1 + \upsigma} \mathbb{D}^2(s)
			\, ds,
	\end{align*}
	and that (in view of Definition~\ref{D:SOLUTIONNORMS})
	the three integrals involving
	the terms $\mathfrak{K}_{IJ}^{(\textnormal{Junk};\iota)}$,
	$\mathfrak{G}_{IJB}^{(\textnormal{Junk};\iota)}$,
	and $\mathfrak{M}_J^{(\textnormal{Junk};\iota)}$
	are bounded by 
		$
		\leq
		C \int_t^1
				s^{-1 + \upsigma} \mathbb{D}^2(s)
			\, ds
	$.
	Moreover, appealing to Definition~\ref{D:SOLUTIONNORMS},
	we see that the integrals
	$$
	C
		\sum_{|\iota| = N}
			\sum_{I,J,B,E,F=1,\cdots,\mydim}
		\int_t^1
			s^{-1+\upsigma}
			\| s^{\blowupexp + 1} \partial^{\iota} k_{IJ} \|_{L^2(\Sigma_s)}
			\left\lbrace
				\| s^{\blowupexp + 1} \partial^{\iota} \upgamma_{BEF} \|_{L^2(\Sigma_s)} 
				+ 
				s^{\upsigma} \mathbb{D}(s)
			\right\rbrace
		\, ds
	$$
	on the next-to-last line of RHS~\eqref{E:FIRSTSTEPTOPORDERENERGYINEQUALITYFORSPATIALMETRIC}
	are bounded by
		$
		\leq
		C \int_t^1
				s^{-1 + \upsigma} \mathbb{D}^2(s)
			\, ds
	$.
	From these estimates,
	we arrive,
	in view of Definition~\ref{D:SOLUTIONNORMS},
	at the desired estimate \eqref{E:TOPORDERENERGYINEQUALITYFORSPATIALMETRIC}.
	
	The inequality \eqref{E:TOPORDERENERGYINEQUALITYFORSCALARFIELD}
	follows from a similar argument based on the
	scalar field differential energy identity \eqref{E:TOPORDERDIFFERENTIALENERGYIDENTITYFORSCALARFIELD}
	and the error term estimates of Lemma~\ref{L:SCALARFIELDL2CONTROLOFERRORTERMS};
	we omit the details.
	
	To prove \eqref{E:TOPORDERENERGYINEQUALITYFORSPATIALFRAME},
	we first set $\Pow := \blowupexp + q$ in equation \eqref{E:COMMUTEDFRAMEEVOLUTION} 
	and multiply it by $2[t^{\blowupexp + q}\partial^{\iota}({e_I^i}-\widetilde{e}_I^i)]$
	to deduce:
	\begin{align} \label{E:COMMUTEDFRAMEEVOLUTION2}
	\begin{split}
	\partial_t
		\left\lbrace
			\left[t^{\blowupexp + q}\partial^{\iota}({e_I^i}-\widetilde{e}_I^i) \right]^2
		\right\rbrace
	& = 
		\frac{2 (\blowupexp + q - \widetilde{q}_{\underline{I}})}{t}
		\left[t^{\blowupexp + q} \partial^{\iota} ({e_{\underline{I}}^i}-\widetilde{e}_{\underline{I}}^i) \right]^2
			\\
& \ \
		+
		2
		(t^{\blowupexp + q} \mathfrak{E}_{\underline{I}}^{\underline{i};(\textnormal{Border};\iota)})
		\left[t^{\blowupexp + q} 
			\partial^{\iota} ({e_{\underline{I}}^{\underline{i}}}-\widetilde{e}_{\underline{I}}^{\underline{i}}) \right]
				\\
& \ \
		+
		2
		(t^{\blowupexp + q} \mathfrak{E}_{\underline{I}}^{\underline{i};(\textnormal{Junk};\iota)})
		\left[t^{\blowupexp + q} 
			\partial^{\iota} ({e_{\underline{I}}^{\underline{i}}}-\widetilde{e}_{\underline{I}}^{\underline{i}}) \right].
\end{split}
\end{align}
We then argue as in the proof of \eqref{E:TOPORDERENERGYINEQUALITYFORSPATIALMETRIC},
but using \eqref{E:COMMUTEDFRAMEEVOLUTION2}
in place of \eqref{E:TOPORDERDIFFERENTIALENERGYIDENTITYFORGAMMAANDK} 
and the error term estimates of	Lemma~\ref{L:FRAMEL2CONTROLOFERRORTERMS}
in place of those of Lemma~\ref{L:METRICL2CONTROLOFTOPORDERERRORTERMS}.
Summing the resulting inequality over $I,i=1,\cdots,\mydim$ and
also noting that $C \varepsilon \leq C_*$, we deduce 
that the following estimate holds for $t \in (T_{Boot)},1]$:
\begin{align} \label{E:PROOFTOPORDERENERGYINEQUALITYFORSPATIALFRAME}
\begin{split}
t^{2(\blowupexp + q)}\|e\|_{\dot{H}^N(\Sigma_t)}^2
& \leq 
\| e \|_{\dot{H}^N(\Sigma_1)}^2
+
C_*
\int_t^1
	s^{-1} \mathbb{H}_{(\upgamma,k)}^2(s)
\, ds
+ 
C_*
\int_t^1
	s^{-1} \mathbb{H}_{(e,\upomega)}^2(s)
\, ds	\\
& \ \
- 
\blowupexp
\int_t^1
	s^{-1} \left\lbrace s^{2(\blowupexp + q)} \| e \|_{\dot{H}^N(\Sigma_s)}^2 \right\rbrace
\, ds
+
C 
\int_t^1
	s^{-1+\upsigma} \mathbb{D}^2(s)
\, ds.
\end{split}
\end{align}
Next, we note that the one-form components $\lbrace \oe_i^I \rbrace_{I,i=1,\cdots,\mydim}$
satisfy the same inequality, that is, \eqref{E:PROOFTOPORDERENERGYINEQUALITYFORSPATIALFRAME}
holds with $\oe$ in place of $e$; to see this, one argues as in the proof of
\eqref{E:PROOFTOPORDERENERGYINEQUALITYFORSPATIALFRAME}, but using the 
evolution equation \eqref{E:COMMUTEDCOFRAMEEVOLUTION} with $\Pow :=\blowupexp + q$
and the last two error term estimates in Lemma~\ref{L:FRAMEL2CONTROLOFERRORTERMS}.
Adding this top-order energy inequality
for the $\lbrace \oe_i^I \rbrace_{I,i=1,\cdots,\mydim}$
to the inequality
\eqref{E:PROOFTOPORDERENERGYINEQUALITYFORSPATIALFRAME},
and considering the definition \eqref{norms.high}
of $\mathbb{H}_{(e,\oe)}(t)$,
we arrive at the desired estimate \eqref{E:TOPORDERENERGYINEQUALITYFORSPATIALFRAME}.
We have therefore proved the proposition.
\end{proof}

\subsection{Proof of Proposition~\ref{prop:overall}} 
\label{PROOFOFprop:overall}
We start by adding the integral inequalities 
\eqref{low.est}
and
\eqref{E:TOPORDERENERGYINEQUALITYFORSPATIALMETRIC}--\eqref{E:TOPORDERENERGYINEQUALITYFORSPATIALFRAME}
to obtain, in view of Definition~\ref{D:SOLUTIONNORMS} and \eqref{E:INITIALNORMOFDYNAMICVARIABLES},
the following inequality for $t \in (T_{\textnormal{Boot}},1]$,
valid under largeness/smallness assumptions on the parameters that we describe just below
(and we again stress that constants labeled ``$C_*$'' -- though we allow them 
to vary from line to line -- are always independent of $N_0, N$ and $\blowupexp$):
\begin{align} \label{MAINSTEPoverall.est}
	\mathbb{D}^2(t)
	& \leq 
				C \mathring{\upepsilon}^2
				+
				(C_* - \blowupexp)
				\int_t^1
					s^{-1} \mathbb{H}_{(e,\oe,\upgamma,k,\psi)}^2(s)
				\, ds
				+
				C
				\int_t^1
					s^{-1 + \upsigma} \mathbb{D}^2(s)
				\, ds.
\end{align}
We now fix $\blowupexp$ to be sufficiently large so that
the factor $C_* - \blowupexp$ on {RHS~\eqref{MAINSTEPoverall.est}} is negative.
For this fixed value of $\blowupexp$ and any fixed integer $N_0 \geq 1$, we choose
$N$ to be sufficiently large (in a manner that depends on $N_0, \blowupexp, \mydim, q,$ and $\upsigma$)
and then $\varepsilon$ to be sufficiently small
(in a manner that depends on $N, N_0, \blowupexp, \mydim, q,$ and $\upsigma$)
such that all of the previous estimates proved in the paper hold true.
For this fixed value of $\blowupexp$, this justifies inequality \eqref{MAINSTEPoverall.est}.
We now note that the negativity
of the factor $C_* - \blowupexp$ ensures that we can discard the first time integral on RHS~\eqref{MAINSTEPoverall.est},
that is, for $t \in (T_{\textnormal{Boot}},1]$, 
we have
$
\mathbb{D}^2(t)
	\leq 
				C \mathring{\upepsilon}^2
				+
				C
				\int_t^1
					s^{-1 + \upsigma} \mathbb{D}^2(s)
				\, ds
$.
From this inequality and Gr\"{o}nwall's lemma, we deduce that
$\mathbb{D}^2(t) \leq C \mathring{\upepsilon}^2$.
From this estimate and \eqref{E:LAPSECONTROLLEDBYDYNAMICVARIABLES}, 
we conclude the desired bound \eqref{overall.est}.
\hfill $\qed$

\subsection{Existence of perturbed solutions on the entire half-slab $(0,1]\times\mathbb{T}^{\mydim}$}
In the next proposition, we use the a priori estimates of
Proposition~\ref{prop:overall} and standard local well-posedness/continuation results 
to show that the perturbed solution exists on $(0,1] \times\mathbb{T}^{\mydim}$.
\begin{proposition}[Existence of perturbed solutions on the entire half-slab ${(0,1]}\times\mathbb{T}^{\mydim}$]
\label{P:EXISTENCEONHALFSLAB}
Let $(\Sigma_1=\mathbb{T}^{\mydim},\mathring{g},\mathring{k},\mathring{\psi},\mathring{\phi})$ 
be geometric initial data (see Sect.\,\ref{SS:CAUCHYPROBLEM}) 
for the Einstein-scalar field equations verifying the constraint equations \eqref{eq:hamconst}--\eqref{eq:momconst} and the CMC condition $\mathrm{tr}k=-1$ (see Remark \ref{rem:CMC}), 
and let $\lbrace \mathring{e}_I \rbrace_{I=1,\cdots,\mydim}$
be the initial orthonormal frame (on $\Sigma_1$)
constructed in Sect.\,\ref{SS:CONSTRUCTIONOFTHEINTIALORTHONORMALFRAME}.
Recall that 
$\mathbb{L}_{(n)}(t),
\mathbb{H}_{(n)}(t)$,
and $\mathbb{D}(t)$ are norms from Definition~\ref{D:SOLUTIONNORMS}
and that $\mathring{\upepsilon} := \mathbb{D}(1)$ (see \eqref{E:INITIALNORMOFDYNAMICVARIABLES}).
Assume that the following conditions are satisfied:
\begin{itemize}
	\item $N_0 \geq 1$.
	\item $\blowupexp \geq 1$ is sufficiently large.
	\item $N$ is sufficiently large in a manner 
		that depends on $N_0, \blowupexp, \mydim, q,$ and $\upsigma$.
	\item The norm $\mathring{\upepsilon}$ defined in \eqref{E:INITIALNORMOFDYNAMICVARIABLES}
		is sufficiently small 
		in a manner that depends on $N, N_0, \blowupexp, \mydim, q,$ and $\upsigma$.
\end{itemize}
Then there exists a constant $C_{N,N_0,\blowupexp,\mydim,q,\upsigma} > 0$
such these data launch a perturbed solution 
$$(n,k_{IJ},\upgamma_{IJB},e_I^i,\oe_i^I,\psi)_{I,J,B,i=1,\cdots,\mydim}$$
to the reduced equations of Proposition~\ref{P:redeq}
that exists classically on $(0,1]\times\mathbb{T}^{\mydim}$ 
and satisfies the following estimate for $t \in (0,1]$:
\begin{align}\label{overall.est.[0,1]}
\mathbb{D}(t) + \mathbb{L}_{(n)}(t) + \mathbb{H}_{(n)}(t)
& \leq C_{N,N_0,\blowupexp,\mydim,q,\upsigma} \mathring{\upepsilon}.
\end{align}

Moreover, if we define $g_{ij}$ and ${\bf{g}}$ in terms of the reduced variables by
$g_{ij} := \oe_i^A \oe_j^A$
and
${\bf g} := - n^2 dt \otimes dt + g_{ab} dx^a \otimes dx^b$
(where $t$ is the CMC time function and 
$\lbrace x^i \rbrace_{i=1,\cdots,\mydim}$ are the transported spatial coordinates),
then the tensorfields $({\bf g},\psi)$
are also classical solutions to the Einstein-scalar field system \eqref{EE}--\eqref{SF}
on $(0,1]\times\mathbb{T}^{\mydim}$.
\end{proposition}
\begin{proof}
We first fix $N_0 \geq 1$, $\blowupexp$ sufficiently large, and $N$ sufficiently large
such that if the bootstrap smallness parameter $\varepsilon$ is sufficiently small, then
all of the estimates proved in the previous subsections hold true
on $(T_{\textnormal{Boot}},1] \times \mathbb{T}^{\mydim}$, 
as long as the bootstrap assumption \eqref{Boots} holds for $t \in (T_{\textnormal{Boot}},1]$.
By standard local well-posedness, if $\mathring{\upepsilon}$ is
sufficiently small and $C$ is sufficiently large, 
then there exists a minimal time 
$T_{\textnormal{Min}} \in [0,1)$, such that the solution 
$(n,k,\upgamma,e,\oe,n,\psi)$ 
to the reduced equations of Proposition~\ref{P:redeq}
exists classically for $(t,x) \in (T_{\textnormal{Min}},1] \times\mathbb{T}^{\mydim}$ and such that the
bootstrap assumptions \eqref{Boots} hold with $T_{\textnormal{Boot}}= T_{\textnormal{Min}}$ and $\varepsilon := C \mathring{\upepsilon}$. 
By enlarging $C$ if necessary,
we can assume that $C \ge 2C_{N,N_0,\blowupexp,\mydim,q,\upsigma}$, where $C_{N,N_0,\blowupexp,\mydim,q,\upsigma}$ 
is the constant on RHS~\eqref{overall.est}. For the reader's convenience, we now comment on
the ``standard local well-posedness'' mentioned above. Specifically, readers
can consult \cite{lAvM2003} for the main ideas behind the proof of local well-posedness
	in a similar but distinct elliptic-hyperbolic gauge for Einstein's equations, 
	or \cite[Theorem~14.1]{RodSp2} for a sketch of a proof 
	of local well-posedness in CMC-transported spatial coordinates;
	local well-posedness for the equations of Proposition~\ref{P:redeq}
	can be proved via similar arguments.
	We emphasize that, as is stated in Proposition~\ref{P:redeq},
	solutions to the reduced equations
	(including the constraints)
	are also solutions to the Einstein-scalar field system \eqref{EE}--\eqref{SF},
	where the spacetime metric can be reconstructed from the reduced variables via the equations
	${\bf g} = - n^2 dt \otimes dt + g_{ab} dx^a \otimes dx^b$
	and
	$g_{ij} = g(\partial_i,\partial_j) = \oe_i^A \oe_j^A$ 
	(see \eqref{E:COORDINATEVECTORFIELDSINTERMSOFFRAMEVECTORFIELDS} and \eqref{E:ORTHONORMALSPATIALFRAMECONDITION}).
	Moreover, in view of the norms defined in Definition~\ref{D:SOLUTIONNORMS},
	it is a standard result (again, see \cite{lAvM2003} for the main ideas)
	that if $\varepsilon$ is sufficiently small, then either \textbf{i)} $T_{\textnormal{Min}}=0$ or 
	\textbf{ii)} $T_{\textnormal{Min}} \in (0,1)$ and
	the bootstrap assumptions 
	are saturated on the time interval $(T_{\textnormal{Min}},1]$, that is, 
\begin{align}
\sup_{t \in (T_{\textnormal{Min}},1]}
\left\lbrace
	\mathbb{D}(t) + \mathbb{L}_{(n)}(t) + \mathbb{H}_{(n)}(t)
\right\rbrace
& = \varepsilon.
\end{align}
The latter possibility is ruled out by inequality \eqref{overall.est} when $\mathring{\upepsilon}$ is small enough. 
Thus, $T_{\textnormal{Min}}= 0$. In particular, the solution exists classically for $(t,x)\in(0,1]\times\mathbb{T}^{\mydim}$, 
and the estimate \eqref{overall.est.[0,1]} holds for $t\in (0,1]$.
\end{proof}

\subsection{Construction of the initial orthonormal spatial frame}
\label{SS:CONSTRUCTIONOFTHEINTIALORTHONORMALFRAME}
Thus far, we have not explained how to construct
the initial orthonormal spatial frame $\lbrace \mathring{e}_I \rbrace_{I=1,\cdots,\mydim}$
on $\Sigma_1$. To achieve this away from symmetry, 
we simply apply the Gram--Schmidt process to the transported spatial coordinate vectorfield frame
$\lbrace \partial_i \rbrace_{i=1,\cdots,\mydim}$.
More precisely, with $\mathring{g}$ denoting the Riemannian metric on $\Sigma_1$,
we set:
\begin{subequations}
\begin{align} \label{E:FIRSTINITIALFRAMEVECTOR}
	\mathring{e}_1 & 
		:= \frac{\partial_1}{\sqrt{\mathring{g}_{11}}}
			= \frac{\partial_1}{\sqrt{\mathring{g}(\partial_1,\partial_1)}},
		\\
	\mathring{E}_{M+1} 
	& := \partial_{M+1}
			 -
			\sum_{L=1,\cdots,M} 
			\underbrace{\mathring{g}_{cd} \updelta_{M+1}^c \mathring{e}_L^d}_{\mathring{g}(\partial_{M+1},\mathring{e_L})}
			\mathring{e_L},
				\label{E:INITIALFRAMEINDUCTION} 
		&
		& M=1,\cdots,\mydim-1,
		\\
	\mathring{e}_{M+1}
	& := \frac{\mathring{E}_{M+1} }{\sqrt{\mathring{g}_{cd} \mathring{E}_{M+1}^c \mathring{E}_{M+1}^d}},
		&
		& M=1,\cdots,\mydim-1.
		\label{E:INITIALFRAMENORMALIZEDINDUCTION}
\end{align}
\end{subequations}
By construction, for $1 \leq I,J \leq \mydim$,
we have the desired identity 
$\mathring{g}(\mathring{e}_I,\mathring{e}_J) = \updelta_{IJ}$,
where $\updelta_{IJ}$ is the Kronecker delta.

In the polarized $U(1)$-symmetric case with $\mydim=3$, 
we proceed in a similar fashion,
but starting with
$
\mathring{e}_3
		:= \frac{\partial_3}{\sqrt{\mathring{g}_{33}}}
			= \frac{\partial_3}{\sqrt{\mathring{g}(\partial_3,\partial_3)}}
$.
Note that for metrics that are initially polarized and $U(1)$-symmetric in the 
sense described in Lemma~\ref{L:PROPOFSYM},
this Gram--Schmidt process leads to an initial frame that respects the $\partial_3$ symmetry:
$\mathcal{L}_{\partial_3} \mathring{e}_I = 0$ for $I=1,2,3$.
Hence, Lemma~\ref{lem:U1} ensures that throughout the classical evolution,
we have
$e_3 = \frac{\partial_3}{\sqrt{g_{33}}}$
and
$\mathcal{L}_{\partial_3} e_I = 0$ for $I=1,2,3$.

\subsection{The near-Kasner smallness condition on the geometric initial data}
\label{SS:SMALLNESSCONDITIONONDATA}
Before proving our main theorems, we will first define a norm of 
the ``geometric'' initial data 
$(\Sigma_1=\mathbb{T}^{\mydim},\mathring{g},\mathring{k},\mathring{\psi},\mathring{\phi})$
minus the background Kasner data.
The smallness of this difference will be sufficient for the validity of our main results. 
We highlight that the lapse $n$ is not among the geometric initial data;
it is a gauge-dependent quantity that can be controlled in terms of the geometric data.
Then, in Lemma~\ref{L:CONSEQUENCEOFNEARKASNERSMALLNESSCONDITIONGEOMETRICDATA},
we show that if the geometric data are sufficiently near-Kasner, 
then the full data norm 
$
\mathbb{D}(1)
+
\mathbb{L}_{(n)}(1)
+
\mathbb{H}_{(n)}(1)
$
is small, i.e., we have smallness not only for the geometric data,
but also for all of the gauge-dependent quantities such as $n-1$,
$e_I^i - \widetilde{e}_I^i$,
$k_{IJ} - \widetilde{k}_{IJ}$,
etc.

To proceed, we let $(\Sigma_1=\mathbb{T}^{\mydim},\mathring{g},\mathring{k},\mathring{\psi},\mathring{\phi})$
be a geometric initial data set, as described in Sect.\,\ref{SS:CAUCHYPROBLEM}. 
Recall that relative to standard coordinates on $\mathbb{T}^{\mydim}$,
the Kasner background data (on $\Sigma_1$) have the following components:
$\mathring{g}_{ij}^{\textnormal{KAS}} := \updelta_{ij}$, 
$\mathring{k}_{ij}^{\textnormal{KAS}} := - \widetilde{q}_{\underline{i}} \updelta_{\underline{i}j}$,
$\mathring{\psi}^{\textnormal{KAS}} := 0$, 
$\mathring{\phi}^{\textnormal{KAS}} := \widetilde{B}$,
where $\updelta_{ij}$ is the Kronecker delta, 
we do not sum over repeated underlined indices,
and by assumption, the Kasner exponent constraints \eqref{sumpi} are satisfied.
For $N \in\mathbb{N}$, we define the following norm which,
relative to the standard coordinates on $\mathbb{T}^{\mydim}$, measures the perturbation of the geometric initial data
set away from the Kasner background:
\begin{align} \label{E:INITIALCONTROLLINGNORM}
\begin{split}	
	\mathring{\upalpha}
	= \mathring{\upalpha}(N) 
	 :=&
	\sum_{i,j=1,\cdots,\mydim} \| \mathring{g}_{ij} - \updelta_{ij}  \|_{H^{N+1}(\mathbb{T}^{\mydim})}
	+
	\sum_{i,j=1,\cdots,\mydim} 
	\| \mathring{k}_{ij} + \widetilde{q}_{\underline{i}} \updelta_{\underline{i}j} \|_{H^N(\mathbb{T}^{\mydim})}\\
& \ \
	+
	 \| \mathring{\psi} \|_{H^{N+1}(\mathbb{T}^{\mydim})}
	+
	\| \mathring{\phi} - \widetilde{B} \|_{H^N(\mathbb{T}^{\mydim})}.
	\end{split}
\end{align}

In the next lemma, we show that the norms appearing in the bootstrap assumptions \eqref{Boots}
are initially small, provided $\mathring{\upalpha}$ is sufficiently small.

\begin{lemma}[A near-Kasner smallness condition on the geometric initial data implies smallness
of all reduced solution variables along $\Sigma_1$]
\label{L:CONSEQUENCEOFNEARKASNERSMALLNESSCONDITIONGEOMETRICDATA}
Recall that $\mathbb{D}(t)$ is the total norm of the dynamic variables
and that $\mathbb{L}_{(n)}$ and $\mathbb{H}_{(n)}$ are the norms of the lapse
(see Definition~\ref{D:SOLUTIONNORMS}).
For $N \in \mathbb{N}$, we define:
\begin{align}\label{E:INITIALNORMOALLREDUCEDSOLUTIONVARIABLES}
\mathring{\upeta}
=
\mathring{\upeta}(N)
& 
:= 
\mathbb{D}(1) 
+ 
\mathbb{L}_{(n)}(1)
+
\mathbb{H}_{(n)}(1).
\end{align}
Let $\mathring{\upalpha}$ be the norm of the perturbation of the geometric initial data away from the Kasner data,
as defined in \eqref{E:INITIALCONTROLLINGNORM}. Let
$\lbrace \mathring{e}_I \rbrace_{I=1,\cdots,\mydim}$
be the initial orthonormal frame
constructed in Sect.\,\ref{SS:CONSTRUCTIONOFTHEINTIALORTHONORMALFRAME},
and let the initial lapse $\mathring{n} := n|_{\Sigma_1}$ be the solution
to the elliptic PDE~\eqref{n.eq} (with $t = 1$).
Fix $N_0 \geq 1$.
There exists a constant $C=C_{N,N_0,\mydim} > 0$
such that if $N$ is sufficiently large in a manner that depends on $N_0$ and $\mydim$,
and if $\mathring{\upalpha}$ is sufficiently small,
then: 
\begin{align} \label{E:FULLINITIALNORMCONTROLLEDBYGEOMETRICDATANORM}
	\mathring{\upeta}
	& \leq C \mathring{\upalpha}.
\end{align}
In particular, since the initial norm $\mathring{\upepsilon}$ of the dynamic variables
defined in \eqref{E:INITIALNORMOFDYNAMICVARIABLES} satisfies $\mathring{\upepsilon} \leq \mathring{\upeta}$, 
it follows from \eqref{E:FULLINITIALNORMCONTROLLEDBYGEOMETRICDATANORM} that:
\begin{align} \label{E:DYNAMICINITIALNORMCONTROLLEDBYGEOMETRICDATANORM}
	\mathring{\upepsilon}
	& \leq 
	C 
	\mathring{\upalpha}.
\end{align}

\end{lemma}

\begin{proof}[Sketch of the proof]
	This is a standard result, so we will only sketch the proof.
	Throughout, we will assume that
	$\mathring{\upalpha}$ is sufficiently small.
	From \eqref{E:INITIALCONTROLLINGNORM}, we see that 
	the $\mydim \times \mydim$ matrix $\mathring{g}_{ij}$ 
	is equal to the identity matrix up to an error matrix whose components are bounded
	in the norm $\| \cdot \|_{H^{N+1}(\mathbb{T}^{\mydim})}$
	by $\lesssim \mathring{\upalpha}$.
	From this fact,
	the Gram--Schmidt process described in Sect.\,\ref{SS:CONSTRUCTIONOFTHEINTIALORTHONORMALFRAME},
	and the standard Sobolev calculus (i.e., estimates of the type appearing in Lemma~\ref{lem:basic.ineq}),
	it follows that for $1 \leq I,i \leq \mydim$,
	we have
	$\| \mathring{e}_I^i - \updelta_I^i \|_{H^{N+1}(\mathbb{T}^{\mydim})} \lesssim \mathring{\upalpha}$,
	where $\updelta_I^i$ denotes the Kronecker delta.
	To complete the proof of \eqref{E:FULLINITIALNORMCONTROLLEDBYGEOMETRICDATANORM},
	we must show that when $t=1$,
	the remaining norms in Definition~\ref{D:SOLUTIONNORMS} are all $\lesssim \mathring{\upalpha}$.
	This can be achieved by working relative to the
	standard spatial coordinates $\lbrace x^i \rbrace_{i=1,\cdots,\mydim}$ on $\Sigma_1$
	and using the definition of $\mathring{\upalpha}$,
	the definitions of the quantities appearing in the norms of Definition~\ref{D:SOLUTIONNORMS},
	the standard Sobolev calculus, and elliptic estimates for the lapse, similar to the ones we used to prove
	Proposition~\ref{prop:n}. As one example, we will show that
	$\| \upgamma_{IJB} \|_{H^N(\Sigma_1)} \lesssim \mathring{\upalpha}$.
	First, we note that
	$
	\upgamma_{IJB}|_{\Sigma_1} 
	= 
	\mathring{g}_{ab}\mathring{e}_I^c (\partial_c \mathring{e}_J^a) \mathring{e}_B^b
	+
	\mathring{e}_I^i \mathring{e}_J^j \mathring{e}_B^b \mathring{\Gamma}_{ibj}
	$,
	where $\mathring{\Gamma}_{ibj} 
	= \frac{1}{2} 
	\left\lbrace 
	\partial_i \mathring{g}_{bj} 
	+   
	\partial_j \mathring{g}_{ib}
	-
	\partial_b \mathring{g}_{ij}
	\right\rbrace$
	are the (lowered) Christoffel symbols of $\mathring{g}$ relative to the spatial coordinates 
	$\lbrace x^i \rbrace_{i=1,\cdots,\mydim}$ on $\Sigma_1$.
	Thus, from this expression for $\upgamma_{IJB}|_{\Sigma_1}$,
	definition~\eqref{E:INITIALCONTROLLINGNORM},
	the estimates $\| \mathring{e}_I^i - \updelta_I^i \|_{H^{N+1}(\mathbb{T}^{\mydim})} \lesssim \mathring{\upalpha}$
	and $\| \mathring{g}_{ij} - \updelta_{ij} \|_{H^{N+1}(\mathbb{T}^{\mydim})} \lesssim \mathring{\upalpha}$,
	and the standard Sobolev calculus,
	we conclude the desired bound $\| \upgamma_{IJB} \|_{H^N(\Sigma_1)} \leq C \mathring{\upalpha}$.
	This concludes our proof sketch.
\end{proof}

\section{The two stable blowup theorems}\label{sec:sol}
In this section, we prove our two main theorems.
The derivation of the a priori estimate \eqref{overall.est.[0,1]} 
was the difficult part of the proof, and based on this estimate,
the proofs of the main results will unfold in a natural fashion.

\subsection{Statement of the theorems}
\label{SS:STATEMENTOFMAINTHEOREMS}
In this section, we state the two theorems. 
The proofs are located in Sect.\,\ref{SS:PROOFOFMAINTHEOREMS}.
Before proving the theorems, 
we will first establish, 
in separate sections, some of their key aspects.
We start by stating our main theorem for solutions without symmetry.

\begin{theorem}[Precise version of stable Big Bang formation without symmetry assumptions]\label{thm:precise}
Let $\widetilde{\bf g} := -dt \otimes dt + \sum_{I=1}^{\mydim} t^{2 \widetilde{q}_I} dx^I \otimes dx^I$, 
$\widetilde{\psi} : =\widetilde{B}\log t$ be an explicit generalized Kasner solution on 
$(0,\infty)\times\mathbb{T}^{\mydim}$,
where the constants $\lbrace \widetilde{q}_I \rbrace_{I=1,\cdots,\mydim}$ and $\widetilde{B}$
satisfy the algebraic constraints
$\sum_{I=1}^{\mydim}\widetilde{q}_I=1$ and $\sum_{I=1}^{\mydim}\widetilde{q}_I^2=1-\widetilde{B}^2$
as well as the following \textbf{stability condition}:
\begin{align}\label{Kasner.cond2}
\mathop{\max_{I,J,B=1,\cdots,\mydim}}_{I < J}
\{\widetilde{q}_I+\widetilde{q}_J-\widetilde{q}_B\}<1.
\end{align}
Note that in the vacuum case, we have $\widetilde{B}=0$. 
As we discussed in Sect.\,\ref{subsec:models}, in the vacuum case,
the set of Kasner solutions satisfying the algebraic constraints
and the condition \eqref{Kasner.cond2} is non-empty
when $\mydim \geq 10$, while in the presence of a scalar field,
the set of Kasner solutions satisfying the algebraic constraints
and the condition \eqref{Kasner.cond2} is non-empty
when $\mydim \geq 3$. 
Let $\widetilde{k}_{IJ}= - \widetilde{q}_{\underline{I}} \updelta_{\underline{I}J} t^{-1}$ 
be the components of the second fundamental form of $\Sigma_t$ relative to the Kasner metric,
with respect to the background orthonormal frame vectors
$\widetilde{e}_I = t^{-\widetilde{q}_{\underline{I}}} \partial_{\underline{I}}$,
where we recall that we do not sum repeated underlined indices.
Let $(\Sigma_1=\mathbb{T}^{\mydim},\mathring{g},\mathring{k},\mathring{\psi},\mathring{\phi})$ 
be geometric initial data (see Sect.\,\ref{SS:CAUCHYPROBLEM}) 
for the Einstein-scalar field equations verifying the constraint equations \eqref{eq:hamconst}--\eqref{eq:momconst} and the CMC condition $\mathrm{tr}k=-1$ (see Remark \ref{rem:CMC}), 
and let $\lbrace \mathring{e}_I \rbrace_{I=1,\cdots,\mydim}$
be the initial orthonormal frame (on $\Sigma_1$)
constructed in Sect.\,\ref{SS:CONSTRUCTIONOFTHEINTIALORTHONORMALFRAME}.
Note that $\mathring{\psi} = \mathring{\phi} = 0$
corresponds to the Einstein-vacuum equations.
Let:
\begin{align} \label{E:AGAININITIALCONTROLLINGNORM}
	\begin{split}
	\mathring{\upalpha}
	 :=&
	\sum_{i,j=1,\cdots,\mydim} \| \mathring{g}_{ij} - \updelta_{ij}  \|_{H^{N+1}(\mathbb{T}^{\mydim})}
	+
	\sum_{i,j=1,\cdots,\mydim} 
	\| \mathring{k}_{ij} + \widetilde{q}_{\underline{i}} \updelta_{\underline{i}j} \|_{H^N(\mathbb{T}^{\mydim})}
		\\
& \ \ 
	+
	 \| \mathring{\psi} \|_{H^{N+1}(\mathbb{T}^{\mydim})}
	+
	\| \mathring{\phi} - \widetilde{B} \|_{H^N(\mathbb{T}^{\mydim})}
\end{split}
\end{align}
denote the norm of the perturbation of the geometric initial data away from the Kasner data,
as defined in \eqref{E:INITIALCONTROLLINGNORM}.
Assume that:
\begin{itemize}
\item $N_0 \geq 1$ is a fixed positive integer (we are free to choose $N_0$).
\item $\blowupexp$ is sufficiently large in a manner that depends on $\mydim$ and the parameters 
$q$ and $\upsigma$ fixed in \eqref{sigma,q}.
\item $N$ is sufficiently large in a manner that depends on $N_0, \blowupexp, \mydim, q,$ and $\upsigma$. 
\item $\mathring{\upalpha}$ is sufficiently small in a manner that depends on
$N, N_0, \blowupexp, \mydim, q,$ and $\upsigma$.
\end{itemize}
Then the following conclusions hold.

\bigskip

\noindent \underline{{\bf Existence and norm estimates on $(0,1]\times\mathbb{T}^{\mydim}$.}} 
 The initial data launch a unique
solution 
$$(n,k_{IJ},\upgamma_{IJB},e_I^i,\oe_i^I,\psi)_{I,J,B,i=1,\cdots,\mydim}$$
to the reduced Einstein-scalar field equations of Proposition~\ref{P:redeq}
existing on the slab $(t,x)\in(0,1]\times\mathbb{T}^{\mydim}$.
Moreover, if we define $g_{ij}$ and ${\bf{g}}$ in terms of the reduced variables by
$g_{ij} := \oe_i^A \oe_j^A$
and
${\bf g} := - n^2 dt \otimes dt + g_{ab} dx^a \otimes dx^b$
(where $t$ is the CMC time function and 
$\lbrace x^i \rbrace_{i=1,\cdots,\mydim}$ are the transported spatial coordinates),
then the tensorfields $({\bf g},\psi)$
are also classical solutions to the Einstein-scalar field system \eqref{EE}--\eqref{SF}
on $(0,1]\times\mathbb{T}^{\mydim}$.
In addition, there exists a constant $C = C_{N,N_0,\blowupexp,\mydim,q,\upsigma}$
such that the following estimates hold for $t \in (0,1]$:
\begin{subequations}
\begin{align}\label{overall.low.est}
\begin{split}
&
\sum_{I,i=1}^{\mydim} t^q \|e_I^i - \widetilde{e}_I^i\|_{W^{N_0,\infty}(\Sigma_t)}
+
\sum_{I,i=1}^{\mydim} t^q \|\upomega_i^I - \widetilde{\upomega}_i^I\|_{W^{N_0,\infty}(\Sigma_t)}
	\\
& \ \ 
+
\sum_{I,J,B=1}^{\mydim} t^q \| \upgamma_{IJB} \|_{W^{N_0,\infty}(\Sigma_t)}	
+
\sum_{I,J=1}^{\mydim} t \|k_{IJ}-\widetilde{k}_{IJ} \|_{W^{N_0+1,\infty}(\Sigma_t)}
	\\
& \ \
+
\sum_{I=1}^{\mydim} t^q  \|e_I \psi \|_{W^{N_0,\infty}(\Sigma_t)}
+
\|t \partial_t \psi-\widetilde{B} \|_{W^{N_0+1,\infty}(\Sigma_t)}
	\\
& \ \
+
t^{-\upsigma} \| n-1 \|_{W^{N_0+1,\infty}(\Sigma_t)}
+
\sum_I^{\mydim} t^{q-\upsigma} \| e_I n \|_{W^{N_0,\infty}(\Sigma_t)}
	\\
& \leq C \mathring{\upalpha},
\end{split}
	\\
\begin{split} \label{overall.HIGH.est.mainthm} 
&
\sum_{I,i=1}^{\mydim} t^{\blowupexp + q} \|e_I^i - \widetilde{e}_I^i\|_{\dot{H}^N(\Sigma_t)}
+
\sum_{I,i=1}^{\mydim} t^{\blowupexp + q} \|\upomega_i^I - \widetilde{\upomega}_i^I\|_{\dot{H}^N(\Sigma_t)}
	\\
& \ \ 
+
\sum_{I,J,B=1}^{\mydim} t^{\blowupexp + 1} \| \upgamma_{IJB} \|_{\dot{H}^N(\Sigma_t)}
+
\sum_{I,J=1}^{\mydim} t^{\blowupexp + 1} \|k_{IJ} \|_{\dot{H}^N(\Sigma_t)}
	\\
& \ \
+
\sum_{I=1}^{\mydim} t^{\blowupexp + 1} \|e_I \psi \|_{\dot{H}^N(\Sigma_t)}
+
t^{\blowupexp + 1} \| \partial_t \psi \|_{\dot{H}^N(\Sigma_t)}
	\\
& \ \
+
t^{\blowupexp} \| n \|_{\dot{H}^N(\Sigma_t)}
+ 
t^{\blowupexp + 1} \|\vec{e} n \|_{\dot{H}^N(\Sigma_t)}
	\\
& \leq C \mathring{\upalpha}.
\end{split} 
\end{align}
\end{subequations}

\bigskip

\noindent \underline{{\bf Kasner-like behavior.}} 
The scalar component functions $\left \lbrace t k_{IJ}(t,x) \right\rbrace_{I,J = 1,\cdots,\mydim}$ of the
normalized second fundamental form of $\Sigma_t$ with respect to the $g$-orthonormal 
frame $\lbrace e_I(t,x) \rbrace_{I=1,\cdots,\mydim}$, 
as well as the normalized time derivative $t\partial_t \psi(t,x)$ of the scalar field, 
have continuous $W^{N_0+1,\infty}(\mathbb{T}^{\mydim})$ limits, 
denoted respectively by\footnote{Here, we are slightly abusing notation by, for example, using the 
expression $\upkappa_{IJ}^{(\infty)}(x)$ to denote the function $x \rightarrow \upkappa_{IJ}^{(\infty)}(x)$.} 
$\left\lbrace \upkappa_{IJ}^{(\infty)}(x) \right\rbrace_{I,J = 1,\cdots,\mydim}$ 
and $B^{(\infty)}(x)$, as $t \downarrow 0$. 
Moreover, the following estimates hold for $t \in (0,1]$:
\begin{subequations}
\begin{align} 
		\sum_{I,J=1,\cdots,\mydim}
		\| t k_{IJ}(t,\cdot) 
			-
			\upkappa_{IJ}^{(\infty)}
		\|_{W^{N_0+1,\infty}(\mathbb{T}^{\mydim})}
	& \leq C \mathring{\upalpha} t^{\upsigma},
	&
	\| t \partial_t \psi(t,\cdot) - B^{(\infty)} \|_{W^{N_0+1,\infty}(\mathbb{T}^{\mydim})}
	& \leq C \mathring{\upalpha} t^{\upsigma},
		\label{E:MAINTHMCONVERGENCERATESOFLIMITSATSINGULARITY} \\
	\sum_{I,J=1,\cdots,\mydim}
	\| \upkappa_{IJ}^{(\infty)}
		+ \widetilde{q}_{\underline{I}} \updelta_{\underline{I}J} \|_{W^{N_0+1,\infty}(\mathbb{T}^{\mydim})}
	& \leq C \mathring{\upalpha},
	&
	\| B^{(\infty)} - \widetilde{B} \|_{W^{N_0+1,\infty}(\mathbb{T}^{\mydim})}
	& \leq C \mathring{\upalpha}.
		\label{E:MAINTHMLIMITINGSECONDFUNDAMENTALFORMCLOSETODATA}
\end{align}
\end{subequations}

In addition, for each $x \in \mathbb{T}^{\mydim}$,
the symmetric $\mydim \times \mydim$ matrix
$(-\upkappa^{(\infty)}_{IJ}(x))_{I,J=1,\cdots,\mydim}$ 
has $\mydim$ (possibly repeated) eigenvalues 
$q_I^{(\infty)}(x)$ 
-- which are the ``final'' Kasner exponents of the perturbed spacetime --
that can be ordered such that:
\begin{align}\label{qi:C01}
q_1^{(\infty)},\cdots,q_{\mydim}^{(\infty)} \in C^{0,1}(\mathbb{T}^{\mydim}),
\end{align}
where $C^{0,1}(\mathbb{T}^{\mydim})$ is the space of Lipschitz-continuous functions on $\mathbb{T}^{\mydim}$.
Moreover, the following estimate holds, where
$\| f \|_{C^{0,1}(\mathbb{T}^{\mydim})} := \| f \|_{L^{\infty}(\mathbb{T}^{\mydim})}
+ \sup_{x, y \in \mathbb{T}^{\mydim}, \, x \neq y} \frac{|f(x) - f(y)|}{d(x,y)}$,
and $d(x,y)$ is the Euclidean distance between $x$ and $y$ in $\mathbb{T}^{\mydim}$:
\begin{align}\label{E:ONEOVERDHOLDEREXPONENTFORFINALKASNEREXPONENTS}
\|q_I^{(\infty)} -\widetilde{q}_I\|_{C^{0,1}(\mathbb{T}^{\mydim})}
\leq C \mathring{\upalpha}.
\end{align}
Moreover, the $\left\lbrace q_I^{(\infty)}(x) \right\rbrace_{I=1,\cdots,\mydim}$ and $B^{(\infty)}(x)$ 
satisfy the following pointwise algebraic relations:
\begin{align}
\sum_{I=1}^{\mydim}q_I^{(\infty)}(x)& = 1,
& \sum_{I=1}^{\mydim}\left[q_I^{(\infty)}(x) \right]^2 & =1-\left[B^{(\infty)}(x) \right]^2.
\label{ALGEBRAqi,B}
\end{align}

\bigskip

\noindent \underline{{\bf Curvature-blowup.}}
The Kretschmann scalar of ${\bf g}$, 
namely ${\bf Riem}^{\alpha\mu\beta\nu}{\bf Riem}_{\alpha\mu\beta\nu}$,
blows up as $t \downarrow 0$, as 
is evident from the following pointwise estimate, valid 
for $(t,x) \in (0,1] \times \mathbb{T}^{\mydim}$:
\begin{align}\label{Krets.fin}
\begin{split}
{\bf Riem}^{\alpha\mu\beta\nu}{\bf Riem}_{\alpha\mu\beta\nu}(t,x)
& = 4 t^{-4} \left\lbrace
	\sum_{I=1}^{\mydim}\left[(q_I^{(\infty)}(x))^2-q_I^{(\infty)}(x) \right]^2+\sum_{1\leq I<J \leq \mydim}(q_I^{(\infty)}(x))^2 
	(q_J^{(\infty)}(x))^2 
\right\rbrace
		\\
& \ \
	+
	\mathcal{O}(\mathring{\upalpha} t^{-4 + \upsigma})
	 \\
& = 4t^{-4}
\left\lbrace
	\sum_{I=1}^{\mydim} \left[\widetilde{q}_I^2-\widetilde{q}_I \right]^2
	+\sum_{1\leq I<J \leq \mydim}\widetilde{q}_I^2\widetilde{q}_J^2 \right\rbrace
+
\mathcal{O}(\mathring{\upalpha} t^{-4}).
\end{split}
\end{align}

\bigskip

\noindent \underline{{\bf Inextendibility.}}
The spacetime is past-inextendible as a $C^2$ Lorentzian manifold. 
%Also, relative to the CMC-transported framework of choice (introduced in Sect.\,\ref{subsec:redEE}), the reduced %variables $e,\oe,k,\upgamma,n,\psi$ are past-inextendible as a weak solution to the reduced system of equations %\eqref{dt.k}--\eqref{boxpsi}.
\end{theorem}
\begin{remark}[No regular limit is claimed for the orthonormal frame vectorfields]
	\label{R:NOLIMITFORFRAME}
	Despite the convergence of the
	normalized component functions $\left \lbrace t k_{IJ}(t,x) \right\rbrace_{I,J = 1,\cdots,\mydim}$,
	our proof does not yield (or require!) 
	that the component functions $\lbrace e_I^i(t,x) \rbrace_{I,i=1,\cdots,\mydim}$
	of the frame vectorfields with respect to the transported spatial coordinates
	can be rescaled by powers of $t$
	so as to have non-trivial, regular limits as $t \downarrow 0$.
\end{remark}	
\begin{remark}[Sharper asymptotics]
	\label{R:SHARPERASYMPTOTICS}
		Although Theorems~\ref{thm:precise} and \ref{thm:precise.U1} yield the most interesting and salient
		features of the stable blowup, by using the estimates provided by the theorems,
		one could try to derive sharper asymptotics for the solution 
		by treating the evolution equations
		as ODEs (with derivative-losing source terms), perhaps also employing
		a different gauge for the already constructed singular solution.
		In some symmetric regimes, gauges are known in which one can prove stability and sharp asymptotics
		at the level of the metric components.
		For example, in polarized $U(1)$-symmetry \cite{AF},\footnote{The work \cite{AF} is concerned with 
		the near-Schwarzschild black hole interior problem, where the symmetry class is called ``polarized axi-symmetry.''
		This regime has some analytical commonalities with the polarized $U(1)$-symmetric solutions that we  
		treat in Theorem~\ref{thm:precise.U1}; see the end of Sect.\,\ref{SSS:BigBangSymm}.}
		the authors proved Big Bang formation and
		derived sharp asymptotics for various solution variables by using 
		frames that are well-adapted to the different Kasner directions.
		See also the recent works \cite{ABIO,ABIO2} 
		on the Einstein-vacuum equations in three spatial dimensions,
		in which the authors used an areal time foliation
		to prove stability and sharper asymptotics at the level of the metric components 
		for a subset of the Kasner solutions 
		near their Big Bang singularities
		under polarized $\mathbb{T}^2$-symmetric perturbations of the initial data.
\end{remark}
\begin{remark}[In general, no additional regularity is claimed for the final Kasner exponents]
	\label{R:NOREGqI}
Although $\upkappa_{IJ}^{(\infty)} \in W^{N_0+1}(\mathbb{T}^{\mydim})$ (where we assume $N_0 \geq 1$), 
in general, the function space $C^{0,1}(\mathbb{T}^{\mydim})$ in \eqref{qi:C01}
and the norm $\| \cdot \|_{C^{0,1}(\mathbb{T}^{\mydim})}$ 
on LHS~\eqref{E:ONEOVERDHOLDEREXPONENTFORFINALKASNEREXPONENTS} are optimal 
(e.g., they cannot be improved to $C^m(\mathbb{T}^{\mydim})$ for any integer $m \geq 2$),
due to the fact that the background solution is allowed to have repeated Kasner exponents.
However, by analyzing the dependence of the characteristic polynomial of
the matrix $\left(t k_{IJ}(t,x) \right)_{I,J = 1,\cdots,\mydim}$ on the entries $t k_{IJ}$,
we could show that if $\widetilde{q}_I \neq \widetilde{q}_J$ 
for $1 \leq I < J \leq \mydim$, then the space $C^{0,1}(\mathbb{T}^{\mydim})$ in \eqref{qi:C01}
could be replaced with $C^{N_0+1}(\mathbb{T}^{\mydim})$
and
the norm $\| \cdot \|_{C^{0,1}(\mathbb{T}^{\mydim})}$ 
on LHS~\eqref{E:ONEOVERDHOLDEREXPONENTFORFINALKASNEREXPONENTS}
could be replaced with $\| \cdot \|_{W^{N_0+1,\infty}(\mathbb{T}^{\mydim})}$. 
This would require an additional smallness assumption on
$\mathring{\upalpha}$: $\mathring{\upalpha} \ll \underset{1 \leq I < J \leq \mydim}{\min}|\widetilde{q}_I - \widetilde{q}_J|$.
We thank one of the referees for providing helpful comments tied to this issue.
\end{remark}	
\begin{remark}\label{rem:fail.diag}
The eigen\emph{vectors} of the symmetric matrix $(\upkappa_{IJ}^{(\infty)}(x))_{I,J=1,\cdots,\mydim}$
might fail to be continuous in $x$, for example in the case
where the $q_I^{(\infty)}(x)$'s have contact points of infinite order; 
see \cite[Chapter 2, Example 5.3]{Kato}. 
\end{remark}

We now state our main theorem for polarized $U(1)$-symmetric solutions.
\begin{theorem}[Precise version of stable Big Bang formation for
polarized $U(1)$-symmetric Einstein-vacuum solutions in $1+3$ dimensions]
	\label{thm:precise.U1}
Let $\widetilde{\bf g}
=
-dt \otimes dt
+
t^{2\widetilde{q}_1}dx^1 \otimes dx^1
+
t^{2\widetilde{q}_2}dx^2 \otimes dx^2
+
t^{2\widetilde{q}_3}dx^3 \otimes dx^3$ 
be a ``background'' Kasner solution on $(0,\infty)\times\mathbb{T}^3$
with Kasner exponents satisfying:
\begin{align}\label{sumpi2}
\sum_{I=1}^3\widetilde{q}_I=\sum_{I=1}^3\widetilde{q}_I^2=1,\qquad \max_{I=1,2,3}\widetilde{q}_I<1.
\end{align}
Let $\widetilde{k}_{IJ} = - \widetilde{q}_{\underline{I}} \updelta_{\underline{I}J} t^{-1}$ 
be the components of the second fundamental form of $\Sigma_t$ relative to the Kasner metric,
with respect to the background orthonormal frame vectors
$\widetilde{e}_I = t^{-\widetilde{q}_{\underline{I}}} \partial_{\underline{I}}$.
Let $(\Sigma_1=\mathbb{T}^3,\mathring{g},\mathring{k})$ be polarized $U(1)$-symmetric initial data 
(see Sects.\,\ref{SS:CAUCHYPROBLEM} and \ref{subsubsec:U1.data}) for the Einstein-vacuum equations 
verifying the constraint equations \eqref{eq:hamconst}--\eqref{eq:momconst} 
(with $\mathring{\psi} = \mathring{\phi} \equiv 0$) 
and the CMC condition $\mathrm{tr} \mathring{k}=-1$ (see Remark \ref{rem:CMC}),
and such that $\mathring{X}:=\partial_3$ is the hypersurface-orthogonal Killing vectorfield of the data.
Let $\mathring{\upalpha}$ be the norm of the perturbation of the initial data away from the Kasner data,
as defined in \eqref{E:INITIALCONTROLLINGNORM} 
(where the scalar field data on RHS~\eqref{E:INITIALCONTROLLINGNORM} are vanishing by assumption).
Let $\lbrace e_I \rbrace_{I=1,2,3}$ be the $g$-orthonormal frame 
obtained by constructing the initial orthonormal frame as in
Sect.\,\ref{SS:CONSTRUCTIONOFTHEINTIALORTHONORMALFRAME}
and then using Lemma~\ref{lem:U1}
to ensure that throughout the evolution,
the corresponding frame solution to the Fermi--Walker transport equation \eqref{frame.prop}
verifies $e_3 = \frac{\partial_3}{\sqrt{g_{33}}}$
and $\mathcal{L}_{\partial_3} e_I = 0$ for $I=1,2,3$, where $\mathcal{L}$ denotes Lie differentiation.
Assume that the parameters $N$, $N_0$, $\blowupexp$, $q$, $\upsigma$, $\mathring{\upalpha}$ satisfy the assumptions of Theorem 
\ref{thm:precise}, where in polarized $U(1)$-symmetry, $q,\upsigma$ are fixed constants satisfying:
\begin{align}\label{q.U1}
0<2 \upsigma< 2 \upsigma+\max\{|\widetilde{q}_1|,|\widetilde{q}_2|,|\widetilde{q}_3|\}<q<1- 2 \upsigma.
\end{align}
Then the conclusions stated in Theorem~\ref{thm:precise} hold for the 
solution to the reduced equations of Proposition~\ref{P:redeq} 
(which also yields a solution to the Einstein-vacuum equations, i.e., \eqref{EE} with $\psi\equiv0$) 
that arises from the prescribed polarized $U(1)$-symmetric initial data $(\mathring{g},\mathring{k})$.
Moreover, the solution is polarized $U(1)$-symmetric
in the sense that relative to the transported spatial coordinates, 
$\partial_3$ is a hypersurface-orthogonal Killing vectorfield of the spacetime metric ${\bf g}$,
and ${\bf g}$ is of the form \eqref{polarizedmetric}.
\end{theorem}
\subsection{Limiting functions and Kasner-like behavior}
\label{SS:LIMITINGFUNCTIONS}
In the next proposition, we show that the scalar functions
$\left\lbrace tk_{IJ}(t,\cdot) \right\rbrace_{I,J=1,\cdots,\mydim}$ 
and $t \partial_t \psi(t,\cdot)$ have limits in $W^{N_0+1}(\mathbb{T}^d)$,
as $t \downarrow 0$. Moreover, the limiting fields obey a limiting Hamiltonian constraint equation
and exhibit other ``Kasner-like'' properties.

\begin{proposition}[Asymptotic, Kasner-like limits]
\label{prop:tk}
Under the assumptions and conclusions of Proposition~\ref{P:EXISTENCEONHALFSLAB}, 
%perhaps with $\mathring{\upepsilon}$ chosen to be smaller,
the scalar component functions $\left \lbrace t k_{IJ}(t,x) \right\rbrace_{I,J = 1,\cdots,\mydim}$ 
of the normalized second fundamental form of $\Sigma_t$
with respect to the $g$-orthonormal 
frame $\lbrace e_I(t,x) \rbrace_{I=1,\cdots,\mydim}$
and the normalized scalar field velocity 
$t\partial_t \psi(t,x)$ have continuous limits in $W^{N_0+1,\infty}(\mathbb{T}^{\mydim})$, 
denoted respectively by
$\left\lbrace \upkappa_{IJ}^{(\infty)}(x) \right\rbrace_{I,J = 1,\cdots,\mydim}$ 
and $B^{(\infty)}(x)$, as $t\downarrow 0$. 
Moreover, the following estimates hold:
\begin{subequations}
\begin{align} 
		\sum_{I,J=1,\cdots,\mydim}
		\| t k_{IJ}(t,\cdot) 
			-
			\upkappa_{IJ}^{(\infty)}
		\|_{W^{N_0+1,\infty}(\mathbb{T}^{\mydim})}
	& \lesssim \mathring{\upepsilon} t^{\upsigma},
	&
	\| t \partial_t \psi(t,\cdot) - B^{(\infty)} \|_{W^{N_0+1,\infty}(\mathbb{T}^{\mydim})}
	& \lesssim \mathring{\upepsilon} t^{\upsigma},
		\label{E:CONVERGENCERATESOFLIMITSATSINGULARITY} \\
	\sum_{I,J=1,\cdots,\mydim}
	\| \upkappa_{IJ}^{(\infty)} 
		+ 
		\widetilde{q}_{\underline{I}} \updelta_{\underline{I}J} \|_{W^{N_0+1,\infty}(\mathbb{T}^{\mydim})}
	& \lesssim \mathring{\upepsilon},
	&
	\| B^{(\infty)} - \widetilde{B} \|_{W^{N_0+1,\infty}(\mathbb{T}^{\mydim})}
	& \lesssim \mathring{\upepsilon}.
		\label{E:LIMITINGSECONDFUNDAMENTALFORMCLOSETODATA}
\end{align}
\end{subequations}

In addition, for each $x \in \mathbb{T}^{\mydim}$,
the symmetric $\mydim \times \mydim$ matrix
$(-\upkappa^{(\infty)}_{IJ}(x))_{I,J=1,\cdots,\mydim}$ 
has $\mydim$ (possibly repeated) eigenvalues 
$q_I^{(\infty)}(x)$ -- which are the ``final'' Kasner exponents of the perturbed spacetime --
that can be ordered such that $q_1^{(\infty)},\cdots,q_{\mydim}^{(\infty)} \in C^{0,1}(\mathbb{T}^{\mydim})$
(see the discussion surrounding \eqref{qi:C01} for the definition of this function space and the norm)
and such that the following estimate holds:
\begin{align}\label{qi-qi.tilde}
\sum_{I=1,\cdots,\mydim}
\|q_I^{(\infty)} - \widetilde{q}_I\|_{C^{0,1}(\mathbb{T}^{\mydim})}
& \lesssim \mathring{\upepsilon}.
\end{align}

Moreover, the $\left\lbrace q_I^{(\infty)}(x) \right\rbrace_{I=1,\cdots,\mydim}$ and $B^{(\infty)}(x)$ 
satisfy the following pointwise algebraic relations:
\begin{align}\label{qi,B.rel}
\sum_{I=1}^{\mydim} q_I^{(\infty)}(x)=1,\qquad \sum_{I=1}^{\mydim}\left[q_I^{(\infty)}(x) \right]^2=1-\left[B^{(\infty)}(x) \right]^2.
\end{align}
\end{proposition}
\begin{proof}
Let $\lbrace t_n \rbrace_{n=1}^{\infty} \subset (0,1]$ be a decreasing sequence of times
such that $\lim_{n \to \infty} t_n = 0$.
A straightforward modification of the proof of \eqref{SUBSTEP2MAINSTEPlow.est},
based on the evolution equation \eqref{E:COMMUTEDKEQUATION}
and the estimate \eqref{overall.est.[0,1]}, 
yields that when $0 < a < b \leq 1$, we have
$
\| a k_{IJ}(a,\cdot) - b k_{IJ}(b,\cdot) \|_{W^{N_0+1,\infty}(\mathbb{T}^{\mydim})}
\lesssim \mathring{\upepsilon} \int_a^b s^{-1 + \upsigma} \mathbb{D}(s) \, ds
\lesssim \mathring{\upepsilon} b^{\upsigma}
$.
Hence, $\lbrace t_n k_{IJ}(t_n,\cdot) \rbrace_{n=1}^{\infty}$ is a Cauchy sequence
in $W^{N_0+1,\infty}(\mathbb{T}^{\mydim})$, and its limit, which we denote by $\upkappa_{IJ}^{(\infty)}$,
verifies
$
\| \upkappa_{IJ}^{(\infty)} - t k_{IJ}(t,\cdot) \|_{W^{N_0+1,\infty}(\mathbb{T}^{\mydim})}
\lesssim \mathring{\upepsilon} t^{\upsigma}
$
for $t \in (0,1]$.
In particular,
$
\| \upkappa_{IJ}^{(\infty)} - k_{IJ}(1,\cdot) \|_{W^{N_0+1,\infty}(\mathbb{T}^{\mydim})}
\lesssim \mathring{\upepsilon}
$.
Since 
$
\|  k_{IJ}(1,\cdot) + \widetilde{q}_{\underline{I}} \updelta_{\underline{I}J} \|_{W^{N_0+1,\infty}(\mathbb{T}^{\mydim})}
=
\|  k_{IJ}(1,\cdot) - \widetilde{k}_{IJ} \|_{W^{N_0+1,\infty}(\mathbb{T}^{\mydim})}
\lesssim \mathring{\upepsilon}
$,
we infer from the triangle inequality that
$
\| \upkappa_{IJ}^{(\infty)} + \widetilde{q}_{\underline{I}} \updelta_{\underline{I}J} \|_{W^{N_0+1,\infty}(\mathbb{T}^{\mydim})}
\lesssim \mathring{\upepsilon}
$. We have therefore proved 
\eqref{E:CONVERGENCERATESOFLIMITSATSINGULARITY} 
and
\eqref{E:LIMITINGSECONDFUNDAMENTALFORMCLOSETODATA} for $\upkappa_{IJ}^{(\infty)}$.
Moreover, the symmetric matrix $(\upkappa_{IJ}^{(\infty)}(x))_{I,J=1,\cdots,\mydim}$ 
is $\mathcal{O}(\mathring{\upepsilon})$-close to the diagonal matrix 
$\mbox{\upshape diag}(-\widetilde{q}_1,\cdots,-\widetilde{q}_{\mydim})$. 
Thus, at each fixed $x$,
$- (\upkappa_{IJ}^{(\infty)}(x))_{I,J=1,\cdots,\mydim}$ is diagonalizable,
and by standard perturbation theory
(see \cite[Equation~(3.6) in Chapter IV]{SS}),
its (possibly repeated) eigenvalues $q_I(x)$ can be ordered such that
$q_1^{(\infty)}(x),\cdots,q_{\mydim}^{(\infty)}(x)\in C^{0,1}(\mathbb{T}^{\mydim})$
and such that the following pointwise estimate holds for all $x,y \in \mathbb{T}^{\mydim}$:
\begin{align} \label{pert.qI.est}
\sum_{I=1,\cdots,\mydim}
	\left|q_I^{(\infty)}(x) - \widetilde{q}_I \right|
& \lesssim
\max_{I,J=1,\cdots,\mydim}
\left|
	\upkappa_{IJ}^{(\infty)}(x) 
	+ 
	\widetilde{q}_{\underline{I}}\updelta_{\underline{I}J}
\right|,
\\
\sum_{I=1,\cdots,\mydim}
	\left|q_I^{(\infty)}(x) - q_I^{(\infty)}(y) \right|
& \lesssim
\max_{I,J=1,\cdots,\mydim}
\left|
	\upkappa_{IJ}^{(\infty)}(x) 
	- 
		\upkappa_{IJ}^{(\infty)}(y) 
\right|.
\label{PreHolder.qI.est}
\end{align}
From \eqref{pert.qI.est}--\eqref{PreHolder.qI.est}, 
the standard inequality 
$\left|
	\upkappa_{IJ}^{(\infty)}(x) 
	- 
		\upkappa_{IJ}^{(\infty)}(y) \right|
		\lesssim 
		\| \upkappa_{IJ}^{(\infty)} \|_{\dot{W}^{1,\infty}(\mathbb{T}^\mydim)}d(x,y)$
(where $d(x,y)$ is the Euclidean distance between $x$ and $y$ in $\mathbb{T}^{\mydim}$),
and the first estimate in \eqref{E:LIMITINGSECONDFUNDAMENTALFORMCLOSETODATA},
we conclude \eqref{qi-qi.tilde}.

The convergence results and estimates for $t \partial_t \psi$
can be proved in a similar fashion by making
straightforward modifications to the proof of \eqref{SUBSTEP4MAINSTEPlow.est}.

To derive the first equation in \eqref{qi,B.rel}, we employ the CMC condition \eqref{trk} 
and the estimate 
$
\| \upkappa_{IJ}^{(\infty)} - t k_{IJ}(t,\cdot) \|_{W^{N_0+1,\infty}(\mathbb{T}^{\mydim})}
\lesssim \mathring{\upepsilon} t^{\upsigma}
$
proved above to deduce the following pointwise estimate:
\begin{align}\label{qi,B.rel2}
-1=t\text{tr}k(t,x)=\mathcal{O}(\mathring{\upepsilon} t^{\upsigma})+\text{tr}\upkappa^{(\infty)}(x)
=
\mathcal{O}(\mathring{\upepsilon} t^{\upsigma})-\sum_{I=1}^{\mydim} q_I^{(\infty)}(x),
\end{align}
where to obtain the last equality, we used that the trace of the $\mydim \times \mydim$
matrix $\left(\upkappa_{IJ}^{(\infty)} \right)_{I,J=1,\cdots,\mydim}$ is the sum of its eigenvalues 
$-q_1^{(\infty)},\cdots,-q_{\mydim}^{(\infty)}$. Taking the limit $t\downarrow 0$
on RHS~\eqref{qi,B.rel2}, we obtain the desired equation.

To derive the second equation in \eqref{qi,B.rel}, 
we multiply the Hamiltonian constraint \eqref{Hamconst} by $t^2$
and use Definition~\ref{D:SOLUTIONNORMS},
the estimate \eqref{overall.est.[0,1]},
the inequalities in \eqref{sigma,q},
and the estimates
$
\| \upkappa_{IJ}^{(\infty)} - t k_{IJ}(t,\cdot) \|_{W^{N_0+1,\infty}(\mathbb{T}^{\mydim})}
\lesssim \mathring{\upepsilon} t^{\upsigma}
$
and
$
\| B^{(\infty)} - t \partial_t \psi(t,\cdot) \|_{W^{N_0+1,\infty}(\mathbb{T}^{\mydim})}
\lesssim \mathring{\upepsilon} t^{\upsigma}
$
noted above to deduce the following pointwise estimate:
\begin{align} \label{qi,B.rel3}
\begin{split}
1  
& 
=
t^2 k_{CD}(t,x) k_{CD}(t,x) 
-
t^2
\left\lbrace
	2e_C \upgamma_{DDC}
	-
	\upgamma_{CDE}\upgamma_{EDC}
	-
	\upgamma_{CCD}\upgamma_{EED}
\right\rbrace
(t,x)
	\\
&	\ \
	+
	n^{-2} \left[t\partial_t\psi(t,x) \right]^2
	+
	t^2 \left[e_C \psi(t,x) \right] e_C \psi(t,x)
	\\
& =  
\upkappa_{CD}^{(\infty)}(x) \upkappa_{CD}^{(\infty)}(x)
+ 
\left[B^{(\infty)}(x) \right]^2 + \mathcal{O}(\mathring{\upepsilon} t^{\upsigma})
=
\sum_{I=1}^{\mydim}\left[q_I^{(\infty)}(x) \right]^2 + \left[B^{(\infty)}(x) \right]^2 
+
\mathcal{O}(\mathring{\upepsilon} t^{\upsigma}),
\end{split}
\end{align}
where to obtain the last equality, we used the fact that
$\upkappa_{CD}^{(\infty)}(x) \upkappa_{CD}^{(\infty)}(x)$ 
is equal to the sum of the squares of eigenvalues of the matrix $\left(\upkappa_{IJ}^{(\infty)} \right)_{I,J=1,\cdots,\mydim}$.
The desired second equation in \eqref{qi,B.rel} 
now follows from taking the limit $t\downarrow 0$ on RHS~\eqref{qi,B.rel3}.
This completes the proof of the proposition.
\end{proof} 
\subsection{Monotonic blowup of curvature}
In the following proposition, 
we show that the Kretschmann scalars of the solutions studied in the present paper blow up like $t^{-4}$. 
\begin{proposition}[Monotonic blow up of the Kretschmann scalar]
\label{prop:curv}
Under the assumptions and conclusions of Proposition~\ref{P:EXISTENCEONHALFSLAB}, 
%perhaps with $\mathring{\upepsilon}$ chosen to be smaller,
the Kretschmann scalar
$
{\bf Riem}^{\alpha\mu\beta\nu}{\bf Riem}_{\alpha\mu\beta\nu}
$
obeys the following pointwise estimate for $(t,x) \in (0,1] \times \mathbb{T}^{\mydim}$,
where the functions $\left\lbrace q_I^{(\infty)}(x) \right\rbrace_{I=1,\cdots,\mydim}$ are 
as in the conclusions of Proposition~\ref{prop:tk}: 
\begin{align} \label{Krets.est}
\begin{split}
{\bf Riem}^{\alpha\mu\beta\nu}{\bf Riem}_{\alpha\mu\beta\nu}(t,x)
%&\,4t^{-4}\bigg[\sum_{I=1}^{\mydim}(q_I^2-q_I)^2+\sum_{1\leq I<J \leq \mydim}q_I^2q_J^2\bigg]+\mathcal{O}(\mathring{\upepsilon} t^{-4+\upsigma})\\
& = 4 t^{-4} \left\lbrace
	\sum_{I=1}^{\mydim}\left[(q_I^{(\infty)}(x))^2-q_I^{(\infty)}(x) \right]^2+\sum_{1\leq I<J \leq \mydim}(q_I^{(\infty)}(x))^2 
	(q_J^{(\infty)}(x))^2 
\right\rbrace
		\\
& \ \
	+
	\mathcal{O}(\mathring{\upepsilon} t^{-4 + \upsigma})
	 \\
& = 4t^{-4}
\left\lbrace
	\sum_{I=1}^{\mydim} \left[\widetilde{q}_I^2-\widetilde{q}_I \right]^2
	+
	\sum_{1\leq I<J \leq \mydim}\widetilde{q}_I^2\widetilde{q}_J^2 \right\rbrace
+
\mathcal{O}(\mathring{\upepsilon} t^{-4}).
\end{split}
\end{align}
\end{proposition}
\begin{proof}
We first use the standard symmetries and antisymmetries of the Riemann curvature
tensor of ${\bf{g}}$ to compute the following identity for its Kretschmann scalar:
\begin{align}\label{Krets}
\begin{split}
{\bf Riem}^{\alpha\mu\beta\nu}{\bf Riem}_{\alpha\mu\beta\nu}
& =
{\bf Riem}(e_A,e_I,e_B,e_J){\bf Riem}(e_A,e_I,e_B,e_J)
	\\
& 
\ \
+
4{\bf Riem}(e_0,e_I,e_0,e_J){\bf Riem}(e_0,e_I,e_0,e_J)
	\\
& \ \
-
4
{\bf Riem}(e_A,e_I,e_0,e_J){\bf Riem}(e_A,e_I,e_0,e_J).
\end{split}
\end{align}
Next, using Gauss' equation \eqref{E:GAUSS},
\eqref{E:FRAMEORTHONORMALITY},
\eqref{E:DEIEJFRAMEEXPANSION},
Definition~\ref{D:SOLUTIONNORMS},
the estimate \eqref{overall.est.[0,1]},
the inequalities in \eqref{sigma,q},
and the convergence estimate \eqref{E:CONVERGENCERATESOFLIMITSATSINGULARITY} for $t k_{IJ}$,
we derive the following pointwise estimate:
\begin{align}
\begin{split} \label{GaussDOMINANT}
t^2 {\bf Riem}(e_A,e_I,e_B,e_J)
& = (tk_{IJ})(tk_{AB})-(tk_{AJ})(tk_{BI})
		+
		t^2 e_A \upgamma_{IJB} 
		-
		t^2 e_I \upgamma_{AJB}
		\\
& 	\	\
		-
		t^2 \upgamma_{AIC} \upgamma_{CJB}
		-
		t^2
		\upgamma_{IJC} \upgamma_{ABC}
		+
		t^2 \upgamma_{IAC} \upgamma_{CJB}
		+
		t^2
		\upgamma_{AJC} \upgamma_{IBC}
			\\
& = (tk_{IJ})(tk_{AB})-(tk_{AJ})(tk_{BI})
	+
	\mathcal{O}(\mathring{\upepsilon}) t^{\upsigma}
	= 
	\upkappa_{IJ}^{(\infty)} \upkappa_{AB}^{(\infty)} - \upkappa_{AJ}^{(\infty)} \upkappa_{BI}^{(\infty)}
	+
	\mathcal{O}(\mathring{\upepsilon}) t^{\upsigma}.
\end{split}
\end{align}
Similarly, with the help of the Codazzi equations \eqref{E:CODAZZI} and \eqref{E:DEIEJFRAMEEXPANSION},
we compute the following pointwise estimate:
\begin{align}\label{CodDOMINANT}
\begin{split}
t^2{\bf Riem}(e_A,e_I,e_0,e_J)
& =  
t^2
\left\lbrace
e_A^c \partial_c k_{IJ}
-
e_I^c\partial_c k_{AJ}
-
\upgamma_{AIB}k_{BJ}
-
\upgamma_{AJB}k_{IB}
+
\upgamma_{IAB}k_{BJ}
+
\upgamma_{IJB}k_{AB}
\right\rbrace
	\\
& =  
\mathcal{O}(\mathring{\upepsilon}) t^{\upsigma}.
\end{split}
\end{align}
Similarly, with the help of \eqref{R0i0j} and the evolution equation \eqref{dt.k},
we deduce the following pointwise estimate:
\begin{align}\label{R0i0j3DOMINANT}
\begin{split}
t^2{\bf Riem}(e_0,e_I,e_0,e_J)
& = 
- 
t k_{IJ}
-
(tk_{IC})(tk_{CJ})
	\\
& \ \
+
t^2 
\left\lbrace
	e_D^c\partial_c\upgamma_{IJD}
	-
	e_I^c\partial_c\upgamma_{DJD}
	-
	\upgamma_{DIC}\upgamma_{CJD}-\upgamma_{DDC}\upgamma_{IJC}
	-
	(e_I \psi) e_J \psi
\right\rbrace
	 \\
& = - 
		\upkappa_{IJ}^{(\infty)}
		- 
		\upkappa_{IC}^{(\infty)} \upkappa_{CJ}^{(\infty)}
		+
		\mathcal{O}(\mathring{\upepsilon}) t^{\upsigma}.
\end{split}
\end{align}
Inserting \eqref{GaussDOMINANT}--\eqref{R0i0j3DOMINANT} into \eqref{Krets}, 
we deduce the following pointwise estimate:
\begin{align}\label{Krets.est2}
\begin{split}
&
{\bf Riem}^{\alpha\mu\beta\nu}{\bf Riem}_{\alpha\mu\beta\nu}
	\\
&
=
t^{-4}
\Big\lbrace
	\left(\upkappa_{IJ}^{(\infty)}\upkappa_{AB}^{(\infty)}-\upkappa_{AJ}^{(\infty)}\upkappa_{BI}^{(\infty)} \right)
	\left(\upkappa_{IJ}^{(\infty)}\upkappa_{AB}^{(\infty)}-\upkappa_{AJ}^{(\infty)}\upkappa_{BI}^{(\infty)} \right)
		\\
& \ \ \ \ \ \ \ \ \ \ \ \ \ \ \ \
	+
	4 \left(\upkappa_{IJ}^{(\infty)}+\upkappa_{IB}^{(\infty)} \upkappa_{BJ}^{(\infty)} \right)
	\left(\upkappa_{IJ}^{(\infty)}+\upkappa_{IC}^{(\infty)} \upkappa_{CJ}^{(\infty)} \right)
\Big\rbrace
	\\
& \ \
+
\mathcal{O}(\mathring{\upepsilon} t^{-4+\upsigma}).
\end{split}
\end{align}
Consider now the symmetric matrix $K :=\left(\upkappa_{IJ}^{(\infty)} \right)_{I,J=1,\cdots,\mydim}$, 
whose eigenvalues are $-q_I^{(\infty)},\cdots,-q_{\mydim}^{(\infty)}$. 
Using that
for $m\in\mathbb{N}$, we have
$\text{tr}(K^m)=\sum_{I=1}^{\mydim} \left[-q_I^{(\infty)} \right]^m$,
we rewrite the expression in braces on RHS~\eqref{Krets.est2} as follows:
\begin{align}\label{Krets.est3}
\begin{split}
&
\left(\upkappa_{IJ}^{(\infty)}\upkappa_{AB}^{(\infty)}-\upkappa_{AJ}^{(\infty)}\upkappa_{BI}^{(\infty)} \right)
\left(\upkappa_{IJ}^{(\infty)}\upkappa_{AB}^{(\infty)}-\upkappa_{AJ}^{(\infty)}\upkappa_{BI}^{(\infty)} \right)
+
4
\left(\upkappa_{IJ}^{(\infty)}+\upkappa_{IB}^{(\infty)}\upkappa_{BJ}^{(\infty)} \right)
\left(\upkappa_{IJ}^{(\infty)}+\upkappa_{IC}^{(\infty)}\upkappa_{CJ}^{(\infty)} \right)
	\\
& = 
	2\left[\text{tr}(KK) \right]^2
	+
	4\text{tr}(KK)
	+
	8\text{tr}(KKK)
	+
	2\text{tr}(KKKK)
		\\
& = 
4\left\lbrace
	\sum_{I=1}^{\mydim} \left[(q_I^{(\infty)})^2-q_I^{(\infty)} \right]^2
	+
	\sum_{1\leq I<J \leq \mydim}(q_I^{(\infty)})^2 (q_J^{(\infty)})^2
	\right\rbrace.
\end{split}
\end{align}
Combining \eqref{Krets.est2}--\eqref{Krets.est3}, 
we arrive at the first equality stated in \eqref{Krets.est}.
To prove the second equality stated in \eqref{Krets.est},
we simply use \eqref{E:LIMITINGSECONDFUNDAMENTALFORMCLOSETODATA}
to replace all factors of 
$\upkappa_{IJ}^{(\infty)}$ on RHS~\eqref{Krets.est2} 
with $-\widetilde{q}_{\underline{I}} \updelta_{\underline{I}J}$
up to $\mathcal{O}(\mathring{\upepsilon})$ error terms
(which, in view of the factor of $t^{-4}$ in front of the braces in \eqref{Krets.est2},
leads to the error term $\mathcal{O}(\mathring{\upepsilon} t^{-4})$
on RHS~\eqref{Krets.est}).
This completes the proof of the proposition.
\end{proof}

\subsection{Proof of Theorems~\ref{thm:precise} and \ref{thm:precise.U1}}
\label{SS:PROOFOFMAINTHEOREMS}
We first prove Theorem~\ref{thm:precise}.
The conclusions regarding existence and norm estimates, generalized Kasner behavior, and blow up of curvature follow from Propositions~\ref{P:EXISTENCEONHALFSLAB}, \ref{prop:tk}, and \ref{prop:curv},
and the estimate \eqref{E:DYNAMICINITIALNORMCONTROLLEDBYGEOMETRICDATANORM}. 
The $C^2$-inextendibility is a direct consequence of the curvature-blowup. 

To prove Theorem~\ref{thm:precise.U1}, we simply note that
the polarized $U(1)$-symmetric solutions satisfy the same estimates
as the solutions from Theorem~\ref{thm:precise}.
Hence, the same arguments used to prove Theorem~\ref{thm:precise}
also yield Theorem~\ref{thm:precise.U1}.
Finally, we note that the symmetry properties of polarized $U(1)$-symmetric solutions relative
to CMC-transported spatial coordinates stated in the conclusions of the theorem
are provided by Lemma~\ref{L:PROPOFSYM}.

\hfill $\qed$

\end{document}